\documentclass[11pt,a4paper]{amsart}

\usepackage{mathtools}
\usepackage{graphicx}
\usepackage{amsthm}
\usepackage{amssymb}
\usepackage{caption}
\usepackage{hyperref}
\usepackage{nicematrix}
\usepackage{enumitem}
\usepackage{booktabs}
\usepackage[dvipsnames]{xcolor}
\usepackage[left=30mm, right=30mm]{geometry}

\usepackage{tikz}
\usepackage{tikz-cd}
\usetikzlibrary{calc}
\usetikzlibrary{decorations.pathmorphing}
\usetikzlibrary{decorations.pathreplacing}

\theoremstyle{definition}
\newtheorem{definition}{Definition}[section]

\newtheorem{example}[definition]{Example}

\newtheorem{problem}[definition]{Problem}

\theoremstyle{plain}
\newtheorem{corollary}[definition]{Corollary}
\newtheorem{lemma}[definition]{Lemma}
\newtheorem{proposition}[definition]{Proposition}
\newtheorem{theorem}[definition]{Theorem}

\numberwithin{equation}{section}

\newcommand*\defterm{\emph}

\newcommand*\diam[1]{\mathrm{diam}(#1)}
\newcommand*\dist[3]{d_{#1}(#2,#3)}
\newcommand*\girth[1]{\mathrm{girth}(#1)}
\newcommand*\knitdegree[1]{\mathrm{kd}(#1)}

\newcommand*\cliquenumber[2][]{\omega\parens[#1]{#2}}
\newcommand*\chromaticnumber[2][]{\chi\parens[#1]{#2}}
\newcommand*\degree[2]{\mathrm{deg}_{#1}#2}
\newcommand*\maxdegree[1]{\Delta(#1)}

\newcommand*\graphjointwo{\mathbin{\nabla}}

\newcommand*\commgraph[1]{\mathcal{G}(#1)}
\newcommand*\extendedcommgraph[1]{\mathcal{G}^*(#1)}
\newcommand*\generalcommgraph[2]{\mathcal{G}(#1,#2)}

\newcommand*\centre[1]{Z(#1)}

\DeclarePairedDelimiter{\abs}{\lvert}{\rvert}

\DeclarePairedDelimiter{\parens}{\lparen}{\rparen}
\DeclarePairedDelimiter{\bracks}{\lbrack}{\rbrack}
\DeclarePairedDelimiter{\braces}{\{}{\}}

\DeclarePairedDelimiter{\set}{\{}{\}}
\DeclarePairedDelimiterX{\gset}[2]{\{}{\}}{\,#1:#2\,}

\newcommand*\Xn{\set{1,\ldots,n}}
\newcommand*\X[1]{\set{1,\ldots,#1}}

\newcommand*\Rees[4]{\mathcal{M}[#1; \allowbreak #2,\allowbreak #3; \allowbreak #4]}
\newcommand*\Reeszero[4]{\mathcal{M}_0[#1; \allowbreak #2,\allowbreak #3; \allowbreak #4]}
\newcommand*\simplifiedgraph[3]{\mathcal{G}(#1,\allowbreak #2,\allowbreak #3)}

\newcommand*\submatrix{$\leftrightarrow$-submatrix} 
\newcommand*\submatrices{$\leftrightarrow$-submatrices}
\newcommand*\equivalent{$\leftrightarrow$-equivalent}

\newcommand*\zeroindex[2]{\zeta(#2,#1)}
\newcommand*\zeroclosuresubmatrix[2]{$0$-closure $(#2,#1)$-submatrix of $P$}
\newcommand*\zeroclosuresub{$0$-closure submatrix of $P$}

\newcommand*\timeszero[1]{\overline{#1}}
\newcommand*\Psubmatrix[3]{#1[#3|#2]}

\makeatletter
\newcommand*{\sizeddelimiter}[2]{\bBigg@{#1}#2}
\makeatother

\newcommand*\blocktimes[3]{
	\begin{tikzpicture}[baseline=#3]
		\pgfmathsetlengthmacro{\d}{#1}
		\pgfmathsetlengthmacro{\p}{#2}
		\begin{scope}[node font=\scriptsize,inner sep=0pt]
			\node (ne) at (\p,\d) {\(\times\)};
			\node (nw) at (-\p,\d) {\(\times\)};
			\node (sw) at (-\p,-\d) {\(\times\)};
			\node (se) at (\p,-\d) {\(\times\)};
		\end{scope}
		\begin{scope}[dash pattern=on 0pt off 1.83pt,line cap=round, thick]
			\draw (sw) -- (se);
			\draw (sw) -- (nw);
			\draw (se) -- (ne);
			\draw (nw) -- (ne);
			\draw (nw) -- (se);
		\end{scope}
	\end{tikzpicture}
}

\newcommand*\rowtimes[2]{
	\begin{tikzpicture}[baseline=#2]
		\pgfmathsetlengthmacro{\d}{#1}
		\begin{scope}[node font=\scriptsize,inner sep=0pt]
			\node (ne) at (\d,\d) {\(\times\)};
			\node (nw) at (-\d,\d) {\(\times\)};
		\end{scope}
		\begin{scope}[dash pattern=on 0pt off 2.3pt,line cap=round, thick]
			\draw (nw) -- (ne);
		\end{scope}
	\end{tikzpicture}
}

\newcommand\columntimes{
	\begin{tikzpicture}[baseline=0em]
		\pgfmathsetlengthmacro{\d}{0.7em}
		\begin{scope}[node font=\scriptsize,inner sep=0pt]
			\node (ne) at (\d,\d) {\(\times\)};
			\node (se) at (\d,-\d) {\(\times\)};
		\end{scope}
		\begin{scope}[dash pattern=on 0pt off 1.95pt,line cap=round, thick]
			\draw (se) -- (ne);
		\end{scope}
	\end{tikzpicture}
}

\setlist[description]{font=\normalfont\itshape}

\tikzset{
	vertex/.style={
		circle,
		minimum size=1.8mm,
		fill,
		inner sep=0,
		outer sep=0,
	},
	edge/.style={
		line width=.25mm,
	}
}

\allowdisplaybreaks

\begin{document}

\title{Commuting graphs of completely $\mathbf{0}$-simple semigroups}

\author{Tânia Paulista}
\address[T. Paulista]{%
Center for Mathematics and Applications (NOVA Math) \& Department of Mathematics\\
NOVA School of Science and Technology\\
NOVA University of Lisbon\\
2829--516 Caparica\\
Portugal
}
\email{%
tpl.paulista@gmail.com
}
\thanks{This work is funded by national funds through the FCT -- Fundação para a Ciência e a Tecnologia, I.P., under the scope of the projects UID/297/2025 and UID/PRR/297/2025 (Center for Mathematics and Applications - NOVA Math). The author is also funded by national funds through the FCT -- Fundação para a Ciência e a Tecnologia, I.P., under the scope of the studentship 2021.07002.BD}

\thanks{The author is thankful to her supervisors António Malheiro and Alan J. Cain for all the support, encouragement and guidance; and also for reading a draft of this paper}


\begin{abstract}
	The aim of this paper is to study commuting graphs of completely $0$-simple semigroups, using the characterization of these semigroups as $0$-Rees matrix semigroups over a groups. We establish a method to decide whether the commuting graph of this semigroup construction is connected or not. If it is not connected, we also supply a way to identify the connected components of the commuting graph. We show how to obtain the diameter of the commuting graph (when it is connected) and the diameters of the connected components of the commuting graph (when it is not connected). Moreover, we obtain the clique number and girth of the commuting graph of such a semigroup, as well as two upper bounds (either of which can be the best in different situations) for its chromatic number. We also determine the knit degree of such a semigroup.
	
	Finally, we use the results regarding the properties of the commuting graph of a $0$-Rees matrix semigroup over a group to determine the set of possible values for the diameter, clique number, girth, chromatic number and knit degree of the commuting graph of a completely $0$-simple semigroup.
\end{abstract}

\maketitle

\section{Introduction}

The commuting graph of a semigroup $S$ on a non-empty subset $T$ of $S$ is the simple graph whose vertex set is $T$, and where two distinct vertices $x,y\in T$ are adjacent if and only if $xy=yx$. We denote this graph by $\generalcommgraph{S}{T}$. Over the years there have been various choices for $T$: the most common ones seem to be $S\setminus\centre{S}$ and $S$ (where $\centre{S}=\gset{x\in S}{xy=yx \text{ for all } y\in S}$ is the center of $S$). If $T=S\setminus\centre{S}$, then we call this graph the commuting graph of $S$ and, if $T=S$, we call this graph the extended commuting graph of $S$.

Commuting graphs seem to have been used for the first time in 1955 by Brauer and Fowler \cite{First_paper_commuting_graphs}, who showed that, if $G$ is a finite group of even order that contains more than one conjugacy class of involutions, then the distance between two involutions in $\generalcommgraph{G}{G\setminus\set{1_G}}$ is at most $3$.

Ever since commuting graphs were first introduced, they have been studied from different perspectives. For instance, several authors have determined various properties of the commuting graphs of important groups and semigroups. Among these groups/semigroups we highlight the symmetric group \cite{Commuting_graph_I_X, Symmetric_group, Diameter_commuting_graph_symmetric_group, Commuting_graph_symmetric_alternating_groups}, the alternating group \cite{Commuting_graph_symmetric_alternating_groups, Alternating_group}, the transformation semigroup \cite{Commuting_graph_T_X} and the symmetric inverse semigroup \cite{Commuting_graph_I_X}.

Several authors contributed to establishing that, if $G$ is a finite non-abelian simple group and $H$ is a group, then their commuting graphs are isomorphic if and only if $H\simeq G$ \cite{Isomorphic_commuting_graphs_simple_group_alternating, Isomorphic_commuting_graphs_simple_group_sporadic, Isomorphic_commuting_graphs_simple_group_Lie_type}. Cameron showed that every graph is isomorphic to an induced subgraph of the extended commuting graph of a finite group \cite{Cameron_commuting_graphs_notes}. Recently, Arvind et al. \cite{Graphs_arise_as_commuting_graphs_groups} developed a quasipolynomial-time algorithm that decides whether a given simple graph is the extended commuting graph of some group. Another interesting result is that a simple graph is isomorphic to the commuting graph of some semigroup if and only if it contains at least two vertices and no vertex is adjacent to all the other vertices of the graph \cite{Graphs_that_arise_as_commuting_graphs_of_semigroups, Graphs_that_arise_as_commuting_graphs_of_semigroups_2}.

Other authors focused on the problem of determining the possible values for some properties (diameter, clique number, girth, chromatic number, knit degree) of the commuting graphs of semigroups/groups. In 2011, Araújo, Kinyon and Konieczny \cite{Commuting_graph_T_X} proved that, for each $n\in\mathbb{N}$ such that $n\geqslant 2$, there is a semigroup whose commuting graph has diameter equal to $n$. More recently, Cutolo \cite{Group_whose_commuting_graph_has_diameter_n} established the same result for groups. Araújo, Kinyon and Konieczny \cite{Commuting_graph_T_X} also proved that for each $n\in\mathbb{N}\setminus\set{1,3}$ there is a semigroup whose knit degree is equal to $n$ and, in 2016, Bauer and Greenfeld \cite{Graphs_that_arise_as_commuting_graphs_of_semigroups} constructed a semigroup of knit degree $3$. It follows from \cite{Graphs_that_arise_as_commuting_graphs_of_semigroups, Graphs_that_arise_as_commuting_graphs_of_semigroups_2} that the set of possible values for the clique/chromatic number of the commuting graph of a semigroup is $\mathbb{N}$, and that the set of possible values for the girth of the commuting graph of a semigroup is $\mathbb{N}\setminus\set{1,2}$. Recently, the present author \cite{Completely_simple_semigroups_paper} contributed to this problem by showing that every positive integer is a possible values for the clique/chromatic number of the commuting graph of a completely simple semigroup, that $3$ is the unique possible value for the girth of the commuting graph of a completely simple semigroup, and that there are no possible values for the knit degree of a completely simple semigroup.

Commuting graphs of semigroups have also been used to solve some group/semigroup problems. For example, the study of the graphs $\generalcommgraph{G}{C}$, where $G$ is a group and $C$ is a conjugacy class of $3$-transpositions, played an important role in the discovery of three sporadic simple groups (now known as the Fischer groups) \cite{Sporadic_simple_groups}. Commuting graphs and extended commuting graphs were also used to determine an upper bound for the size of the abelian subgroups of a finite group \cite{Importance_commuting_graphs_1}. Furthermore, commuting graphs were important in establishing some results concerning finite dimensional division algebras (see \cite{Importance_commuting_graphs_2, Importance_commuting_graphs_3, Importance_commuting_graphs_5, Importance_commuting_graphs_6, Importance_commuting_graphs_4}). More recently, the notions of left path and knit
degree were introduced (for commuting graphs) to answer (positively, except in one case) a conjecture made by Schein (see \cite{Schein_conjecture}) in the process of determining a characterization for
$r$-semisimple bands \cite{Commuting_graph_T_X}.

The aim of this paper is to study commuting graphs of completely $0$-simple semigroups. These semigroups are `close to' groups with a zero adjoined. In fact, completely $0$-simple semigroups and groups with a zero adjoined share several properties: both contain a zero and their unique ideals are the semigroup itself and the singleton formed by the zero; and all their non-zero idempotents are minimal. Completely $0$-simple semigroups are also `close to' completely simple semigroups: completely simple semigroups become completely $0$-simple by adjoining a zero element (although most completely $0$-simple semigroups do not arise in this way). In \cite{Completely_simple_semigroups_paper} the commuting graphs of completely simple semigroups were investigated through the determination of properties of the commuting graph of a Rees matrix semigroup over a group --- a semigroup construction that characterizes completely simple semigroups.

The commuting graphs of particular completely $0$-simple semigroups have already been investigated by other authors: in \cite{Brandt_semigroups_1, Brandt_semigroups_2} Kumar, Dalal and Pandey determined various properties of the commuting graphs of some Brandt semigroups --- a semigroup construction used to characterize the semigroups that are simultaneously completely $0$-simple and inverse. There is also a semigroup construction that characterizes completely $0$-simple semigroups, which is called the $0$-Rees matrix semigroup over a group (this semigroup construction is a more complex version of a Rees matrix semigroup over a group). According to the Rees--Suschkewitsch Theorem, a semigroup is completely $0$-simple if and only if it is isomorphic to some $0$-Rees matrix semigroup over a group. This way, we conduct the study of commuting graphs of completely $0$-simple semigroups through the analysis of commuting graphs of $0$-Rees matrix semigroups over groups.

This paper contains ten sections. Section~\ref{Preliminaries} contains definitions and notations that will be used frequently in the paper and Section~\ref{sec: 0Rees} contains some basic properties regarding the commuting graph of a $0$-Rees matrix semigroup over a group. The succeeding five sections are occupied with determining properties of this commuting graph. More specifically, in Section~\ref{sec: connectedness diameter 0Rees} we provide a way to decide when the commuting graph is connected and we show how to find the connected components of the commuting graph (when it is not connected). Furthermore, we describe a method for determining the diameter of the commuting graph (respectively, connected components of the commuting graph) when it is connected (respectively, it is not connected). We also determine the clique number of the commuting graph (Section~\ref{sec: clique number 0Rees}); study the existence of cycles and determine the girth of the commuting graph (Section~\ref{sec: girth 0Rees}); find upper bounds for the chromatic number (Section~\ref{sec: chromatic number 0Rees}) and study the existence left paths, which will lead to the determination of the knit degree (Section~\ref{sec: knit degree 0Rees}). 

In the process of determining these properties, we end up highlighting how the group, index sets and matrix chosen to construct the $0$-Rees matrix semigroup over a group influence the characteristics of its commuting graph. Using the information gathered in Sections~\ref{sec: connectedness diameter 0Rees}--\ref{sec: knit degree 0Rees}, we then answer some questions regarding commuting graphs of completely $0$-simple semigroups in Sections~\ref{sec: completely 0-simple smg}. More precisely, we identify the possible values for the diameter, clique number, chromatic number and girth of commuting graphs of completely $0$-simple semigroups, as well as the possible values for the knit degree of completely $0$-simple semigroups.

Finally, in Section~\ref{sec: open problems} we discuss some open problems concerning commuting graphs of completely $0$-simple semigroups.

This paper is based on Chapter 9 of the author's Ph.D. thesis~\cite{My_thesis}. A paper in preparation, based on Chapters 10 and 11 of the author's Ph.D. thesis, will examine commuting graphs of inverse semigroups and completely regular semigroups \cite{Commuting_graphs_inverse_completely_regular}.

\section{Preliminaries} \label{Preliminaries}

For general background on graphs see, for example, \cite{Graphs_Wilson}. For general background on semigroups we use \cite{Nine_chapters_Cain}.

\subsection{Graphs}\label{Subsection graphs}

A \defterm{simple graph} $G=(V,E)$ consists of a non-empty set $V$ --- whose elements are called \defterm{vertices} --- and a set $E$ --- whose elements are called \defterm{edges} --- formed by $2$-subsets of $V$. Throughout this subsection we will assume that $G=(V,E)$ is a simple graph.

Let $x$ and $y$ be vertices of $G$. If $\set{x,y}\in E$, then we say that the vertices $x$ and $y$ are \defterm{adjacent}, and that the vertices $x$ and $y$ are \defterm{incident} with the edge $\set{x,y}$.


If $H=\parens{V',E'}$ is also a simple graph, then we say that $G$ and $H$ are \defterm{isomorphic} if there exists a bijection $\varphi: V\to V'$ such that for all $x,y\in V$ we have $\set{x,y}\in E$ if and only if $\set{x\varphi,y\varphi}\in E'$ (that is, for all $x,y\in V$ we have that $x$ and $y$ are adjacent in $G$ if and only if $x\varphi$ and $y\varphi$ are adjacent in $H$).


A simple graph $H=\parens{V',E'}$ is a \defterm{subgraph} of $G$ if $V'\subseteq V$ and $E'\subseteq E$. Note that, since $H$ is a simple graph, the elements of $E'$ are $2$-subsets of $V'$.

Given $V'\subseteq V$, the \defterm{subgraph induced by $V'$} is the subgraph of $G$ whose set of vertices is $V'$ and where two vertices are adjacent if and only if they are adjacent in $G$ (that is, the set of edges of the induced subgraph is $\braces{\braces{x,y}\in E: x,y\in V'}$).

A \defterm{complete graph} is a simple graph where all distinct vertices are adjacent to each other. The unique (up to isomorphism) complete graph with $n$ vertices is denoted $K_n$.


A \defterm{path} in $G$ from a vertex $x$ to a vertex $y$ is a sequence of pairwise distinct vertices (except, possibly, $x$ and $y$) $x=x_1,x_2,\ldots,x_n=y$ such that $\braces{x_1,x_2}, \braces{x_2,x_3},\ldots, \braces{x_{n-1},x_n}$ are pairwise distinct edges of $G$. The \defterm{length} of the path is the number of edges of the path; thus, the length of our example path is $n-1$. If $n=1$, then we call the path --- which has only one vertex and whose length is $0$ --- a \defterm{trivial path}. If $x=y$ then we call the path a \defterm{cycle}. Whenever we want to mention a path, we will write that $x=x_1-x_2-\cdots-x_n=y$ is a path (instead of writing that $x=x_1,x_2,\ldots,x_n=y$ is a path). The \defterm{distance} between the vertices $x$ and $y$, denoted $\dist{G}{x}{y}$, is the length of a shortest path from $x$ to $y$. If there is no such path between the vertices $x$ and $y$, then the distance between $x$ and $y$ is defined to be infinity, that is, $\dist{G}{x}{y}=\infty$.

We say that $G$ is \defterm{connected} if for all vertices $x,y\in V$ there is a path from $x$ to $y$. We can partition $V$, the vertex set of $G$, into several non-empty sets $V_1,\ldots,V_n$ such that
\begin{enumerate}
	\item For all $i\in\Xn$ and vertices $x,y\in V_i$ there is a path from $x$ to $y$.
	
	\item For all distinct $i,j\in\Xn$ and $x\in V_i$ and $y\in V_j$ there is no path from $x$ to $y$.
\end{enumerate}
Then each subgraph of $G$ induced by $V_i$, where $i\in\Xn$, is connected and we call it a \defterm{connected component} of $G$. It is clear that $G$ is connected if and only if $G$ contains exactly one connected component.


The \defterm{diameter} of $G$, denoted $\diam{G}$, is the maximum distance between vertices of $G$, that is, $\diam{G}=\max\gset{\dist{G}{x}{y}}{x,y\in V}$. We notice that the diameter of $G$ is finite if and only if $G$ is connected.

If $x$ and $y$ are vertices of $G$, then we are going to use the notation $x\sim y$ to mean that either $x=y$ or $\set{x,y}\in E$. Note that if $x_1-x_2-\cdots-x_n$ is a path, then we have $x_1\sim x_2 \sim\cdots\sim x_n$. However, if we have $x_1\sim x_2 \sim\cdots\sim x_n$, then that sequence of vertices does not necessarily form a path because there might exist distinct $i,j\in\Xn$ such that $x_i=x_j$.

Given a vertex $x\in V$ of $G$, the \defterm{degree of $x$} is the number of edges of $G$ that are incident with $x$. We denote the degree of $x$ by $\degree{G}{x}$ and we denote by $\maxdegree{G}$ the maximum degree of a vertex of $G$, that is, $\maxdegree{G}=\max\gset{\degree{G}{x}}{x\in V}$.

Let $K\subseteq V$. We say that $K$ is a \defterm{clique} in $G$ if $\braces{x,y}\in E$ for all $x,y\in K$, that is, if the subgraph of $G$ induced by $K$ is complete. The \defterm{clique number} of $G$, denoted $\cliquenumber{G}$, is the size of a largest clique in $G$, that is, $\cliquenumber{G}=\max\left\{|K|: K \text{ is a clique in } G\right\}$.

If the graph $G$ contains cycles, then the \defterm{girth} of $G$, denoted $\girth{G}$, is the length of a shortest cycle in $G$. If $G$ contains no cycles, then $\girth{G}=\infty$.

Let $n\in\mathbb{N}$. We say that $G$ is \defterm{$n$-colourable} if it is possible to colour the vertices of $G$ with $n$ colours in a way such that adjacent vertices have different colours. More formally, we say that $G$ is $n$-colourable if there exists a set $C$ of size $n$ and a map $\varphi: V \to C$ such that for all $c\in C$ the set $\set{c}\varphi^{-1}$ contains no adjacent vertices of $G$. The idea of this map is that each element of $C$ represents a colour and for each $v\in V$ we have that $v\varphi$ is the colour assigned to vertex $v$. Hence for each $c\in C$ the set $\set{c}\varphi^{-1}$ contains all the vertices of $G$ assigned with the colour $c$. The smallest $n\in\mathbb{N}$ such that $G$ is $n$-colourable is called \defterm{chromatic number} of $G$ and it is denoted by $\chromaticnumber{G}$.

The following two results are known upper bounds for the chromatic number of a simple graph.

\begin{lemma}\label{preli: 0Rees chromatic number}
	Let $G=(V,E)$ be a simple graph. Then $\chromaticnumber{G}\parens{\chromaticnumber{G}-1}\leqslant 2\abs{E}$.
\end{lemma}

\begin{theorem}[Brooks' Theorem]\label{preli: ORees Brooks}
	Let $G$ be a connected simple graph. If $G$ is neither a complete graph nor an odd cycle, then $\chromaticnumber{G}\leqslant\maxdegree{G}$.
\end{theorem}

Let $G=(V,E)$ and $H=(V',E')$ be two simple graphs. We can assume, without loss of generality, that $V\cap V'=\emptyset$. The \defterm{graph join} of $G$ and $H$, denoted $G\graphjointwo H$, is defined to be the (simple) graph whose set of vertices is $V\cup V'$ and whose set of edges is $E\cup E'\cup\gset{\set{x,y}}{x\in V \text{ and } y\in V'}$. This means that, in the graph $G\graphjointwo H$, two vertices $x,y\in V\cup V'$ are adjacent if and only if one of the following conditions is satisfied:
\begin{enumerate}
	\item $x\in V$ and $y\in V'$ (or vice versa).
	\item $x,y\in V$ and $\braces{x,y}\in E$ (or $x,y\in V'$ and $\braces{x,y}\in E'$).
\end{enumerate}

The next lemma, which is easy to prove, shows the relationship between the clique numbers of two graphs and of their graph join.  

\begin{lemma}\label{preli: graph join clique/chromatic numbers}
	Let $G$ and $H$ be two simple graphs. Then $\cliquenumber{G\graphjointwo H}=\cliquenumber{G}+\cliquenumber{H}$.
\end{lemma}

\subsection{Commuting graphs and extended commuting graphs}\label{Subsection commuting graphs}

In this subsection we present the two most common definitions of commuting graph of a semigroup. In both definitions, the condition that determines adjacency of vertices is the same, but the vertex set is distinct: in one definition the vertices of the graph are the non-central elements of the semigroup --- we call this graph the commuting graph of the semigroup --- and in the other one the vertices are all the elements of the semigroup --- we call this graph the extended commuting graph of the semigroup. These terminologies were also used in \cite{Completely_simple_semigroups_paper}.

The \defterm{center} of a semigroup $S$ is the set
\begin{displaymath}
	\centre{S}=\gset{x\in S}{xy=yx \text{ for all } y\in S}.
\end{displaymath}

Let $S$ be a finite non-commutative semigroup. The \defterm{commuting graph} of $S$, denoted $\commgraph{S}$, is the simple graph whose set of vertices is $S\setminus Z(S)$ and where two distinct vertices $x,y\in S\setminus Z(S)$ are adjacent if and only if $xy=yx$. (This is the definition of commuting graph used in \cite{Commuting_graph_T_X, Commuting_graph_I_X, Commuting_graph_symmetric_alternating_groups}, for example.)

Let $S$ be a finite semigroup. The \defterm{extended commuting graph} of $S$, denoted $\extendedcommgraph{S}$, is the simple graph whose set of vertices is $S$ and where two distinct vertices $x,y\in S$ are adjacent if and only if $xy=yx$. (This is the definition of commuting graph used in \cite{Graphs_arise_as_commuting_graphs_groups, Cameron_commuting_graphs_notes, Commuting_graphs_groups_split}, for example.)

It follows from both definitions that, for all vertices $x$ and $y$ of $\commgraph{S}$ (respectively $\extendedcommgraph{S}$), we have $x\sim y$ if and only if $xy=yx$.

Note that in the first definition the semigroup must be non-commutative (because otherwise we would obtain an empty vertex set), but in the second one we allow the semigroup to be commutative. Furthermore, as a consequence of the first definition we have $\diam{\commgraph{S}}\geqslant 2$ because, since $S$ must be non-commutative, then there exist $x,y\in S$ such that $xy\neq yx$, which implies that $\diam{\commgraph{S}}\geqslant\dist{\commgraph{S}}{x}{y}>1$. Additionally, the second definition implies that the center of the semigroup is a clique in the extended commuting graph of the semigroup.

The next lemma, which is easy to prove, gives a characterization of the extended commuting graph of a semigroup. When the semigroup is not commutative, this characterization shows a relationship between the commuting graph and the extended commuting graph of the semigroup.

\begin{lemma}\label{preli: commgraph and extended commgraph}
	Let $S$ be a finite semigroup.
	\begin{enumerate}
		\item If $S$ is commutative, then $\extendedcommgraph{S}$ is isomorphic to $K_{\abs{S}}$.
		
		\item If $S$ is non-commutative, then $\extendedcommgraph{S}$ is isomorphic to $K_{\abs{Z\parens{S}}}\graphjointwo\commgraph{S}$.
	\end{enumerate}
\end{lemma}

The next lemma, whose proof is straightforward, shows the relationship between the largest cliques in a commuting graph of a semigroup and the size of its largest commutative subsemigroups.

\begin{lemma}\label{preli: largest cliques, commutative subsemigroups}
	Let $S$ be a finite non-commutative semigroup and let $\centre{S}\subseteq T\subseteq S$. Then $T$ is a commutative subsemigroup of $S$ of maximum size if and only if $T\setminus\centre{S}$ is a clique in $\commgraph{S}$ of maximum size.
\end{lemma}


The following definitions are of concepts that were first defined (in \cite{Commuting_graph_T_X}) specifically for commuting graphs of semigroups.

Let $S$ be a non-commutative semigroup. A \defterm{left path} in $\commgraph{S}$ is a path $x_1,\ldots,x_n$ in $\commgraph{S}$ such that $x_1\neq x_n$ and $x_1x_i=x_nx_i$ for all $i\in\braces{1,\ldots,n}$. If $\commgraph{S}$ contains left paths, then the \defterm{knit degree} of $S$, denoted $\knitdegree{S}$, is the length of a shortest left path in $\commgraph{S}$.

\subsection{Completely simple semigroups}\label{Subsection completely simple semigroups}

We say that a semigroup $S$ with a zero $0$ is a \defterm{completely $0$-simple semigroup} if it satisfies the following conditions:
\begin{enumerate}
	\item $S$ is \defterm{$0$-simple}, which means that $S$ is not a null semigroup and its ideals are precisely $\set{0}$ and $S$;
	
	\item $S$ contains a \defterm{primitive} idempotent, which means that $S$ contains a minimal idempotent among the set of non-zero idempotents.
\end{enumerate}

Completely $0$-simple semigroups can also be characterized via a semigroup construction called the $0$-Rees matrix semigroup over a group (Theorem~\ref{preli: completely 0-simple semigroup <=> Rees matrix construction}), which is described below.

Let $G$ be a group, $I$ and $\Lambda$ be index sets, and $P$ be a regular $\Lambda\times I$ matrix with entries from $G^0$. (Recall that a \defterm{regular} matrix is a matrix where every row and every column contains at least one non-zero entry.) For each $i\in I$ and $\lambda \in \Lambda$, we denote by $p_{\lambda i}$ the $\parens{\lambda, i}$-th entry of the matrix $P$. A \defterm{$0$-Rees matrix semigroup over a group}, denoted $\Reeszero{G}{I}{\Lambda}{P}$, is the set $\parens{I\times G\times \Lambda}\cup\set{0}$ with multiplication defined as follows
\begin{gather*}
	\parens{i,x,\lambda}\parens{j,y,\mu} =\begin{cases}
		(i,xp_{\lambda j}y,\mu) & \text{if } p_{\lambda j}\neq 0,\\
		0 & \text{if } p_{\lambda j}=0;
	\end{cases}\\
	0\parens{i,x,\lambda}=\parens{i,x,\lambda}0=00=0.
\end{gather*}


\begin{theorem}[Rees--Suschkewitsch Theorem]\label{preli: completely 0-simple semigroup <=> Rees matrix construction}
	A semigroup $S$ is completely $0$-simple if and only if there exist a group $G$, index sets $I$ and $\Lambda$, and a regular $\Lambda\times I$ matrix $P$ with entries from $G^0$ such that $S\simeq \Reeszero{G}{I}{\Lambda}{P}$.
\end{theorem}

This theorem is fundamental and we will use it without explicit reference in the rest of the paper.

\subsection{Matrices}\label{sec: 0Rees definitions/notations}

This subsection is a compilation of definitions of terms and notations that will be used frequently in the course of the paper. We also present notations that will be adopted. 

\begin{definition}
	Given two matrices $Q$ and $M$, we say that $Q$ is \defterm{\equivalent} to $M$ if $Q$ can be obtained from $M$ by exchanging rows and/or columns. 
\end{definition}

It follows immediately from the definition that $\leftrightarrow$-equivalence is, as its name suggests, an equivalence relation.

\begin{definition}
	Given two matrices $Q$ and $M$, we say that $Q$ is a \defterm{\submatrix} of $M$ if $Q$ is \equivalent\ to a submatrix of $M$.
\end{definition}

The previous definition also implies that if $Q$ is a submatrix of $M$, then $Q$ is a \submatrix\ of $M$; and if $Q$ is \equivalent\ to $M$ then $Q$ is a \submatrix\ of $M$.

Let $I$ and $\Lambda$ be index sets and let $P$ be a $\Lambda\times I$ matrix. Given $i\in I$ and $\Lambda\in\lambda$, we denote by $p_{\lambda i}$ the $\parens{\lambda, i}$-th entry of $P$.

Let $I'\subseteq I$ and $\Lambda'\subseteq\Lambda$. Assume that $I'=\set{i_1,\ldots,i_n}$, $\Lambda'=\set{\lambda_1,\ldots,\lambda_m}$ and that the indices $i_1,\ldots,i_n$ (respectively, $\lambda_1,\ldots,\lambda_m$) are in the order in which they appear in the columns (respectively, rows) of $P$.

We denote by $\Psubmatrix{P}{I'}{\Lambda'}$ the submatrix of $P$ formed by the rows and columns of $P$ whose indices belong to $\Lambda'$ and $I'$, respectively. In the matrix $\Psubmatrix{P}{I'}{\Lambda'}$ the columns (respectively, rows) appear in the original order: $i_1,\ldots,i_n$ (respectively, $\lambda_1,\ldots,\lambda_m$). If $\Lambda'=\set{\lambda}$ or $I'=\set{i}$, then we will often replace the relevant singleton sets by their elements, that is, we will write $\Psubmatrix{P}{I'}{\lambda}$ or $\Psubmatrix{P}{i}{\Lambda'}$, respectively.

We extend the notation $\Psubmatrix{P}{I'}{\Lambda'}$ to \submatrices\ by allowing the (unordered) sets $I'$ and $\Lambda'$ to be replaced by (ordered) sequences of elements of $I$ and $\Lambda$, respectively. Let $\alpha$ be a permutation of $\Xn$ and $\beta$ a permutation of $\X{m}$. We denote by $\Psubmatrix{P}{i_{1\alpha},\ldots,i_{n\alpha}}{\lambda_{1\beta},\ldots,\lambda_{m\beta}}$ the \submatrix\ of $P$ obtained by selecting the row and column indices of $P$ in the following orders: $\lambda_{1\beta},\ldots,\lambda_{m\beta}$ and $i_{1\alpha},\ldots,i_{n\alpha}$, respectively. We note that, if $\alpha$ and $\beta$ are both equal to the identity, then we obtain the \submatrix\ $\Psubmatrix{P}{i_1,\ldots,i_n}{\lambda_1,\ldots,\lambda_m}=\Psubmatrix{P}{I'}{\Lambda'}$, which is a submatrix of $P$.


%

Now we introduce a new matrix which will be used several times in the course of the paper. We will reveal its importance in Lemma~\ref{0Rees: lemma (0,X)-representation}.

\begin{definition}
	Let $G$ be a group, $I$ and $\Lambda$ be index sets and $P$ be a $\Lambda\times I$ matrix whose entries are elements of $G^0$. We denote by $\timeszero{P}$ the $\Lambda\times I$ matrix whose entries are elements of $\set{0,\times}$ and such that
	\begin{displaymath}
		\overline{p}_{\lambda i}=\begin{cases}
			0& \text{if } p_{\lambda i}=0,\\
			\times& \text{if } p_{\lambda i}\in G.
		\end{cases}
	\end{displaymath}
\end{definition}

It is straightforward to see that $\timeszero{P}$ can be obtained from $P$ by replacing all its non-zero entries by $\times$. Furthermore, $\timeszero{P}$ satisfies the following properties:
\begin{enumerate}
	\item Let $M$ be a matrix with entries in $\set{0,\times}$. Then $M$ is \equivalent\ to $\timeszero{P}$ if and only if there exists a matrix $Q$ such that $M=\timeszero{Q}$ and $Q$ is \equivalent\ to $P$.
	
	\item Let $M$ be a matrix with entries in $\set{0,\times}$. Then $M$ is a submatrix of $\timeszero{P}$ if and only if there exists a matrix $Q$ such that $M=\timeszero{Q}$ and $Q$ is a submatrix of $P$.
	
	\item Let $M$ be a matrix with entries in $\set{0,\times}$. Then $M$ is a \submatrix\ of $\timeszero{P}$ if and only if there exists a matrix $Q$ such that $M=\timeszero{Q}$ and $Q$ is a \submatrix\ of $P$.
	
	\item Let $i_1,\ldots,i_n\in I$ and $\lambda_1,\ldots,\lambda_m$ and assume that they are pairwise distinct. Then $\timeszero{\Psubmatrix{P}{i_1,\ldots,i_n}{\lambda_1,\ldots,\lambda_m}}=\Psubmatrix{\timeszero{P}}{i_1,\ldots,i_n}{\lambda_1,\ldots,\lambda_m}$.
\end{enumerate}

Finally, we define two types of matrices which will be used several times in the course of the paper. For each $n\in\mathbb{N}$ we define $D_n$ to be the $n\times n$ matrix whose diagonal entries are $\times$ and the remaining entries are $0$, that is
\begin{displaymath}
	D_n=\begin{bNiceMatrix}[first-row,first-col]
		&1&2&\cdots&n\\
		1&\times&0&\cdots&0\\
		2&0&\times&\cdots&0\\
		\vdots&\vdots&\vdots&\ddots&\vdots\\
		n&0&0&\cdots&\times
	\end{bNiceMatrix}.
\end{displaymath}
Additionally, for each $n,m\in\mathbb{N}$ we define $O_{n\times m}$ to be the $n\times m$ matrix whose entries are all zeros.

\section{Basic properties of the commuting graph of a 0-Rees matrix semigroup over a group}\label{sec: 0Rees}

Let $G$ be a group, let $I$ and $\Lambda$ be index sets, and let $P$ be a regular $\Lambda\times I$ matrix whose entries are elements of $G^0$. When all the entries of $P$ are elements of $G$ (that is, when none of the entries of $P$ is a zero), we have that $\Reeszero{G}{I}{\Lambda}{P}=\parens{\Rees{G}{I}{\Lambda}{P}}^0$, which implies that (when $\Reeszero{G}{I}{\Lambda}{P}$ and $\Rees{G}{I}{\Lambda}{P}$ are not commutative and, consequently, their commuting graphs are defined) the graphs $\commgraph{\Reeszero{G}{I}{\Lambda}{P}}$ and $\commgraph{\parens{\Rees{G}{I}{\Lambda}{P}}^0}$ are isomorphic. Moreover, since a zero is a central element of a semigroup, then a zero is not a vertex of the commuting graph of a semigroup. Hence the graphs $\commgraph{\parens{\Rees{G}{I}{\Lambda}{P}}^0}$ and $\commgraph{\Rees{G}{I}{\Lambda}{P}}$ are also isomorphic. Therefore, when all the entries of $P$ are elements of $G$, the graphs $\commgraph{\Reeszero{G}{I}{\Lambda}{P}}$ and $\commgraph{\Rees{G}{I}{\Lambda}{P}}$ are isomorphic. Consequently, determining the properties of $\commgraph{\Reeszero{G}{I}{\Lambda}{P}}$ is equivalent to determining the properties of $\commgraph{\Rees{G}{I}{\Lambda}{P}}$, which was done in \cite{Completely_simple_semigroups_paper}. Since it is already known how to determine the properties of $\commgraph{\Reeszero{G}{I}{\Lambda}{P}}$ when all the entries of $P$ are elements of $G$, then for the remainder of the paper we assume, often without further comment, that $P$ contains at least one zero entry.

Suppose that $P$ contains at least one zero entry. The aim of this section is to identify the vertex set of $\commgraph{\Reeszero{G}{I}{\Lambda}{P}}$ (Proposition~\ref{0Rees: center}) and to provide necessary and sufficient conditions for adjacency of vertices of $\commgraph{\Reeszero{G}{I}{\Lambda}{P}}$ (Lemma~\ref{0Rees: commutativity}). Moreover, in Lemma~\ref{0Rees: lemma (0,X)-representation} we see how the investigation of the properties of the commuting graph of $\Reeszero{G}{I}{\Lambda}{P}$ can be simplified. 

We start with Lemma~\ref{0Rees: p=0 <=> p'=0}, which provides information regarding commutativity in $\Reeszero{G}{I}{\Lambda}{P}$.

\begin{lemma}\label{0Rees: p=0 <=> p'=0}
	Let $i_1,i_2,j_1,j_2\in I$ and $\lambda_1,\lambda_2,\mu_1,\mu_2\in\Lambda$ and $x_1,x_2,y_1,y_2\in G$ be such that $\parens{i_1,x_1,\lambda_1}\parens{i_2,x_2,\lambda_2}\allowbreak=\parens{j_1,y_1,\mu_1}\parens{j_2,y_2,\mu_2}$. Then $p_{\lambda_1 i_2}=0$ if and only if $p_{\mu_1 j_2}=0$.
\end{lemma}

\begin{proof}
	It follows from the definition of multiplication in $\Reeszero{G}{I}{\Lambda}{P}$ that
	\begin{align*}
		p_{\lambda_1 i_2}=0 &\iff \parens{i_1,x_1,\lambda_1}\parens{i_2,x_2,\lambda_2}=0\\
		& \iff \parens{j_1,y_1,\mu_1}\parens{j_2,y_2,\mu_2}=0\\
		& \iff p_{\mu_1 j_2}=0.\qedhere
	\end{align*}
\end{proof}

\begin{lemma}\label{0Rees: commutativity}
	Let $i,j\in I$ and $\lambda,\mu\in\Lambda$ and $x,y\in G$. Then $\parens{i,x,\lambda}\parens{j,y,\mu}=\parens{j,y,\mu}\parens{i,x,\lambda}$ if and only if one of the following conditions is satisfied:
	\begin{enumerate}
		\item $i=j$, $\lambda=\mu$ and $xp_{\lambda i}y=yp_{\lambda i}x$.
		\item $p_{\lambda j}=p_{\mu i}=0$.
	\end{enumerate}
\end{lemma}

\begin{proof} 
	We begin by proving the forward implication. Suppose that $\parens{i,x,\lambda}\parens{j,y,\mu}=\parens{j,y,\mu}\parens{i,x,\lambda}$. It follows from Lemma~\ref{0Rees: p=0 <=> p'=0} that $p_{\lambda j}=p_{\mu i}=0$ or $p_{\lambda j},p_{\mu i}\in G$. If $p_{\lambda j}=p_{\mu i}=0$, then condition 2 is satisfied. If $p_{\lambda j},p_{\mu i}\in G$, then
	\begin{displaymath}
		\parens{i,xp_{\lambda j}y,\mu}=\parens{i,x,\lambda}\parens{j,y,\mu}=\parens{j,y,\mu}\parens{i,x,\lambda}=\parens{j,yp_{\mu i}x,\lambda},
	\end{displaymath}
	which implies that $i=j$, $\lambda=\mu$ and $xp_{\lambda i}y=xp_{\lambda j}y=yp_{\mu i}x=yp_{\lambda i}x$ and, consequently, condition 1 is satisfied.
	
	Now we prove the reverse implication. Assume that $i=j$, $\lambda=\mu$ and $xp_{\lambda i}y=yp_{\lambda i}x$. Hence $xp_{\lambda j}y=yp_{\mu i}x$ and, consequently, we must have $p_{\lambda j},p_{\mu i}\in G$ or $p_{\lambda j}=p_{\mu i}=0$. If $p_{\lambda j},p_{\mu i}\in G$ then
	\begin{displaymath}
		\parens{i,x,\lambda}\parens{j,y,\mu}=\parens{i,xp_{\lambda j}y,\mu}=\parens{j,yp_{\mu i}x,\lambda}=\parens{j,y,\mu}\parens{i,x,\lambda}.
	\end{displaymath}
	and, if $p_{\lambda j}=p_{\mu i}=0$, then
	\begin{displaymath}
		\parens{i,x,\lambda}\parens{j,y,\mu}=0=\parens{j,y,\mu}\parens{i,x,\lambda}. \qedhere
	\end{displaymath}
\end{proof}

\begin{proposition}\label{0Rees: center}
	We have that $\centre{\Reeszero{G}{I}{\Lambda}{P}}=\set{0}$. Moreover, $\Reeszero{G}{I}{\Lambda}{P}$ is not commutative.
\end{proposition}

\begin{proof}
	It is clear that $0\in\centre{\Reeszero{G}{I}{\Lambda}{P}}$. Now we establish that for each $i\in I$, $\lambda\in\Lambda$ and $x\in G$ we have $\parens{i,x,\lambda}\notin\centre{\Reeszero{G}{I}{\Lambda}{P}}$, which is enough to conclude that $\centre{\Reeszero{G}{I}{\Lambda}{P}}\subseteq\set{0}$. Let $i\in I$, $\lambda\in\Lambda$ and $x\in G$.
	
	\smallskip
	
	\textit{Case 1:} Assume that $p_{\lambda i}\in G$. Due to the fact that $P$ contains at least one zero entry, we have that there exist $j\in I$ and $\mu\in\Lambda$ such that $p_{\mu j}=0$. In addition, $P$ is regular, which implies that row $\mu$ and column $i$ contains at least one non-zero entry. Hence $\abs{I}>1$ and $\abs{\Lambda}>1$. Consequently, there exists $i'\in I$ such that $i'\neq i$. Since $p_{\lambda i}\neq 0$ and $i\neq i'$, then Lemma~\ref{0Rees: commutativity} implies that $\parens{i,x,\lambda}\parens{i',x,\lambda}\neq\parens{i',x,\lambda}\parens{i,x,\lambda}$. Consequently, $\parens{i,x,\lambda}$ is not a central element of $\Reeszero{G}{I}{\Lambda}{P}$.
	
	\smallskip
	
	\textit{Case 2:} Assume that $p_{\lambda i}=0$. As a consequence of the fact that $P$ is regular, we have that row $\lambda$ contains a non-zero entry, which implies that there exists $i'\in I$ such that $i'\neq i$ and $p_{\lambda i'}\in G$. Then, by Lemma~\ref{0Rees: commutativity}, we have that $\parens{i,x,\lambda}\parens{i',x,\lambda}\neq\parens{i',x,\lambda}\parens{i,x,\lambda}$ and, consequently, $\parens{i,x,\lambda}$ is not a central element of $\Reeszero{G}{I}{\Lambda}{P}$.
	
	\smallskip
	
	Therefore $\centre{\Reeszero{G}{I}{\Lambda}{P}}=\set{0}\neq\Reeszero{G}{I}{\Lambda}{P}$ and, consequently, $\Reeszero{G}{I}{\Lambda}{P}$ is not commutative.
\end{proof}

Proposition~\ref{0Rees: center} guarantees that we can always talk about the commuting graph of $\Reeszero{G}{I}{\Lambda}{P}$ whenever $P$ contains at least one zero entry. Furthermore, it also implies that the vertices of $\commgraph{\Reeszero{G}{I}{\Lambda}{P}}$ are all the non-zero elements of $\Reeszero{G}{I}{\Lambda}{P}$; that is, the vertex set of $\commgraph{\Reeszero{G}{I}{\Lambda}{P}}$ is $I\times G\times\Lambda$.

\begin{lemma}\label{0Rees: lemma (0,X)-representation}
	Let $Q$ be a regular $\Lambda \times I$ matrix with entries from $G^0$. If $\timeszero{P}=\timeszero{Q}$, then the graphs $\commgraph{\Rees{G}{I}{\Lambda}{P}}$ and $\commgraph{\Rees{G}{I}{\Lambda}{Q}}$ are isomorphic.
\end{lemma}

Note that Lemma~\ref{0Rees: lemma (0,X)-representation} only concerns isomorphism of the graphs; in general, the semigroups will be non-isomorphic.

\begin{proof}
	In the course of this proof, whenever we write $\parens{i,x,\lambda}_P$ we will be referring to an element of $\Reeszero{G}{I}{\Lambda}{P}$ (or a vertex of $\commgraph{\Reeszero{G}{I}{\Lambda}{P}}$) and whenever we write $\parens{i,x,\lambda}_Q$ we will be referring to an element of $\Reeszero{G}{I}{\Lambda}{Q}$ (or a vertex of $\commgraph{\Reeszero{G}{I}{\Lambda}{Q}}$).
	
	Let $\psi:I\times G\times\Lambda\to I\times G\times\Lambda$ be the map defined by
	\begin{displaymath}
		\parens{i,x,\lambda}_P\psi=\begin{cases}
			\parens{i,x,\lambda}_Q& \text{ if } p_{\lambda i}=q_{\lambda i}=0,\\
			\parens{i,xp_{\lambda i}q_{\lambda i}^{-1},\lambda}_Q& \text{ if } p_{\lambda i},q_{\lambda i}\in G
		\end{cases}
	\end{displaymath}
	for all $i\in I$, $\lambda\in\Lambda$ and $x\in G$. We have that $\psi$ is a map from the set of vertices of $\commgraph{\Reeszero{G}{I}{\Lambda}{P}}$ to the set of vertices of $\commgraph{\Reeszero{G}{I}{\Lambda}{Q}}$.
	
	First, we are going to see that $\psi$ is a bijection. Since $I \times G\times \Lambda$ is finite, it is enough to see that $\psi$ is surjective. Let $i\in I$, $\lambda\in\Lambda$ and $x\in G$. We have $p_{\lambda i}=q_{\lambda i}=0$ or $p_{\lambda i},q_{\lambda i}\in G$. If $p_{\lambda i}=q_{\lambda i}=0$, then $\parens{i,x,\lambda}_Q=\parens{i,x,\lambda}_P\psi$. If $p_{\lambda i},q_{\lambda i}\in G$, then $\parens{i,x,\lambda}_Q=\parens{i,\parens{xq_{\lambda i}p_{\lambda i}^{-1}}p_{\lambda i}q_{\lambda i}^{-1},\lambda}_Q=\parens{i,xq_{\lambda i}p_{\lambda i}^{-1},\lambda}_P\psi$.
	
	Now we are going to prove that, for all $i,j\in I$, $\lambda,\mu\in\Lambda$ and $x,y\in G$, the vertices $\parens{i,x,\lambda}_P$ and $\parens{j,y,\mu}_P$ are adjacent in $\commgraph{\Reeszero{G}{I}{\Lambda}{P}}$ if and only if the vertices $\parens{i,x,\lambda}_P\psi$ and $\parens{j,y,\mu}_P\psi$ are adjacent in $\commgraph{\Reeszero{G}{I}{\Lambda}{Q}}$. It is enough to show that for all $i,j\in I$, $\lambda,\mu\in\Lambda$ and $x,y\in G$, $\parens{i,x,\lambda}_P$ and $\parens{j,y,\mu}_P$ commute if and only if $\parens{i,x,\lambda}_P\psi$ and $\parens{j,y,\mu}_P\psi$ commute. Let $i,j\in I$, $\lambda,\mu\in\Lambda$ and $x,y\in G$. Let $x',y'\in G$ be such that
	\begin{displaymath}
		x'=\begin{cases}
			x& \text{if } p_{\lambda i}=q_{\lambda i}=0,\\
			xp_{\lambda i}q_{\lambda i}^{-1}& \text{if } p_{\lambda i},q_{\lambda i}\in G;
		\end{cases}\quad
		y'=\begin{cases}
			y& \text{if } p_{\mu j}=q_{\mu j}=0,\\
			yp_{\mu j}q_{\mu j}^{-1}& \text{if } p_{\mu j},q_{\mu j}\in G.
		\end{cases}
	\end{displaymath}
	Then $\parens{i,x,\lambda}_P\psi=\parens{i,x',\lambda}_Q$ and $\parens{j,y,\mu}_P\psi=\parens{j,y',\mu}_Q$.
	
	
	
	\smallskip
	
	\textit{Case 1:} Suppose that $p_{\lambda j}=q_{\lambda j}=0$. Then we have
	\begin{align*}
		&\parens{i,x,\lambda}_P\parens{j,y,\mu}_P=\parens{j,y,\mu}_P\parens{i,x,\lambda}_P \\
		\iff{} & p_{\mu i}=0 & \smash{\begin{minipage}[c]{55mm}\raggedleft [since 
				$p_{\lambda j}=0$ and by Lemmata~\ref{0Rees: p=0 <=> p'=0} and \ref{0Rees: commutativity}]\end{minipage}}\\
		\iff{} & q_{\mu i}=0\\
		\iff{} &\parens{i,x',\lambda}_Q\parens{j,y',\mu}_Q=\parens{j,y',\mu}_Q\parens{i,x',\lambda}_Q & \smash{\begin{minipage}[c]{50mm}\raggedleft [since 
				$q_{\lambda j}=0$ and by Lemmata~\ref{0Rees: p=0 <=> p'=0} and \ref{0Rees: commutativity}]\end{minipage}}\\
		\iff{} &\parens{i,x,\lambda}_P\psi\parens{j,y,\mu}_P\psi=\parens{j,y,\mu}_P\psi\parens{i,x,\lambda}_P\psi.
	\end{align*}
	
	
	\smallskip
	
	\textit{Case 2:} Now suppose that $p_{\lambda j},q_{\lambda j}\in G$. Then we have
	\begin{align*}
		&\parens{i,x,\lambda}_P\parens{j,y,\mu}_P=\parens{j,y,\mu}_P\parens{i,x,\lambda}_P\\
		\iff{} & xp_{\lambda i}y=yp_{\lambda i}x \text{ and } \parens{i,\lambda}=\parens{j,\mu} &\kern -55mm\bracks{\text{since } p_{\lambda j}\in G \text{ and by Lemma~\ref{0Rees: commutativity}}}\\
		\iff{} & xp_{\lambda i}y=yp_{\lambda i}x \text{ and } \parens{i,\lambda}=\parens{j,\mu} \text{ and } p_{\lambda i},q_{\lambda i}\in G &\kern -55mm\bracks{\text{since } p_{\lambda j},q_{\lambda j}\in G}\\
		\iff{} & \parens{xp_{\lambda i}y}\parens{p_{\lambda i}q_{\lambda i}^{-1}}=\parens{yp_{\lambda i}x}\parens{p_{\lambda i}q_{\lambda i}^{-1}} \\
		&\qquad\text{and } \parens{i,\lambda}=\parens{j,\mu} \text{ and } p_{\lambda i},q_{\lambda i}\in G\\
		\iff{} & xp_{\lambda i}\parens{q_{\lambda i}^{-1}q_{\lambda i}}yp_{\lambda i}q_{\lambda i}^{-1}=yp_{\lambda i}\parens{q_{\lambda i}^{-1}q_{\lambda i}}xp_{\lambda i}q_{\lambda i}^{-1} \\
		&\qquad\qquad\qquad\text{and } \parens{i,\lambda}=\parens{j,\mu} \text{ and } p_{\lambda i},q_{\lambda i}\in G \\
		\iff{} & \parens{xp_{\lambda i}q_{\lambda i}^{-1}}q_{\lambda j}\parens{yp_{\mu j}q_{\mu j}^{-1}}=\parens{yp_{\mu j}q_{\mu j}^{-1}}q_{\mu i}\parens{xp_{\lambda i}q_{\lambda i}^{-1}} &\kern -55mm \smash{\begin{minipage}[t]{40mm}\raggedleft [since 
				$q_{\mu j}=q_{\lambda i}\in G$ and $p_{\lambda i}=p_{\lambda j}\in G$]\end{minipage}}\\
		&\qquad\qquad\qquad\text{and } \parens{i,\lambda}=\parens{j,\mu} \text{ and } q_{\lambda i},q_{\mu j}\in G\\
		\iff{} &\parens{i,xp_{\lambda i}q_{\lambda i}^{-1},\lambda}_Q\parens{j,yp_{\mu j}q_{\mu j}^{-1},\mu}_Q=\parens{j,yp_{\mu j}q_{\mu j}^{-1},\mu}_Q\parens{i,xp_{\lambda i}q_{\lambda i}^{-1},\lambda}_Q\\
		&\qquad\qquad\qquad\qquad\qquad\qquad\qquad\qquad\qquad\qquad\quad \text{and } q_{\lambda i},q_{\mu j}\in G\\
		&&\kern -55mm\bracks{\text{since } q_{\lambda j}\in G \text{ and by Lemma~\ref{0Rees: commutativity}}}\\
		\iff{} &\parens{i,x',\lambda}_Q\parens{j,y',\mu}_Q=\parens{j,y',\mu}_Q\parens{i,x',\lambda}_Q\\
		&& \kern -55mm \text{[since $q_{\lambda j}\in G$ and, by Lemma~\ref{0Rees: commutativity}, $\parens{i,\lambda}=\parens{j,\mu}$]}\\
		\iff{} &\parens{i,x,\lambda}_P\psi\parens{j,y,\mu}_P\psi=\parens{j,y,\mu}_P\psi\parens{i,x,\lambda}_P\psi. &\qedhere
	\end{align*}
\end{proof}

Lemma~\ref{0Rees: lemma (0,X)-representation} reveals that, whenever an entry $\parens{\lambda,i}$ of $P$ is such that $p_{\lambda i}\in G$, it is not important which element of the group $p_{\lambda i}$ is, but rather the fact that it is not a zero. This way, we only need to distinguish between zero and non-zero entries in $P$, which justifies the introduction of the matrix $\timeszero{P}$, defined in the previous section. As a result of Lemma~\ref{0Rees: lemma (0,X)-representation}, we will often think of $\timeszero{P}$ instead of $P$.

\section{Connectedness and diameter of the commuting graph of a 0-Rees matrix semigroup over a group}\label{sec: connectedness diameter 0Rees}

Let $G$ be a finite group, let $I$ and $\Lambda$ be finite index sets and let $P$ be a regular $\Lambda \times I$ matrix whose entries are elements of $G^0$. Assume that $P$ contains at least one zero entry.

In this section we will ascertain when $\commgraph{\Reeszero{G}{I}{\Lambda}{P}}$ is connected and determine its diameter. Moreover, when $\commgraph{\Reeszero{G}{I}{\Lambda}{P}}$ is not connected we will present a way to identify its connected components, as well as a way to determine their diameter.

We start by determining when $\commgraph{\Reeszero{G}{I}{\Lambda}{P}}$ is connected. Lemma~\ref{0Rees: subgraph induced by ixGxlambda} provides a first step in that direction by characterizing the subgraphs of $\commgraph{\Reeszero{G}{I}{\Lambda}{P}}$ induced by $\set{i}\times G\times\set{\lambda}$ (where $i\in I$ and $\lambda\in\Lambda$).

\begin{lemma}\label{0Rees: subgraph induced by ixGxlambda}
	Let $i\in I$ and $\lambda\in \Lambda$. We have that the subgraph of $\commgraph{\Reeszero{G}{I}{\Lambda}{P}}$ induced by $\set{i}\times G\times\set{\lambda}$ is isomorphic to
	\begin{displaymath}
		\begin{cases}
			K_{\abs{G}}& \text{if } G \text{ is abelian or } p_{\lambda i}=0,\\
			K_{\abs{\centre{G}}}\graphjointwo\commgraph{G}& \text{if } G \text{ is not abelian and } p_{\lambda i}\in G.
		\end{cases}
	\end{displaymath}
	
\end{lemma}

\begin{proof}
	Let $\mathcal{C}$ be the subgraph of $\commgraph{\Reeszero{G}{I}{\Lambda}{P}}$ induced by $\set{i}\times G\times\set{\lambda}$. We consider two cases.
	
	\smallskip
	
	\textit{Case 1:} Assume that $G$ is abelian or $p_{\lambda i}=0$. Let $x,y\in G$. Then we have $xp_{\lambda i}y=yp_{\lambda i}x$, which implies, by Lemma~\ref{0Rees: commutativity}, that $\parens{i,x,\lambda}\parens{i,y,\lambda}=\parens{i,y,\lambda}\parens{i,x,\lambda}$. Since $x$ and $y$ are arbitrary elements of $G$, then this means we proved that all the distinct vertices of $\mathcal{C}$ are adjacent to each other. Hence $\mathcal{C}$ is isomorphic to $K_{\abs{G}}$.
	
	\smallskip
	
	\textit{Case 2:} Assume that $G$ is not abelian and $p_{\lambda i}\in G$. Let $\varphi: G \to \braces{i}\times G \times \braces{\lambda}$ be the map defined by $x\varphi=\parens{i,p_{\lambda i}^{-1}x, \lambda}$ for all $x\in G$. It is easy to verify that $\varphi$ is a group isomorphism. This implies that $\varphi$ preserves adjacency. Since $G$ and $\set{i}\times G\times\set{\lambda}$ are the sets of vertices of $\extendedcommgraph{G}$ and $\mathcal{C}$, respectively, then the graphs $\extendedcommgraph{G}$ and $\mathcal{C}$ are isomorphic. Furthermore, it follows from Lemma~\ref{preli: commgraph and extended commgraph} that $\extendedcommgraph{G}$ is isomorphic to $K_{\abs{\centre{G}}}\graphjointwo\commgraph{G}$, which concludes the proof.
\end{proof}

Lemma~\ref{0Rees: subgraph induced by ixGxlambda} shows that the subgraphs of $\commgraph{\Reeszero{G}{I}{\Lambda}{P}}$ induced by $\set{i}\times G\times\set{\lambda}$ (where $i\in I$ and $\lambda\in\Lambda$) are all connected. This means that, in order to ascertain if $\commgraph{\Reeszero{G}{I}{\Lambda}{P}}$ is connected, we only need to verify if, for all $i,j\in I$ and $\lambda,\mu\in\Lambda$ such that $\parens{i,\lambda}\neq\parens{j,\mu}$, there is a path from $\parens{i,1_G,\lambda}$ to $\parens{j,1_G,\mu}$. In order to do this we create a new graph which ignores the elements of $G$ and where adjacency is determined solely by condition 2 of Lemma~\ref{0Rees: commutativity} (which characterizes adjacency in $\commgraph{\Reeszero{G}{I}{\Lambda}{P}}$). Condition 1 will be excluded in the new graph since we no longer have to determine adjacency when the first and last components of the vertices of $\commgraph{\Reeszero{G}{I}{\Lambda}{P}}$ coincide).

\begin{definition}\label{0Rees: definition G(I,^,P)}
	Let $\simplifiedgraph{I}{\Lambda}{P}$ be the simple graph whose set of vertices is $I\times \Lambda$ and where two distinct vertices $\parens{i,\lambda},\parens{j,\mu}\in I\times \Lambda$ are adjacent if and only if $p_{\lambda j}=p_{\mu i}=0$.
\end{definition}

In Lemma~\ref{0Rees: adjacency in G(I,^) and G(M0[...])} we describe a relationship between $\simplifiedgraph{I}{\Lambda}{P}$ and $\commgraph{\Reeszero{G}{I}{\Lambda}{P}}$. Moreover, it follows from this lemma that $\simplifiedgraph{I}{\Lambda}{P}$ is isomorphic to $\commgraph{\Reeszero{G}{I}{\Lambda}{P}}$ if and only if $G$ is trivial. Furthermore, (the succeeding) Lemma~\ref{0Rees: connected component G(I,^) => connected component G(M0[...])} justifies the choice of $\simplifiedgraph{I}{\Lambda}{P}$ in determining whether $\commgraph{\Reeszero{G}{I}{\Lambda}{P}}$ is connected or not, and in the determination of the diameter of $\commgraph{\Reeszero{G}{I}{\Lambda}{P}}$.



\begin{lemma}\label{0Rees: adjacency in G(I,^) and G(M0[...])}
	Let $i,j\in I$ and $\lambda,\mu\in\Lambda$ be such that $\parens{i,\lambda}\neq\parens{j,\mu}$. The following statements are equivalent:
	\begin{enumerate}
		\item $\parens{i,\lambda}$ and $\parens{j,\mu}$ are adjacent in $\simplifiedgraph{I}{\Lambda}{P}$.
		\item For all $x,y\in G$ we have that $\parens{i,x,\lambda}$ and $\parens{j,y,\mu}$ are adjacent in $\commgraph{\Reeszero{G}{I}{\Lambda}{P}}$.
		\item There exist $x,y\in G$ such that $\parens{i,x,\lambda}$ and $\parens{j,y,\mu}$ are adjacent in $\commgraph{\Reeszero{G}{I}{\Lambda}{P}}$.
	\end{enumerate}
\end{lemma}

\begin{proof}
	We are going to prove that $1 \implies 2 \implies 3 \implies 1$.
	
	\medskip
	
	\textbf{Part 1} $[1\implies 2]$. Suppose that $\parens{i,\lambda}$ and $\parens{j,\mu}$ are adjacent in $\simplifiedgraph{I}{\Lambda}{P}$. Then $p_{\lambda j}=p_{\mu i}=0$, which implies, by Lemma~\ref{0Rees: commutativity}, that $\parens{i,x,\lambda}\parens{j,y,\mu}=\parens{j,y,\mu}\parens{i,x,\lambda}$ for all $x,y\in G$. Therefore for all $x,y\in G$ we have that $\parens{i,x,\lambda}$ and $\parens{j,y,\mu}$ are adjacent in $\commgraph{\Reeszero{G}{I}{\Lambda}{P}}$.
	
	\medskip
	
	\textbf{Part 2} $[2 \implies 3]$. Statement $3$ is an immediate consequence of statement $2$.
	
	\medskip
	
	\textbf{Part 3} $[3 \implies 1]$. Suppose that there exist $x,y\in G$ such that $\parens{i,x,\lambda}$ and $\parens{j,y,\mu}$ are adjacent in $\commgraph{\Reeszero{G}{I}{\Lambda}{P}}$. Then $\parens{i,x,\lambda}\parens{j,y,\mu}=\parens{j,y,\mu}\parens{i,x,\lambda}$. Additionally, we have that $\parens{i,\lambda}\neq\parens{j,\mu}$, which implies, by Lemma~\ref{0Rees: commutativity}, that $p_{\lambda j}=p_{\mu i}=0$. Thus $\parens{i,\lambda}$ and $\parens{j,\mu}$ are adjacent in $\simplifiedgraph{I}{\Lambda}{P}$.
\end{proof}

\begin{lemma}\label{0Rees: connected component G(I,^) => connected component G(M0[...])}
	Let $\mathcal{D}$ be a connected component of $\simplifiedgraph{I}{\Lambda}{P}$ and let $D$ be its vertex set. Then the subgraph $\mathcal{C}$ of $\commgraph{\Reeszero{G}{I}{\Lambda}{P}}$ induced by $\bigcup_{\parens{i,\lambda}\in D}\set{i}\times G\times\set{\lambda}$ is a connected component of $\commgraph{\Reeszero{G}{I}{\Lambda}{P}}$. Furthermore, if $\diam{\mathcal{D}}\geqslant 1$, then $\diam{\mathcal{D}}=\diam{\mathcal{C}}$.   
\end{lemma}

\begin{proof}
	Let $C=\bigcup_{\parens{i,\lambda}\in D}\set{i}\times G\times\set{\lambda}$. In order to show that $\mathcal{C}$ is a connected component of $\commgraph{\Reeszero{G}{I}{\Lambda}{P}}$ we need to prove two things: first we are going to see that for any $\parens{i,x,\lambda},\parens{j,y,\mu}\in C$ there exists a path from $\parens{i,x,\lambda}$ to $\parens{j,y,\mu}$ in $\commgraph{\Reeszero{G}{I}{\Lambda}{P}}$; and then we are going to see that if $\parens{i,x,\lambda}\in C$ and there is a path from $\parens{i,x,\lambda}$ to $\parens{j,y,\mu}$ in $\commgraph{\Reeszero{G}{I}{\Lambda}{P}}$, then $\parens{j,y,\mu}\in C$.
	
	\medskip
	
	\textbf{Part 1.} Let $\parens{i,x,\lambda},\parens{j,y,\mu}\in C$. We are going to see that there is a path from $\parens{i,x,\lambda}$ to $\parens{j,y,\mu}$ in $\commgraph{\Reeszero{G}{I}{\Lambda}{P}}$. It follows from the definition of $C$ that $\parens{i,\lambda},\parens{j,\mu}\in D$. We consider two cases: $\parens{i,\lambda}\neq\parens{j,\mu}$ and $\parens{i,\lambda}=\parens{j,\mu}$.
	
	\smallskip
	
	\textit{Case 1}: Assume that $\parens{i,\lambda}\neq\parens{j,\mu}$. Since $\mathcal{D}$ is a connected component of $\simplifiedgraph{I}{\Lambda}{P}$, then there exists a path from $\parens{i,\lambda}$ to $\parens{j,\mu}$ in $\simplifiedgraph{I}{\Lambda}{P}$. Let
	\begin{displaymath}
		\parens{i,\lambda}=\parens{i_1,\lambda_1}-\parens{i_2,\lambda_2}-\cdots-\parens{i_k,\lambda_k}=\parens{j,\mu}
	\end{displaymath}
	be one of those paths and assume that $k-1=\dist{\simplifiedgraph{I}{\Lambda}{P}}{\parens{i,\lambda}}{\parens{j,\mu}}$. (We observe that $k\geqslant 2$ because $\parens{i,\lambda}\neq\parens{j,\mu}$.) Then, by Lemma~\ref{0Rees: adjacency in G(I,^) and G(M0[...])}, $\parens{i_m,z,\lambda_m}$ is adjacent to $\parens{i_{m+1},w,\lambda_{m+1}}$ for all $z,w\in G$ and $m\in\X{k-1}$. Hence
	\begin{displaymath}
		\parens{i,x,\lambda}=\parens{i_1,x,\lambda_1}\sim \parens{i_2,y,\lambda_2}\sim\cdots\sim \parens{i_k,y,\lambda_k}=\parens{j,y,\mu}
	\end{displaymath}
	(in $\commgraph{\Reeszero{G}{I}{\Lambda}{P}}$), which implies that there exists a path from $\parens{i,x,\lambda}$ to $\parens{j,y,\mu}$ in $\commgraph{\Reeszero{G}{I}{\Lambda}{P}}$. Moreover, we have
	\begin{displaymath}
		\dist{\commgraph{\Reeszero{G}{I}{\Lambda}{P}}}{\parens{i,x,\lambda}}{\parens{j,y,\mu}}\leqslant k-1 =\dist{\simplifiedgraph{I}{\Lambda}{P}}{\parens{i,\lambda}}{\parens{j,\mu}}.
	\end{displaymath}
	
	\smallskip
	
	\textit{Case 2:} Assume that $\parens{i,\lambda}=\parens{j,\mu}$. We have $p_{\lambda i}=0$ or $p_{\lambda i}\in G$. 
	
	If $p_{\lambda i}=0$, then Lemma~\ref{0Rees: commutativity} ensures that $\parens{i,x,\lambda}\parens{i,y,\lambda}=\parens{i,y,\lambda}\parens{i,x,\lambda}$. Thus
	\begin{displaymath}
		\parens{i,x,\lambda}\sim\parens{i,y,\lambda}=\parens{j,y,\mu}
	\end{displaymath}
	(in $\commgraph{\Reeszero{G}{I}{\Lambda}{P}}$), which implies that there exists a path from $\parens{i,x,\lambda}$ to $\parens{j,y,\mu}$ in $\commgraph{\Reeszero{G}{I}{\Lambda}{P}}$ and that $\dist{\commgraph{\Reeszero{G}{I}{\Lambda}{P}}}{\parens{i,x,\lambda}}{\parens{j,y,\mu}}\leqslant 1$. 
	
	If $p_{\lambda i}\in G$, then $xp_{\lambda i}p_{\lambda i}^{-1}=x=p_{\lambda i}^{-1}p_{\lambda i}x$ and $yp_{\lambda i}p_{\lambda i}^{-1}=y=p_{\lambda i}^{-1}p_{\lambda i}y$, which implies, by Lemma~\ref{0Rees: commutativity}, that $\parens{i,p_{\lambda i}^{-1},\lambda}$ commutes with $\parens{i,x,\lambda}$ and $\parens{i,y,\lambda}$. Hence
	\begin{displaymath}
		\parens{i,x,\lambda}\sim\parens{i,p_{\lambda i}^{-1},\lambda}\sim\parens{i,y,\lambda}=\parens{j,y,\mu}
	\end{displaymath}
	(in $\commgraph{\Reeszero{G}{I}{\Lambda}{P}}$) and, consequently, there exists a path from $\parens{i,x,\lambda}$ to $\parens{j,y,\mu}$ in $\commgraph{\Reeszero{G}{I}{\Lambda}{P}}$. In addition, we have $\dist{\commgraph{\Reeszero{G}{I}{\Lambda}{P}}}{\parens{i,x,\lambda}}{\parens{j,y,\mu}}\leqslant 2$.
	
	\medskip
	
	\textbf{Part 2.} Let $\parens{i,x,\lambda}\in C$. Then $\parens{i,\lambda}\in D$. Let
	\begin{displaymath}
		\parens{i,x,\lambda}=\parens{i_1,x_1,\lambda_1}-\parens{i_2,x_2,\lambda_2}-\cdots-\parens{i_k,x_k,\lambda_k}
	\end{displaymath}
	be a path in $\commgraph{\Reeszero{G}{I}{\Lambda}{P}}$ and assume that $k-1=\dist{\commgraph{\Reeszero{G}{I}{\Lambda}{P}}}{\parens{i,x,\lambda}}{\parens{i_k,x_k,\lambda_k}}$. Our aim is to prove that $\parens{i_k,x_k,\lambda_k}\in C$. As a consequence of Lemma~\ref{0Rees: adjacency in G(I,^) and G(M0[...])} we have
	\begin{displaymath}
		\parens{i,\lambda}=\parens{i_1,\lambda_1}\sim\parens{i_2,\lambda_2}\sim\cdots\sim\parens{i_k,\lambda_k}
	\end{displaymath}
	(in $\simplifiedgraph{I}{\Lambda}{P}$), which implies that there is a path from $\parens{i,\lambda}=\parens{i_1,\lambda_1}$ to $\parens{i_k,\lambda_k}$ in $\simplifiedgraph{I}{\Lambda}{P}$. Since $\mathcal{D}$ is a connected component of $\simplifiedgraph{I}{\Lambda}{P}$ and $\parens{i,\lambda}\in D$, then we have $\parens{i_k,\lambda_k}\in D$. Thus $\parens{i_k,x_k,\lambda_k}\in C$. Additionally, we have
	\begin{displaymath}
		\dist{\simplifiedgraph{I}{\Lambda}{P}}{\parens{i,\lambda}}{\parens{i_k,\lambda_k}}\leqslant k-1=\dist{\commgraph{\Reeszero{G}{I}{\Lambda}{P}}}{\parens{i,x,\lambda}}{\parens{i_k,x_k,\lambda_k}}.
	\end{displaymath}
	
	We just proved (through part 1 and part 2) that $\mathcal{C}$ is a connected component of $\commgraph{\Reeszero{G}{I}{\Lambda}{P}}$.
	
	\medskip
	
	\textbf{Part 3.} Suppose that $\diam{\mathcal{D}}\geqslant 1$. We are going to demonstrate that $\diam{\mathcal{D}}=\diam{\mathcal{C}}$.
	
	First, we are going to verify that $\diam{\mathcal{D}}\leqslant\diam{\mathcal{C}}$. Let $\parens{i,\lambda},\parens{j,\mu}\in D$ be such that $\dist{\simplifiedgraph{I}{\Lambda}{P}}{\parens{i,\lambda}}{\parens{j,\mu}}=\diam{\mathcal{D}}$. Then $\parens{i,1_G,\lambda},\parens{j,1_G,\mu}\in C$ and, since $\mathcal{C}$ is a connected component of $\commgraph{\Reeszero{G}{I}{\Lambda}{P}}$, there is a path from $\parens{i,1_G,\lambda}$ to $\parens{j,1_G,\mu}$ in $\commgraph{\Reeszero{G}{I}{\Lambda}{P}}$. By part 2 of the proof we have
	\begin{displaymath}
		\diam{\mathcal{D}}=\dist{\simplifiedgraph{I}{\Lambda}{P}}{\parens{i,\lambda}}{\parens{j,\mu}}\leqslant\dist{\commgraph{\Reeszero{G}{I}{\Lambda}{P}}}{\parens{i,1_G,\lambda}}{\parens{j,1_G,\mu}}\leqslant\diam{\mathcal{C}}.
	\end{displaymath}
	
	Now we are going to show that $\diam{\mathcal{D}}\geqslant\diam{\mathcal{C}}$. Let $\parens{i,x,\lambda},\parens{j,y,\mu}\in C$ be such that $\dist{\commgraph{\Reeszero{G}{I}{\Lambda}{P}}}{\parens{i,x,\lambda}}{\parens{j,y,\mu}}=\diam{\mathcal{C}}$. We divide the proof into three cases.
	
	\smallskip
	
	\textit{Case 1:} Assume that $\parens{i,\lambda}\neq\parens{j,\mu}$. It follows from case 1 of part 1 of the proof that
	\begin{displaymath}
		\diam{\mathcal{C}}=\dist{\commgraph{\Reeszero{G}{I}{\Lambda}{P}}}{\parens{i,x,\lambda}}{\parens{j,y,\mu}}\leqslant
		\dist{\simplifiedgraph{I}{\Lambda}{P}}{\parens{i,\lambda}}{\parens{j,\mu}}\leqslant\diam{\mathcal{D}}.
	\end{displaymath}
	
	\smallskip
	
	\textit{Case 2:} Assume that $\parens{i,\lambda}=\parens{j,\mu}$ and $\diam{\mathcal{D}}=1$. This implies that $\mathcal{D}$ contains at least two vertices, that is, $\abs{D}>1$. Let $\parens{i',\lambda'}\in D$ be such that $\parens{i',\lambda'}\neq\parens{i,\lambda}$. Then we also have $\parens{i',\lambda}\neq\parens{i,\lambda'}$ and, consequently, $0\neq\dist{\simplifiedgraph{I}{\Lambda}{P}}{\parens{i',\lambda}}{\parens{i,\lambda'}}\leqslant\diam{\mathcal{D}}=1$, that is, $\dist{\simplifiedgraph{I}{\Lambda}{P}}{\parens{i',\lambda}}{\parens{i,\lambda'}}=1$. Hence $\parens{i',\lambda}$ and $\parens{i,\lambda'}$ are adjacent and, consequently, we have $p_{\lambda i}=p_{\lambda' i'}=0$. Since $p_{\lambda i}=0$, then it follows from Lemma~\ref{0Rees: commutativity} that $\parens{i,x,\lambda}\parens{i,y,\lambda}=\parens{i,y,\lambda}\parens{i,x,\lambda}$; that is, we have $\parens{i,x,\lambda}\sim\parens{i,y,\lambda}=\parens{j,y,\mu}$ (in $\commgraph{\Reeszero{G}{I}{\Lambda}{P}}$). Thus
	\begin{displaymath}
		\diam{\mathcal{C}}=\dist{\commgraph{\Reeszero{G}{I}{\Lambda}{P}}}{\parens{i,x,\lambda}}{\parens{j,y,\mu}}\leqslant 1=\diam{\mathcal{D}}.
	\end{displaymath}
	
	\smallskip
	
	\textit{Case 3:} Assume that $\parens{i,\lambda}=\parens{j,\mu}$ and $\diam{\mathcal{D}}\geqslant 2$. It follows from case 2 of part 1 of the proof that
	\begin{displaymath}
		\diam{\mathcal{C}}=\dist{\commgraph{\Reeszero{G}{I}{\Lambda}{P}}}{\parens{i,x,\lambda}}{\parens{j,y,\mu}}\leqslant 2\leqslant\diam{\mathcal{D}}.\qedhere
	\end{displaymath}
\end{proof}

Next we define what we will call $0$-closure method (and we will see an illustration of the method in Example~\ref{0Rees: example 0-closure method}). We will see later its importance: for instance, the method allows us to ascertain if $\commgraph{\Reeszero{G}{I}{\Lambda}{P}}$ is connected (Theorem~\ref{0Rees: connectedness}) and it can be used to find the connected components of $\commgraph{\Reeszero{G}{I}{\Lambda}{P}}$ when it is not connected (Theorem~\ref{0Rees: finding connected components}). Furthermore, it can also be used for determining the diameter of $\commgraph{\Reeszero{G}{I}{\Lambda}{P}}$ (respectively, diameters of some connected components of $\commgraph{\Reeszero{G}{I}{\Lambda}{P}}$) when it is connected (respectively, not connected) (Theorem~\ref{0Rees: diameter}).

\begin{definition}[$0$-closure method]\label{0Rees: definition 0-closure method}
	We call \defterm{$0$-closure method} the procedure described below:
	\begin{enumerate}
		\item Choose a zero entry from the matrix $P$. Assume that the chosen entry is the $\parens{\lambda,i}$-th entry of $P$ and let $Q_0=\Psubmatrix{P}{i}{\lambda}=\begin{bmatrix}
			p_{\lambda i}
		\end{bmatrix}=\begin{bmatrix}
			0
		\end{bmatrix}$.
		
		\item Let $k\geqslant 0$ and suppose that the submatrices $Q_0,Q_1,\ldots,Q_k$ have already been constructed. Assume that $Q_k=\Psubmatrix{P}{I_k}{\Lambda_k}$. We are going to construct a new matrix $Q_{k+1}$, if possible.
		\begin{enumerate}
			\item If there are zero entries in the rows and/or columns intersecting $Q_k$ that are not in $Q_k$ itself, then we mark all those zero entries and we define $Q_{k+1}$ as the smallest submatrix of $P$ such that:
			\begin{itemize}
				\item $Q_k$ is a submatrix of $Q_{k+1}$.
				\item The marked zero entries of $P$ are entries of $Q_{k+1}$.
			\end{itemize}
			That is, if $\Lambda'\subseteq\Lambda\setminus\Lambda_k$ is the set formed by the indices of the rows that contain marked zero entries and $I'\subseteq I\setminus I_k$ is the set formed by the indices of the columns that contain marked zero entries, then $Q_{k+1}=\Psubmatrix{P}{I_k\cup I'}{\Lambda_k\cup \Lambda'}$. Go to 2.
			
			\item Otherwise, halt and set $\zeroindex{i}{\lambda}=k$.
		\end{enumerate}
	\end{enumerate}
	
	The result of the $0$-closure method is the sequence $\parens{Q_0,Q_1,\ldots,Q_{\zeroindex{i}{\lambda}}}$. We call $Q_{\zeroindex{i}{\lambda}}$ the \defterm{\zeroclosuresubmatrix{i}{\lambda}} and we say that a matrix $Q$ is a \defterm{\zeroclosuresub} if there exist $i\in I$ and $\lambda\in\Lambda$ such that $p_{\lambda i}=0$ and $Q$ is the \zeroclosuresubmatrix{i}{\lambda}. Moreover, we denote by \defterm{step $0$} the part of the $0$-closure method where we choose an initial zero entry of $P$ and construct $Q_0$, and for each $k\in\X{\zeroindex{i}{\lambda}}$ we denote by \defterm{step $k$} the part of the $0$-closure method which starts with the (already constructed) matrix $Q_{k-1}$ and ends with the construction of the matrix $Q_k$. The unique \defterm{entry selected at step $0$} is the $\parens{\lambda, i}$-th entry of $P$, and for each $k\in\X{\zeroindex{i}{\lambda}}$ we call \defterm{entries selected at step $k$} the entries of $P$ that are entries of $Q_k$, but are not entries of $Q_{k-1}$.
\end{definition}

Note that, given a chosen initial zero entry, the behaviour of the algorithm is fixed and, consequently, the sequence of matrices it produces is also fixed. Thus the algorithm is deterministic.

In the $0$-closure method the only thing that matters is the positions of the zeros in the matrix $P$. Since $P$ and $\timeszero{P}$ have zeros in the exact same positions, we will often do the $0$-closure method in the matrix $\timeszero{P}$ instead.

\begin{example}\label{0Rees: example 0-closure method}
	Let $I=\set{1,2,3,4,5,6,7,8}$, $\Lambda=\set{1,2,3,4,5,6}$ and assume that $P$ is such that
	\begin{displaymath}
		\timeszero{P}=\begin{bNiceMatrix}[first-row,first-col]
			& 1 & 2 & 3 & 4 & 5 & 6 & 7 & 8 \\
			1 & 0 & \times & 0 & 0 & \times & \times & \times & \times \\
			2 & \times & 0 & \times & \times & 0 & \times & \times & 0 \\
			3 & 0 & \times & 0 & \times & \times & 0 & 0 & \times \\
			4 & 0 & \times & 0 & \times & \times & 0 & 0 & \times \\
			5 & \times & 0 & \times & \times & \times & \times & \times & \times \\ 
			6 & \times & \times & \times & \times & \times & \times & \times & \times
		\end{bNiceMatrix}.
	\end{displaymath}
	We are going to demonstrate how the $0$-closure method works.
	
	\begin{itemize}
		\item Step $0$: We choose (in red) the $\parens{4,6}$-th entry of $P$ to start the $0$-closure method and form the submatrix $Q_0=\Psubmatrix{P}{6}{4}$.
		\begin{displaymath}
			\begin{bNiceMatrix}[first-row,first-col]
				& 1 & 2 & 3 & 4 & 5 & 6 & 7 & 8 \\
				1 & 0 & \times & 0 & 0 & \times & \times & \times & \times \\
				2 & \times & 0 & \times & \times & 0 & \times & \times & 0 \\
				3 & 0 & \times & 0 & \times & \times & 0 & 0 & \times \\
				4 & 0 & \times & 0 & \times & \times & \cellcolor{Red} 0 & 0 & \times \\
				5 & \times & 0 & \times & \times & \times & \times & \times & \times \\ 
				6 & \times & \times & \times & \times & \times & \times & \times & \times
			\end{bNiceMatrix}
		\end{displaymath}
		
		\item Since row $4$ and column $6$ contain more zero entries, the $0$-closure method continues to step $1$.
		
		\item Step $1$: We mark (in yellow) all the uncoloured zero entries of $P$ that are in row $4$ or column $6$ (see matrix below on the left). We form the smallest submatrix of $P$ which contains the entries in red and in yellow. That matrix is $Q_1=\Psubmatrix{P}{\set{1,3,6,7}}{\set{3,4}}$ (which corresponds to the submatrix whose entries are coloured in red and yellow on the right matrix below).
		\begin{displaymath}
			\begin{bNiceMatrix}[first-row,first-col]
				& 1 & 2 & 3 & 4 & 5 & 6 & 7 & 8 \\
				1 & 0 & \times & 0 & 0 & \times & \times & \times & \times \\
				2 & \times & 0 & \times & \times & 0 & \times & \times & 0 \\
				3 & 0 & \times & 0 & \times & \times & \cellcolor{Goldenrod} 0 & 0 & \times \\
				4 & \cellcolor{Goldenrod} 0 & \times & \cellcolor{Goldenrod} 0 & \times & \times & \cellcolor{Red} 0 & \cellcolor{Goldenrod} 0 & \times \\
				5 & \times & 0 & \times & \times & \times & \times & \times & \times \\ 
				6 & \times & \times & \times & \times & \times & \times & \times & \times
			\end{bNiceMatrix}\quad
			\begin{bNiceMatrix}[first-row,first-col]
				& 1 & 2 & 3 & 4 & 5 & 6 & 7 & 8 \\
				1 & 0 & \times & 0 & 0 & \times & \times & \times & \times \\
				2 & \times & 0 & \times & \times & 0 & \times & \times & 0 \\
				3 & \cellcolor{Goldenrod} 0 & \times & \cellcolor{Goldenrod} 0 & \times & \times & \cellcolor{Goldenrod} 0 & \cellcolor{Goldenrod} 0 & \times \\
				4 & \cellcolor{Goldenrod} 0 & \times & \cellcolor{Goldenrod} 0 & \times & \times & \cellcolor{Red} 0 & \cellcolor{Goldenrod} 0 & \times \\
				5 & \times & 0 & \times & \times & \times & \times & \times & \times \\ 
				6 & \times & \times & \times & \times & \times & \times & \times & \times
			\end{bNiceMatrix}
		\end{displaymath}
		
		\item Since columns $1$ and $3$ contain more zero entries (besides the ones in yellow), the $0$-closure method continues to step $2$.
		
		\item Step $2$: We mark (in green) all the uncoloured zero entries of $P$ that are in rows $3$ or $4$ or in columns $1$, $3$, $6$ or $7$ (see matrix below on the left). We form the smallest submatrix of $P$ that contains the entries in red, yellow and green, which corresponds to the matrix $Q_2=\Psubmatrix{P}{\set{1,3,6,7}}{\set{1,3,4}}$ (see the matrix below on the right).
		\begin{displaymath}
			\begin{bNiceMatrix}[first-row,first-col]
				& 1 & 2 & 3 & 4 & 5 & 6 & 7 & 8 \\
				1 & \cellcolor{Green} 0 & \times & \cellcolor{Green} 0 & 0 & \times & \times & \times & \times \\
				2 & \times & 0 & \times & \times & 0 & \times & \times & 0 \\
				3 & \cellcolor{Goldenrod} 0 & \times & \cellcolor{Goldenrod} 0 & \times & \times & \cellcolor{Goldenrod} 0 & \cellcolor{Goldenrod} 0 & \times \\
				4 & \cellcolor{Goldenrod} 0 & \times & \cellcolor{Goldenrod} 0 & \times & \times & \cellcolor{Red} 0 & \cellcolor{Goldenrod} 0 & \times \\
				5 & \times & 0 & \times & \times & \times & \times & \times & \times \\ 
				6 & \times & \times & \times & \times & \times & \times & \times & \times
			\end{bNiceMatrix}\quad
			\begin{bNiceMatrix}[first-row,first-col]
				& 1 & 2 & 3 & 4 & 5 & 6 & 7 & 8 \\
				1 & \cellcolor{Green} 0 & \times & \cellcolor{Green} 0 & 0 & \times & \cellcolor{Green} \times & \cellcolor{Green} \times & \times \\
				2 & \times & 0 & \times & \times & 0 & \times & \times & 0 \\
				3 & \cellcolor{Goldenrod} 0 & \times & \cellcolor{Goldenrod} 0 & \times & \times & \cellcolor{Goldenrod} 0 & \cellcolor{Goldenrod} 0 & \times \\
				4 & \cellcolor{Goldenrod} 0 & \times & \cellcolor{Goldenrod} 0 & \times & \times & \cellcolor{Red} 0 & \cellcolor{Goldenrod} 0 & \times \\
				5 & \times & 0 & \times & \times & \times & \times & \times & \times \\ 
				6 & \times & \times & \times & \times & \times & \times & \times & \times
			\end{bNiceMatrix}
		\end{displaymath}
		
		\item Row $1$ of $P$ contains an uncoloured zero entry. Consequently, the $0$-closure method continues to step $3$.
		
		\item Step $3$: We mark (in blue) all the uncoloured zero entries of $P$ that are in rows $1$, $3$ or $4$ or in columns $1$, $3$, $6$ or $7$ (see the left matrix below). We notice that there are no uncoloured zero entries of $P$ in columns $1$, $3$, $6$ or $7$. We form the smallest submatrix of $P$ that contains the entries in red, yellow, green and blue, which corresponds to the matrix $Q_3=\Psubmatrix{P}{\set{1,3,4,6,7}}{\set{1,3,4}}$ (see the right matrix below).
		\begin{displaymath}
			\begin{bNiceMatrix}[first-row,first-col]
				& 1 & 2 & 3 & 4 & 5 & 6 & 7 & 8 \\
				1 & \cellcolor{Green} 0 & \times & \cellcolor{Green} 0 & \cellcolor{cyan!40} 0 & \times & \cellcolor{Green} \times & \cellcolor{Green} \times & \times \\
				2 & \times & 0 & \times & \times & 0 & \times & \times & 0 \\
				3 & \cellcolor{Goldenrod} 0 & \times & \cellcolor{Goldenrod} 0 & \times & \times & \cellcolor{Goldenrod} 0 & \cellcolor{Goldenrod} 0 & \times \\
				4 & \cellcolor{Goldenrod} 0 & \times & \cellcolor{Goldenrod} 0 & \times & \times & \cellcolor{Red} 0 & \cellcolor{Goldenrod} 0 & \times \\
				5 & \times & 0 & \times & \times & \times & \times & \times & \times \\ 
				6 & \times & \times & \times & \times & \times & \times & \times & \times
			\end{bNiceMatrix}\quad
			\begin{bNiceMatrix}[first-row,first-col]
				& 1 & 2 & 3 & 4 & 5 & 6 & 7 & 8 \\
				1 & \cellcolor{Green} 0 & \times & \cellcolor{Green} 0 & \cellcolor{cyan!40} 0 & \times & \cellcolor{Green} \times & \cellcolor{Green} \times & \times \\
				2 & \times & 0 & \times & \times & 0 & \times & \times & 0 \\
				3 & \cellcolor{Goldenrod} 0 & \times & \cellcolor{Goldenrod} 0 & \cellcolor{cyan!40} \times & \times & \cellcolor{Goldenrod} 0 & \cellcolor{Goldenrod} 0 & \times \\
				4 & \cellcolor{Goldenrod} 0 & \times & \cellcolor{Goldenrod} 0 & \cellcolor{cyan!40} \times & \times & \cellcolor{Red} 0 & \cellcolor{Goldenrod} 0 & \times \\
				5 & \times & 0 & \times & \times & \times & \times & \times & \times \\ 
				6 & \times & \times & \times & \times & \times & \times & \times & \times
			\end{bNiceMatrix}
		\end{displaymath}
		
		\item There are no uncoloured zero entries of $P$ in rows $1$, $3$ and $4$ and in columns $1$, $3$, $4$, $6$ and $7$. This implies that the $0$-closure method ends at step $3$. Then $\zeroindex{6}{4}=3$ and $Q_3=\Psubmatrix{P}{\set{1,3,4,6,7}}{\set{1,3,4}}$ is the \zeroclosuresubmatrix{6}{4}.
	\end{itemize}
\end{example}

The next two lemmata describe properties of the matrices belonging to the sequence of submatrices of $P$ constructed in the $0$-closure method.

\begin{lemma}\label{0Rees: every row/column of 0-closure submatrix has a 0}
	Let $i\in I$ and $\lambda\in\Lambda$ be such that $p_{\lambda i}=0$. Suppose that we start the $0$-closure method with the $\parens{\lambda, i}$-th entry of $P$. Let $\parens{Q_0,\ldots,Q_{\zeroindex{i}{\lambda}}}$ be the sequence of submatrices of $P$ obtained from the $0$-closure method. Then for each $k\in\set{0,\ldots,\zeroindex{i}{\lambda}}$ we have that every row and every column of $Q_k$ contains at least one zero entry.
\end{lemma}

\begin{proof}

	The result follows from the fact that $Q_0$ consists only of a zero entry, and that every row and column added at each step $k$ of the $0$-closure method (to construct $Q_k$) contains at least one zero which was previously marked at the beginning of that step.
\end{proof}

We observe that Lemma~\ref{0Rees: every row/column of 0-closure submatrix has a 0} implies, in particular, that every row and every column of a \zeroclosuresub\ contains at least one zero entry.

Let $i\in I$ and $\lambda\in\Lambda$ be such that $p_{\lambda i}=0$. The following result characterizes the entries $\parens{\mu,j}$ of the \zeroclosuresubmatrix{i}{\lambda} in terms of the distance (in $\simplifiedgraph{I}{\Lambda}{P}$) between the vertices $\parens{i,\lambda}$ and $\parens{j,\mu}$. 

\begin{lemma}\label{0Rees: d(0,?) in G(I,^)}
	Let $i\in I$ and $\lambda\in\Lambda$ be such that $p_{\lambda i}=0$. Suppose that we start the $0$-closure method with the $\parens{\lambda, i}$-th entry of $P$. Then for all $j\in I$, $\mu\in\Lambda$ and $k\in\set{0,\ldots,\zeroindex{i}{\lambda}}$ we have that the $\parens{\mu,j}$-th entry of $P$ is selected at step $k$ of the $0$-closure method if and only if $\dist{\simplifiedgraph{I}{\Lambda}{P}}{\parens{i,\lambda}}{\parens{j,\mu}}=k$.
	
	
\end{lemma}

\begin{proof}
	For each $k\in\set{0,\ldots,\zeroindex{i}{\lambda}}$ let
	\begin{align*}
		A_k&=\gset{\parens{j,\mu}\in I\times\Lambda}{\text{the } \parens{\mu,j}\text{-th entry of } P \text{ is selected at step } k}\\
		\shortintertext{and}
		B_k&=\gset{\parens{j,\mu}\in I\times\Lambda}{\dist{\simplifiedgraph{I}{\Lambda}{P}}{\parens{i,\lambda}}{\parens{j,\mu}}=k}.
	\end{align*}
	It is clear that proving Lemma~\ref{0Rees: d(0,?) in G(I,^)} is equivalent to proving that $A_k=B_k$ for all $k\in\set{0,\ldots,\zeroindex{i}{\lambda}}$. We are going to prove this by induction on the step $k$ of the $0$-closure method.
	
	Suppose that $k=0$. Then it is immediate that $A_0=\set{\parens{i,\lambda}}=B_0$.
	
	Now suppose that $0<k\leqslant \zeroindex{i}{\lambda}$ and that for all $l\in\set{0,\ldots,k-1}$ we have $A_l=B_l$. Let $Q_{k-1}=\Psubmatrix{P}{I_{k-1}}{\Lambda_{k-1}}$ be the submatrix of $P$ constructed in step $k-1$. Since $Q_{k-1}$ corresponds to the submatrix of $P$ formed by the entries of $P$ selected from step $0$ to step $k-1$, then, by the induction hypothesis,
	\begin{displaymath}
		I_{k-1}\times\Lambda_{k-1}=\bigcup_{l=0}^{k-1} A_l=\bigcup_{l=0}^{k-1} B_l=\gset{\parens{j,\mu}\in I\times \Lambda}{\dist{\simplifiedgraph{I}{\Lambda}{P}}{\parens{i,\lambda}}{\parens{j,\mu}}\leqslant k-1}.
	\end{displaymath}
	
	Let
	\begin{align*}
		I'&=\gset{j\in I\setminus I_{k-1}}{p_{\mu j}=0 \text{ for some } \mu \in \Lambda_{k-1}}\\
		\shortintertext{and}
		\Lambda'&=\gset{\mu\in \Lambda\setminus \Lambda_{k-1}}{p_{\mu j}=0 \text{ for some } j \in I_{k-1}},
	\end{align*}
	that is, $I'$ (respectively, $\Lambda'$) corresponds to the set of indices of the columns (respectively, rows) of $P$ that contain a zero entry that appears in a row (respectively, column) that intersects $Q_{k-1}$, but that is not an entry of $Q_{k-1}$. Since $k\leqslant \zeroindex{i}{\lambda}$, then this means that there are more zero entries in the rows and/or columns whose indices belong to $\Lambda_{k-1}$ and $I_{k-1}$, respectively, that are not entries of $Q_{k-1}$; that is, $I'\cup \Lambda'\neq\emptyset$. (Note that we can have $I'=\emptyset$ or $\Lambda'=\emptyset$.) Let $I_k=I_{k-1}\cup I'$ and $\Lambda_k=\Lambda_{k-1}\cup\Lambda'$. Then the new submatrix of $P$ that we construct in step $k$ is $Q_k=\Psubmatrix{P}{I_k}{\Lambda_k}$.
	
	We begin by proving that $A_k\subseteq B_k$. Let $\parens{j,\mu}\in A_k$. Then the $\parens{\mu,j}$-th entry of $P$ is selected at step $k$. This implies that $\parens{\mu,j}$ is an entry of $Q_k$ and, consequently, $\parens{j,\mu}\in I_k\times\Lambda_k$. Furthermore, it also implies that $\parens{j,\mu}\notin \bigcup_{l=0}^{k-1} A_l=I_{k-1}\times\Lambda_{k-1}$. Hence $\parens{j,\mu}\in \parens{I_k\times\Lambda_k}\setminus\parens{I_{k-1}\times\Lambda_{k-1}}$ and, consequently, $\parens{j,\mu}\in I' \times\Lambda_{k-1}$ or $\parens{j,\mu}\in I_{k-1}\times\Lambda'$ or $\parens{j,\mu}\in I'\times\Lambda'$.

	\smallskip
	
	\textit{Case 1:} Assume that $\parens{j,\mu}\in I' \times\Lambda_{k-1}$. It follows from the definition of $I'$ that there exists $\mu'\in\Lambda_{k-1}$ such that $p_{\mu' j}=0$. In addition, Lemma~\ref{0Rees: every row/column of 0-closure submatrix has a 0} guarantees the existence of a zero entry in row $\mu$ of $Q_{k-1}$. Hence there exists $j'\in I_{k-1}$ such that $p_{\mu j'}=0$.
	
	The following diagram illustrates the indices mentioned so far. In the diagram we have an arrow from entry $\parens{\mu,j}$ to entry $\parens{\mu',j'}$, and an arrow from entry $\parens{\mu',j'}$ to entry $\parens{\lambda,i}$. These arrows indicate that we are going to show that there are paths (in $\simplifiedgraph{I}{\Lambda}{P}$) from $\parens{j,\mu}$ to $\parens{j',\mu'}$, and from $\parens{j',\mu'}$ to $\parens{i,\lambda}$; and that we will use them to prove the existence of a path from $\parens{j,\mu}$ to $\parens{i,\lambda}$ and to determine the distance between $\parens{j,\mu}$ and $\parens{i,\lambda}$.
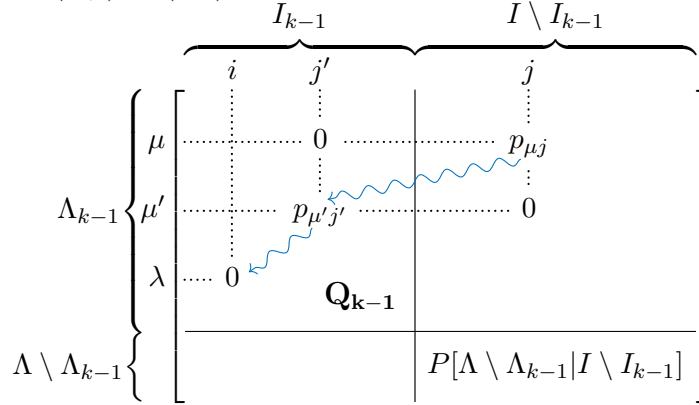
	\begin{displaymath}
		\begin{bNiceArray}{cccccc|ccccccccc}[first-row,first-col,margin]
			\rule{0pt}{22pt} & & i & & j' & & &&  & & j &&&&&\\
			\\
			\mu & & & & 0  &  &  &  &  & & p_{\mu j} &&&&&  \\
			\\
			\mu' & & & & p_{\mu' j'} & & & & && 0 &&&&& \\
			& & & & & & & & & &&&&&& \\
			\lambda & & 0 & & & & & & & & & & & & & \\
			\\
			\hline
			& & & & & & & \Block{2-9}{\Psubmatrix{P}{I\setminus I_{k-1}}{\Lambda\setminus\Lambda_{k-1}}} &&&&&&&& \\
			\\
			\CodeAfter
			\begin{tikzpicture}
				\node[node font = \large] at (7-|5) {$\mathbf{Q_{k-1}}$};
				\node[xshift=-1.09cm] at (4.5-|1) {$\Lambda_{k-1}\sizeddelimiter{7}{\{}$};
				\node[xshift=-1.39cm] at (9-|1) {$\Lambda\setminus\Lambda_{k-1}\sizeddelimiter{2}{\{}$};
				\begin{scope}[dash pattern=on 0pt off 2.3pt,line cap=round, thick]
					\draw (2.5-|1) -- ($(2.5-|4)+(3mm,0mm)$);
					\draw ($(2.5-|5)+(-2mm,0mm)$) -- ($(2.5-|10)+(0mm,0mm)$);
					\draw (4.5-|1) -- (4.5-|4);
					\draw (4.5-|5) -- ($(4.5-|10)+(1.8mm,0mm)$);
					\draw (6.5-|1) -- ($(6.5-|2)+(0.5mm,0mm)$);
					\draw (2.5|-1) -- ($(2.5|-6)+(0mm,0.5mm)$);
					\draw (4.5|-1) -- ($(4.5|-2)+(0mm,0.5mm)$);
					\draw (4.5|-3) -- (4.5|-4);
					\draw (10.5|-1) -- (10.5|-2);
					\draw ($(10.5|-3.5)+(0mm,0.7mm)$) -- ($(10.5|-4)+(0mm,0.5mm)$);
				\end{scope}
				
				\begin{scope}[decoration = {snake, amplitude = 0.7mm}, ->, RoyalBlue]
					\draw[decorate] ($(4.5|-5)+(-1mm,0mm)$) -- ($(6.5-|3)+(-0.5mm,1mm)$);
					\draw[decorate] ($(3-|10.5)+(-1mm,0mm)$) -- ($(4.5|-4.5)+(1mm,1.5mm)$);
				\end{scope}
			\end{tikzpicture}
			\OverBrace[yshift=15pt]{1-1}{1-6}{I_{k-1}}
			\OverBrace[yshift=15pt]{1-7}{1-15}{I\setminus I_{k-1}}
		\end{bNiceArray}
	\end{displaymath}
	
	Since $p_{\mu j'}=p_{\mu' j}=0$, then $\parens{j,\mu}-\parens{j',\mu'}$ is a path in $\simplifiedgraph{I}{\Lambda}{P}$. Furthermore, since $\parens{j',\mu'}\in I_{k-1}\times\Lambda_{k-1}$, then $\dist{\simplifiedgraph{I}{\Lambda}{P}}{\parens{i,\lambda}}{\parens{j',\mu'}}\leqslant k-1$. Thus $\dist{\simplifiedgraph{I}{\Lambda}{P}}{\parens{i,\lambda}}{\parens{j,\mu}}\leqslant k$. Moreover, $\parens{j,\mu}\notin I_{k-1}\times\Lambda_{k-1}$, which implies that $\dist{\simplifiedgraph{I}{\Lambda}{P}}{\parens{i,\lambda}}{\parens{j,\mu}}>k-1$. Therefore $\dist{\simplifiedgraph{I}{\Lambda}{P}}{\parens{i,\lambda}}{\parens{j,\mu}}=k$ and, consequently, $\parens{j,\mu}\in B_k$.
	
	\smallskip

	\textit{Case 2:} Assume that $\parens{j,\mu}\in I_{k-1}\times\Lambda'$. We can show, in a similar way to case 1, that $\parens{j,\mu}\in B_k$.
	
	\smallskip
	
	\textit{Case 3:} Assume that $\parens{j,\mu}\in I'\times\Lambda'$. It follows from the definitions of $I'$ and $\Lambda'$ that there exist $\mu'\in\Lambda_{k-1}$ and $j'\in I_{k-1}$ such that $p_{\mu' j}=p_{\mu j'}=0$.
	
	We can see an illustration of these indices in the diagram below. The arrows from entry $\parens{\mu,j}$ to entry $\parens{\mu',j'}$, and from entry $\parens{\mu',j'}$ to entry $\parens{\lambda,i}$, indicate that we are going to establish the existence of paths (in $\simplifiedgraph{I}{\Lambda}{P}$) from $\parens{j,\mu}$ to $\parens{j',\mu'}$, and from $\parens{j',\mu'}$ to $\parens{i,\lambda}$. These paths will be used to construct a path from $\parens{j,\mu}$ to $\parens{i,\lambda}$ and to determine the distance between $\parens{j,\mu}$ and $\parens{i,\lambda}$.
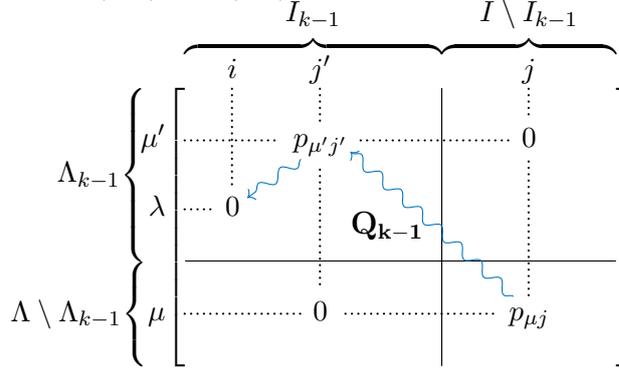
	\begin{displaymath}
		\begin{bNiceArray}{ccccccc|ccccc}[first-row,first-col,margin]
			\rule{0pt}{22pt} & & i & & j' & & & &&  & j & \\
			\\
			\mu' & & & & p_{\mu' j'} & & & & && 0 &&\\
			\\
			\lambda & & 0 & & & &  &  &  &  & && \\
			\\
			\hline
			\\
			\mu & & & & 0 & & & & && p_{\mu j} && \\
			\\
			\CodeAfter
			\begin{tikzpicture}
				\node[node font = \large] at (5-|6) {$\mathbf{Q_{k-1}}$};
				\node[xshift=-1.09cm] at (3.5-|1) {$\Lambda_{k-1}\sizeddelimiter{5}{\{}$};
				\node[xshift=-1.39cm] at (7.5-|1) {$\Lambda\setminus\Lambda_{k-1}\sizeddelimiter{3}{\{}$};
				\begin{scope}[dash pattern=on 0pt off 2.3pt,line cap=round, thick]
					\draw (2.5-|1) -- (2.5-|4);
					\draw (2.5-|5) -- ($(2.5-|10)+(2mm,0mm)$);
					\draw (4.5-|1) -- ($(4.5-|2)+(0.5mm,0mm)$);
					\draw (7.5-|1) -- ($(7.5-|4)+(3mm,0mm)$);
					\draw ($(7.5-|5)+(-2mm,0mm)$) -- (7.5-|10);
					\draw (2.5|-1) -- ($(2.5|-4)+(0mm,0.6mm)$);
					\draw (4.5|-1) -- (4.5|-2);
					\draw ($(4.5|-3.5)+(0mm,0.7mm)$) -- ($(4.5|-7)+(0mm,1mm)$);
					\draw (10.5|-1) -- ($(10.5|-2)+(0mm,0.5mm)$);
					\draw ($(10.5|-3)+(0mm,-0.5mm)$) -- (10.5|-7);
				\end{scope}
				
				\begin{scope}[decoration = {snake, amplitude = 0.7mm}, ->, RoyalBlue]
					\draw[decorate] ($(2.5-|4.5)+(-2.5mm,-2.5mm)$) -- ($(4.5-|2.5)+(2mm,1.5mm)$);
					\draw[decorate] ($(7-|10.5)+(-2mm,0mm)$) -- ($(2.5-|4.5)+(4mm,-1.5mm)$);
				\end{scope}
			\end{tikzpicture}
			\OverBrace[yshift=15pt]{1-1}{1-7}{I_{k-1}}
			\OverBrace[yshift=15pt]{1-8}{1-12}{I\setminus I_{k-1}}
		\end{bNiceArray}
	\end{displaymath}
	
	We have that $\parens{j',\mu'}-\parens{j,\mu}$ is a path in $\simplifiedgraph{I}{\Lambda}{P}$. Since $\parens{j',\mu'}\in I_{k-1}\times\Lambda_{k-1}$, then $\dist{\simplifiedgraph{I}{\Lambda}{P}}{\parens{i,\lambda}}{\parens{j',\mu'}}\leqslant k-1$, which implies that $\dist{\simplifiedgraph{I}{\Lambda}{P}}{\parens{i,\lambda}}{\parens{j,\mu}}\leqslant k$. Due to the fact that $\parens{j,\mu}\notin I_{k-1}\times\Lambda_{k-1}$, we also have $\dist{\simplifiedgraph{I}{\Lambda}{P}}{\parens{i,\lambda}}{\parens{j,\mu}}>k-1$. Therefore $\dist{\simplifiedgraph{I}{\Lambda}{P}}{\parens{i,\lambda}}{\parens{j,\mu}}=k$, which implies that $\parens{j,\mu}\in B_k$.
	
	\smallskip
	
	Now we are going to see that $B_k\subseteq A_k$. Let $\parens{j,\mu}\in B_k$. Then $\dist{\simplifiedgraph{I}{\Lambda}{P}}{\parens{i,\lambda}}{\parens{j,\mu}}=k$. Then there exists a path from $\parens{i,\lambda}$ to $\parens{j,\mu}$ in $\simplifiedgraph{I}{\Lambda}{P}$ of length $k$. Let
	\begin{displaymath}
		\parens{i,\lambda}=\parens{i_1,\lambda_1}-\parens{i_2,\lambda_2}-\cdots-\parens{i_{k+1},\lambda_{k+1}}=\parens{j,\mu}
	\end{displaymath}
	be one of those paths. Then we have $\dist{\simplifiedgraph{I}{\Lambda}{P}}{\parens{i,\lambda}}{\parens{i_k,\lambda_k}}=k-1$, which implies that $\parens{i_k,\lambda_k}\in I_{k-1}\times\Lambda_{k-1}$. We also have $\parens{j,\mu}\notin I_{k-1}\times \Lambda_{k-1}$ because $\dist{\simplifiedgraph{I}{\Lambda}{P}}{\parens{i,\lambda}}{\parens{j,\mu}}=k$. In addition, we have $p_{\lambda_k j}=p_{\mu i_k}=0$ because $\parens{i_k,\lambda_k}$ is adjacent to $\parens{i_{k+1},\lambda_{k+1}}=\parens{j,\mu}$. Since $\parens{j,\lambda_k}\in \parens{I\setminus I_{k-1}}\times\Lambda_{k-1}$ and $p_{\lambda_k j}=0$, then $j\in I'\subseteq I_k$; and since $\parens{i_k,\mu}\in I_{k-1}\times \Lambda\setminus\Lambda_{k-1}$ and $p_{\mu i_k}=0$, then $\mu\in\Lambda'\subseteq \Lambda_k$. Hence $\parens{j,\mu}\in I_k\times\Lambda_k$, that is, $\parens{\mu,j}$ is an entry of $Q_k$. Additionally, we have that $\parens{\mu,j}$ is not an entry of $Q_{k-1}$ because $\parens{j,\mu}\notin I_{k-1}\times \Lambda_{k-1}$. Thus the $\parens{\mu,j}$-th entry of $P$ must have been selected at step $k$ of the $0$-closure method, which implies that $\parens{j,\mu}\in A_k$.
\end{proof}

Let $i\in I$ and $\lambda\in\Lambda$ be such that $p_{\lambda i}=0$ and let $Q=\Psubmatrix{P}{I_Q}{\Lambda_Q}$ be the \zeroclosuresubmatrix{i}{\lambda}. We observe that Lemma~\ref{0Rees: d(0,?) in G(I,^)} implies that $\dist{\simplifiedgraph{I}{\Lambda}{P}}{\parens{i,\lambda}}{\parens{j,\mu}}\leqslant \zeroindex{i}{\lambda}$ for all $j\in I_Q$ and $\lambda\in\Lambda_Q$. Furthermore, it also guarantees that there exist $j\in I_Q$ and $\mu\in\Lambda_Q$ such that $\dist{\simplifiedgraph{I}{\Lambda}{P}}{\parens{i,\lambda}}{\parens{j,\mu}}=\zeroindex{i}{\lambda}$.

Moreover, it follows from Lemma~\ref{0Rees: d(0,?) in G(I,^)} that performing the $0$-closure method is equivalent to drawing $\simplifiedgraph{I}{\Lambda}{P}$, choosing a vertex $\parens{i,\lambda}$ such that $p_{\lambda i}=0$ (step $0$), selecting the vertices of $\simplifiedgraph{I}{\Lambda}{P}$ adjacent to $\parens{i,\lambda}$ (step $1$), selecting the vertices of $\simplifiedgraph{I}{\Lambda}{P}$ adjacent to the ones selected in the previous step (step $2$), and so on, until we select all the vertices whose distance to $\parens{i,\lambda}$ is at most $\zeroindex{i}{\lambda}$. It follows from Lemma~\ref{0Rees: connected component G(I,^,P) G(M0[...])} below that the vertices at distance $\zeroindex{i}{\lambda}$ are precisely the ones that are furthest away from $\parens{i,\lambda}$ (and in its connected component). This means that the procedure we described for the graph $\simplifiedgraph{I}{\Lambda}{P}$ ends when no new vertices are selected; that is, when all the vertices of the connected component of $\parens{i,\lambda}$ have been selected.

\begin{example}
	In Example~\ref{0Rees: example 0-closure method} the entries of $P$ coloured in red, yellow, green and blue are the entries selected at steps $0$, $1$, $2$ and $3$, respectively, of the $0$-closure method. It follows from Lemma~\ref{0Rees: d(0,?) in G(I,^)} that the entries $\parens{\mu,j}$ of $P$ in red, blue, yellow, green and blue are precisely the ones such that $\dist{\simplifiedgraph{I}{\Lambda}{P}}{\parens{4,6}}{\parens{j,\mu}}$ is equal to $0$, $1$, $2$ and $3$, respectively.
\end{example}

Lemma~\ref{0Rees: connected component G(I,^,P) G(M0[...])} provides a first step in showing the importance of the $0$-closure method for studying connectedness in $\commgraph{\Reeszero{G}{I}{\Lambda}{P}}$. More precisely, it highlights the significance of identifying $0$-closure submatrices of $P$ as a method for finding connected components of $\commgraph{\Reeszero{G}{I}{\Lambda}{P}}$.

\begin{lemma}\label{0Rees: connected component G(I,^,P) G(M0[...])}
	Let $Q=\Psubmatrix{P}{I_Q}{\Lambda_Q}$ be a \zeroclosuresub. Then
	\begin{enumerate}
		\item The subgraph of $\simplifiedgraph{I}{\Lambda}{P}$ induced by $I_Q\times\Lambda_Q$ is a connected component of $\simplifiedgraph{I}{\Lambda}{P}$.
		
		\item The subgraph of $\commgraph{\Reeszero{G}{I}{\Lambda}{P}}$ induced by $I_Q\times G\times \Lambda_Q$ is a connected component of $\commgraph{\Reeszero{G}{I}{\Lambda}{P}}$.
	\end{enumerate}
\end{lemma}

\begin{proof}
	\textbf{Part 1.} Since $Q$ is a \zeroclosuresub, then there exist $i\in I$ and $\lambda\in\Lambda$ such that $p_{\lambda i}=0$ and $Q$ is the \zeroclosuresubmatrix{i}{\lambda}. Let $\mathcal{C}$ be the subgraph of $\simplifiedgraph{I}{\Lambda}{P}$ induced by $I_Q\times \Lambda_Q$. Our goal is to prove that $\mathcal{C}$ is a connected component of $\simplifiedgraph{I}{\Lambda}{P}$. It follows from Lemma~\ref{0Rees: d(0,?) in G(I,^)} that $I_Q\times\Lambda_Q$ comprises all the vertices of $\simplifiedgraph{I}{\Lambda}{P}$ whose distance to $\parens{i,\lambda}$ is at most $\zeroindex{i}{\lambda}$. In order to demonstrate that $\mathcal{C}$ is a connected component of $\simplifiedgraph{I}{\Lambda}{P}$, it is enough to show that the vertices of $\simplifiedgraph{I}{\Lambda}{P}$ which are furthest away from $\parens{i,\lambda}$ (and in the same connected component as $\parens{i,\lambda}$) are precisely the ones at distance $\zeroindex{i}{\lambda}$. This can be done by analysing the vertices which are adjacent to the ones whose distance to $\parens{i,\lambda}$ is $\zeroindex{i}{\lambda}$.

	Let $\parens{j,\mu}$ be a vertex of $\simplifiedgraph{I}{\Lambda}{P}$ such that $\dist{\simplifiedgraph{I}{\Lambda}{P}}{\parens{i,\lambda}}{\parens{j,\mu}}=\zeroindex{i}{\lambda}$ and let $\parens{j',\mu'}$ be a vertex of $\simplifiedgraph{I}{\Lambda}{P}$ adjacent to $\parens{j,\mu}$. Then $p_{\mu j'}=p_{\mu' j}=0$. Due to the fact that $Q=\Psubmatrix{P}{I_Q}{\Lambda_Q}$ is the submatrix of $P$ obtained at the end of $0$-closure method, we have that all the zero entries located in the rows (respectively, columns) whose indices belong to $\Lambda_Q$ (respectively, $I_Q$) are entries of $Q=\Psubmatrix{P}{I_Q}{\Lambda_Q}$. This implies that $p_{\nu t}\neq 0$ for all $\parens{t,\nu}\in \parens{\parens{I\setminus I_Q}\times\Lambda_Q}\cup\parens{I_Q\times\parens{\Lambda\setminus\Lambda_Q}}$. Then, since $\parens{j,\mu}\in I_Q\times\Lambda_Q$ (because $\dist{\simplifiedgraph{I}{\Lambda}{P}}{\parens{i,\lambda}}{\parens{j,\mu}}=\zeroindex{i}{\lambda}$), we must have $\parens{j',\mu'}\in I_Q\times\Lambda_Q$, which implies that $\dist{\simplifiedgraph{I}{\Lambda}{P}}{\parens{i,\lambda}}{\parens{j',\mu'}}\leqslant\zeroindex{i}{\lambda}$.
	
	We just proved that the vertices of $\simplifiedgraph{I}{\Lambda}{P}$ whose distance to $\parens{i,\lambda}$ is $\zeroindex{i}{\lambda}$ are only adjacent to vertices of $\simplifiedgraph{I}{\Lambda}{P}$ whose distance to $\parens{i,\lambda}$ is at most $\zeroindex{i}{\lambda}$. This implies that all the vertices of $\simplifiedgraph{I}{\Lambda}{P}$ which are in the same connected component as $\parens{i,\lambda}$ are precisely the ones whose distance to $\parens{i,\lambda}$ is at most $\zeroindex{i}{\lambda}$; that is, the vertices belonging to $I_Q\times\Lambda_Q$. Therefore $\mathcal{C}$ is a connected component of $\simplifiedgraph{I}{\Lambda}{P}$.

	\medskip

	\textbf{Part 2.} By part 1 of the theorem, we have that $\mathcal{C}$ is a connected component of $\simplifiedgraph{I}{\Lambda}{P}$. Hence it follows from Lemma~\ref{0Rees: connected component G(I,^) => connected component G(M0[...])} that the subgraph of $\commgraph{\Reeszero{G}{I}{\Lambda}{P}}$ induced by
	\begin{displaymath}
		\bigcup_{\parens{j,\mu}\in C} \set{j}\times G\times \set{\mu}=\bigcup_{\parens{j,\mu}\in I_Q\times \Lambda_Q} \set{j}\times G\times \set{\mu} = I_Q\times G\times \Lambda_Q
	\end{displaymath}
	is a connected component of $\commgraph{\Reeszero{G}{I}{\Lambda}{P}}$.
\end{proof}


	

The following lemma allows us to compare $0$-closure submatrices of $P$ constructed from distinct zero entries of $P$.

\begin{lemma}\label{0Rees: comparing 0-closure submatrices}
	Let $Q=\Psubmatrix{P}{I_Q}{\Lambda_Q}$ be a \zeroclosuresub.
	\begin{enumerate}
		
		\item If $i\in I_Q$ and $\lambda\in\Lambda_Q$ are such that $p_{\lambda i}=0$, and $M=\Psubmatrix{P}{I_M}{\Lambda_M}$ is the \zeroclosuresubmatrix{i}{\lambda}, then $I_Q=I_M$ and $\Lambda_Q=\Lambda_M$.
		
		\item If $i\in I$ and $\lambda\in \Lambda$ are such that $\parens{i,\lambda}\in\parens{I\times\Lambda}\setminus\parens{I_Q\times\Lambda_Q}$ and $p_{\lambda i}=0$, and $M=\Psubmatrix{P}{I_M}{\Lambda_M}$ is the \zeroclosuresubmatrix{i}{\lambda}, then $I_Q\cap I_M=\Lambda_Q\cap\Lambda_M=\emptyset$.
	\end{enumerate}
\end{lemma}

\begin{proof}
	\textbf{Part 1.} Let $i\in I_Q$ and $\lambda\in\Lambda_Q$ be such that $p_{\lambda i}=0$. Assume that $M=\Psubmatrix{P}{I_M}{\Lambda_M}$ is the \zeroclosuresubmatrix{i}{\lambda}. It follows from part 2 of Lemma~\ref{0Rees: connected component G(I,^,P) G(M0[...])} that the subgraphs of $\commgraph{\Reeszero{G}{I}{\Lambda}{P}}$ induced by $I_Q\times G\times\Lambda_Q$ and $I_M\times G \times \Lambda_M$ are connected components of $\commgraph{\Reeszero{G}{I}{\Lambda}{P}}$. We have that $\parens{i,1_G,\lambda}\in \parens{I_Q\times G\times\Lambda_Q}\cap\parens{I_M\times G\times\Lambda_M}$ and, since $I_Q\times G\times\Lambda_Q$ and $I_M\times G\times\Lambda_M$ are the vertex sets of connected components of $\commgraph{\Reeszero{G}{I}{\Lambda}{P}}$, then we must have $I_Q\times G\times\Lambda_Q=I_M\times G\times\Lambda_M$. Thus $I_Q=I_M$ and $\Lambda_Q=\Lambda_M$.
	
	\medskip
	
	\textbf{Part 2.} Let $i\in I$ and $\lambda\in \Lambda$ be such that $\parens{i,\lambda}\in\parens{I\times\Lambda}\setminus\parens{I_Q\times\Lambda_Q}$ and $p_{\lambda i}=0$, and let $M=\Psubmatrix{P}{I_M}{\Lambda_M}$ be the \zeroclosuresubmatrix{i}{\lambda}. It follows from part 2 of Lemma~\ref{0Rees: connected component G(I,^,P) G(M0[...])} that the subgraphs of $\commgraph{\Reeszero{G}{I}{\Lambda}{P}}$ induced by $I_Q\times G\times\Lambda_Q$ and $I_M\times G\times \Lambda_M$ are connected components of $\commgraph{\Reeszero{G}{I}{\Lambda}{P}}$. Additionally, we have $\parens{i,1_G,\lambda}\in \parens{I_M\times G\times \Lambda_M}\setminus\parens{I_Q\times G\times \Lambda_Q}$. Due to the fact that $I_Q\times G\times \Lambda_Q$ and $I_M\times G\times \Lambda_M$ are the vertex sets of connected components of $\commgraph{\Reeszero{G}{I}{\Lambda}{P}}$, then we have $\parens{I_Q\times G\times \Lambda_Q}\cap\parens{I_M\times G\times \Lambda_M}=\emptyset$. Hence $\parens{I_Q\times\Lambda_Q}\cap\parens{I_M\times\Lambda_M}=\emptyset$, which implies that $I_Q\cap I_M=\emptyset$ or $\Lambda_Q\cap\Lambda_M=\emptyset$. Assume, without loss of generality, that $I_Q\cap I_M=\emptyset$. We are going to see that we also have $\Lambda_Q\cap\Lambda_M=\emptyset$. The fact that $Q$ is the \zeroclosuresubmatrix{i}{\lambda} implies that there are no zero entries in the rows of $P$ intersecting $Q$ that are not entries of $Q$. Consequently, $p_{\lambda' i'}\neq 0$ for all $i'\in I\setminus I_Q$ and $\lambda'\in\Lambda_Q$. In particular, we have $p_{\lambda' i'}\neq 0$ for all $i'\in I_M$ and $\lambda'\in\Lambda_Q$ (because $I_M\subseteq I\setminus I_Q$). Furthermore, it follows from Lemma~\ref{0Rees: every row/column of 0-closure submatrix has a 0} that every row of $M$ has a zero entry, which implies that for all $\mu\in \Lambda_M$ there exists $j\in I_M$ such that $p_{\mu j}=0$. Therefore $\Lambda_M\subseteq\Lambda\setminus\Lambda_Q$, that is, $\Lambda_Q\cap\Lambda_M=\emptyset$.
\end{proof}

Let $Q=\Psubmatrix{P}{I_Q}{\Lambda_Q}$ be a \zeroclosuresub. We observe that part 1 of Lemma~\ref{0Rees: comparing 0-closure submatrices} implies that, regardless of the zero entry of $Q$ that we choose to start the $0$-closure method, the \zeroclosuresub\ we construct at the end of the method is always the same, namely, $Q$. Moreover, if $M=\Psubmatrix{P}{I_M}{\Lambda_M}$ is also a \zeroclosuresub, then it also follows from Lemma~\ref{0Rees: comparing 0-closure submatrices} that we have $I_Q=I_M$ and $\Lambda_Q=\Lambda_M$, or we have $I_Q\cap I_M=\Lambda_Q\cap\Lambda_M=\emptyset$. With this in mind, from now onwards we will say that $Q$ and $M$ are \defterm{distinct $0$-closure submatrices of $P$} if and only if $I_Q\cap I_M=\Lambda_Q\cap\Lambda_M=\emptyset$.

Lemmata~\ref{0Rees: lemma connected component determined Q M} and \ref{0Rees: two paths zeros --> path non-zeros} add information about adjacency and distance in $\simplifiedgraph{I}{\Lambda}{P}$, respectively.

\begin{lemma}\label{0Rees: lemma connected component determined Q M}
	Let $Q=\Psubmatrix{P}{I_Q}{\Lambda_Q}$ and $M=\Psubmatrix{P}{I_M}{\Lambda_M}$ be distinct $0$-closure submatrices of $P$. Then
	\begin{enumerate}
		\item Let $\parens{i,\lambda}\in I_Q\times\Lambda_M$. If there exists $\parens{j,\mu}\in I\times\Lambda$ such that $\parens{j,\mu}$ is adjacent to $\parens{i,\lambda}$ (in $\simplifiedgraph{I}{\Lambda}{P}$), then $\parens{j,\mu}\in I_M\times\Lambda_Q$.
		
		\item Let $\parens{i,\lambda}\in I_M\times\Lambda_Q$. If there exists $\parens{j,\mu}\in I\times\Lambda$ such that $\parens{j,\mu}$ is adjacent to $\parens{i,\lambda}$ (in $\simplifiedgraph{I}{\Lambda}{P}$), then $\parens{j,\mu}\in I_Q\times\Lambda_M$.
	\end{enumerate}
\end{lemma}

\begin{proof}
	We are only going to prove statement 1 (statement 2 can be proved analogously). Let $\parens{i,\lambda}\in I_Q\times\Lambda_M$. Let $\parens{j,\mu}\in I\times\Lambda$ such that $\parens{j,\mu}$ is adjacent to $\parens{i,\lambda}$ (in $\simplifiedgraph{I}{\Lambda}{P}$). Then $p_{\lambda j}=p_{\mu i}=0$. Moreover, since $Q$ and $M$ are $0$-closure submatrices of $P$, then all the zero entries of row $\lambda$ (respectively, column $i$) of $P$ are entries of $M$ (respectively, $Q$). Thus $p_{\lambda i'}\neq 0$ for all $i'\in I\setminus I_M$ and $p_{\lambda' i}\neq 0$ for all $\lambda'\in\Lambda\setminus\Lambda_Q$, which implies that $j\in I_M$ and $\mu\in\Lambda_Q$ (that is, $\parens{j,\mu}\in I_M\times\Lambda_Q$).
\end{proof}

\begin{lemma}\label{0Rees: two paths zeros --> path non-zeros}
	Let $i,i',j,j\in I$ and $\lambda,\lambda',\mu,\mu'\in\Lambda$ and assume that $p_{\mu j'}=p_{\mu' j}=0$. Then
	\begin{enumerate}
		\item If there exists a path from $\parens{i,\lambda'}$ to $\parens{j,\mu'}$ in $\simplifiedgraph{I}{\Lambda}{P}$ and a path from $\parens{i',\lambda}$ to $\parens{j',\mu}$ in $\simplifiedgraph{I}{\Lambda}{P}$, then
		\begin{displaymath}
			\dist{\simplifiedgraph{I}{\Lambda}{P}}{\parens{i,\lambda}}{\parens{j,\mu}}\leqslant\begin{cases}
				\max\set{n,m}& \text{if } \max\set{n,m} \text{ is even,}\\
				1+\max\set{n,m}& \text{if } \max\set{n,m} \text{ is odd,}
			\end{cases}
		\end{displaymath}
		where $n=\dist{\simplifiedgraph{I}{\Lambda}{P}}{\parens{i,\lambda'}}{\parens{j,\mu'}}$ and $m=\dist{\simplifiedgraph{I}{\Lambda}{P}}{\parens{i',\lambda}}{\parens{j',\mu}}$.
		
		\item If there exists a path from $\parens{i,\lambda'}$ to $\parens{j',\mu}$ in $\simplifiedgraph{I}{\Lambda}{P}$ and a path from $\parens{i',\lambda}$ to $\parens{j,\mu'}$ in $\simplifiedgraph{I}{\Lambda}{P}$, then
		\begin{displaymath}
			\dist{\simplifiedgraph{I}{\Lambda}{P}}{\parens{i,\lambda}}{\parens{j,\mu}}\leqslant\begin{cases}
				1+\max\set{n,m}& \text{if } \max\set{n,m} \text{ is even,}\\
				\max\set{n,m}& \text{if } \max\set{n,m} \text{ is odd,}
			\end{cases}
		\end{displaymath}
		where $n=\dist{\simplifiedgraph{I}{\Lambda}{P}}{\parens{i,\lambda'}}{\parens{j',\mu}}$ and $m=\dist{\simplifiedgraph{I}{\Lambda}{P}}{\parens{i',\lambda}}{\parens{j,\mu'}}$.
	\end{enumerate}
\end{lemma}

The following two diagrams illustrate parts 1 and 2 of Lemma~\ref{0Rees: two paths zeros --> path non-zeros}, respectively. In the first diagram we have two solid arrows: one from entry $\parens{\lambda,i'}$ to entry $\parens{\mu, j'}$, and another from entry $\parens{\lambda',i}$ to entry $\parens{\mu',j}$; and we also have a dashed arrow from entry $\parens{\lambda,i}$ to entry $\parens{\mu,j}$. The solid arrows represent minimal paths (in $\simplifiedgraph{I}{\Lambda}{P}$) from $\parens{i',\lambda}$ to $\parens{j',\mu}$, and from $\parens{i,\lambda'}$ to $\parens{j,\mu'}$, that we will use to construct a path from $\parens{i,\lambda}$ to $\parens{j,\mu}$ (which is represented by the dashed arrow) and to determine an upper bound for $\dist{\simplifiedgraph{I}{\Lambda}{P}}{\parens{i,\lambda}}{\parens{j,\mu}}$. The second diagram can be interpreted analogously. 
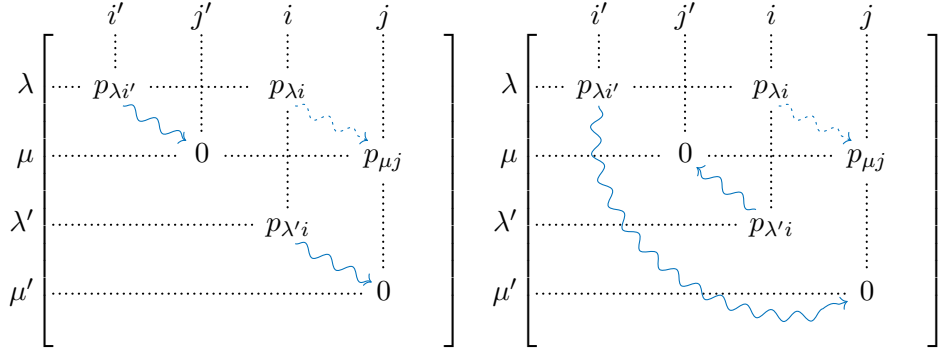
\begin{displaymath}
	\begin{bNiceMatrix}[first-row,first-col,margin]
		& & i' & & j' & & i & & j &\\
		\\
		\lambda & & p_{\lambda i'} & & & & p_{\lambda i} & & &\\
		\\
		\mu & & & & 0 & &  &  & p_{\mu j} &  \\
		\\
		\lambda' & & & & & & p_{\lambda' i} &  &  &  \\
		\\
		\mu' & & & & & & & & 0 & \\
		\\
		\CodeAfter
		\begin{tikzpicture}
			\begin{scope}[dash pattern=on 0pt off 2.3pt,line cap=round, thick]
				\draw (2.5-|1) -- (2.5-|2);
				\draw (2.5-|3) -- ($(2.5-|6)+(0mm,0mm)$);
				\draw (4.5-|1) -- ($(4.5-|4)+(0.5mm,0mm)$);
				\draw (4.5-|5) -- ($(4.5-|8)+(0mm,0mm)$);
				\draw ($(6.5-|1)+(0mm,0mm)$) -- (6.5-|6);
				\draw (8.5-|1) -- ($(8.5-|8)+(1.5mm,0mm)$);
				\draw (6.5|-1) -- (6.5|-2);
				\draw ($(6.5|-3)+(0mm,-1mm)$) -- ($(6.5|-6)+(0mm,0mm)$);
				\draw (8.5|-1) -- ($(8.5|-4)+(0mm,0mm)$);
				\draw ($(8.5|-5)+(0mm,-1.5mm)$) -- ($(8.5|-8)+(0mm,1mm)$);
				\draw ($(2.5|-1)+(0mm,0mm)$) -- (2.5|-2);
				\draw ($(4.5|-1)+(0mm,0mm)$) -- ($(4.5|-4)+(0mm,1mm)$);
			\end{scope}
			
			\begin{scope}[decoration = {snake, amplitude = 0.7mm}, ->, RoyalBlue]
				\draw[decorate, dash pattern=on 1pt off 1.8pt] ($(2.5-|6.5)+(1mm,-2.5mm)$) -- ($(4.5-|8.5)+(-2mm,2mm)$);
				\draw[decorate] ($(6.5-|6.5)+(1mm,-2.5mm)$) -- ($(8.5-|8.5)+(-1.5mm,1.5mm)$);
				\draw[decorate] ($(2.5-|2.5)+(1mm,-2.5mm)$) -- ($(4.5-|4.5)+(-2mm,2mm)$);
			\end{scope}
		\end{tikzpicture}
	\end{bNiceMatrix}\quad
	\begin{bNiceMatrix}[first-row,first-col,margin]
		& & i' & & j' & & i & & j &\\
		\\
		\lambda & & p_{\lambda i'} & & & & p_{\lambda i} & & &\\
		\\
		\mu & & & & 0 & &  &  & p_{\mu j} &  \\
		\\
		\lambda' & & & & & & p_{\lambda' i} &  &  &  \\
		\\
		\mu' & & & & & & & & 0 & \\
		\\
		\CodeAfter
		\begin{tikzpicture}
			\begin{scope}[dash pattern=on 0pt off 2.3pt,line cap=round, thick]
				\draw (2.5-|1) -- (2.5-|2);
				\draw (2.5-|3) -- ($(2.5-|6)+(0mm,0mm)$);
				\draw (4.5-|1) -- ($(4.5-|4)+(0.5mm,0mm)$);
				\draw (4.5-|5) -- ($(4.5-|8)+(0mm,0mm)$);
				\draw ($(6.5-|1)+(0mm,0mm)$) -- (6.5-|6);
				\draw (8.5-|1) -- ($(8.5-|8)+(1.5mm,0mm)$);
				\draw (6.5|-1) -- (6.5|-2);
				\draw ($(6.5|-3)+(0mm,-1mm)$) -- ($(6.5|-6)+(0mm,0mm)$);
				\draw (8.5|-1) -- ($(8.5|-4)+(0mm,0mm)$);
				\draw ($(8.5|-5)+(0mm,-1.5mm)$) -- ($(8.5|-8)+(0mm,1mm)$);
				\draw ($(2.5|-1)+(0mm,0mm)$) -- (2.5|-2);
				\draw ($(4.5|-1)+(0mm,0mm)$) -- ($(4.5|-4)+(0mm,1mm)$);
			\end{scope}
			
			\begin{scope}[decoration = {snake, amplitude = 0.7mm}, ->, RoyalBlue]
				\draw[decorate, dash pattern=on 1pt off 1.8pt] ($(2.5-|6.5)+(1mm,-2.5mm)$) -- ($(4.5-|8.5)+(-2mm,2mm)$);
				\draw[decorate] ($(6.5-|6.5)+(-2mm,2mm)$) -- ($(4.5-|4.5)+(1.5mm,-1.5mm)$);
				\draw[decorate] ($(2.5-|2.5)+(0mm,-2.5mm)$) to [out=-100,in=-160] ($(8.5-|8.5)+(-2.5mm,-1mm)$);
			\end{scope}
		\end{tikzpicture}
	\end{bNiceMatrix}
\end{displaymath}

\begin{proof}
	We are only going to demonstrate part 1 of the lemma --- part 2 can be proved analogously to part 1.
	
	Suppose that there exists a path from $\parens{i,\lambda'}$ to $\parens{j,\mu'}$ in $\simplifiedgraph{I}{\Lambda}{P}$ and a path from $\parens{i',\lambda}$ to $\parens{j',\mu}$ in $\simplifiedgraph{I}{\Lambda}{P}$. Let $n=\dist{\simplifiedgraph{I}{\Lambda}{P}}{\parens{i,\lambda'}}{\parens{j,\mu'}}$ and $m=\dist{\simplifiedgraph{I}{\Lambda}{P}}{\parens{i',\lambda}}{\parens{j',\mu}}$. Assume, without loss of generality, that $n\leqslant m$. Let
	\begin{gather*}
		\parens{i,\lambda'}=\parens{i_1,\lambda_1}-\parens{i_2,\lambda_2}-\cdots-\parens{i_{n+1},\lambda_{n+1}}=\parens{j,\mu'}\\
		\shortintertext{and}
		\parens{i',\lambda}=\parens{j_1,\mu_1}-\parens{j_2,\mu_2}-\cdots-\parens{j_{m+1},\mu_{m+1}}=\parens{j',\mu}
	\end{gather*}
	be paths in $\simplifiedgraph{I}{\Lambda}{P}$ from $\parens{i,\lambda'}$ to $\parens{j,\mu'}$ and from $\parens{i',\lambda}$ to $\parens{j',\mu}$, respectively. Let $i_l=i_{n+1}=j$ and $\lambda_l=\lambda_{n+1}=\mu'$ for all $l\in\set{n+2,\ldots,m+1}$. Then $p_{\lambda_l i_{l+1}}=p_{\lambda_{l+1} i_l}=0$ for all $l\in\set{n+1,\ldots,m}$ (because $p_{\mu' j}=0$). Additionally, we have $p_{\lambda_l i_{l+1}}=p_{\lambda_{l+1} i_l}=0$ for all $l\in\Xn$ (because $\parens{i_l,\lambda_l}$ is adjacent to $\parens{i_{l+1},\lambda_{l+1}}$ for all $l\in\Xn$), and we have $p_{\mu_l j_{l+1}}=p_{\mu_{l+1} j_l}=0$ for all $l\in\X{m}$ (because $\parens{j_l,\mu_l}$ is adjacent to $\parens{j_{l+1},\mu_{l+1}}$ for all $l\in\X{m}$). Therefore $\parens{i_l,\mu_l}\sim\parens{j_{l+1},\lambda_{l+1}}$ and $\parens{j_l,\lambda_l}\sim\parens{i_{l+1},\mu_{l+1}}$ for all $l\in\X{m}$.
	
	\smallskip
	
	\textit{Case 1:} Suppose that $m=\max\set{n,m}$ is even. We have
	\begin{displaymath}
		\parens{i,\lambda}=\parens{i_1,\mu_1}\sim\parens{j_2,\lambda_2}\sim\parens{i_3,\mu_3}\sim\cdots\sim\parens{j_m,\lambda_m}\sim\parens{i_{m+1},\mu_{m+1}}=\parens{j,\mu},
	\end{displaymath}
	which implies that there is a path from $\parens{i,\lambda}$ to $\parens{j,\mu}$ in $\simplifiedgraph{I}{\Lambda}{P}$ whose length is at most $m$. Hence
	\begin{displaymath}
		\dist{\simplifiedgraph{I}{\Lambda}{P}}{\parens{i,\lambda}}{\parens{j,\mu}}\leqslant m=\max\set{n,m}.
	\end{displaymath}
	
	\smallskip
	
	\textit{Case 2:} Suppose that $m=\max\set{n,m}$ is odd. We have
	\begin{displaymath}
		\parens{i,\lambda}=\parens{i_1,\mu_1}\sim\parens{j_2,\lambda_2}\sim\parens{i_3,\mu_3}\sim\cdots\sim\parens{i_m,\mu_m}\sim\parens{j_{m+1},\lambda_{m+1}}=\parens{j',\mu'}
	\end{displaymath}
	and $\parens{j',\mu'}\sim\parens{j,\mu}$ (because $p_{\mu' j}=p_{\mu j'}=0$). Hence there is a path from $\parens{i,\lambda}$ to $\parens{j,\mu}$ in $\simplifiedgraph{I}{\Lambda}{P}$ whose length is at most $1+m$. Consequently,
	\begin{displaymath}
		\dist{\simplifiedgraph{I}{\Lambda}{P}}{\parens{i,\lambda}}{\parens{j,\mu}}\leqslant 1+m=1+\max\set{n,m}.\qedhere
	\end{displaymath}
\end{proof}

Theorem~\ref{0Rees: connectedness} provides a few ways to check if $\commgraph{\Reeszero{G}{I}{\Lambda}{P}}$ is connected. The most important is the one that uses the $0$-closure method (and $0$-closure submatrices of $P$). We will also see that connectedness of $\commgraph{\Reeszero{G}{I}{\Lambda}{P}}$ only depends on the matrix $P$.

\begin{theorem}\label{0Rees: connectedness}
	The following statements are equivalent:
	\begin{enumerate}
		\item $\commgraph{\Rees{G}{I}{\Lambda}{P}}$ is connected.
		
		\item $\simplifiedgraph{I}{\Lambda}{P}$ is connected.
		
		\item $\timeszero{P}$ is not \equivalent\ to any of the following matrices
		\begin{gather*}
			\begin{bNiceArray}{cc}[margin]
				\Block{2-2}{\timeszero{A}}&\\
				\\
				\hline
				\Block{1-2}{\rowtimes{0.7em}{0.45em}}&
			\end{bNiceArray},\quad
			\begin{bNiceArray}{cc|c}[margin]
				\Block{2-2}{\timeszero{A}}&&\Block{2-}{\columntimes}\\
				\\
			\end{bNiceArray},\quad
			\begin{bNiceArray}{cc|cc}[margin]
				\Block{2-2}{\timeszero{A}} & & \Block{2-2}{\blocktimes{0.61em}{0.61em}{-0.25em}} & \\
				\\
				\hline
				\Block{2-2}{\blocktimes{0.61em}{0.61em}{-0.25em}} & & \Block{2-2}{\timeszero{B}} & \\
				\\
			\end{bNiceArray},
		\end{gather*}
		where $A$ and $B$ are submatrices of $P$.
		
		\item For all $i\in I$ and $\lambda\in\Lambda$ such that $p_{\lambda i}=0$, the \zeroclosuresubmatrix{i}{\lambda} is $P$.
		
		\item There exist $i\in I$ and $\lambda\in\Lambda$ such that $p_{\lambda i}=0$ and the \zeroclosuresubmatrix{i}{\lambda} is $P$.
	\end{enumerate}
\end{theorem}

\begin{proof}
	\textbf{Part 1} $\bracks{1 \implies 2}$. Suppose that $\commgraph{\Reeszero{G}{I}{\Lambda}{P}}$ is connected. Then $\commgraph{\Reeszero{G}{I}{\Lambda}{P}}$ has only one connected component, which is itself. Let $C$ be the vertex set of a connected component of $\simplifiedgraph{I}{\Lambda}{P}$. It follows from Lemma~\ref{0Rees: connected component G(I,^) => connected component G(M0[...])} that the subgraph of $\commgraph{\Reeszero{G}{I}{\Lambda}{P}}$ induced by $\bigcup_{\parens{i,\lambda}\in C} \set{i}\times G\times\set{\lambda}$ is a connected component of $\commgraph{\Reeszero{G}{I}{\Lambda}{P}}$. Since the only connected component of $\commgraph{\Reeszero{G}{I}{\Lambda}{P}}$ is itself (and its vertex set is $I\times G\times \Lambda$), then we must have $\bigcup_{\parens{i,\lambda}\in C} \set{i}\times G\times\set{\lambda}=I\times G\times \Lambda$. Thus $C=I\times \Lambda$, which implies that the only connected component of $\simplifiedgraph{I}{\Lambda}{P}$ is itself; that is, $\simplifiedgraph{I}{\Lambda}{P}$ is connected.
	
	\medskip
	
	\textbf{Part 2} $\bracks{\neg 3 \implies \neg 2}$. Suppose that $\timeszero{P}$ is \equivalent\ to at least one of the matrices described in 3. We consider three cases.
	
	\smallskip
	
	\textit{Case 1:} Suppose that there exists a submatrix $A$ of $P$ such that the matrix
	\begin{displaymath}
		\begin{bNiceArray}{cc}[margin]
			\Block{2-2}{\timeszero{A}}&\\
			\\
			\hline
			\Block{1-2}{\rowtimes{0.7em}{0.45em}}&
		\end{bNiceArray}
	\end{displaymath}
	is \equivalent\ to $\timeszero{P}$. Then $P$ has at least one row with no zero entries. Let $\lambda\in\Lambda$ be the index of one of those rows and let $i\in I$. We have $p_{\lambda j}\neq 0$ for all $j\in I$, which implies that there is no vertex adjacent to the vertex $\parens{i,\lambda}$. Hence $\simplifiedgraph{I}{\Lambda}{P}$ is not connected.
	
	\smallskip
	
	\textit{Case 2:} Suppose that there exists a submatrix $A$ of $P$ such that the matrix
	\begin{displaymath}
		\begin{bNiceArray}{cc|c}[margin]
			\Block{2-2}{\timeszero{A}}&&\Block{2-}{\columntimes}\\
			\\
		\end{bNiceArray}
	\end{displaymath}
	is \equivalent\ to $\timeszero{P}$. The proof of this case is similar to the one of case 1 and is omitted.
	
	\smallskip
	
	\textit{Case 3:} Suppose that there exist submatrices $A$ and $B$ of $P$ such that the matrix
	\begin{displaymath}
		\begin{bNiceArray}{cc|cc}[margin]
			\Block{2-2}{\timeszero{A}} & & \Block{2-2}{\blocktimes{0.61em}{0.61em}{-0.25em}} & \\
			\\
			\hline
			\Block{2-2}{\blocktimes{0.61em}{0.61em}{-0.25em}} & & \Block{2-2}{\timeszero{B}} & \\
			\\
		\end{bNiceArray}
	\end{displaymath}
	is \equivalent\ to $\timeszero{P}$. Then there exist $I'\subseteq I$ and $\Lambda'\subseteq \Lambda$ such that $A=\Psubmatrix{P}{I'}{\Lambda'}$ and $B=\Psubmatrix{P}{I\setminus I'}{\Lambda\setminus\Lambda'}$. Let $\parens{j,\mu}\in I'\times\Lambda'$ and $\parens{j',\mu'}\in I\times\Lambda$ be adjacent vertices. Then $p_{\mu j'}=p_{\mu' j}=0$ and, consequently, $\parens{j',\mu},\parens{j,\mu'}\in\parens{I'\times\Lambda'}\cup\parens{\parens{I\setminus I'}\times\parens{\Lambda\setminus\Lambda'}}$ (note that $p_{\lambda i}\neq 0$ if we have either $i\in I\setminus I'$ and $\lambda\in\Lambda'$ or $i\in I'$ and $\lambda\in\Lambda\setminus\Lambda'$). Since $j\in I'$ and $\mu\in\Lambda'$, then we must have $\mu'\in\Lambda'$ and $j'\in I'$; that is, $\parens{j',\mu'}	\in I'\times \Lambda'$. This proves that all vertices of $I'\times \Lambda'$ are only adjacent to vertices of $I'\times\Lambda'$. Hence there exists a connected component of $\simplifiedgraph{I}{\Lambda}{P}$ whose vertex set is contained in $I'\times\Lambda'$. Therefore $\simplifiedgraph{I}{\Lambda}{P}$ must contain more than one connected component and, thus, $\simplifiedgraph{I}{\Lambda}{P}$ is not connected.

	\medskip
	
	\textbf{Part 3} $\bracks{\neg 4\implies\neg 3}$. Suppose that there exist $i\in I$ and $\lambda\in\Lambda$ such that $p_{\lambda i}=0$ and the \zeroclosuresubmatrix{i}{\lambda} is not $P$. Let $Q=\Psubmatrix{P}{I_Q}{\Lambda_Q}$ be the \zeroclosuresubmatrix{i}{\lambda}. Since $Q\neq P$, then we have $\abs{I_Q}<\abs{I}$ or $\abs{\Lambda_Q}<\abs{\Lambda}$.
	
	\smallskip
	
	\textit{Case 1:} Suppose that $\abs{I_Q}=\abs{I}$ and $\abs{\Lambda_Q}<\abs{\Lambda}$. The fact that $Q=\Psubmatrix{P}{I_Q}{\Lambda_Q}$ is the \zeroclosuresubmatrix{i}{\lambda} implies that the zero entries in the rows of $P$ that intersect $Q$ are all entries of $Q$ and, consequently, we have $p_{\mu j}\neq 0$ for all $j\in I_Q=I$ and $\mu\in\Lambda\setminus\Lambda_Q$. Let $\mu\in\Lambda\setminus\Lambda_Q$. We have that row $\mu$ of $P$ has no zero entries. Hence row $\mu$ of $\timeszero{P}$ has no zero entries, and if we move row $\mu$ to the bottom, we obtain the matrix
	\begin{displaymath}
		\begin{bNiceArray}{ccccccc}[margin]
			\Block{2-7}{\timeszero{\Psubmatrix{P}{I}{\Lambda\setminus\set{\mu}}}}&&&&&&\\
			\\
			\hline
			\Block{1-7}{\rowtimes{2.95em}{2.65em}}&&&&&&
		\end{bNiceArray}
	\end{displaymath}
	which is \equivalent\ to $\timeszero{P}$.
	
	\smallskip
	
	\textit{Case 2:} Suppose that $\abs{I_Q}<\abs{I}$ and $\abs{\Lambda_Q}=\abs{\Lambda}$. If we use an argument symmetrical to the one used in case 1, then we can see that there exists $j\in I$ such that $\timeszero{P}$ is \equivalent\ to
	\begin{displaymath}
		\begin{bNiceArray}{ccccccc|c}[margin]
			\Block{2-7}{\timeszero{\Psubmatrix{P}{I\setminus\set{j}}{\Lambda}}}&&&&&&&\Block{2-}{\columntimes}\\
			\\
		\end{bNiceArray}.
	\end{displaymath}
	
	\smallskip
	
	\textit{Case 3:} Suppose that $\abs{I_Q}<\abs{I}$ and $\abs{\Lambda_Q}<\abs{\Lambda}$. Due to the fact that $Q=\Psubmatrix{P}{I_Q}{\Lambda_Q}$ is the \zeroclosuresubmatrix{i}{\lambda}, then $Q$ contains all the zero entries of the rows (respectively, columns) of $P$ whose indices belong to $\Lambda_Q$ (respectively, $I_Q$). Hence $p_{\mu j}\neq 0$ for all $\parens{j,\mu}\in \parens{\parens{I\setminus I_Q}\times \Lambda_Q}\cup \parens{I_Q\times\parens{\Lambda\setminus\Lambda_Q}}$. If we move the rows of $\timeszero{P}$ whose indices belong to $\Lambda_Q$ to the top of the matrix (without changing the order in which the rows indexed by $\Lambda_Q$ appear in $\timeszero{P}$), and then move the columns of the resulting matrix whose indices belong to $I_Q$ to the left of the matrix (without changing the order in which the columns indexed by $I_Q$ appear in $\timeszero{P}$), then we obtain the matrix
	\begin{displaymath}
		\begin{bNiceArray}{cc|ccccccccc}[first-row,first-col,margin]
			& \Block{1-2}{I_Q} & & \Block{1-9}{I\setminus I_Q} &&&&&&&&\\
			\Block{2-1}{\Lambda_Q} & \Block{2-2}{\timeszero{Q}} & & \Block{2-9}{\blocktimes{0.61em}{3.7em}{0.15em}} &&&&&&&& \\
			\\
			\hline
			\Block{2-1}{\Lambda\setminus\Lambda_Q} & \Block{2-2}{\blocktimes{0.61em}{0.61em}{-0.25em}} & & \Block{2-9}{\timeszero{\Psubmatrix{P}{I\setminus I_Q}{\Lambda\setminus\Lambda_Q}}} &&&&&&&& \\
			\\
		\end{bNiceArray},
	\end{displaymath}
	which is \equivalent\ to $\timeszero{P}$.
	
	\medskip
	
	\textbf{Part 4} $\bracks{4 \implies 5}$. Statement 5 is an immediate consequence of statement 4 and the fact that $P$ contains at least one zero entry.
	
	\medskip
	
	\textbf{Part 5} $\bracks{5 \implies 1}$. Suppose that there exist $i\in I$ and $\lambda\in\Lambda$ such that $p_{\lambda i}=0$ and the \zeroclosuresubmatrix{i}{\lambda} is $P$. We have $P=\Psubmatrix{P}{I}{\Lambda}$ and it follows from Lemma~\ref{0Rees: connected component G(I,^,P) G(M0[...])} that the subgraph of $\commgraph{\Reeszero{G}{I}{\Lambda}{P}}$ induced by $I\times G\times\Lambda$ is a connected component of $\commgraph{\Reeszero{G}{I}{\Lambda}{P}}$. Since $\commgraph{\Reeszero{G}{I}{\Lambda}{P}}$ is the subgraph of $\commgraph{\Reeszero{G}{I}{\Lambda}{P}}$ induced by $I\times G\times\Lambda$, then $\commgraph{\Reeszero{G}{I}{\Lambda}{P}}$ is a connected component of $\commgraph{\Reeszero{G}{I}{\Lambda}{P}}$, that is, $\commgraph{\Reeszero{G}{I}{\Lambda}{P}}$ is connected.
\end{proof}

It follows from Theorem~\ref{0Rees: connectedness} that, regardless of the zero entry we choose to start the $0$-closure method, if $\commgraph{\Reeszero{G}{I}{\Lambda}{P}}$ is connected, then the \zeroclosuresub\ constructed is always going to be $P$ and, if $\commgraph{\Reeszero{G}{I}{\Lambda}{P}}$ is not connected, then the \zeroclosuresub\ constructed is never equal to $P$. This implies that, in order to see if $\commgraph{\Reeszero{G}{I}{\Lambda}{P}}$ is connected, we only need to choose one zero entry, start the $0$-closure method with that entry and verify if the \zeroclosuresub\ obtained is $P$. 

\begin{example}\label{0Rees: example connectedness}
	In Example~\ref{0Rees: example 0-closure method} we showed that the \zeroclosuresubmatrix{6}{4} is $\Psubmatrix{P}{\set{1,3,4,6,7}}{\set{1,3,4}}$, which is not $P$. Thus Theorem~\ref{0Rees: connectedness} implies that $\commgraph{\Reeszero{G}{I}{\Lambda}{P}}$ is not connected.
\end{example}

Now we are going to see how we can identify the connected components of $\commgraph{\Reeszero{G}{I}{\Lambda}{P}}$ (when it is not connected) by analyzing $P$.

\begin{theorem}\label{0Rees: finding connected components}
	Suppose that $\commgraph{\Reeszero{G}{I}{\Lambda}{P}}$ is not connected. Then $\timeszero{P}$ is \equivalent\ to one of the following matrices
	\begin{displaymath}
		\underbrace{\begin{bNiceArray}{cc|cc|cc}[margin]
				\Block{2-2}{\timeszero{A_1}} & & \Block{2-2}{\blocktimes{0.61em}{0.61em}{-0.25em}} & & \Block{2-2}{\blocktimes{0.61em}{0.61em}{-0.25em}} & \\
				\\
				\hline
				\Block{2-2}{\blocktimes{0.61em}{0.61em}{-0.25em}} & & \Block{2-2}{\ddots} & & \Block{2-2}{\blocktimes{0.61em}{0.61em}{-0.25em}} & \\
				\\
				\hline
				\Block{2-2}{\blocktimes{0.61em}{0.61em}{-0.25em}} & & \Block{2-2}{\blocktimes{0.61em}{0.61em}{-0.25em}} & & \Block{2-2}{\timeszero{A_n}}& \\
				\\
		\end{bNiceArray}}_{\text{\footnotesize $n\geqslant 2$}},\quad
		\underbrace{\begin{bNiceArray}{cc|cc|cc|cc}[margin]
				\Block{2-2}{\timeszero{A_1}} & & \Block{2-2}{\blocktimes{0.61em}{0.61em}{-0.25em}} & & \Block{2-2}{\blocktimes{0.61em}{0.61em}{-0.25em}} & & \Block{2-2}{\blocktimes{0.61em}{0.61em}{-0.25em}} & \\
				\\
				\hline
				\Block{2-2}{\blocktimes{0.61em}{0.61em}{-0.25em}} & & \Block{2-2}{\ddots} & & \Block{2-2}{\blocktimes{0.61em}{0.61em}{-0.25em}} & & \Block{2-2}{\blocktimes{0.61em}{0.61em}{-0.25em}} & \\
				\\
				\hline
				\Block{2-2}{\blocktimes{0.61em}{0.61em}{-0.25em}} & & \Block{2-2}{\blocktimes{0.61em}{0.61em}{-0.25em}} & & \Block{2-2}{\timeszero{A_n}}& & \Block{2-2}{\blocktimes{0.61em}{0.61em}{-0.25em}} & \\
				\\
		\end{bNiceArray}}_{\text{\footnotesize $n\geqslant 1$}},\quad
		\underbrace{\begin{bNiceArray}{cc|cc|cc}[margin]
				\Block{2-2}{\timeszero{A_1}} & & \Block{2-2}{\blocktimes{0.61em}{0.61em}{-0.25em}} & & \Block{2-2}{\blocktimes{0.61em}{0.61em}{-0.25em}} & \\
				\\
				\hline
				\Block{2-2}{\blocktimes{0.61em}{0.61em}{-0.25em}} & & \Block{2-2}{\ddots} & & \Block{2-2}{\blocktimes{0.61em}{0.61em}{-0.25em}} & \\
				\\
				\hline
				\Block{2-2}{\blocktimes{0.61em}{0.61em}{-0.25em}} & & \Block{2-2}{\blocktimes{0.61em}{0.61em}{-0.25em}} & & \Block{2-2}{\timeszero{A_n}}& \\
				\\
				\hline
				\Block{2-2}{\blocktimes{0.61em}{0.61em}{-0.25em}} & & \Block{2-2}{\blocktimes{0.61em}{0.61em}{-0.25em}} & & \Block{2-2}{\blocktimes{0.61em}{0.61em}{-0.25em}} & \\
				\\
		\end{bNiceArray}}_{\text{\footnotesize $n\geqslant 1$}},\quad
		\underbrace{\begin{bNiceArray}{cc|cc|cc|cc}[margin]
				\Block{2-2}{\timeszero{A_1}} & & \Block{2-2}{\blocktimes{0.61em}{0.61em}{-0.25em}} & & \Block{2-2}{\blocktimes{0.61em}{0.61em}{-0.25em}} & & \Block{2-2}{\blocktimes{0.61em}{0.61em}{-0.25em}} & \\
				\\
				\hline
				\Block{2-2}{\blocktimes{0.61em}{0.61em}{-0.25em}} & & \Block{2-2}{\ddots} & & \Block{2-2}{\blocktimes{0.61em}{0.61em}{-0.25em}} & & \Block{2-2}{\blocktimes{0.61em}{0.61em}{-0.25em}} & \\
				\\
				\hline
				\Block{2-2}{\blocktimes{0.61em}{0.61em}{-0.25em}} & & \Block{2-2}{\blocktimes{0.61em}{0.61em}{-0.25em}} & & \Block{2-2}{\timeszero{A_n}}& & \Block{2-2}{\blocktimes{0.61em}{0.61em}{-0.25em}} & \\
				\\
				\hline
				\Block{2-2}{\blocktimes{0.61em}{0.61em}{-0.25em}} & & \Block{2-2}{\blocktimes{0.61em}{0.61em}{-0.25em}} & & \Block{2-2}{\blocktimes{0.61em}{0.61em}{-0.25em}} & & \Block{2-2}{\blocktimes{0.61em}{0.61em}{-0.25em}} & \\
				\\
		\end{bNiceArray}}_{\text{\footnotesize $n\geqslant 1$}},
	\end{displaymath}
	where, for each $k\in\Xn$, $A_k$ is a \zeroclosuresub. Moreover, we can identify the connected components of $\commgraph{\Reeszero{G}{I}{\Lambda}{P}}$ the following way:
	\begin{enumerate}
		\item Let $k\in\Xn$ and suppose that $A_k=\Psubmatrix{P}{I_k}{\Lambda_k}$. Then the subgraph of $\commgraph{\Reeszero{G}{I}{\Lambda}{P}}$ induced by $I_k\times G\times \Lambda_k$ is a connected component of $\commgraph{\Reeszero{G}{I}{\Lambda}{P}}$.
		
		\item Let $k,m\in\Xn$ be such that $k\neq m$ and suppose that $A_k=\Psubmatrix{P}{I_k}{\Lambda_k}$ and $A_m=\Psubmatrix{P}{I_m}{\Lambda_m}$. Then the subgraph of $\commgraph{\Reeszero{G}{I}{\Lambda}{P}}$ induced by $\parens{I_k\times G\times \Lambda_m}\cup\parens{I_m\times G\times\Lambda_k}$ is a connected component of $\commgraph{\Reeszero{G}{I}{\Lambda}{P}}$.
		
		\item Let $i \in I$ and let $\lambda\in\Lambda$ be such that row $\lambda$ has no zero entries. Then the subgraph of $\commgraph{\Reeszero{G}{I}{\Lambda}{P}}$ induced by $\set{i}\times G\times \set{\lambda}$ is a connected component of $\commgraph{\Reeszero{G}{I}{\Lambda}{P}}$.
		
		\item Let $\lambda \in \Lambda$ and let $i \in I$ be such that column $i$ has no zero entries. Then the subgraph of $\commgraph{\Reeszero{G}{I}{\Lambda}{P}}$ induced by $\set{i}\times G\times \set{\lambda}$ is a connected component of $\commgraph{\Reeszero{G}{I}{\Lambda}{P}}$.
	\end{enumerate}
\end{theorem}

In the course of the proof of Theorem~\ref{0Rees: finding connected components} we will make some reasonings regarding distances between vertices. Although these are not necessary for this proof, since we are only establishing what the connected components of $\commgraph{\Reeszero{G}{I}{\Lambda}{P}}$ are, they will be useful to demonstrate Proposition~\ref{0Rees: diameter connected component Q M}.

\begin{proof}
	We start by showing that $P$ is \equivalent\ to one of the matrices described in the theorem statement.
	
	Since $P$ contains zero entries, then we can choose one of them to start the $0$-closure method and construct a \zeroclosuresub. Let $A_1$ be that submatrix. If there are other zero entries in $P$ that are not entries of $A_1$, then we choose one of them to start (a second time) the $0$-closure method and construct a (second) \zeroclosuresub. Let $A_2$ be that submatrix. If there are other zero entries in $P$ that are entries of neither $A_1$ nor $A_2$, then we choose one of them to start (a third time) the $0$-closure method and construct a (third) \zeroclosuresub.
	
	Since $I$ and $\Lambda$ are finite, then this process eventually stops. Let $n$ be the number of times the $0$-closure method was performed and let $A_1,\ldots,A_n$ be the $0$-closure submatrices of $P$ obtained in each application of the method. For each $k\in\Xn$ assume that $A_k=\Psubmatrix{P}{I_k}{\Lambda_k}$. We observe that, due to the fact that for each $k\in\Xn$ we have that $A_k$ is a \zeroclosuresub, then this implies that for each $k\in\Xn$ there are no zero entries in the rows and columns of $P$ intersecting $A_k$ that are not entries of $A_k$, that is, for each $k\in\Xn$ we have $p_{\lambda i}\neq 0$ for all $\parens{i,\lambda}\in\parens{I_k\times\parens{\Lambda\setminus\Lambda_k}}\cup\parens{\parens{I\setminus I_k}\times \Lambda_k}$. Thus $\timeszero{P}$ is \equivalent\ to
	\begin{displaymath}
		\begin{bNiceArray}{cc|ccccccccc}[first-row,first-col,margin]
			& \Block{1-2}{I_k} & & \Block{1-9}{I\setminus I_k} &&&&&&&&\\
			\Block{2-1}{\Lambda_k} & \Block{2-2}{\timeszero{A_k}} & & \Block{2-9}{\blocktimes{0.61em}{3.7em}{0.15em}} &&&&&&&& \\
			\\
			\hline
			\Block{2-1}{\Lambda\setminus\Lambda_k} & \Block{2-2}{\blocktimes{0.61em}{0.61em}{-0.25em}} & & \Block{2-9}{\timeszero{\Psubmatrix{P}{I\setminus I_k}{\Lambda\setminus\Lambda_k}}} &&&&&&&& \\
			\\
		\end{bNiceArray}
	\end{displaymath}
	for all $k\in\Xn$.
	
	In order to conclude the proof, we only need to see that for all distinct $k,m\in\Xn$ we have $I_k\cap I_m=\Lambda_k\cap\Lambda_m=\emptyset$. Let $k,m\in\Xn$ be such that $k<m$. Let $i\in I$ and $\lambda\in\Lambda$ be such that $p_{\lambda i}=0$ and $A_m$ is the \zeroclosuresubmatrix{i}{\lambda}. Due to the fact that each application of the $0$-closure method starts with an entry that does not belong to any of the $0$-closure submatrices of $P$ constructed in the preceding applications of the method, then the $\parens{\lambda,i}$-th entry of $P$ is not an entry of any of the submatrices $A_1,\ldots,A_{m-1}$, that is, $\parens{i,\lambda}\in\parens{I\times\Lambda}\setminus\bigcup_{l=1}^{m-1} \parens{I_l\times\Lambda_l}\subseteq\parens{I\times\Lambda}\setminus\parens{I_k\times\Lambda_k}$. Then, by part 2 of Lemma~\ref{0Rees: comparing 0-closure submatrices}, we have $I_k\cap I_m=\Lambda_k\cap\Lambda_m=\emptyset$.
	
	We observe that, as a consequence of Theorem~\ref{0Rees: connectedness} and the fact that $\commgraph{\Reeszero{G}{I}{\Lambda}{P}}$ is not connected, then we have $A_1\neq P$. Hence, if $n=1$, $P$ is \equivalent\ to one of the following matrices
	\begin{displaymath}
		\begin{bNiceArray}{cc|cc}[margin]
			\Block{2-2}{\timeszero{A_1}} & & \Block{2-2}{\blocktimes{0.61em}{0.61em}{-0.25em}} & \\
			\\
		\end{bNiceArray},\quad
		\begin{bNiceArray}{cc}[margin]
			\Block{2-2}{\timeszero{A_1}} & \\
			\\
			\hline
			\Block{2-2}{\blocktimes{0.61em}{0.61em}{-0.25em}} & \\
			\\
		\end{bNiceArray},\quad
		\begin{bNiceArray}{cc|cc}[margin]
			\Block{2-2}{\timeszero{A_1}} & & \Block{2-2}{\blocktimes{0.61em}{0.61em}{-0.25em}} & \\
			\\
			\hline
			\Block{2-2}{\blocktimes{0.61em}{0.61em}{-0.25em}} & & \Block{2-2}{\blocktimes{0.61em}{0.61em}{-0.25em}} & \\
			\\
		\end{bNiceArray}.
	\end{displaymath}

	\medskip
	
	\textbf{Part 1.} Let $k\in\Xn$ and assume that $A_k=\Psubmatrix{P}{I_k}{\Lambda_k}$. Since $A_k$ is a \zeroclosuresub, then part 2 of Lemma~\ref{0Rees: connected component G(I,^,P) G(M0[...])} guarantees that the subgraph of $\commgraph{\Reeszero{G}{I}{\Lambda}{P}}$ induced by $I_k\times G\times\Lambda_k$ is a connected component of $\commgraph{\Reeszero{G}{I}{\Lambda}{P}}$.
	
	\medskip
	
	\textbf{Part 2.} Let $k,m\in\Xn$ be such that $k\neq m$ and assume that $A_k=\Psubmatrix{P}{I_k}{\Lambda_k}$ and $A_m=\Psubmatrix{P}{I_m}{\Lambda_m}$. Let $\mathcal{C}_k$ and $\mathcal{C}_m$ be the subgraphs of $\commgraph{\Reeszero{G}{I}{\Lambda}{P}}$ induced by $I_k\times G\times\Lambda_k$ and $I_m\times G\times \Lambda_m$, respectively. By part 1, $\mathcal{C}_k$ and $\mathcal{C}_m$ are connected components of $\commgraph{\Reeszero{G}{I}{\Lambda}{P}}$. (We observe that $\mathcal{C}_k$ and $\mathcal{C}_m$ are distinct because $I_k\cap I_m=\Lambda_k\cap\Lambda_m=\emptyset$.) Let $\mathcal{D}_k$ and $\mathcal{D}_m$ be the (distinct) subgraphs of $\simplifiedgraph{I}{\Lambda}{P}$ induced by $I_k\times\Lambda_k$ and $I_m\times\Lambda_m$, respectively. We have that $\mathcal{D}_k$ and $\mathcal{D}_m$ are connected components of $\simplifiedgraph{I}{\Lambda}{P}$ (by part 1 of Lemma~\ref{0Rees: connected component G(I,^,P) G(M0[...])}). Let $\mathcal{D}_{km}$ be the subgraph of $\simplifiedgraph{I}{\Lambda}{P}$ induced by $\parens{I_k\times\Lambda_m}\cup\parens{I_m\times\Lambda_k}$. We want to see that $\mathcal{D}_{km}$ is a connected component of $\simplifiedgraph{I}{\Lambda}{P}$.
	
	Our first goal is to show that $\mathcal{D}_{km}$ is connected. Let $\parens{i,\lambda},\parens{j,\mu}\in\parens{I_k\times\Lambda_m}\cup\parens{I_m\times\Lambda_k}$. We consider the following two cases:
	
	\smallskip
	
	\textit{Case 1:} Assume that $\parens{i,\lambda},\parens{j,\mu}\in I_k\times \Lambda_m$. (The proof is identical if we assume that $\parens{i,\lambda},\parens{j,\mu}\in I_m\times \Lambda_k$.) It follows from Lemma~\ref{0Rees: every row/column of 0-closure submatrix has a 0} that columns $i$ and $j$ of $A_k$ and rows $\lambda$ and $\mu$ of $A_m$ contain at least one zero entry. Hence there exist $\lambda',\mu'\in\Lambda_k$ and $i',j'\in I_m$ such that $p_{\lambda' i}=p_{\mu' j}=p_{\lambda i'}=p_{\mu j'}=0$. As a consequence of the fact that $\mathcal{D}_k$ and $\mathcal{D}_m$ are connected components of $\simplifiedgraph{I}{\Lambda}{P}$, we have that there exists a path from $\parens{i,\lambda'}$ to $\parens{j,\mu'}$ in $\mathcal{D}_k$ and there exists a path from $\parens{i',\lambda}$ to $\parens{j',\mu}$ in $\mathcal{D}_m$. Then part 1 of Lemma~\ref{0Rees: two paths zeros --> path non-zeros} ensures that there is a path from $\parens{i,\lambda}$ to $\parens{j,\mu}$ in $\simplifiedgraph{I}{\Lambda}{P}$ and that
	\begin{align*}
		\dist{\simplifiedgraph{I}{\Lambda}{P}}{\parens{i,\lambda}}{\parens{j,\mu}}&\leqslant 1+\max\set[\big]{\dist{\simplifiedgraph{I}{\Lambda}{P}}{\parens{i,\lambda'}}{\parens{j,\mu'}},\dist{\simplifiedgraph{I}{\Lambda}{P}}{\parens{i',\lambda}}{\parens{j',\mu}}}\\
		&\leqslant 1+\max\set[\big]{\diam{\mathcal{D}_k},\diam{\mathcal{D}_m}}.
	\end{align*}

	\smallskip
	
	\textit{Case 2:} Assume that $\parens{i,\lambda}\in I_k\times \Lambda_m$ and $\parens{j,\mu}\in I_m\times \Lambda_k$. (The proof is similar if we have $\parens{i,\lambda}\in I_m\times \Lambda_k$ and $\parens{j,\mu}\in I_k\times \Lambda_m$.) Lemma~\ref{0Rees: every row/column of 0-closure submatrix has a 0} ensures the existence of at least one zero entry in column $i$ of $A_k$, row $\lambda$ of $A_m$, column $j$ of $A_m$ and row $\mu$ of $A_k$, that is, it ensures the existence of $\lambda'\in\Lambda_k$, $i'\in I_m$, $\mu'\in\Lambda_m$ and $j'\in I_k$ such that $p_{\lambda' i}=p_{\lambda i'}=p_{\mu' j}=p_{\mu j'}=0$. Since $\mathcal{D}_k$ and $\mathcal{D}_m$ are connected components of $\simplifiedgraph{I}{\Lambda}{P}$, then there exists a path from $\parens{i,\lambda'}$ to $\parens{j',\mu}$ in $\mathcal{D}_k$ and there exists a path from $\parens{i',\lambda}$ to $\parens{j,\mu'}$ in $\mathcal{D}_m$. Thus, by part 2 of Lemma~\ref{0Rees: two paths zeros --> path non-zeros}, there is a path from $\parens{i,\lambda}$ to $\parens{j,\mu}$ in $\simplifiedgraph{I}{\Lambda}{P}$ and we have
	\begin{align*}
		\dist{\simplifiedgraph{I}{\Lambda}{P}}{\parens{i,\lambda}}{\parens{j,\mu}}&\leqslant 1+\max\set[\big]{\dist{\simplifiedgraph{I}{\Lambda}{P}}{\parens{i,\lambda'}}{\parens{j',\mu}},\dist{\simplifiedgraph{I}{\Lambda}{P}}{\parens{i',\lambda}}{\parens{j,\mu'}}}\\
		&\leqslant 1+\max\set[\big]{\diam{\mathcal{D}_k},\diam{\mathcal{D}_m}}.
	\end{align*}

	\smallskip
	
	This concludes the proof that $\mathcal{D}_{km}$ is connected. Moreover, we have $\diam{\mathcal{D}_{km}}\leqslant1+\max\set[\big]{\diam{\mathcal{D}_k},\diam{\mathcal{D}_m}}$. In order to conclude that $\mathcal{D}_{km}$ is a connected component of $\simplifiedgraph{I}{\Lambda}{P}$, we only need to observe that it follows from Lemma~\ref{0Rees: lemma connected component determined Q M} that the vertices of $\simplifiedgraph{I}{\Lambda}{P}$ that belong to $I_k\times\Lambda_m$ (respectively, $I_m\times\Lambda_k$) are only adjacent to vertices that belong to $I_m\times\Lambda_k$ (respectively, $I_k\times\Lambda_m$). Then Lemma~\ref{0Rees: connected component G(I,^) => connected component G(M0[...])} guarantees that the subgraph of $\commgraph{\Reeszero{G}{I}{\Lambda}{P}}$ induced by
	\begin{displaymath}
		\bigcup_{\parens{i,\lambda}\in \parens{I_k\times\Lambda_m}\cup\parens{I_m\times\Lambda_k}}\set{i}\times G\times\set{\lambda}=\parens{I_k\times G\times\Lambda_m}\cup\parens{I_m\times G\times\Lambda_k}.
	\end{displaymath}
	is a connected component of $\commgraph{\Reeszero{G}{I}{\Lambda}{P}}$.
	
	\medskip
	
	\textbf{Part 3.} Let $i \in I$ and let $\lambda\in\Lambda$ be such that row $\lambda$ has no zero entries. Then $p_{\lambda i}\neq 0$. Let $\mathcal{C}$ be the subgraph of $\commgraph{\Reeszero{G}{I}{\Lambda}{P}}$ induced by $\set{i}\times G\times\set{\lambda}$. It follows from Lemma~\ref{0Rees: subgraph induced by ixGxlambda} that $\mathcal{C}$ is isomorphic to $K_{\abs{G}}$, if $G$ is abelian, or $\mathcal{C}$ is isomorphic to $K_{\abs{\centre{G}}}\graphjointwo\commgraph{G}$, if $G$ is not abelian. In both cases we have that $\mathcal{C}$ is connected. In order to see that $\mathcal{C}$ is a connected component of $\commgraph{\Reeszero{G}{I}{\Lambda}{P}}$, it is enough to see that the vertices of $\commgraph{\Reeszero{G}{I}{\Lambda}{P}}$ belonging to $\set{i}\times G\times\set{\lambda}$ are only adjacent to vertices belonging to $\set{i}\times G\times\set{\lambda}$. Let $j\in I$, $\mu\in\Lambda$ and $x,y\in G$ be such that $\parens{i,x,\lambda}\parens{j,y,\mu}=\parens{j,y,\mu}\parens{i,x,\lambda}$. Since row $\lambda$ has no zero entries, then $p_{\lambda j}\neq0$, which implies, by Lemma~\ref{0Rees: commutativity}, that $i=j$ and $\lambda=\mu$. Thus $\parens{j,y,\mu}\in\set{i}\times G\times\set{\lambda}$, that is, $\parens{j,y,\mu}$ is a vertex of $\mathcal{C}$.
	
	\medskip
	
	\textbf{Part 4.} Let $\lambda \in \Lambda$ and let $i \in I$ be such that column $i$ has no zero entries. We can prove, in a similar way to part 3, that the subgraph of $\commgraph{\Reeszero{G}{I}{\Lambda}{P}}$ induced by $\set{i}\times G\times \set{\lambda}$ is a connected component of $\commgraph{\Reeszero{G}{I}{\Lambda}{P}}$.
\end{proof}

The proof of the previous theorem provides a way to extract from $P$ the connected components of $\commgraph{\Reeszero{G}{I}{\Lambda}{P}}$. The only thing we need to do is execute the $0$-closure method sequentially. In each new application of the $0$-closure method we make sure to choose a zero entry that does not belong to any of the $0$-closure submatrices of $P$ already formed. We stop when there are no more zero entries to choose, that is, when all the zero entries of $P$ lie in some \zeroclosuresub\ already constructed. (Example~\ref{0Rees: example finding connected components} illustrates how to do this with a particular matrix.)

\begin{definition}
	We define the following terminology:
	\begin{description}[style = unboxed]
		{\sloppy \item[Connected component determined by $Q$] the connected component of $\commgraph{\Reeszero{G}{I}{\Lambda}{P}}$ whose set of vertices is $I_Q\times G\times\Lambda_Q$, where $Q=\Psubmatrix{P}{I_Q}{\Lambda_Q}$ is a \zeroclosuresub.\par} 
		
		\item[Connected component determined by $Q$ and $M$] the connected component of $\commgraph{\Reeszero{G}{I}{\Lambda}{P}}$ whose set of vertices is $\parens{I_Q\times G\times\Lambda_M}\cup\parens{I_M\times G\times\Lambda_Q}$, where $Q=\Psubmatrix{P}{I_Q}{\Lambda_Q}$ and $M=\Psubmatrix{P}{I_M}{\Lambda_M}$ are distinct $0$-closure submatrices of $P$. 
		
		\item[Connected component determined by $\parens{\lambda,i}$] the connected component of $\commgraph{\Reeszero{G}{I}{\Lambda}{P}}$ whose set of vertices is $\set{i}\times G\times\set{\lambda}$, where $i\in I$ and $\lambda\in\Lambda$ are such that column $i$ has no zero entries or row $\lambda$ has no zero entries.
	\end{description}
\end{definition}

We observe that, when $\commgraph{\Reeszero{G}{I}{\Lambda}{P}}$ is connected, then $I\times G\times \Lambda$ is the vertex set of a connected component of $\commgraph{\Reeszero{G}{I}{\Lambda}{P}}$. Moreover, in this case we have that $P=\Psubmatrix{P}{I}{\Lambda}$ is a $0$-closure submatrix of $P$ (see Theorem~\ref{0Rees: connectedness}). This means that, when $\commgraph{\Reeszero{G}{I}{\Lambda}{P}}$ is connected, $\commgraph{\Reeszero{G}{I}{\Lambda}{P}}$ is a connected component determined by $P$.

In light of the previous comment, we observe that whenever we are referring to a connected components determined by a \zeroclosuresub, we are also considering the case where $\commgraph{\Reeszero{G}{I}{\Lambda}{P}}$ is connected.

\begin{example}\label{0Rees: example finding connected components}
	In Example~\ref{0Rees: example 0-closure method} we saw that $Q=\Psubmatrix{P}{\set{1,3,4,6,7}}{\set{1,3,4}}$ is a \zeroclosuresub. Since $P$ contains zero entries that are not entries of $Q$, then we need to perform the $0$-closure method again. We choose the entry $\parens{2,2}$ of $P$. Then $\zeroindex{2}{2}=1$ and $M=\Psubmatrix{P}{\set{2,5,8}}{\set{2,5}}$ is the \zeroclosuresubmatrix{2}{2}. We observe that all the zero entries of $P$ are entries of $Q$ or $M$ and, consequently, we do not need to start the $0$-closure method again. We have that $\timeszero{P}$ is \equivalent\ to
	\begin{displaymath}
		\begin{bNiceArray}{cccccc|cccc}[first-row, first-col, margin]
			& \Block{1-6}{\set{1,3,4,6,7}} & & &&&& \Block{1-4}{\set{2,5,8}} &&& \\
			\Block{2-1}{\set{1,3,4}} & \Block{2-6}{\timeszero{Q}} & & &&&& \Block{2-4}{\blocktimes{0.61em}{1.5em}{0.2em}} &&& \\
			\\
			\hline
			\Block{2-1}{\set{2,5}} & \Block{2-6}{\blocktimes{0.61em}{2.45em}{-0.25em}} & & &&&& \Block{2-4}{\timeszero{M}} &&& \\
			\\
			\hline
			\set{6} & \Block{1-10}{\rowtimes{4.3em}{4em}} & & & &&&&&&
		\end{bNiceArray}.
	\end{displaymath}
	
	Now we can use Theorem~\ref{0Rees: finding connected components} to identify the connected components of $\commgraph{\Reeszero{G}{I}{\Lambda}{P}}$:
	\begin{itemize}
		\item Connected components determined by (a unique) \zeroclosuresub: since there are only two $0$-closure submatrices of $P$, we only have two connected components of this type, namely the ones whose vertex sets are $\set{1,3,4,6,7}\times G\times \set{1,3,4}$ and $\set{2,5,8}\times G\times\set{2,5}$.
		\item Connected components determined by two distinct $0$-closure submatrices of $P$: there are only two $0$-closure submatrices of $P$, which implies that there is only one connected component of this type, whose vertex comprises $\set{1,3,4,6,7}\times G\times \set{2,5}$ and $\set{2,5,8}\times G\times\set{1,3,4}$.
		\item Connected components determined by entries of $P$: due to the fact that there is only one row with no zero entries (row $6$) and no columns with no zero entries, then we have eight connected components of this type. Their vertex sets are $\set{i}\times G\times\set{6}$, where $i\in\set{1,2,3,4,5,6,7,8}$.
	\end{itemize}
\end{example}

The succeeding results focus on obtaining the diameter of the connected components of $\commgraph{\Reeszero{G}{I}{\Lambda}{P}}$: results~\ref{0Rees: diameter 0 or 1 or 2}--\ref{0Rees: connected component determined by Q, simplification repeated rows/columns} concern connected components determined by (a unique) \zeroclosuresub, result~\ref{0Rees: diameter connected component Q M} concerns connected components determined by two distinct $0$-closure submatrices of $P$ and result~\ref{0Rees: diameter connected component (i,lambda)} concerns connected components determined by an entry of $P$.

\begin{proposition} \label{0Rees: diameter 0 or 1 or 2}
	Let $Q=\Psubmatrix{P}{I_Q}{\Lambda_Q}$ be a \zeroclosuresub\ and let $\mathcal{C}$ be the connected component of $\commgraph{\Reeszero{G}{I}{\Lambda}{P}}$ determined by $Q$. Then
	\begin{enumerate}
		\item $\diam{\mathcal{C}}=0$ if and only if $\abs{I_Q}=\abs{\Lambda_Q}=1$ and $G$ is trivial.
		
		\item $\diam{\mathcal{C}}=1$ if and only if the following conditions are satisfied:
		\begin{enumerate}
			\item $\abs{I_Q}>1$ or $\abs{\Lambda_Q}>1$ or $G$ is not trivial.
			
			\item $Q=O_{\abs{\Lambda_Q}\times\abs{I_Q}}$.
		\end{enumerate}
		
		\item $\diam{\mathcal{C}}=2$ if and only if the following conditions are satisfied:
		\begin{enumerate}
			\item $Q$ contains a non-zero entry.
			
			\item For all distinct $\lambda,\mu\in\Lambda_Q$ there exists $i\in I_Q$ such that $p_{\lambda i}=p_{\mu i}=0$.
			
			\item For all distinct $i,j\in I_Q$ there exists $\lambda\in\Lambda_Q$ such that $p_{\lambda i}=p_{\lambda j}=0$.
		\end{enumerate}
	\end{enumerate}
\end{proposition}

\begin{proof}
	\textbf{Part 1.} We have
	\begin{align*}
		\diam{\mathcal{C}}=0 & \iff \mathcal{C} \text{ has only one vertex}\\
		& \iff \abs{I_Q\times G\times \Lambda_Q}=1\\
		& \iff \abs{I_Q}=\abs{\Lambda_Q}=1 \text{ and } G \text{ is trivial.}
	\end{align*}
	
	\medskip
	
	\textbf{Part 2.} We are going to start by proving the reverse implication. Suppose that $\abs{I_Q}>1$ or $\abs{\Lambda_Q}>1$ or $G$ is not trivial, and that $Q=O_{\abs{\Lambda_Q}\times\abs{I_Q}}$. First we are going to see that $\diam{\mathcal{C}}\leqslant 1$. Let $i,j\in I_Q$, $\lambda,\mu\in\Lambda_Q$ and $x,y\in G$. Due to the fact that $Q=O_{\abs{\Lambda_Q}\times\abs{I_Q}}$, we have $p_{\lambda j}=p_{\mu i}=0$, which implies, by Lemma~\ref{0Rees: commutativity}, that $\parens{i,x,\lambda}\parens{j,y,\mu}=\parens{j,y,\mu}\parens{i,x,\lambda}$ and, consequently, $\parens{i,x,\lambda}\sim\parens{j,y,\mu}$. Therefore $\diam{\mathcal{C}}\leqslant 1$. Additionally, the fact that $\abs{I_Q}>1$ or $\abs{\Lambda_Q}>1$ or $G$ is not trivial implies (by part 1) that $\diam{\mathcal{C}}\neq 0$, which concludes the proof.
	
	Now we prove the forward implication. Suppose that $\diam{\mathcal{C}}=1$. Then $\diam{\mathcal{C}}\neq 0$, which implies, by part 1, that $\abs{I_Q}>1$ or $\abs{\Lambda_Q}>1$ or $G$ is not trivial. The only thing left to do is to see that $Q=O_{\abs{\Lambda_Q}\times\abs{I_Q}}$. Let $i\in I_Q$ and $\lambda\in\Lambda_Q$. We consider two cases.
	
	\smallskip
	
	\textit{Case 1:} Assume that $\abs{I_Q}=\abs{\Lambda_Q}=1$. Then $I_Q=\set{i}$ and $\Lambda_Q=\set{\lambda}$, which implies that the $0$-closure method yielding $Q$ must have started with the $\parens{\lambda, i}$-th entry of $P$. Since the $0$-closure method starts with a zero entry, then we must have $p_{\lambda i}=0$.
	
	\smallskip
	
	\textit{Case 2:} Assume that $\abs{I_Q}>1$ or $\abs{\Lambda_Q}>1$. Hence $\abs{I_Q\times\Lambda_Q}> 1$ and, consequently, there exist $j\in I_Q$ and $\mu\in\Lambda_Q$ such that $\parens{j,\mu}\neq\parens{i,\lambda}$. Let $x\in G$. Due to the fact that $\diam{\mathcal{C}}=1$ and $\parens{j,\lambda}\neq\parens{i,\mu}$, we have that $\parens{j,x,\lambda}$ is adjacent to $\parens{i,x,\mu}$ and, by Lemma~\ref{0Rees: commutativity}, we have $p_{\lambda i}=p_{\mu j}=0$.
	
	\smallskip
	
	In both cases we saw that $p_{\lambda i}=0$. Since $i$ and $\lambda$ are arbitrary elements of $I_Q$ and $\Lambda_Q$, respectively, then this means we have $Q=O_{\abs{\Lambda_Q}\times\abs{I_Q}}$.
	
	\medskip
	
	\textbf{Part 3.} First we prove the forward implication. Suppose that $\diam{\mathcal{C}}=2$. Since $\diam{\mathcal{C}}\neq 0$, then (by part 1) $\abs{I_Q}>1$ or $\abs{\Lambda_Q}>1$ or $G$ is not trivial. Additionally, $\diam{\mathcal{C}}\neq 1$, which implies (by part 2) that $Q\neq O_{\abs{\Lambda_Q}\times\abs{I_Q}}$. Hence $Q$ contains a non-zero entry and, consequently, condition a) holds.
	
	Now we are going to see that condition b) is also satisfied. We can prove in a symmetrical way that condition c) is satisfied. Let $\lambda,\mu\in\Lambda_Q$ be such that $\lambda\neq\mu$. Let $i\in I_Q$ and $x\in G$. We have $\parens{i,x,\lambda}\neq\parens{i,x,\mu}$ and $\diam{\mathcal{C}}=2$, which implies that $\dist{\commgraph{\Reeszero{G}{I}{\Lambda}{P}}}{\parens{i,x,\lambda}}{\parens{i,x,\mu}}\in\set{1,2}$.
	
	\smallskip
	
	\textit{Case 1:} Assume that $\dist{\commgraph{\Reeszero{G}{I}{\Lambda}{P}}}{\parens{i,x,\lambda}}{\parens{i,x,\mu}}=1$. Then $\parens{i,x,\lambda}$ is adjacent to $\parens{i,x,\mu}$ (that is, $\parens{i,x,\lambda}\parens{i,x,\mu}=\parens{i,x,\mu}\parens{i,x,\lambda}$). As a consequence of Lemma~\ref{0Rees: commutativity}, and the fact that $\lambda\neq\mu$, we have $p_{\lambda i}=p_{\mu i}=0$.
	
	\smallskip
	
	\textit{Case 2:} Assume that $\dist{\commgraph{\Reeszero{G}{I}{\Lambda}{P}}}{\parens{i,x,\lambda}}{\parens{i,x,\mu}}=2$. Then there exist $i'\in I$, $\lambda'\in\Lambda$ and $y\in G$ such that
	\begin{displaymath}
		\parens{i,x,\lambda}-\parens{i',y,\lambda'}-\parens{i,x,\mu}
	\end{displaymath}
	is a path from $\parens{i,x,\lambda}$ to $\parens{i,x,\mu}$ in $\commgraph{\Reeszero{G}{I}{\Lambda}{P}}$. We observe that, due to the fact that $\mathcal{C}$ is a connected component of $\commgraph{\Reeszero{G}{I}{\Lambda}{P}}$ whose set of vertices is $I_Q\times G\times\Lambda_Q$, then $i'\in I_Q$ and $\lambda'\in\Lambda_Q$. Furthermore, we have that $\parens{i,x,\lambda}$ is not adjacent to $\parens{i,x,\mu}$, which implies, by Lemma~\ref{0Rees: adjacency in G(I,^) and G(M0[...])}, that $\parens{i,x,\lambda}$ is not adjacent to $\parens{i,y,\mu}$ and $\parens{i,x,\mu}$ is not adjacent to $\parens{i,y,\lambda}$. Then we have $\parens{i',\lambda'}\neq\parens{i,\lambda}$ and $\parens{i',\lambda'}\neq\parens{i,\mu}$ (because $\parens{i,x,\lambda}$ and $\parens{i,x,\mu}$ are both adjacent to $\parens{i',y,\lambda'}$). Consequently, Lemma~\ref{0Rees: adjacency in G(I,^) and G(M0[...])} guarantees that $\parens{i,x,\lambda}$ and $\parens{i,x,\mu}$ are both adjacent to $\parens{i',x,\lambda'}$. Therefore, by Lemma~\ref{0Rees: commutativity}, $p_{\lambda i'}=p_{\lambda' i}=p_{\mu i'}=p_{\lambda' i}=0$. Thus $p_{\lambda i'}=p_{\mu i'}=0$.
	
	\smallskip
	
	Now we are going to prove the reverse implication. Suppose that conditions a), b) and c) are satisfied. We start by proving that $\diam{\mathcal{C}}\leqslant 2$. Let $i,j\in I_Q$, $\lambda,\mu\in\Lambda_Q$ and $x,y\in G$. In the following two cases we prove that there exists $\lambda'\in\Lambda_Q$ such that $p_{\lambda' i}=p_{\lambda' j}=0$. We can prove in a similar way that there exists $i'\in I_Q$ such that $p_{\lambda i'}=p_{\mu i'}=0$.
	
	\smallskip
	
	\textit{Case 1:} Assume that $i=j$. It follows from Lemma~\ref{0Rees: every row/column of 0-closure submatrix has a 0} that column $i=j$ of $Q$ contains a zero entry. Hence there exists $\lambda'\in\Lambda_Q$ such that $p_{\lambda' i}=p_{\lambda' j}=0$.
	
	\smallskip
	
	\textit{Case 2:} Assume that $i\neq j$. It follows from condition c) that there exists $\lambda'\in\Lambda_Q$ such that $p_{\lambda' i}=p_{\lambda' j}=0$.
	
	\smallskip
	
	Due to the fact that $p_{\lambda i'}=p_{\lambda' i}=0$ and $p_{\lambda' j}=p_{\mu i'}=0$ (and by Lemma~\ref{0Rees: commutativity}), we have $\parens{i,x,\lambda}\parens{i',x,\lambda'}=\parens{i',x,\lambda'}\parens{i,x,\lambda}$ and $\parens{i',x,\lambda'}\parens{j,y,\mu}=\parens{j,y,\mu}\parens{i',x,\lambda'}$, respectively. Thus
	\begin{displaymath}
		\parens{i,x,\lambda}\sim\parens{i',x,\lambda'}\sim\parens{j,y,\mu}.
	\end{displaymath}
	and, consequently, $\dist{\commgraph{\Reeszero{G}{I}{\lambda}{P}}}{\parens{i,x,\lambda}}{\parens{j,y,\mu}}\leqslant 2$. Since $\parens{i,x,\lambda}$ and $\parens{j,y,\mu}$ are arbitrary vertices of $\mathcal{C}$, this implies that $\diam{\mathcal{C}}\leqslant 2$. 
	
	We only need to verify that $\diam{\mathcal{C}}\geqslant 2$. Condition a) guarantees that $Q$ contains a non-zero entry. Thus $Q\neq O_{\abs{\lambda_Q}\times\abs{I_Q}}$ and, consequently, $\diam{\mathcal{C}}\neq 1$ (by part 2). Furthermore, $Q$ contains at least one zero entry (because $Q$ is a \zeroclosuresub). Hence $Q$ contains at least two entries (that is, $\abs{I_Q}>1$ or $\abs{\lambda_Q}>1$), which implies that $\diam{\mathcal{C}}\neq 0$ (by part 1).
\end{proof}

We observe that in part 1 of Theorem~\ref{0Rees: diameter 0 or 1 or 2} we have $Q=O_{1 \times 1}=O_{\abs{\Lambda_Q}\times\abs{I_Q}}$.

We also observe that it is only possible to have connected components determined by $0$-closure submatrices of $P$ with diameter equal to $0$ or $1$ when $\commgraph{\Reeszero{G}{I}{\Lambda}{P}}$ is not connected. In fact, when $\commgraph{\Reeszero{G}{I}{\Lambda}{P}}$ is connected, we have that the unique connected component determined by a \zeroclosuresub\ is $\commgraph{\Reeszero{G}{I}{\Lambda}{P}}$ and we have $\diam{\commgraph{\Reeszero{G}{I}{\Lambda}{P}}}\geqslant 2$.

The theorem that follows shows the significance of the $0$-closure method in determining the diameter of a connected component determined by a \zeroclosuresub. We will see that, in most cases, the diameter of such connected components can be obtained by applying the $0$-closure method for each zero entry of the \zeroclosuresub.

\begin{theorem}\label{0Rees: diameter}
	Let $Q=\Psubmatrix{P}{I_Q}{\Lambda_Q}$ be a \zeroclosuresub\ and let $\mathcal{C}$ be the connected component of $\commgraph{\Reeszero{G}{I}{\Lambda}{P}}$ determined by $Q$. If $\abs{I_Q}>1$ or $\abs{\Lambda_Q}>1$ or $G$ is trivial, then
	\begin{displaymath}
		\diam{\mathcal{C}}=\max\gset{\zeroindex{i}{\lambda}}{i\in I_Q,\text{ } \lambda\in\Lambda_Q \text{ and } p_{\lambda i}=0}.
	\end{displaymath}
\end{theorem}

We observe that if $\abs{I_Q}=\abs{\Lambda_Q}=1$ and $G$ is not trivial, then the result is not true. In fact, if $I_Q=\set{i}$ and $\Lambda_Q=\set{\lambda}$, then Lemma~\ref{0Rees: every row/column of 0-closure submatrix has a 0} guarantees that column $i$ of $Q$ contains a zero entry, and consequently, we have that $Q=\begin{bmatrix}0\end{bmatrix}$. Consequently, part 2 of Proposition~\ref{0Rees: diameter 0 or 1 or 2} implies that $\diam{\mathcal{C}}=1$. In addition, it is easy to see that the $0$-closure method started with the $\parens{\lambda,i}$-th entry of $P$ (the unique zero entry of $Q$) yields $Q$ and it stops at step $0$, which implies that $\zeroindex{i}{\lambda}=0$. Hence
\begin{displaymath}
	\diam{\mathcal{C}}=1\neq 0=\max\gset{\zeroindex{i}{\lambda}}{i\in I_Q,\text{ } \lambda\in\Lambda_Q \text{ and } p_{\lambda i}=0}.
\end{displaymath}

\begin{proof}
	Suppose that $\abs{I_Q}>1$ or $\abs{\Lambda_Q}>1$ or $G$ is trivial. Let $\zeta=\max\gset{\zeroindex{i}{\lambda}}{i\in I_Q,\text{ } \lambda\in\Lambda_Q \text{ and } p_{\lambda i}=0}$.
	
	\smallskip
	
	\textit{Case 1:} Suppose that ${I_Q}=\set{i}$ and $\Lambda_Q=\set{\lambda}$. Then $\abs{I_Q}=\abs{\Lambda_Q}=1$ and, consequently, $G$ must be trivial. Hence, by part 1 of Proposition~\ref{0Rees: diameter 0 or 1 or 2}, $\diam{\mathcal{C}}=0$. Furthermore, $Q=\begin{bmatrix}0\end{bmatrix}$ because, by Lemma~\ref{0Rees: every row/column of 0-closure submatrix has a 0}, every column of a \zeroclosuresub\ contains at least one zero entry, which implies that the $0$-closure method started with the unique entry of $Q$ (which is a zero) stops at step $0$. Consequently, $\zeroindex{i}{\lambda}=0$. Thus $\diam{\mathcal{C}}=0=\max\set{\zeroindex{i}{\lambda}}=\zeta$, which concludes the proof of this case.
	
	\smallskip
	
	\textit{Case 2:} Suppose that $\abs{I_Q}>1$ or $\abs{\Lambda_Q}>1$. Let $\mathcal{D}$ be the subgraph of $\simplifiedgraph{I}{\Lambda}{P}$ induced by $I_Q\times\Lambda_Q$. (We observe that, by part 1 of Lemma~\ref{0Rees: connected component G(I,^,P) G(M0[...])}, $\mathcal{D}$ is a connected component of $\simplifiedgraph{I}{\Lambda}{P}$.) Due to the fact that $\abs{I_Q}>1$ or $\abs{\Lambda_Q}>1$, then we have that $\abs{I_Q\times\Lambda_Q}>1$; that is, $\mathcal{D}$ contains more than one vertex. Thus $\diam{\mathcal{D}}\geqslant 1$ and, consequently, Lemma~\ref{0Rees: connected component G(I,^) => connected component G(M0[...])} guarantees that $\diam{\mathcal{C}}=\diam{\mathcal{D}}$. This equality shows that, in order to prove that $\diam{\mathcal{C}}=\zeta$, we just need to see that $\diam{\mathcal{D}}=\zeta$, which is done in the remainder of the proof.
	
	First, we prove that $\zeta\leqslant\diam{\mathcal{D}}$. Let $i\in I_Q$ and $\lambda\in\Lambda_Q$ be such that $p_{\lambda i}=0$ and $\zeroindex{i}{\lambda}=\zeta$. Let $M=\Psubmatrix{P}{I_M}{\Lambda_M}$ be the \zeroclosuresubmatrix{i}{\lambda}. Then, by part 1 of Lemma~\ref{0Rees: comparing 0-closure submatrices}, $I_Q=I_M$ and $\Lambda_Q=\Lambda_M$. Furthermore, by Lemma~\ref{0Rees: d(0,?) in G(I,^)}, there exist $j\in I_M=I_Q$ and $\mu\in\Lambda_M=\Lambda_Q$ such that $\dist{\simplifiedgraph{I}{\Lambda}{P}}{\parens{i,\lambda}}{\parens{j,\mu}}=\zeroindex{i}{\lambda}$. Thus
	\begin{displaymath}
		\zeta=\zeroindex{i}{\lambda}=\dist{\simplifiedgraph{I}{\Lambda}{P}}{\parens{i,\lambda}}{\parens{j,\mu}}\leqslant \diam{\mathcal{D}}
	\end{displaymath}
	(because $\parens{i,\lambda},\parens{j,\mu}\in I_Q\times\Lambda_Q$, the set of vertices of $\mathcal{D}$).

	Now we demonstrate that $\diam{\mathcal{D}}\leqslant \zeta$. Let $i,j\in I_Q$ and $\lambda,\mu\in\Lambda_Q$ be such that $\dist{\simplifiedgraph{I}{\Lambda}{P}}{\parens{i,\lambda}}{\parens{j,\mu}}=\diam{\mathcal{D}}$. Since $Q$ is a \zeroclosuresub, then Lemma~\ref{0Rees: every row/column of 0-closure submatrix has a 0} guarantees that rows $\lambda$ and $\mu$ of $Q$ and columns $i$ and $j$ of $Q$ contain at least one zero entry. Hence there exist $i',j'\in I_Q$ and $\lambda',\mu'\in\Lambda_Q$ such that $p_{\lambda i'}=p_{\mu j'}=p_{\lambda' i}=p_{\mu' j}=0$. Furthermore, $\mathcal{D}$ is a connected component of $\simplifiedgraph{I}{\Lambda}{P}$ and $\parens{i',\lambda},\parens{j',\mu},\parens{i,\lambda'},\parens{j,\mu'}\in I_Q\times\Lambda_Q$, the vertex set of $\mathcal{D}$, which implies that there exists a path from $\parens{i,\lambda'}$ to $\parens{j,\mu'}$ in $\simplifiedgraph{I}{\Lambda}{P}$, a path from $\parens{i',\lambda}$ to $\parens{j',\mu}$ in $\simplifiedgraph{I}{\Lambda}{P}$, a path from $\parens{i,\lambda'}$ to $\parens{j',\mu}$ in $\simplifiedgraph{I}{\Lambda}{P}$ and a path from $\parens{i',\lambda}$ to $\parens{j,\mu'}$ in $\simplifiedgraph{I}{\Lambda}{P}$. Let
	\begin{align*}
		n&=\dist{\simplifiedgraph{I}{\Lambda}{P}}{\parens{i,\lambda'}}{\parens{j,\mu'}},& m=\dist{\simplifiedgraph{I}{\Lambda}{P}}{\parens{i',\lambda}}{\parens{j',\mu}},\\
		n'&=\dist{\simplifiedgraph{I}{\Lambda}{P}}{\parens{i,\lambda'}}{\parens{j',\mu}},& m'=\dist{\simplifiedgraph{I}{\Lambda}{P}}{\parens{i',\lambda}}{\parens{j,\mu'}}.
	\end{align*}
	We observe that we have $\diam{\mathcal{D}}\geqslant \max\set{n,m,n',m'}$. In the next three cases we are going to see that $\diam{\mathcal{D}}=\max\set{n,m,n',m'}$.
	
	\smallskip
	
	\textsc{Sub-case 1:} Assume that $\max\set{n,m}$ is even. Then, by part 1 of Lemma~\ref{0Rees: two paths zeros --> path non-zeros}, we have
	\begin{displaymath}
		\diam{\mathcal{D}}=\dist{\simplifiedgraph{I}{\Lambda}{P}}{\parens{i,\lambda}}{\parens{j,\mu}}\leqslant \max\set{n,m}\leqslant\max\set{n,m,n',m'}\leqslant \diam{\mathcal{D}},
	\end{displaymath}
	which implies that $\diam{\mathcal{D}}=\max\set{n,m,n',m'}$.
	
	\smallskip
	
	\textsc{Sub-case 2:} Assume that $\max\set{n',m'}$ is odd. Then, by part 2 of Lemma~\ref{0Rees: two paths zeros --> path non-zeros}, we have
	\begin{displaymath}
		\diam{\mathcal{D}}=\dist{\simplifiedgraph{I}{\Lambda}{P}}{\parens{i,\lambda}}{\parens{j,\mu}}\leqslant \max\set{n',m'}\leqslant\max\set{n,m,n',m'}\leqslant \diam{\mathcal{D}},
	\end{displaymath}
	which implies that $\diam{\mathcal{D}}=\max\set{n,m,n',m'}$.
	
	\smallskip
	
	\textsc{Sub-case 3:} Assume that $\max\set{n,m}$ is odd and $\max\set{n',m'}$ is even. It follows from part 1 of Lemma~\ref{0Rees: two paths zeros --> path non-zeros} that
	\begin{displaymath}
		\underbrace{\max\set{n,m}}_{\text{\small odd}}\leqslant\diam{\mathcal{D}}=\dist{\simplifiedgraph{I}{\Lambda}{P}}{\parens{i,\lambda}}{\parens{j,\mu}}\leqslant \underbrace{1+\max\set{n,m}}_{\text{\small even}},
	\end{displaymath}
	and it follows from part 2 of Lemma~\ref{0Rees: two paths zeros --> path non-zeros} that
	\begin{displaymath}
		\underbrace{\max\set{n',m'}}_{\text{\small even}}\leqslant\diam{\mathcal{D}}=\dist{\simplifiedgraph{I}{\Lambda}{P}}{\parens{i,\lambda}}{\parens{j,\mu}}\leqslant \underbrace{1+\max\set{n',m'}}_{\text{\small odd}}.
	\end{displaymath}
	Consequently, if $\diam{\mathcal{D}}$ is odd, then $\diam{\mathcal{D}}=\max\set{n,m}=\max\set{n,m,n',m'}$; and if $\diam{\mathcal{D}}$ is even, then $\diam{\mathcal{D}}=\max\set{n',m'}=\max\set{n,m,n',m'}$.
	
	
	
	\smallskip
	
	We just finished proving that $\diam{\mathcal{D}}=\max\set{n,m,n',m'}$. Now we are going to see that $\diam{\mathcal{D}}\leqslant\zeta$. Assume, without loss of generality, that $\max\set{n,n',m,m'}=n$. We have $p_{\lambda' i}=0$. Let $M=\Psubmatrix{P}{I_M}{\Lambda_M}$ be the \zeroclosuresubmatrix{i}{\lambda'}. It follows from part 1 of Lemma~\ref{0Rees: comparing 0-closure submatrices} that $I_Q=I_M$ and $\Lambda_Q=\Lambda_M$. Thus $i,j\in I_M$ and $\lambda',\mu'\in\Lambda_M$, which implies, by Lemma~\ref{0Rees: d(0,?) in G(I,^)}, that
	\begin{displaymath}
		\diam{\mathcal{D}}=\max\set{n,n',m,m'}=n=\dist{\simplifiedgraph{I}{\Lambda}{P}}{\parens{i,\lambda'}}{\parens{j,\mu'}}\leqslant \zeroindex{i}{\lambda'}\leqslant \zeta.\qedhere
	\end{displaymath} 
\end{proof}

Theorem~\ref{0Rees: diameter} provides a method for obtaining the diameter of a connected component of $\commgraph{\Reeszero{G}{I}{\Lambda}{P}}$ determined by a \zeroclosuresub. However, in order to do so, we have to execute the $0$-closure method a considerable number of times --- namely the number of zeros of $Q$. The purpose of the next proposition, as well as the one that succeeds it, is to simplify this process by reducing the number of zero entries of $Q$ used to start the $0$-closure method.

\begin{proposition}\label{0Rees: connected components determined by Q, simplification >=3}
	Let $Q=\Psubmatrix{P}{I_Q}{\Lambda_Q}$ be a \zeroclosuresub\ and let $\mathcal{C}$ be the connected component of $\commgraph{\Reeszero{G}{I}{\Lambda}{P}}$ determined by $Q$. Let $M=\Psubmatrix{Q}{I_M}{\Lambda_M}=\Psubmatrix{P}{I_M}{\Lambda_M}$ be a submatrix of $Q$ that contains a row and/or a column of zeros. Then, if $\diam{\mathcal{C}}\geqslant 3$, we have
	\begin{displaymath}
		\diam{\mathcal{C}}=\max\parens[\big]{\set{3}\cup\gset{\zeroindex{i}{\lambda}}{i\in I_Q\setminus I_M \text{ or } \lambda\in \Lambda_Q\setminus\Lambda_M, \text{ and } p_{\lambda i}=0}}.
	\end{displaymath}
	
		
\end{proposition}

We observe that the presence of the constant term $3$ in this result shows why the characterization in Proposition~\ref{0Rees: diameter 0 or 1 or 2} of connected components (determined by some \zeroclosuresub) whose diameter is either $0$, $1$ or $2$ is important.

\begin{proof}
	Suppose that $M$ contains a row of zeros. (If $M$ contains a column of zeros, then the proof is symmetrical.) Let $\mu\in\Lambda_M$ be such that row $\mu$ of $M$ is a row of zeros. Let $X=\gset{\zeroindex{i}{\lambda}}{i\in I_Q\setminus I_M \text{ or } \lambda\in \Lambda_Q\setminus\Lambda_M, \text{ and } p_{\lambda i}=0}$.
	
	It follows from Theorem~\ref{0Rees: diameter} that there exist $i\in I_Q$ and $\lambda\in\Lambda_Q$ such that $p_{\lambda i}=0$ and $\diam{\mathcal{C}}=\zeroindex{i}{\lambda}$. Furthermore, since $\diam{\mathcal{C}}\geqslant 3$, we have $\diam{\mathcal{C}}\geqslant\max\parens{\set{3}\cup X}$. Our aim is to see that we also have $\diam{\mathcal{C}}\leqslant\max\parens{\set{3}\cup X}$.
	
	\smallskip
	
	\textit{Case 1:} Suppose that $i\in I_Q\setminus I_M$ or $\lambda\in\Lambda_Q\setminus\Lambda_M$. Then $\zeroindex{i}{\lambda}\in X$ and, consequently,
	\begin{displaymath}
		\diam{\mathcal{C}}=\zeroindex{i}{\lambda}\leqslant\max X\leqslant\max\parens{\set{3}\cup X}.
	\end{displaymath}
	
	\smallskip
	
	\textit{Case 2:} Suppose that $i\in I_M$ and $\lambda\in\Lambda_M$. Since $I_M\subseteq I_Q$ and $\Lambda_M\subseteq\Lambda_Q$, then  part 1 of Lemma~\ref{0Rees: comparing 0-closure submatrices} ensures that $Q$ is the \zeroclosuresubmatrix{i}{\lambda}. We have that at each step of the $0$-closure method (started with the $\parens{\lambda,i}$-th entry of $P$) we select at least one zero entry of $P$. In particular, at step $\zeroindex{i}{\lambda}$ we also select a zero entry of $P$. Let $i'\in I$ and $\lambda'\in\Lambda$ be such that $p_{\lambda' i'}=0$ and the $\parens{\lambda',i'}$-th entry of $P$ is selected at step $\zeroindex{i}{\lambda}$ of the $0$-closure method (started with the $\parens{\lambda,i}$-th entry of $P$). It follows from Lemma~\ref{0Rees: d(0,?) in G(I,^)} that $\dist{\simplifiedgraph{I}{\Lambda}{P}}{\parens{i,\lambda}}{\parens{i',\lambda'}}=\zeroindex{i}{\lambda}$, and it follows from the fact that $Q$ is the \zeroclosuresubmatrix{i}{\lambda}, that $i'\in I_Q$ and $\lambda'\in\Lambda_Q$.
	
	\smallskip
	
	\textsc{Sub-case 1:} Assume that $i'\in I_Q\setminus I_M$ or $\lambda'\in\Lambda_Q\setminus\Lambda_M$. Then $\zeroindex{i'}{\lambda'}\in X$. Furthermore, $Q$ is the \zeroclosuresubmatrix{i'}{\lambda'} (by part 1 of Lemma~\ref{0Rees: comparing 0-closure submatrices}) and, consequently,
	\begin{align*}
		\diam{\mathcal{C}}&=\zeroindex{i}{\lambda}\\
		&=\dist{\simplifiedgraph{I}{\Lambda}{P}}{\parens{i,\lambda}}{\parens{i',\lambda'}}\\
		&=\dist{\simplifiedgraph{I}{\Lambda}{P}}{\parens{i',\lambda'}}{\parens{i,\lambda}}\\
		&\leqslant\zeroindex{i'}{\lambda'}& \bracks{\text{by Lemma~\ref{0Rees: d(0,?) in G(I,^)}}}\\
		&\leqslant \max X\\
		&\leqslant\max\parens{\set{3}\cup X}.
	\end{align*}
	
	\smallskip
	
	\textsc{Sub-case 2:} Assume that $i'\in I_M$ and $\lambda'\in\Lambda_M$. Since row $\mu$ of $M$ is a row of zeros, we have $p_{\mu i}=p_{\mu i'}=0$. We also have $p_{\lambda i}=p_{\lambda' i'}=0$. Thus
	\begin{displaymath}
		\parens{i,\lambda}\sim\parens{i,\mu}\sim\parens{i',\mu}\sim\parens{i',\lambda'},
	\end{displaymath}
	which implies that $\dist{\simplifiedgraph{I}{\Lambda}{P}}{\parens{i,\lambda}}{\parens{i',\lambda'}}\leqslant 3$ and, consequently,
	\begin{displaymath}
		\diam{\mathcal{C}}=\zeroindex{i}{\lambda}=\dist{\simplifiedgraph{I}{\Lambda}{P}}{\parens{i,\lambda}}{\parens{i',\lambda'}}\leqslant 3\leqslant\max\parens{\set{3}\cup X}.\qedhere
	\end{displaymath}
\end{proof}

\begin{proposition}\label{0Rees: connected component determined by Q, simplification repeated rows/columns}
	Let $Q=\Psubmatrix{P}{I_Q}{\Lambda_Q}$ be a \zeroclosuresub. Then
	\begin{enumerate}
		\item If there exist distinct $\lambda,\mu\in\Lambda_Q$ such that rows $\lambda$ and $\mu$ of $\timeszero{Q}$ are equal, then $\zeroindex{i}{\lambda}=\zeroindex{i}{\mu}$ for all $i\in I_Q$ such that $p_{\lambda i}=p_{\mu i}=0$.
		
		\item If there exist distinct $i,j\in I_Q$ such that columns $i$ and $j$ of $\timeszero{Q}$ are equal, then $\zeroindex{i}{\lambda}=\zeroindex{j}{\lambda}$ for all $\lambda\in \Lambda_Q$ such that $p_{\lambda i}=p_{\lambda j}=0$.
	\end{enumerate}
\end{proposition}

\begin{proof}

	We are going to demonstrate part 1 --- part 2 can be proved symmetrically. Let $\lambda,\mu\in\Lambda_Q$ be such that $\lambda\neq\mu$ and rows $\lambda$ and $\mu$ of $\timeszero{Q}$ are equal. This means that for all $i\in I_Q$ we have $p_{\lambda i}=p_{\mu i}=0$ or $p_{\lambda i},p_{\mu i}\in G$. Let $i\in I_Q$ be such that $p_{\lambda i}=p_{\mu i}=0$. (We note that such an $i$ exists by Lemma~\ref{0Rees: every row/column of 0-closure submatrix has a 0}.) Let $\parens{M_0,M_1,\ldots,M_{\zeroindex{i}{\lambda}}}$ and $\parens{N_0,N_1,\ldots,N_{\zeroindex{i}{\mu}}}$ be the sequence of matrices obtained from the $0$-closure method when we start with the $\parens{\lambda,i}$-th entry of $P$ and the $\parens{\mu,i}$-th entry of $P$, respectively. We are going to prove that the matrices $M_1$ and $N_1$ are formed by exactly the same entries.
	
	The following diagrams illustrate how the $0$-closure method processes when we start with the $\parens{\lambda,i}$-th entry of $P$ (left) and the $\parens{\mu,i}$-th entry of $P$ (right). The entry in red corresponds to the one we chose to start the $0$-closure method with, the entries surrounded by a red circle are the zeros marked at the beginning of step $1$, and the entries in yellow are the entries selected at step $1$.
	\begin{displaymath}
		\begin{bNiceArray}{ccccccccccc}[first-row,first-col,margin]
			& & & & i & & &&  & & \\
			\\
			\lambda & & \cellcolor{Goldenrod} 0 & & \cellcolor{Red} 0  &  &  & \cellcolor{Goldenrod} 0 &&& \cellcolor{Goldenrod} 0 & \\
			\\
			\mu & & \cellcolor{Goldenrod} 0 & & \cellcolor{Goldenrod} 0 & & & \cellcolor{Goldenrod} 0 &&& \cellcolor{Goldenrod} 0 & \\
			& & & & & & & & & && \\
			& & \cellcolor{Goldenrod} & & \cellcolor{Goldenrod} 0 & & & \cellcolor{Goldenrod} & & & \cellcolor{Goldenrod} & \\
			& & \cellcolor{Goldenrod} & & \cellcolor{Goldenrod} 0 & & & \cellcolor{Goldenrod} & & & \cellcolor{Goldenrod} & \\
			\CodeAfter
			\begin{tikzpicture}
				\begin{scope}[dash pattern=on 0pt off 2.3pt,line cap=round, thick]
					\draw (2.5-|1) -- ($(2.5-|2)+(1mm,0mm)$);
					\draw ($(2.5-|3)+(-0.5mm,0mm)$) -- ($(2.5-|4)+(1mm,0mm)$);
					\draw ($(2.5-|5)+(-0.5mm,0mm)$) -- ($(2.5-|7)+(1mm,0mm)$);
					\draw ($(2.5-|8)+(-0.5mm,0mm)$) -- ($(2.5-|10)+(1mm,0mm)$);
					\draw ($(2.5-|11)+(-0.5mm,0mm)$) -- ($(2.5-|12)+(1mm,0mm)$);
					\draw (4.5-|1) -- ($(4.5-|2)+(1mm,0mm)$);
					\draw ($(4.5-|3)+(-0.5mm,0mm)$) -- ($(4.5-|4)+(1mm,0mm)$);
					\draw ($(4.5-|5)+(-0.5mm,0mm)$) -- ($(4.5-|7)+(1mm,0mm)$);
					\draw ($(4.5-|8)+(-0.5mm,0mm)$) -- ($(4.5-|10)+(1mm,0mm)$);
					\draw ($(4.5-|11)+(-0.5mm,0mm)$) -- ($(4.5-|12)+(1mm,0mm)$);
					\draw (2.5|-1) -- (2.5|-2);
					\draw ($(2.5|-3)+(0mm,0mm)$) -- (2.5|-4);
					\draw (2.5|-5) -- (2.5|-6);
					\draw (4.5|-1) -- (4.5|-2);
					\draw (4.5|-3) -- (4.5|-4);
					\draw (4.5|-5) -- (4.5|-6);
					\draw (7.5|-1) -- (7.5|-2);
					\draw (7.5|-3) -- (7.5|-4);
					\draw (7.5|-5) -- (7.5|-6);
					\draw (10.5|-1) -- (10.5|-2);
					\draw (10.5|-3) -- (10.5|-4);
					\draw (10.5|-5) -- (10.5|-6);
				\end{scope}
				
				\draw[thick, color=Red] (2-2) circle (2.1mm);
				\draw[thick, color=Red] (2-7) circle (2.1mm);
				\draw[thick, color=Red] (2-10) circle (2.1mm);
				\draw[thick, color=Red] (4-4) circle (2.1mm);
				\draw[thick, color=Red] (6-4) circle (2.1mm);
				\draw[thick, color=Red] (7-4) circle (2.1mm);
			\end{tikzpicture}
		\end{bNiceArray}\quad
		\begin{bNiceArray}{ccccccccccc}[first-row,first-col,margin]
			& & & & i & & &&  & & \\
			\\
			\lambda & & \cellcolor{Goldenrod} 0 & & \cellcolor{Goldenrod} 0  &  &  & \cellcolor{Goldenrod} 0 &&& \cellcolor{Goldenrod} 0 & \\
			\\
			\mu & & \cellcolor{Goldenrod} 0 & & \cellcolor{Red} 0 & & & \cellcolor{Goldenrod} 0 &&& \cellcolor{Goldenrod} 0 & \\
			& & & & & & & & & && \\
			& & \cellcolor{Goldenrod} & & \cellcolor{Goldenrod} 0 & & & \cellcolor{Goldenrod} & & & \cellcolor{Goldenrod} & \\
			& & \cellcolor{Goldenrod} & & \cellcolor{Goldenrod} 0 & & & \cellcolor{Goldenrod} & & & \cellcolor{Goldenrod} & \\
			\CodeAfter
			\begin{tikzpicture}
				\begin{scope}[dash pattern=on 0pt off 2.3pt,line cap=round, thick]
					\draw (2.5-|1) -- ($(2.5-|2)+(1mm,0mm)$);
					\draw ($(2.5-|3)+(-0.5mm,0mm)$) -- ($(2.5-|4)+(1mm,0mm)$);
					\draw ($(2.5-|5)+(-0.5mm,0mm)$) -- ($(2.5-|7)+(1mm,0mm)$);
					\draw ($(2.5-|8)+(-0.5mm,0mm)$) -- ($(2.5-|10)+(1mm,0mm)$);
					\draw ($(2.5-|11)+(-0.5mm,0mm)$) -- ($(2.5-|12)+(1mm,0mm)$);
					\draw (4.5-|1) -- ($(4.5-|2)+(1mm,0mm)$);
					\draw ($(4.5-|3)+(-0.5mm,0mm)$) -- ($(4.5-|4)+(1mm,0mm)$);
					\draw ($(4.5-|5)+(-0.5mm,0mm)$) -- ($(4.5-|7)+(1mm,0mm)$);
					\draw ($(4.5-|8)+(-0.5mm,0mm)$) -- ($(4.5-|10)+(1mm,0mm)$);
					\draw ($(4.5-|11)+(-0.5mm,0mm)$) -- ($(4.5-|12)+(1mm,0mm)$);
					\draw (2.5|-1) -- (2.5|-2);
					\draw ($(2.5|-3)+(0mm,0mm)$) -- (2.5|-4);
					\draw (2.5|-5) -- (2.5|-6);
					\draw (4.5|-1) -- (4.5|-2);
					\draw (4.5|-3) -- (4.5|-4);
					\draw (4.5|-5) -- (4.5|-6);
					\draw (7.5|-1) -- (7.5|-2);
					\draw (7.5|-3) -- (7.5|-4);
					\draw (7.5|-5) -- (7.5|-6);
					\draw (10.5|-1) -- (10.5|-2);
					\draw (10.5|-3) -- (10.5|-4);
					\draw (10.5|-5) -- (10.5|-6);
				\end{scope}
				
				\draw[thick, color=Red] (4-2) circle (2.1mm);
				\draw[thick, color=Red] (4-7) circle (2.1mm);
				\draw[thick, color=Red] (4-10) circle (2.1mm);
				\draw[thick, color=Red] (2-4) circle (2.1mm);
				\draw[thick, color=Red] (6-4) circle (2.1mm);
				\draw[thick, color=Red] (7-4) circle (2.1mm);
			\end{tikzpicture}
		\end{bNiceArray}
	\end{displaymath}
	
	Matrix $M_1$ is yielded at the end of step $1$ of the $0$-closure method started with the $\parens{\lambda,i}$-th entry of $P$. Hence $M_1$ is the smallest submatrix of $P$ that contain the $\parens{\lambda,i}$-th entry of $P$, the zero entries of column $i$ and the zero entries of row $\lambda$. This matrix is precisely the one formed by all the entries of $P$ that lie in a row whose intersection with column $i$ is a zero entry, and that lie in a column whose intersection with row $\lambda$ is also a zero entry. Analogously, $N_1$ is the matrix formed by all the entries of $P$ that lie in a row whose intersection with column $i$ is a zero entry, and that lie in a column whose intersection with row $\mu$ is also a zero entry. Since rows $\lambda$ and $\mu$ have zeros in the same positions, then the intersection of a column with row $\lambda$ is a zero entry if and only if the intersection of that column with row $\mu$ is a zero entry. Therefore, an entry of $P$ is an entry of $M_1$ if and only if it is an entry of $N_1$; that is, at the end of step 1 of both $0$-closure methods, the entries of $P$ selected so far are exactly the same. This means that the number of steps required to finish each method is exactly the same. Therefore we must have $\zeroindex{i}{\lambda}=\zeroindex{i}{\mu}$.
\end{proof}

In the procedure we describe below we show how we can use a combination of the results presented above to determine efficiently the diameter of $\mathcal{C}$, a connected component determined by the $0$-closure submatrix $Q$ of $P$. (See Example~\ref{0Rees: example diameter} for an illustration of how to determine the diameter of a connected component determined by a \zeroclosuresub.)
\begin{enumerate}
	\item We use Proposition~\ref{0Rees: diameter 0 or 1 or 2} to check if $\diam{\mathcal{C}}$ is $0$, $1$ or $2$.
	
	\item Assume that $\diam{\mathcal{C}}\notin\set{0,1,2}$. We start by choosing one of the following strategies:
	\begin{enumerate}
		\item We pick a row $\lambda$ of $Q$ and take the biggest subset $I'$ of $I_Q$ such that $\Psubmatrix{Q}{I'}{\lambda}=O_{1\times \abs{I'}}$. It follows from Proposition~\ref{0Rees: connected components determined by Q, simplification >=3} that we only need to do the $0$-closure method with the zero entries of $Q$ that are entries of $M=\Psubmatrix{Q}{I_Q\setminus I'}{\Lambda_Q}$.
		
		\item We pick a column $i$ of $Q$ and take the biggest subset $\Lambda'$ of $\Lambda$ such that $\Psubmatrix{Q}{i}{\Lambda'}=O_{\abs{\Lambda'}\times 1}$. It follows from Proposition~\ref{0Rees: connected components determined by Q, simplification >=3} that we only need to do the $0$-closure method with the zero entries of $Q$ that are entries of $M=\Psubmatrix{Q}{I_Q}{\Lambda_Q\setminus\Lambda'}$.
	\end{enumerate}
	After obtaining the matrix $M$, we can repeat the procedure described ahead several times, so we can exclude more zero entries. If $M$ contains two zero entries located in equal (but not the same) rows (respectively, columns) of $\timeszero{Q}$, then, by Proposition~\ref{0Rees: connected component determined by Q, simplification repeated rows/columns}, we only need to perform the $0$-closure method starting at one of them.
\end{enumerate}

Now we are going to determine the diameter of the connected components of $\commgraph{\Reeszero{G}{I}{\Lambda}{P}}$ determined by either two $0$-closure matrices or an entry of $P$. Proposition~\ref{0Rees: diameter connected component Q M} concerns the first type of connected component, and Proposition~\ref{0Rees: diameter connected component (i,lambda)} concerns the second type. 

\begin{proposition}\label{0Rees: diameter connected component Q M}
	Suppose that $\commgraph{\Reeszero{G}{I}{\Lambda}{P}}$ is not connected. Let $Q=\Psubmatrix{P}{I_Q}{\Lambda_Q}$ and $M=\Psubmatrix{P}{I_M}{\Lambda_M}$ be distinct $0$-closure submatrices of $P$. Let $\mathcal{C}_Q$ and $\mathcal{C}_M$ be the (distinct) connected components of $\commgraph{\Reeszero{G}{I}{\Lambda}{P}}$ determined by $Q$ and $M$, respectively, and let $\mathcal{C}_{QM}$ be the connected component  of $\commgraph{\Reeszero{G}{I}{\Lambda}{P}}$ determined by $Q$ and $M$. Then
	\begin{enumerate}
		\item If $\abs{I_Q}=\abs{\Lambda_Q}=\abs{I_M}=\abs{\Lambda_M}=1$ and $G$ is a non-trivial abelian group, then $\diam{\mathcal{C}_{QM}}=1$.
		
		\item Otherwise, $\diam{\mathcal{C}_{QM}}=1+\max\set[\big]{\diam{\mathcal{C}_Q},\diam{\mathcal{C}_M}}$.
	\end{enumerate}
\end{proposition}

\begin{proof}
	Let $\mathcal{D}_Q$, $\mathcal{D}_M$ and $\mathcal{D}_{QM}$ be the subgraphs of $\simplifiedgraph{I}{\Lambda}{P}$ induced by $I_Q\times\Lambda_Q$, $I_M\times\Lambda_M$ and $\parens{I_Q\times\Lambda_M}\cup\parens{I_M\times\Lambda_Q}$, respectively. We have that $\mathcal{D}_Q$ and $\mathcal{D}_M$ are connected components of $\simplifiedgraph{I}{\Lambda}{P}$ (by part 1 of Lemma~\ref{0Rees: connected component G(I,^,P) G(M0[...])}) and that $\mathcal{D}_{QM}$ is a connected component of $\simplifiedgraph{I}{\Lambda}{P}$ (by part 2 of the proof of Theorem~\ref{0Rees: finding connected components}). We partition the proof into two parts. In part 1 we are going to prove that $\diam{\mathcal{D}_{QM}}= 1+\max\set[\big]{\diam{\mathcal{D}_Q},\diam{\mathcal{D}_M}}$, and in part 2 we are going to use the former equality to prove the statement of Proposition~\ref{0Rees: diameter connected component Q M}.
	
	\medskip
	
	\textbf{Part 1.} We are going to see that $\diam{\mathcal{D}_{QM}}= 1+\max\set[\big]{\diam{\mathcal{D}_Q},\diam{\mathcal{D}_M}}$. We recall that in part 2 of the proof of Theorem~\ref{0Rees: finding connected components} we saw that $\diam{\mathcal{D}_{QM}}\leqslant 1+\max\set[\big]{\diam{\mathcal{D}_Q},\diam{\mathcal{D}_M}}$. Now we are going to establish that $\diam{\mathcal{D}_{QM}}\geqslant 1+\max\set[\big]{\diam{\mathcal{D}_Q},\diam{\mathcal{D}_M}}$.
	
	Assume, without loss of generality, that $\diam{\mathcal{D}_Q}\geqslant\diam{\mathcal{D}_M}$. As a consequence of Theorem~\ref{0Rees: diameter}, there exist $i\in I_Q$ and $\lambda\in\Lambda_Q$ such that $p_{\lambda i}=0$ and $\diam{\mathcal{D}_Q}=\zeroindex{i}{\lambda}$ and, as a consequence of Lemma~\ref{0Rees: d(0,?) in G(I,^)}, there exist $j\in I_Q$ and $\mu\in\Lambda_Q$ such that $\dist{\simplifiedgraph{I}{\Lambda}{P}}{\parens{i,\lambda}}{\parens{j,\mu}}=\zeroindex{i}{\lambda}$. Let $i',j'\in I_M$ and $\lambda',\mu'\in\Lambda_M$. We consider two cases:

	
	\smallskip
	
	\textit{Case 1:} Suppose that $\max\set[\big]{\diam{\mathcal{D}_Q},\diam{\mathcal{D}_M}}=\diam{\mathcal{D}_Q}$ is odd.
	
	The diagram below provides an illustration of the proof of case 1. In the diagram we have two solid arrows: one from entry $\parens{\mu,j'}$ to entry $\parens{\lambda, i'}$, and another from entry $\parens{\mu',j}$ to entry $\parens{\lambda',i}$; and we also have a dashed arrow from entry $\parens{\mu,j}$ to entry $\parens{\lambda,i}$. The solid arrows represent paths (whose existence we will prove) of minimum length (in $\simplifiedgraph{I}{\Lambda}{P}$) from $\parens{j',\mu}$ to $\parens{i',\lambda}$, and from $\parens{j,\mu'}$ to $\parens{i,\lambda'}$. We will then use these paths to construct a path from $\parens{j,\mu}$ to $\parens{i,\lambda}$ (which is represented by the dashed arrow) and to determine an upper bound for the distance between $\parens{j,\mu}$ and $\parens{i,\lambda}$.
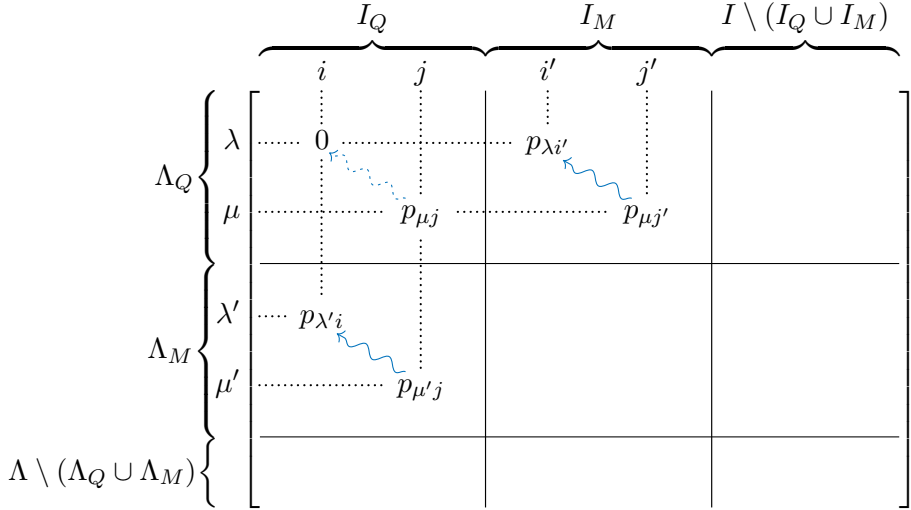
	\begin{displaymath}
		\begin{bNiceArray}{ccccc|ccccc|ccccccc}[first-row,first-col,margin]
			\rule{0pt}{30pt} && i && j &&& i' &&j'&&&&&&& \\
			\\
			\lambda &&0&&&&& p_{\lambda i'} &&&&&&&&&& \\
			\\
			\mu &&&&p_{\mu j}&&&&& p_{\mu j'} &&&&&&&& \\
			\\
			\hline
			\\
			\lambda' && p_{\lambda' i} &&&&&&&&&&&&&&& \\
			\\
			\mu' &&&& p_{\mu' j} &&&&&&&&&&&&& \\
			\\
			\hline
			\\
			\\
			\CodeAfter
			\begin{tikzpicture}
				\node[xshift=-0.95cm] at (3.5-|1) {$\Lambda_Q\sizeddelimiter{5}{\{}$};
				\node[xshift=-1cm] at (8.5-|1) {$\Lambda_M\sizeddelimiter{5}{\{}$};
				\node[xshift=-1.92cm] at (12-|1) {$\Lambda\setminus\parens{\Lambda_Q\cup\Lambda_M}\sizeddelimiter{2}{\{}$};
				\begin{scope}[dash pattern=on 0pt off 2.3pt,line cap=round, thick]
					\draw (2.5-|1) -- ($(2.5-|2)+(1.8mm,0mm)$);
					\draw ($(2.5-|3)+(-1.8mm,0mm)$) -- ($(2.5-|7)+(0mm,0mm)$);
					\draw ($(4.5-|1)+(0mm,0mm)$) -- ($(4.5-|4)+(0mm,0mm)$);
					\draw ($(4.5-|5)+(0mm,0mm)$) -- ($(4.5-|9)+(0mm,0mm)$);
					\draw ($(7.5-|1)+(0mm,0mm)$) -- ($(7.5-|2)+(0mm,0mm)$);
					\draw ($(9.5-|1)+(0mm,0mm)$) -- ($(9.5-|4)+(0mm,0mm)$);
					\draw ($(2.5|-1)+(0mm,0mm)$) -- ($(2.5|-2)+(0mm,0mm)$);
					\draw ($(2.5|-3)+(0mm,0mm)$) -- ($(2.5|-7)+(0mm,0mm)$);
					\draw ($(4.5|-1)+(0mm,0mm)$) -- ($(4.5|-4)+(0mm,0mm)$);
					\draw ($(4.5|-5)+(0mm,-1.5mm)$) -- ($(4.5|-9)+(0mm,0mm)$);
					\draw ($(7.5|-1)+(0mm,0mm)$) -- ($(7.5|-2)+(0mm,0mm)$);
					\draw ($(9.5|-1)+(0mm,0mm)$) -- ($(9.5|-4)+(0mm,0mm)$);
				\end{scope}
				
				\begin{scope}[decoration = {snake, amplitude = 0.7mm}, ->, RoyalBlue]
					\draw[decorate] ($(9.5|-4)+(-2mm,-0.5mm)$) -- ($(7.5|-3)+(2mm,0mm)$);
					\draw[decorate] ($(4.5|-9)+(-2mm,-0.5mm)$) -- ($(2.5|-8)+(2mm,0mm)$);
					\draw[decorate, dash pattern=on 1pt off 1.8pt] ($(4.5|-4)+(-2mm,-0.5mm)$) -- ($(2.5|-3)+(1mm,1mm)$);
				\end{scope}
			\end{tikzpicture}
			\OverBrace[yshift=15pt]{1-1}{1-5}{I_Q}
			\OverBrace[yshift=15pt]{1-6}{1-10}{I_M}
			\OverBrace[yshift=15pt]{1-11}{1-17}{I\setminus\parens{I_Q\cup I_M}}
		\end{bNiceArray}
	\end{displaymath}
	
	Since $\mathcal{D}_{QM}$ is a connected component of $\simplifiedgraph{I}{\Lambda}{P}$ and $\parens{j',\mu},\parens{i',\lambda},\parens{j,\mu'},\parens{i,\lambda'}\in \parens{I_Q\times\Lambda_M}\cup\parens{I_M\times\Lambda_Q}$, then there exists a path from $\parens{j',\mu}$ to $\parens{i',\lambda}$ in $\simplifiedgraph{I}{\Lambda}{P}$ and a path from $\parens{j,\mu'}$ to $\parens{i,\lambda'}$ in $\simplifiedgraph{I}{\Lambda}{P}$. Let
	\begin{gather*}
		\parens{j',\mu}=\parens{i_1,\lambda_1}-\parens{i_2,\lambda_2}-\cdots-\parens{i_k,\lambda_k}=\parens{i',\lambda}\\
		\shortintertext{and}
		\parens{j,\mu'}=\parens{j_1,\mu_1}-\parens{j_2,\mu_2}-\cdots-\parens{j_m,\mu_m}=\parens{i,\lambda'}
	\end{gather*}
	be paths of minimum length in $\simplifiedgraph{I}{\Lambda}{P}$ from $\parens{j',\mu}$ to $\parens{i',\lambda}$ and from $\parens{j,\mu'}$ to $\parens{i,\lambda'}$, respectively, so that $\dist{\simplifiedgraph{I}{\Lambda}{P}}{\parens{j',\mu}}{\parens{i',\lambda}}=k-1$ and $\dist{\simplifiedgraph{I}{\Lambda}{P}}{\parens{j,\mu'}}{\parens{i,\lambda'}}=m-1$. It follows from Lemma~\ref{0Rees: lemma connected component determined Q M} that the vertices of $\mathcal{D}_{QM}$ that belong to $I_Q\times\Lambda_M$ (respectively, $I_M\times\Lambda_Q$) are only adjacent to vertices that belong to $I_M\times\Lambda_Q$ (respectively, $I_Q\times\Lambda_M$). This implies that any path from $\parens{j',\mu}$ to $\parens{i',\lambda}$ has even length and any path from $\parens{j,\mu'}$ to $\parens{i,\lambda'}$ has even length (because $\parens{j',\mu},\parens{i',\lambda}\in I_M\times\Lambda_Q$ and $\parens{j,\mu'},\parens{i,\lambda'}\in I_Q\times\Lambda_M$). Thus $\dist{\simplifiedgraph{I}{\Lambda}{P}}{\parens{j',\mu}}{\parens{i',\lambda}}$ and $\dist{\simplifiedgraph{I}{\Lambda}{P}}{\parens{j,\mu'}}{\parens{i,\lambda'}}$ are even and, consequently, $k$ and $m$ are odd. Assume, without loss of generality, that $k\leqslant m$. Let $n\in\mathbb{Z}$ be such that $m=k+2n$. It is clear that $n\geqslant 0$.
	
	Due to the fact that $\parens{i_l,\lambda_l}$ is adjacent to $\parens{i_{l+1},\lambda_{l+1}}$ for all $l\in\X{k-1}$, we have $p_{\lambda_l i_{l+1}}=p_{\lambda_{l+1} i_l}=0$ for all $l\in\X{k-1}$ and, due to the fact that $\parens{j_l,\mu_l}$ is adjacent to $\parens{j_{l+1},\mu_{l+1}}$ for all $l\in\X{m-1}$, we have $p_{\mu_l j_{l+1}}=p_{\mu_{l+1} j_l}=0$ for all $l\in\X{m-1}$. Therefore, for each $l\in\X{k-1}$ we have that $\parens{j_l,\lambda_l}\sim\parens{i_{l+1},\mu_{l+1}}$ and $\parens{i_l,\mu_l}\sim\parens{j_{l+1},\lambda_{l+1}}$ (because for each $l\in\X{k-1}$ we have $p_{\lambda_l i_{l+1}}=p_{\mu_{l+1} j_l}=0$ and $p_{\mu_l j_{l+1}}=p_{\lambda_{l+1} i_l}=0$, respectively). Thus, since $k$ is odd, we have
	\begin{displaymath}
		\parens{j,\mu}=\parens{j_1,\lambda_1}\sim\parens{i_2,\mu_2}\sim\parens{j_3,\lambda_3}\sim\cdots\sim\parens{i_{k-1},\mu_{k-1}}\sim\parens{j_k,\lambda_k}=\parens{j_k,\lambda},
	\end{displaymath}
	which implies that there is a path from $\parens{j,\mu}$ to $\parens{j_k,\lambda}$ whose length is at most $k-1$. Additionally, for each $l\in\set{k,\ldots,m-1}$ we have that $\parens{j_l,\lambda}\sim\parens{i,\mu_{l+1}}$ and $\parens{i,\mu_l}\sim\parens{j_{l+1},\lambda}$ (because for each $l\in\set{k,\ldots,m-1}$ we have $p_{\lambda i}=p_{\mu_{l+1} j_l}=0$ and $p_{\mu_l j_{l+1}}=p_{\lambda i}=0$, respectively). Hence
	\begin{displaymath}
		\parens{j_k,\lambda}\sim\parens{i,\mu_{k+1}}\sim\parens{j_{k+2},\lambda}\sim\cdots\sim\parens{i,\mu_{k+2n-1}}\sim\parens{j_{k+2n},\lambda}=\parens{j_m,\lambda}=\parens{i,\lambda},
	\end{displaymath}
	which implies that there is a path from $\parens{j_k,\lambda}$ to $\parens{i,\lambda}$ whose length is at most $2n=m-k$. Thus there is a path from $\parens{j,\mu}$ to $\parens{i,\lambda}$ whose length is at most $m-1$. Therefore
	\begin{align*}
		\smash{\underbrace{\max\set[\big]{\diam{\mathcal{D}_Q},\diam{\mathcal{D}_M}}}_{\text{\small odd}}}&=\diam{\mathcal{D}_Q}\\
		&=\zeroindex{i}{\lambda}\\
		&=\dist{\simplifiedgraph{I}{\Lambda}{P}}{\parens{i,\lambda}}{\parens{j,\mu}}\\
		&=\dist{\simplifiedgraph{I}{\Lambda}{P}}{\parens{j,\mu}}{\parens{i,\lambda}}\\
		&\leqslant m-1\\
		&=\underbrace{\dist{\simplifiedgraph{I}{\Lambda}{P}}{\parens{j,\mu'}}{\parens{i,\lambda'}}}_{\text{\small even}},
	\end{align*}
	and, consequently,
	\begin{displaymath}
		1+\max\set[\big]{\diam{\mathcal{D}_Q},\diam{\mathcal{D}_M}}\leqslant \dist{\simplifiedgraph{I}{\Lambda}{P}}{\parens{j,\mu'}}{\parens{i,\lambda'}}\leqslant \diam{\mathcal{D}_{QM}}.
	\end{displaymath}
	
	\smallskip
	
	\textit{Case 2:} Suppose that $\max\set[\big]{\diam{\mathcal{D}_Q},\diam{\mathcal{D}_M}}=\diam{\mathcal{D}_Q}$ is even. We can prove this case in a similar way to the previous one. Instead of using a shortest path from $\parens{j',\mu}$ to $\parens{i',\lambda}$ and a shortest path from $\parens{j,\mu'}$ to $\parens{i,\lambda'}$ (which have an even length), we select a shortest path from $\parens{j,\mu'}$ to $\parens{i',\lambda}$ and a shortest path from $\parens{j',\mu}$ to $\parens{i,\lambda'}$. Since $\parens{j,\mu'},\parens{i,\lambda'}\in I_Q\times\Lambda_M$ and $\parens{i',\lambda},\parens{j',\mu}\in I_M\times\Lambda_Q$ and since the vertices of $\mathcal{D}_{QM}$ that belong to $I_Q\times\Lambda_M$ (respectively, $I_M\times\Lambda_Q$) are only adjacent to vertices that belong to $I_M\times\Lambda_Q$ (respectively, $I_Q\times\Lambda_M$), then these paths will have an odd length. We then use the paths to see that
	\begin{displaymath}
		\dist{\simplifiedgraph{I}{\Lambda}{P}}{\parens{j,\mu}}{\parens{i,\lambda}}\leqslant \max\set[\big]{\dist{\simplifiedgraph{I}{\Lambda}{P}}{\parens{j,\mu'}}{\parens{i',\lambda}}, \dist{\simplifiedgraph{I}{\Lambda}{P}}{\parens{j',\mu}}{\parens{i,\lambda'}}}
	\end{displaymath}
	and, consequently, that
	\begin{displaymath}
		\smash{\underbrace{\max\set[\big]{\diam{\mathcal{D}_Q},\diam{\mathcal{D}_M}}}_{\text{\small even}}}\leqslant \underbrace{\max\set[\big]{\dist{\simplifiedgraph{I}{\Lambda}{P}}{\parens{j,\mu'}}{\parens{i',\lambda}}, \dist{\simplifiedgraph{I}{\Lambda}{P}}{\parens{j',\mu}}{\parens{i,\lambda'}}}}_{\text{\small odd}}.
	\end{displaymath}
	This allow us to conclude that 
	\begin{displaymath}
		1+\max\set[\big]{\diam{\mathcal{D}_Q},\diam{\mathcal{D}_M}}\leqslant \diam{\mathcal{D}_{QM}}.
	\end{displaymath}
	
	\medskip
	
	\textbf{Part 2.} This part of the proof is dedicated to demonstrate the statement of Proposition~\ref{0Rees: diameter connected component Q M}.
	
	\smallskip
	
	\textit{Case 1:} Suppose that $\abs{I_Q}=\abs{\Lambda_Q}=\abs{I_M}=\abs{\Lambda_M}=1$. Assume that $I_Q=\set{i}$, $\Lambda_Q=\set{\lambda}$, $I_M=\set{j}$ and $\Lambda_M=\set{\mu}$. Then the set of vertices of $\mathcal{C}_{QM}$ is $\parens{\set{i}\times G\times \set{\mu}}\cup\parens{\set{j}\times G\times\set{\lambda}}$ and we have $p_{\lambda j},p_{\mu i}\in G$ and $p_{\lambda i}=p_{\mu j}=0$.
	
	\smallskip
	
	\textsc{Sub-case 1:} Suppose that $G$ is a non-trivial abelian group. This implies that $xp_{\mu i}y=yp_{\mu i}x$ and $xp_{\lambda j}y=yp_{\lambda j}x$ for all $x,y \in G$. Hence, by Lemma~\ref{0Rees: commutativity}, $\parens{i,x,\mu}\parens{i,y,\mu}=\parens{i,y,\mu}\parens{i,x,\mu}$ and $\parens{j,x,\lambda}\parens{j,y,\lambda}=\parens{j,y,\lambda}\parens{j,x,\lambda}$ for all $x,y \in G$. Moreover, $p_{\lambda i}=p_{\mu j}=0$ and, as a consequence of Lemma~\ref{0Rees: commutativity}, we also have $\parens{i,x,\mu}\parens{j,y,\lambda}=\parens{j,y,\lambda}\parens{i,x,\mu}$ for all $x,y \in G$. Thus $\diam{\mathcal{C}_{QM}}\leqslant 1$. Furthermore, since $\abs{G}\geqslant 2$, then $\mathcal{C}_{QM}$ contains at least two vertices, which implies that $\diam{\mathcal{C}_{QM}}\geqslant 1$. Therefore $\diam{\mathcal{C}_{QM}}=1$.
	
	\smallskip
	
	\textsc{Sub-case 2:} Suppose that $G$ is trivial. Then $\mathcal{C}_{QM}$ is a connected component with exactly two vertices, which implies that $\diam{\mathcal{C}_{QM}}=1$. Furthermore, we have $\abs{I_Q\times G\times\Lambda_Q}=\abs{I_M\times G\times\Lambda_M}=1$, that is, both $\mathcal{C}_Q$ and $\mathcal{C}_M$ contain exactly one vertex, which implies that $\diam{\mathcal{C}_Q}=\diam{\mathcal{C}_M}=0$. Thus
	\begin{displaymath}
		\diam{\mathcal{C}_{QM}}=1=1+0=1+\max\set{\diam{\mathcal{C}_Q},\diam{\mathcal{C}_M}}.
	\end{displaymath}
	
	\smallskip
	
	\textsc{Sub-case 3:} Suppose that $G$ is non-abelian. Then there exist $x,y\in G$ such that $xy\neq yx$. Due to Lemma~\ref{0Rees: commutativity}, and the fact $p_{\lambda j}\in G$ and $\parens{p_{\lambda j}^{-1}x}p_{\lambda j}\parens{p_{\lambda j}^{-1}y}=p_{\lambda j}^{-1}xy\neq p_{\lambda j}^{-1}yx=\parens{p_{\lambda j}^{-1}y}p_{\lambda j}\parens{p_{\lambda j}^{-1}x}$, we have $\parens{j,p_{\lambda j}^{-1}x,\lambda}\parens{j,p_{\lambda j}^{-1}y,\lambda}\neq\parens{j,p_{\lambda j}^{-1}y,\lambda}\parens{j,p_{\lambda j}^{-1}x,\lambda}$. Hence the vertices $\parens{j,p_{\lambda j}^{-1}x,\lambda}$ and $\parens{j,p_{\lambda j}^{-1}y,\lambda}$ of $\mathcal{C}_{QM}$ are not adjacent and, as a consequence, $\diam{\mathcal{C}_{QM}}\geqslant 2$. Furthermore, for all $x,y\in G$ we have
	\begin{gather*}
		\parens{i,x,\mu}\sim\parens{j,x,\lambda}\sim\parens{i,y,\mu}\\
		\shortintertext{and}
		\parens{j,x,\lambda}\sim\parens{i,x,\mu}\sim\parens{j,y,\lambda}\\
		\shortintertext{and}
		\parens{i,x,\mu}\sim\parens{j,y,\lambda}
	\end{gather*}
	because $p_{\lambda i}=p_{\mu j}=0$ and by Lemma~\ref{0Rees: commutativity}. Thus $\diam{\mathcal{C}_{QM}}\leqslant 2$ and, consequently, $\diam{\mathcal{C}_{QM}}=2$. In addition, since $p_{\lambda i}=p_{\mu j}=0$, then Lemma~\ref{0Rees: subgraph induced by ixGxlambda} implies that $\mathcal{C}_Q$ and $\mathcal{C}_M$ are both isomorphic to $K_{\abs{G}}$ and, since $G$ is not abelian, then $\abs{G}\geqslant 2$, which implies that $\diam{\mathcal{C}_Q}=\diam{\mathcal{C}_M}=\diam{K_{\abs{G}}}=1$. Therefore
	\begin{displaymath}
		\diam{\mathcal{C}_{QM}}=2=1+1=1+\max\set{\diam{\mathcal{C}_Q},\diam{\mathcal{C}_M}}.
	\end{displaymath}
	
	\smallskip
	
	\textit{Case 2:} Suppose that $\abs{I_Q}>1$ or $\abs{\Lambda_Q}>1$ or $\abs{I_M}>1$ or $\abs{\Lambda_M}>1$. Then $\abs{I_Q\times\Lambda_Q}>1$ or $\abs{I_M\times\Lambda_M}>1$, that is, $\mathcal{D}_Q$ contains more than one vertex or $\mathcal{D}_M$ contains more than one vertex. Hence $\diam{\mathcal{D}_Q}\geqslant 1$ or $\diam{\mathcal{D}_M}\geqslant 1$. Assume, without loss of generality, that $\diam{\mathcal{D}_Q}\geqslant\diam{\mathcal{D}_M}$. Then we must have $\diam{\mathcal{D}_Q}\geqslant 1$.
	
	Before determining $\diam{\mathcal{C}_{QM}}$ we need to check that $\max\set{\diam{\mathcal{D}_Q},\diam{\mathcal{D}_M}}=\max\set{\diam{\mathcal{C}_Q},\diam{\mathcal{C}_M}}$. Since we have $\diam{\mathcal{D}_Q}=\diam{\mathcal{C}_Q}$ (by Lemma~\ref{0Rees: connected component G(I,^) => connected component G(M0[...])}) and $\diam{\mathcal{D}_Q}\geqslant\diam{\mathcal{D}_M}$, it is enough to verify that $\diam{\mathcal{C}_Q}\geqslant\diam{\mathcal{C}_M}$. If $\diam{\mathcal{D}_M}\geqslant 1$, then Lemma~\ref{0Rees: connected component G(I,^) => connected component G(M0[...])} ensures that $\diam{\mathcal{C}_M}=\diam{\mathcal{D}_M}\leqslant \diam{\mathcal{D}_Q}=\diam{\mathcal{C}_Q}$. Now assume that $\diam{\mathcal{D}_M}=0$. Then $\abs{I_M\times\Lambda_M}=1$, that is, $\abs{I_M}=\abs{\Lambda_M}=1$. Assume that $I_M=\set{i}$ and $\Lambda_M=\set{\lambda}$. Then the vertex set of $\mathcal{C}_M$ is $\set{i}\times G\times\set{\lambda}$ and we have $p_{\lambda i}=0$. By Lemma~\ref{0Rees: commutativity} we have that $\parens{i,x,\lambda}\parens{i,y,\lambda}= \parens{i,y,\lambda}\parens{i,x,\lambda}$ for all $x,y\in G$. Hence $\parens{i,x,\lambda}\sim\parens{i,y,\lambda}$ for all $x,y\in G$, which implies that $\dist{\commgraph{\Reeszero{G}{I}{\Lambda}{P}}}{\parens{i,x,\lambda}}{\parens{i,y,\lambda}}\leqslant 1$ for all $x,y\in G$ and, consequently, $\diam{\mathcal{C}_M}\leqslant 1\leqslant \diam{\mathcal{D}_Q}=\diam{\mathcal{C}_Q}$. Hence $\max\set{\diam{\mathcal{D}_Q},\diam{\mathcal{D}_M}}=\max\set{\diam{\mathcal{C}_Q},\diam{\mathcal{C}_M}}$.
	
	Since $\diam{\mathcal{D}_{QM}}=1+\max\set[\big]{\diam{\mathcal{D}_Q},\diam{\mathcal{D}_M}}\geqslant 1$, then Lemma~\ref{0Rees: connected component G(I,^) => connected component G(M0[...])} guarantees that $\diam{\mathcal{D}_{QM}}=\diam{\mathcal{C}_{QM}}$. Therefore we have
	\begin{align*}
		\diam{\mathcal{C}_{QM}}&=\diam{\mathcal{D}_{QM}}\\
		&=1+\max\set{\diam{\mathcal{D}_Q},\diam{\mathcal{D}_M}}\\
		&=1+\max\set{\diam{\mathcal{C}_Q},\diam{\mathcal{C}_M}}.\qedhere
	\end{align*}
\end{proof}

\begin{proposition}\label{0Rees: diameter connected component (i,lambda)}
	Suppose that $\commgraph{\Reeszero{G}{I}{\Lambda}{P}}$ is not connected. Let $i \in I$ and $\lambda\in\Lambda$ be such that row $\lambda$ has no zero entries or column $i$ has no zero entries. Let $\mathcal{C}$ be the connected component of $\commgraph{\Reeszero{G}{I}{\Lambda}{P}}$ determined by $\parens{\lambda,i}$. Then
	\begin{displaymath}
		\diam{\mathcal{C}}=\begin{cases}
			0 & \text{if } G \text{ is trivial},\\
			1 & \text{if } G \text{ is abelian and not trivial},\\
			2 & \text{if } G \text{ is not abelian}.
		\end{cases}
	\end{displaymath}
\end{proposition}

\begin{proof}
	We consider three cases:
	
	\smallskip
	
	\textit{Case 1:} Suppose that $G$ is trivial, then $\abs{\set{i}\times G\times\set{\lambda}}=1$, which implies that $\mathcal{C}$ has only one vertex. Thus $\diam{\mathcal{C}}=0$.
	
	\smallskip
	
	\textit{Case 2:} Suppose that $G$ is abelian and not trivial. It follows from Lemma~\ref{0Rees: subgraph induced by ixGxlambda} that $\mathcal{C}$ is isomorphic to $K_{\abs{G}}$. Since $G$ is not trivial, then $\abs{G}\geqslant 2$ and, consequently, $K_{\abs{G}}$ contains at least two vertices, which implies that $\diam{\mathcal{C}}=\diam{K_{\abs{G}}}=1$.
	
	\smallskip
	
	\textit{Case 3:} Suppose that $G$ is not abelian. According to Lemma~\ref{0Rees: subgraph induced by ixGxlambda} we have that $\mathcal{C}$ is isomorphic to $K_{\abs{\centre{G}}}\graphjointwo\commgraph{G}$. Let $x$ be a vertex of $K_{\abs{\centre{G}}}$. We have that $x$ is adjacent to all the other vertices of $K_{\abs{\centre{G}}}\graphjointwo\commgraph{G}$. Hence the distance between any two vertices of $K_{\abs{\centre{G}}}\graphjointwo\commgraph{G}$ is at most $2$. Thus $\diam{K_{\abs{\centre{G}}}\graphjointwo\commgraph{G}} \leqslant 2$. Furthermore, $G$ is not abelian, which guarantees the existence of $x,y\in G$ such that $xy\neq yx$. Hence $x$ and $y$ are not adjacent in $\commgraph{G}$ (and, consequently, $K_{\abs{\centre{G}}}\graphjointwo\commgraph{G}$), which implies that $\diam{K_{\abs{\centre{G}}}\graphjointwo\commgraph{G}} \geqslant \dist{K_{\abs{\centre{G}}}\graphjointwo\commgraph{G}}{x}{y}\geqslant 2$. Therefore $\diam{\mathcal{C}}=\diam{K_{\abs{\centre{G}}}\graphjointwo\commgraph{G}}=2$.
\end{proof}

\begin{example}\label{0Rees: example diameter}
	We are going to determine the diameter of the several connected components of $\commgraph{\Reeszero{G}{I}{\Lambda}{P}}$ determined in Example~\ref{0Rees: example finding connected components}. The $0$-closure submatrices of $P$ are $Q=\Psubmatrix{P}{\set{1,3,4,6,7}}{\set{1,3,4}}$ and $M=\Psubmatrix{P}{\set{2,5,8}}{\set{2,5}}$. Let $\mathcal{C}_Q$ and $\mathcal{C}_M$ be the connected components of $\commgraph{\Reeszero{G}{I}{\Lambda}{P}}$ determined by $Q$ and $M$, respectively; let $\mathcal{C}_{QM}$ be the connected component of $\commgraph{\Reeszero{G}{I}{\Lambda}{P}}$ determined by $Q$ and $M$; and for each $i\in\set{1,2,3,4,5,6,7,8}$ let $\mathcal{C}_{\parens{6,i}}$ be the connected component of $\commgraph{\Reeszero{G}{I}{\Lambda}{P}}$ determined by $\parens{6,i}$.
	
	We start by determining the diameter of $\mathcal{C}_Q$. In Example~\ref{0Rees: example 0-closure method} we verified that $Q$ is the \zeroclosuresubmatrix{6}{4} and $\zeroindex{6}{4}=3$. Thus, by Theorem~\ref{0Rees: diameter}, we have that $\diam{\mathcal{C}_Q}\geqslant\zeroindex{6}{4}=3$. Then we can use Proposition~\ref{0Rees: connected components determined by Q, simplification >=3}. Let $Q'=\Psubmatrix{Q}{\set{1,3,4}}{\set{1,3,4}}=\Psubmatrix{P}{\set{1,3,4}}{\set{1,3,4}}$ (which is represented in red below).
	\begin{displaymath}
		\timeszero{Q}=\begin{bNiceMatrix}[first-row, first-col]
			& 1 & 3 & 4 & 6 & 7\\
			1 & \cellcolor{Red} 0 & \cellcolor{Red} 0 & \cellcolor{Red} 0 & \times & \times\\
			3 & \cellcolor{Red} 0 & \cellcolor{Red} 0 & \cellcolor{Red} \times & 0 & 0\\
			4 & \cellcolor{Red} 0 & \cellcolor{Red} 0 & \cellcolor{Red} \times & 0 & 0
		\end{bNiceMatrix}.
	\end{displaymath}
	It is straightforward to see that $Q'$ contains a row of zeros. Consequently, Proposition~\ref{0Rees: connected components determined by Q, simplification >=3} guarantees that
	\begin{displaymath}
		\diam{\mathcal{C}_Q}=\max\set{3,\zeroindex{6}{3},\zeroindex{6}{4},\zeroindex{7}{3},\zeroindex{7}{4}}.
	\end{displaymath}
	Furthermore, we have that rows $3$ and $4$ of $\timeszero{Q}$ are equal, which implies, by Proposition~\ref{0Rees: connected component determined by Q, simplification repeated rows/columns}, that $\zeroindex{6}{3}=\zeroindex{6}{4}$ and $\zeroindex{7}{3}=\zeroindex{7}{4}$. Thus
	\begin{displaymath}
		\diam{\mathcal{C}_Q}=\max\set{3,\zeroindex{6}{4},\zeroindex{7}{4}}.
	\end{displaymath}
	Columns $6$ and $7$ of $\timeszero{Q}$ are also equal. Consequently, Proposition~\ref{0Rees: connected component determined by Q, simplification repeated rows/columns} implies that $\zeroindex{6}{4}=\zeroindex{7}{4}$. Therefore
	\begin{displaymath}
		\diam{\mathcal{C}_Q}=\max\set{3,\zeroindex{6}{4}}=3.
	\end{displaymath}
	
	Now we determine the diameter of $\mathcal{C}_M$. We have that
	\begin{displaymath}
		\timeszero{M}=\begin{bNiceMatrix}[first-row, first-col]
			& 2 & 5 & 8\\
			2 & 0 & 0 & 0\\
			5 & 0 & \times & \times
		\end{bNiceMatrix}.
	\end{displaymath}
	Then it is easy to see that $M$ contains a non-zero entry. Additionally, the intersection of column $2$ with rows $2$ and $5$ of $M$ is the matrix $O_{2\times 1}$; and for every two columns of $M$ we have that the intersection of row $2$ of $M$ with those columns is $O_{1\times 2}$. That is, we have $p_{2 2}=p_{5 2}=0$ and $p_{2 2}=p_{2 5}=p_{2 8}=0$. Hence Proposition~\ref{0Rees: diameter 0 or 1 or 2} implies that $\diam{\mathcal{C}_M}=2$.
	
	It follows from the fact that $Q$ contains more than one row, and Proposition~\ref{0Rees: diameter connected component Q M}, that
	\begin{displaymath}
		\diam{\mathcal{C}_{QM}}=1+\max\set[\big]{\diam{\mathcal{C}_Q},\diam{\mathcal{C}_M}}=1+\max\set{3,2}=4.
	\end{displaymath}
	
	Finally, we observe that Proposition~\ref{0Rees: diameter connected component (i,lambda)} ensures that $\diam{\mathcal{C}_{\parens{6,1}}}=\diam{\mathcal{C}_{\parens{6,i}}}$ for all $i\in\set{1,2,3,4,5,6,7,8}$, which is equal to $0$, $1$ or $2$ depending if $G$ is trivial, non-trivial and abelian, or non-abelian, respectively.
\end{example}




\section{Clique number of the commuting graph of a 0-Rees matrix semigroup over a group}\label{sec: clique number 0Rees}

Let $G$ be a finite group, let $I$ and $\Lambda$ be finite index sets and let $P$ be a regular $\Lambda \times I$ matrix whose entries are elements of $G^0$. Throughout this section we are going to assume that $P$ contains at least one zero entry.

The purpose of this section is to determine the clique number of $\commgraph{\Reeszero{G}{I}{\Lambda}{P}}$ --- Theorem~\ref{0Rees: clique number}. In order to achieve this, first we need to prove the following two lemmata.

\begin{lemma}\label{0Rees: clique number, characterization of R}
	Let $C$ be a clique in $\commgraph{\Reeszero{G}{I}{\Lambda}{P}}$ and let $R=\gset{\parens{i,\lambda}\in I\times\Lambda}{\parens{i,x,\lambda}\in C \text{ for some } x\in G}$. Then
	\begin{enumerate}
		\item For all $i,j\in I$ and $\lambda,\mu\in\Lambda$ such that $\parens{i,\lambda}\neq\parens{j,\mu}$ and $\parens{i,\lambda},\allowbreak\parens{j,\mu}\in R$ we have $p_{\lambda j}=p_{\mu i}=0$.
		
		\item If $O_{1\times 3}$, $O_{3\times 1}$ and $O_{2 \times 2}$ are not \submatrices\ of $\timeszero{P}$, then $\abs{R}\leqslant 3$.
			
			
		
		\item If $O_{1\times 2}$ and $O_{2 \times 1}$ are not \submatrices\ of $\timeszero{P}$, then $\abs{R}\leqslant 2$.
	\end{enumerate}
\end{lemma}

\begin{proof}
	\textbf{Part 1.} Let $i,j\in I$ and $\lambda,\mu\in\Lambda$ be such that $\parens{i,\lambda}\neq\parens{j,\mu}$ and $\parens{i,\lambda},\allowbreak\parens{j,\mu}\in R$. There exist $x,y\in G$ such that $\parens{i,x,\lambda},\allowbreak\parens{j,y,\mu}\in C$. Hence $\parens{i,x,\lambda}\parens{j,y,\mu}=\parens{j,y,\mu}\parens{j,x,\mu}$ and, as a consequence of Lemma~\ref{0Rees: commutativity} and the fact that $\parens{i,\lambda}\neq\parens{j,\mu}$, we have $p_{\lambda j}=p_{\mu i}=0$.
	
	\medskip
	
	\textbf{Part 2.} Suppose that $O_{1\times 3}$, $O_{3\times 1}$ and $O_{2 \times 2}$ are not \submatrices\ of $\timeszero{P}$. We divide the proof into the following three cases.
	
	\smallskip
	
	\textit{Case 1:} Suppose that there exist $i\in I$ and distinct $\lambda_1,\lambda_2\in \Lambda$ such that $\parens{i,\lambda_1},\parens{i,\lambda_2}\in R$. Our aim is to prove that $R=\set{\parens{i,\lambda_1},\parens{i,\lambda_2}}$. Let $\parens{j,\mu}\in R$. We have $\parens{j,\mu}\neq\parens{i,\lambda_1}$ or $\parens{j,\mu}\neq\parens{i,\lambda_2}$. Interchanging $\lambda_1$ and $\lambda_2$ if necessary, assume that $\parens{j,\mu}\neq\parens{i,\lambda_1}$.
	
	\smallskip
	
	\textsc{Sub-case 1:} Assume that $p_{\lambda_2 j}\neq 0$. Then, due to the fact that $\parens{j,\mu},\parens{i,\lambda_2}\in R$, and by part 1, we have $\parens{j,\mu}=\parens{i,\lambda_2}$.
	
	\smallskip
	
	\textsc{Sub-case 2:} Assume that $p_{\lambda_2 j}=0$. We have $\parens{j,\mu},\parens{i,\lambda_1},\parens{i,\lambda_2}\in R$, $\parens{j,\mu}\neq\parens{i,\lambda_1}$ and $\parens{i,\lambda_1}\neq\parens{i,\lambda_2}$. Hence, by part 1, we have $p_{\mu i}=p_{\lambda_1 j}=p_{\lambda_1 i}=p_{\lambda_2 i}=0$. Since $O_{2\times 2}$ is not a \submatrix\ of $\timeszero{P}$ and $p_{\lambda_1 i}=p_{\lambda_2 i}=p_{\lambda_1 j}=p_{\lambda_2 j}=0$, then $i=j$ (because $\lambda_1\neq\lambda_2$). Additionally, since $O_{3\times 1}$ is not a \submatrix\ of $\timeszero{P}$ and $p_{\lambda_1 i}=p_{\lambda_2 i}=p_{\mu i}=0$, then we must have $\mu\in\set{\lambda_1,\lambda_2}$ (because $\lambda_1\neq\lambda_2$). As a consequence of the fact that $\parens{i,\mu}=\parens{j,\mu}\neq\parens{i,\lambda_1}$, we must have $\mu=\lambda_2$. Thus $\parens{j,\mu}=\parens{i,\lambda_2}$.
	
	\smallskip
	
	It follows from sub-cases 1 and 2 that $R=\set{\parens{i,\lambda_1},\parens{i,\lambda_2}}$. Thus $\abs{R}\leqslant 3$.
	
	\smallskip
	
	\textit{Case 2:} Suppose that there exist distinct $i_1,i_2\in I$ and $\lambda\in\Lambda$ such that $\parens{i_1,\lambda},\parens{i_2,\lambda}\in R$. If we use a reasoning similar to the one used in case 1, we can see that $R=\set{\parens{i_1,\lambda},\parens{i_2,\lambda}}$. Therefore $\abs{R}\leqslant 3$.
	
	\smallskip
	
	\textit{Case 3:} Suppose that there exist $n\in\mathbb{N}$, pairwise distinct $i_1,\ldots,i_n\in I$ and pairwise distinct $\lambda_1,\ldots,\lambda_n\in\Lambda$ such that $R=\gset{\parens{i_j,\lambda_j}}{j\in\Xn}$. Therefore, by part 1, $p_{\lambda_j i_k}=0$ for all distinct $j,k\in\Xn$, which implies that all the rows of $\timeszero{P}$ contain at least $n-1$ zero entries. Since $O_{1\times 3}$ is not a \submatrix\ of $\timeszero{P}$, then every row of $\timeszero{P}$ contains at most two zero entries. Thus $n-1\leqslant 2$ and, consequently, $\abs{R}=n\leqslant 3$.
	
	\medskip
	
	\textbf{Part 3.} Suppose that $O_{1\times 2}$ and $O_{2\times 1}$ are not \submatrices\ of $\timeszero{P}$. If $\abs{R}=1$, then the result follows. Now assume that $\abs{R}\geqslant 2$. Then there exist $i,j\in I$ and $\lambda_1,\lambda_2\in\Lambda$ such that $\parens{i,\lambda}\neq\parens{j,\mu}$ and $\parens{i,\lambda},\parens{j,\mu}\in R$. By part 1, $p_{\lambda j}=p_{\mu i}=0$. Additionally, $O_{1 \times 2}$ and $O_{2 \times 1}$ are not \submatrices\ of $\timeszero{P}$, which implies that $\lambda\neq\mu$ and $i\neq j$. Due to the fact that $\parens{i,\lambda}$ and $\parens{j,\mu}$ are arbitrary (distinct) elements of $R$, then this implies that there exist $n\in\mathbb{N}$, distinct $i_1,\ldots,i_n\in I$ and distinct $\lambda_1,\ldots,\lambda_n\in\Lambda$ such that $R=\gset{\parens{i_j,\lambda_j}}{j\in\Xn}$. Thus we have that $p_{\lambda_j i_k}=0$ for all distinct $j,k\in\Xn$. Consequently, each row of $\timeszero{P}$ has at least $n-1$ zero entries. Since $O_{1\times 2}$ is not a \submatrix\ of $\timeszero{P}$, then each row of $\timeszero{P}$ contains at most one zero entry. Thus $n-1\leqslant 1$ and, consequently, $\abs{R}=n\leqslant 2$.
\end{proof}

\begin{lemma}\label{0Rees: clique number, Lagrange}
	Suppose that $G$ is not abelian. Then
	$\abs{\centre{G}}+\cliquenumber{\commgraph{G}}\leqslant \abs{G}/2$.
\end{lemma}

\begin{proof}
	In order to prove this lemma, we require Lagrange's Theorem, which we recall:
	
	\begin{theorem}[Lagrange's Theorem]\label{0Rees: Lagrange}
		If $H$ is a subgroup of $G$, then $\abs{G}=\bracks{G:H}\cdot\abs{H}$.
	\end{theorem}
	
	Now we demonstrate Lemma~\ref{0Rees: clique number, Lagrange}. Let $C$ be a clique of $\commgraph{G}$ such that $\abs{C}=\cliquenumber{\commgraph{G}}$ and let $H=\centre{G}\cup C$. It follows from Lemma~\ref{preli: largest cliques, commutative subsemigroups} that $H$ is a maximum-order commutative subsemigroup of $G$. Consequently, we have $x^{-1}y=x^{-1}yxx^{-1}=x^{-1}xyx^{-1}=yx^{-1}$ for all $x,y\in H$, which implies that $x^{-1}\in H$ for all $x\in H$. In addition, $1_G\in\centre{G}\subseteq H$. Hence $H$ is an abelian subgroup of $G$ and, by Theorem~\ref{0Rees: Lagrange}, we have $\abs{G}=\bracks{G:H}\cdot\abs{H}$. In addition, since $H$ is abelian and $G$ is not abelian, we have $H\neq G$, which implies that $\abs{H}<\abs{G}$. Thus $\bracks{G:H}\geqslant 2$ and, consequently,
	\begin{displaymath}
		\abs{\centre{G}}+\cliquenumber{\commgraph{G}}=\abs{\centre{G}}+\abs{C}=\abs{H}=\abs{G}/\bracks{G:H}\leqslant\abs{G}/2.\qedhere
	\end{displaymath}
\end{proof}

\begin{theorem}\label{0Rees: clique number}
	\begin{enumerate}
		\item Suppose that $G$ is abelian. If $D_3$ is a \submatrix\ of $\timeszero{P}$ and $O_{1\times 3}$, $O_{3\times 1}$ and $O_{2 \times 2}$ are not \submatrices\ of $\timeszero{P}$, then
		\begin{displaymath}
			\cliquenumber{\commgraph{\Rees{G}{I}{\Lambda}{P}}}=3|G|.
		\end{displaymath}
		
		\item Suppose that $G$ is abelian. If $D_2$ is a \submatrix\ of $\timeszero{P}$ and $O_{1\times 2}$ and $O_{2 \times 1}$ are not \submatrices\ of $\timeszero{P}$, then
		\begin{displaymath}
			\cliquenumber{\commgraph{\Rees{G}{I}{\Lambda}{P}}}=2|G|.
		\end{displaymath}
		
		\item For the remaining cases, we have
		\begin{displaymath}
			\cliquenumber{\commgraph{\Rees{G}{I}{\Lambda}{P}}}=\abs{G}\cdot\max\gset{nm}{O_{n\times m} \text{ is a \submatrix\ of } \timeszero{P}}.
		\end{displaymath}
	\end{enumerate}
\end{theorem}

\begin{proof}
	\textbf{Part 1.} Suppose that $G$ is abelian, $D_3$ is a \submatrix\ of $\timeszero{P}$ and $O_{1\times 3}$, $O_{3\times 1}$ and $O_{2 \times 2}$ are not \submatrices\ of $\timeszero{P}$. Our purpose is to establish that $\cliquenumber{\commgraph{\Reeszero{G}{I}{\Lambda}{P}}}=3\abs{G}$.
	
	We begin by showing that $\cliquenumber{\commgraph{\Reeszero{G}{I}{\Lambda}{P}}}\leqslant 3\abs{G}$. Let $C$ be a clique in $\commgraph{\Reeszero{G}{I}{\Lambda}{P}}$ such that $\abs{C}=\cliquenumber{\commgraph{\Reeszero{G}{I}{\Lambda}{P}}}$. Let
	\begin{displaymath}
		R=\gset{\parens{i,\lambda}\in I\times\Lambda}{\parens{i,x,\lambda}\in C \text{ for some } x\in G}.
	\end{displaymath}
	We have $C\subseteq\bigcup_{\parens{i,\lambda}\in R}\set{i}\times G\times\set{\lambda}$. Furthermore, it follows from part 2 of Lemma~\ref{0Rees: clique number, characterization of R}, and the fact that $O_{1\times 3}$, $O_{3\times 1}$ and $O_{2 \times 2}$ are not \submatrices\ of $\timeszero{P}$, that $\abs{R}\leqslant 3$. Thus
	\begin{align*}
		\cliquenumber{\commgraph{\Reeszero{G}{I}{\Lambda}{P}}}&=\abs{C}\\
		&\leqslant\abs[\bigg]{\,\bigcup_{\parens{i,\lambda}\in R}\set{i}\times G\times\set{\lambda}}\\
		&=\sum_{\parens{i,\lambda}\in R}\abs{\set{i}\times G\times\set{\lambda}}\\
		&=\sum_{\parens{i,\lambda}\in R}\abs{G}\\
		&=\abs{R}\cdot\abs{G}\\
		&\leqslant 3\abs{G}.
	\end{align*}
	
	Now we are going to show that $\cliquenumber{\commgraph{\Reeszero{G}{I}{\Lambda}{P}}}\geqslant 3\abs{G}$. Since $D_3$ is a \submatrix\ of $\timeszero{P}$, then there exist pairwise distinct $i_1,i_2,i_3\in I$ and pairwise distinct $\lambda_1,\lambda_2,\lambda_3\in\Lambda$ such that $\Psubmatrix{\timeszero{P}}{i_1,i_2,i_3}{\lambda_1,\lambda_2,\lambda_3}=D_3$. Hence $p_{\lambda_j i_k}=0$ for all distinct $j,k\in\set{1,2,3}$. Let $C=\parens{\set{i_1}\times G \times \set{\lambda_1}}\cup\parens{\set{i_2}\times G \times \set{\lambda_2}}\cup\parens{\set{i_3}\times G \times \set{\lambda_3}}$. Since $G$ is abelian, then $xp_{\lambda_j i_j}y=yp_{\lambda_j i_j}x$ for all $x,y\in G$ and $j\in\set{1,2,3}$. This implies, together with Lemma~\ref{0Rees: commutativity}, that $\parens{i_j,x,\lambda_j}\parens{i_j,y,\lambda_j}=\parens{i_j,y,\lambda_j}\parens{i_j,x,\lambda_j}$ for all $x,y\in G$ and $j\in\set{1,2,3}$. Furthermore, since $p_{\lambda_j i_k}=0$ for all distinct $j,k\in\set{1,2,3}$, then Lemma~\ref{0Rees: commutativity} also guarantees that $\parens{i_j,x,\lambda_j}\parens{i_k,y,\lambda_k}=\allowbreak\parens{i_k,y,\lambda_k}\parens{i_j,x,\lambda_j}$ for all $x,y\in G$ and distinct $j,k\in\set{1,2,3}$. Thus $C$ is a clique in $\commgraph{\Reeszero{G}{I}{\Lambda}{P}}$, which implies that
	\begin{displaymath}
		\cliquenumber{\commgraph{\Reeszero{G}{I}{\Lambda}{P}}}\geqslant \abs{C}=3\abs{G}.
	\end{displaymath}
	
	\medskip
	
	\textbf{Part 2.} Suppose that $G$ is abelian, $D_2$ is a \submatrix\ of $\timeszero{P}$ and $O_{1\times 2}$ and $O_{2\times 1}$ are not \submatrices\ of $\timeszero{P}$. We can check, in a similar way to part 1 (and by using part 3 of Lemma~\ref{0Rees: clique number, characterization of R}), that $\cliquenumber{\commgraph{\Reeszero{G}{I}{\Lambda}{P}}}=2\abs{G}$.
	
	\medskip
	
	\textbf{Part 3.} Suppose that the following conditions hold:
	\begin{itemize}
		\item[a)] $G$ is not abelian or $D_3$ is not a \submatrix\ of $\timeszero{P}$ or at least one of the matrices $O_{1\times 3}$, $O_{3\times 1}$ and $O_{2\times 2}$ is a \submatrix\ of $\timeszero{P}$.
		
		\item[b)] $G$ is not abelian or $D_2$ is not a \submatrix\ of $\timeszero{P}$ or at least one of the matrices $O_{1\times 2}$ and $O_{2\times 1}$ is a \submatrix\ of $\timeszero{P}$.
	\end{itemize}
	
	Let $N=\gset{nm}{O_{n\times m} \text{ is a \submatrix\ of } \timeszero{P}}$. We notice that, since $P$ (and, consequently, $\timeszero{P}$) contains at least one zero entry, then $O_{1\times 1}$ is a \submatrix\ of $\timeszero{P}$ and, consequently, $N\neq\emptyset$ and $\max N\geqslant 1$. We want to demonstrate that $\cliquenumber{\commgraph{\Reeszero{G}{I}{\Lambda}{P}}}=\abs{G}\cdot\max N$.
	
	We begin by checking that $\cliquenumber{\commgraph{\Reeszero{G}{I}{\Lambda}{P}}}\geqslant \abs{G}\cdot\max N$. Let $n,m\in\mathbb{N}$ be such that $nm=\max N$ and $O_{n \times m}$ is a \submatrix\ of $\timeszero{P}$. There exist $I'\subseteq I$ and $\Lambda'\subseteq\Lambda$ such that $\abs{I'}=m$, $\abs{\Lambda'}=n$ and $\Psubmatrix{\timeszero{P}}{I'}{\Lambda'}=O_{n\times m}$. Our aim is to verify that $I'\times G\times\Lambda'$ is a clique in $\commgraph{\Reeszero{G}{I}{\Lambda}{P}}$. Let $i,j\in I'$, $\lambda,\mu\in\Lambda'$ and $x,y\in G$. We have $p_{\lambda j}=p_{\mu i}=0$ (because $\Psubmatrix{\timeszero{P}}{I'}{\Lambda'}=O_{n\times m}$), which implies, by Lemma~\ref{0Rees: commutativity}, that $\parens{i,x,\lambda}\parens{j,y,\mu}=\parens{j,y,\mu}\parens{i,x,\lambda}$. Thus $\parens{i,x,\lambda}\sim\parens{j,y,\mu}$ (in $\commgraph{\Reeszero{G}{I}{\Lambda}{P}}$). Therefore $I'\times G\times\Lambda'$ is a clique in $\commgraph{\Reeszero{G}{I}{\Lambda}{P}}$ and, consequently,
	\begin{displaymath}
		\cliquenumber{\commgraph{\Reeszero{G}{I}{\Lambda}{P}}}\geqslant\abs{I'\times G\times\Lambda'}=\abs{G}\cdot\abs{I'}\cdot\abs{\Lambda'}=\abs{G}\cdot nm=\abs{G}\cdot\max N.
	\end{displaymath}
	
	Now we establish that $\cliquenumber{\commgraph{\Reeszero{G}{I}{\Lambda}{P}}}\leqslant\abs{G}\cdot\max N$. Let $C$ be a clique in $\commgraph{\Reeszero{G}{I}{\Lambda}{P}}$. We intend to check that $\abs{C}\leqslant\abs{G}\cdot\max N$. Let
	\begin{gather*}
		R=\gset{\parens{i,\lambda}\in I\times\Lambda}{\parens{i,x,\lambda}\in C \text{ for some } x\in G}\\
		\shortintertext{and}
		I'=\gset{i\in I}{\parens{i,\lambda}\in R \text{ for some } \lambda\in\Lambda}\\
		\shortintertext{and}
		\Lambda'=\gset{\lambda\in\Lambda}{\parens{i,\lambda}\in R \text{ for some } i\in I}.
	\end{gather*}
	We have $C\subseteq \bigcup_{\parens{i,\lambda}\in R}\set{i}\times G\times\set{\lambda}$, which implies that
	\begin{displaymath}
		\abs{C}\leqslant\abs[\bigg]{\,\bigcup_{\parens{i,\lambda}\in R}\set{i}\times G\times\set{\lambda}}=\sum_{\parens{i,\lambda}\in R}\abs{\set{i}\times G\times\set{\lambda}}=\sum_{\parens{i,\lambda}\in R}\abs{G}=\abs{R}\cdot\abs{G}.
	\end{displaymath}
	Additionally, we have $\abs{I'}\leqslant \abs{R}$ and $\abs{\Lambda'}\leqslant\abs{R}$.
	
	Assume that $R=\set{\parens{i_1,\lambda_1},\parens{i_2,\lambda_2},\ldots,\parens{i_k,\lambda_k}}$. Then $I'=\set{i_1,\ldots,i_k}$ and $\Lambda'=\set{\lambda_1,\ldots,\lambda_k}$. We observe that $i_1,\ldots,i_k$ (respectively, $\lambda_1,\ldots,\lambda_k$) are not necessarily pairwise distinct; that is, there might exist distinct $l,l'\in\X{k}$ such that $i_l=i_{l'}$ (respectively, $\lambda_l=\lambda_{l'}$). Then we can have $\abs{I'}<k$ (respectively, $\abs{\Lambda'}<k$). We consider several cases:
	
	\smallskip
	
	\textit{Case 1:} Suppose that $\Psubmatrix{\timeszero{P}}{I'}{\Lambda'}$ is \equivalent\ to $D_3$ and that $\abs{I'}=\abs{\Lambda'}=k$. This implies that $\abs{R}=k=\abs{I'}=\abs{\Lambda'}=3$ and, consequently, $R=\set{\parens{i_1,\lambda_1},\parens{i_2,\lambda_2},\parens{i_3,\lambda_3}}$. Moreover, we have that $i_1,i_2,i_3$ are pairwise distinct and $\lambda_1,\lambda_2,\lambda_3$ are pairwise distinct. Since part 1 of Lemma~\ref{0Rees: clique number, characterization of R} guarantees that $p_{\lambda_l i_{l'}}=0$ for all distinct $l,l'\in\set{1,2,3}$, then we have that
	\begin{displaymath}
		\Psubmatrix{P}{i_1,i_2,i_3}{\lambda_1,\lambda_2,\lambda_3}=\begin{bmatrix}
			p_{\lambda_1 i_1}&0&0\\
			0&p_{\lambda_2 i_2}&0\\
			0&0&p_{\lambda_3 i_3}
		\end{bmatrix}.
	\end{displaymath}
	Due to that fact that $\Psubmatrix{\timeszero{P}}{I'}{\Lambda'}$ is \equivalent\ to $D_3$, then we must have $p_{\lambda_1 i_1},p_{\lambda_2 i_2},p_{\lambda_3 i_3}\in G$.
	
	\smallskip
	
	\textsc{Sub-case 1:} Suppose that $G$ is not abelian. It follows from the fact that $R=\set{\parens{i_1,\lambda_1},\parens{i_2,\lambda_2},\parens{i_3,\lambda_3}}$ that there exist $H_1,H_2,H_3\subseteq G$ such that $C=\bigcup_{l\in\set{1,2,3}}\set{i_l}\times H_l \times \set{\lambda_l}$. For each $l\in\set{1,2,3}$ let $\mathcal{H}_l$ be the subgraph of $\commgraph{\Reeszero{G}{I}{\Lambda}{P}}$ induced by $\set{i_l}\times G\times\set{\lambda_l}$. As a consequence of the fact that $C$ is a clique in $\commgraph{\Reeszero{G}{I}{\Lambda}{P}}$, then we also have that $\set{i_l}\times H_l\times\set{\lambda_l}$ is a clique in $\commgraph{\Reeszero{G}{I}{\Lambda}{P}}$ for all $l\in\set{1,2,3}$. Furthermore, for each $l\in\set{1,2,3}$ we have that $\set{i_l}\times H_l\times\set{\lambda_l}\subseteq\set{i_l}\times G\times\set{\lambda_l}$, the vertex set of $\mathcal{H}_l$, which implies that $\set{i_l}\times H_l\times\set{\lambda_l}$ is a clique in $\mathcal{H}_l$ for all $l\in\set{1,2,3}$. Moreover, according to Lemma~\ref{0Rees: subgraph induced by ixGxlambda}, for each $l\in\set{1,2,3}$ we have that $\mathcal{H}_l$ is isomorphic to $K_{\abs{\centre{G}}}\graphjointwo\commgraph{G}$ (because $G$ is not abelian and $p_{\lambda_l i_l}\in G$ for all $l\in\set{1,2,3}$). Thus for each $l\in\set{1,2,3}$ we have
	\begin{align*}
		\abs{\set{i_l}\times H_l\times\set{\lambda_l}} &\leqslant\cliquenumber{\mathcal{H}_l}\\
		&=\cliquenumber{K_{\abs{\centre{G}}}\graphjointwo\commgraph{G}}\\
		&=\cliquenumber{K_{\abs{\centre{G}}}}+\cliquenumber{\commgraph{G}}& \bracks{\text{by Lemma~\ref{preli: graph join clique/chromatic numbers}}}\\
		&=\abs{\centre{G}}+\cliquenumber{\commgraph{G}}.
	\end{align*}
	Therefore
	\begin{align*}
		\abs{C}&=\abs[\bigg]{\,\bigcup_{l\in\set{1,2,3}}\set{i_l}\times H_l \times \set{\lambda_l}} \kern -15mm \\
		&=\sum_{l\in\set{1,2,3}}\abs{\set{i_l}\times H_l \times \set{\lambda_l}} \kern -15mm \\
		&\leqslant \sum_{l\in\set{1,2,3}} \parens[\big]{\abs{\centre{G}}+\cliquenumber{\commgraph{G}}} \kern -15mm \\
		&=3\parens[\big]{\abs{\centre{G}}+\cliquenumber{\commgraph{G}}} \kern -15mm \\
		&\leqslant 3\parens{\abs{G}/2} &\bracks{\text{by Lemma~\ref{0Rees: clique number, Lagrange}}}\\
		&\leqslant \abs{G}\cdot 2\\
		&\leqslant \abs{G}\cdot\max N. & \bracks{\text{because }\Psubmatrix{\timeszero{P}}{i_2,i_3}{\lambda_1}=O_{1\times 2} \text{ is a \submatrix\ of } \timeszero{P}}
	\end{align*}
	
	\smallskip
	
	\textsc{Sub-case 2:} Suppose that $G$ is abelian. Since $\Psubmatrix{\timeszero{P}}{I'}{\Lambda'}$ is \equivalent\ to $D_3$, then $D_3$ is a \submatrix\ of $\timeszero{P}$. Consequently, since condition a) holds, we must have that at least one the matrices $O_{1\times 3}$, $O_{3\times 1}$ and $O_{2\times 2}$ is a \submatrix\ of $\timeszero{P}$. Then $\set{3,4}\cap N\neq\emptyset$ and, consequently, $\max N\geqslant 3$. Thus
	\begin{displaymath}
		\abs{C}\leqslant \abs{G}\cdot\abs{R}=\abs{G}\cdot 3\leqslant\abs{G}\cdot\max N.
	\end{displaymath}
	
	\smallskip
	
	\textit{Case 2:} Suppose that $\Psubmatrix{\timeszero{P}}{I'}{\Lambda'}$ is \equivalent\ to $D_2$ and that $\abs{I'}=\abs{\Lambda'}=k$. We can prove that $\abs{C}\leqslant \abs{G}\cdot\max N$ in a similar way to case 1, using condition b) instead of a).
	
	\smallskip
	
	\textit{Case 3:} Suppose that $\Psubmatrix{\timeszero{P}}{I'}{\Lambda'}$ is not \equivalent\ to $D_3$ or $D_2$, and that $\abs{I'}=\abs{\Lambda'}=k$. Then $i_1,\ldots,i_k$ are pairwise distinct and $\lambda_1,\ldots,\lambda_k$ are pairwise distinct. It follows from part 1 of Lemma~\ref{0Rees: clique number, characterization of R} that $p_{\lambda_l i_{l'}}=0$ for all distinct $l,l'\in\X{k}$. Then
	\begin{displaymath}
		\Psubmatrix{P}{i_1,\ldots,i_k}{\lambda_1,\ldots,\lambda_k}=\begin{bmatrix}
			p_{\lambda_1 i_1}&0&0&\cdots&0\\
			0&p_{\lambda_2 i_2}&0&\cdots&0\\
			0&0&p_{\lambda_3 i_3}&\cdots&0\\
			\vdots&\vdots&\vdots&\ddots&\vdots\\
			0&0&0&\cdots&p_{\lambda_k i_k}
		\end{bmatrix}.
	\end{displaymath}
	
	\smallskip
	
	\textsc{Sub-case 1:} Suppose that there exists $l\in\X{k}$ such that $p_{\lambda_l i_l}=0$. Then $\Psubmatrix{P}{i_1,\ldots,i_k}{\lambda_l}=O_{1\times k}$, which implies that $O_{1\times k}$ is a \submatrix\ of $\timeszero{P}$. Hence $k\in N$ and, consequently, $k\leqslant\max N$. Therefore
	\begin{displaymath}
		\abs{C}\leqslant\abs{G}\cdot\abs{R}=\abs{G}\cdot k\leqslant\abs{G}\cdot \max N.
	\end{displaymath}
	
	\smallskip
	
	\textsc{Sub-case 2:} Suppose that $p_{\lambda_l i_l}\neq 0$ for all $l\in\X{k}$. Then $\Psubmatrix{\timeszero{P}}{i_1,\ldots,i_k}{\lambda_1,\ldots,\lambda_k}=D_k$ and, consequently, $\Psubmatrix{\timeszero{P}}{I'}{\Lambda'}$ is \equivalent\ to $D_k$. Due to the fact that $\Psubmatrix{\timeszero{P}}{I'}{\Lambda'}$ is not \equivalent\ to $D_3$ or $D_2$, then we must have $k=1$ or $k\geqslant 4$.
	
	Assume that $k=1$. We have
	\begin{displaymath}
		\abs{C}\leqslant\abs{G}\cdot\abs{R}=\abs{G}\cdot 1\leqslant\abs{G}\cdot\max N.
	\end{displaymath}
	
	Now assume that $k\geqslant 4$. We have $\Psubmatrix{P}{i_3,\ldots,i_k}{\lambda_1,\lambda_2}=O_{\abs{\set{\lambda_1,\lambda_2}}\times\abs{\set{i_3,\ldots,i_k}}}=O_{2\times\parens{k-2}}$. Hence $O_{2\times\parens{k-2}}$ is a \submatrix\ of $\timeszero{P}$, which implies that $2\parens{k-2}\in N$ and, consequently, $2\parens{k-2}\leqslant\max N$. Therefore
	\begin{displaymath}
		\abs{C}\leqslant\abs{G}\cdot\abs{R}=\abs{G}\cdot k\leqslant\abs{G}\cdot\parens{k+\parens{k-4}}=\abs{G}\cdot 2\parens{k-2}\leqslant\abs{G}\cdot\max N.
	\end{displaymath}
	
	\smallskip
	
	\textit{Case 4:} Suppose that $\abs{I'}<k$ or $\abs{\Lambda'}<k$. Assume, without loss of generality, that $\abs{\Lambda'}<k$, that is, that $\lambda_1,\ldots,\lambda_k$ are not pairwise distinct. For each $\lambda\in\Lambda'$ we define $I_\lambda=\gset{i\in I}{\parens{i,\lambda}\in R}$. Let $\lambda'\in\Lambda'$ be such that $\abs{I_{\lambda'}}=\max\gset{\abs{I_\lambda}}{\lambda\in\Lambda'}$. In what follows we ascertain that $\Psubmatrix{P}{I_{\lambda'}}{\Lambda'}=O_{\abs{\Lambda'}\times\abs{I_{\lambda'}}}$. Let $i\in I_{\lambda'}$ and $\lambda\in\Lambda'$. We are going to see that $p_{\lambda i}=0$.
	
	Assume that $\lambda=\lambda'$. As a consequence of the fact that $\abs{\Lambda'}<k=\abs{R}$, then there exist $\mu\in\Lambda'$ and distinct $j,j'\in I$ such that $\parens{j,\mu},\parens{j',\mu}\in R$, which implies that $j,j'\in I_\mu$ and, consequently, $\abs{I_{\lambda'}}\geqslant\abs{I_\mu}\geqslant 2$. This implies the existence of $i'\in I_{\lambda'}$ such that $i'\neq i$. We have $\parens{i,\lambda'},\parens{i',\lambda'}\in R$ and, by part 1 of Lemma~\ref{0Rees: clique number, characterization of R}, we have $p_{\lambda' i'}=p_{\lambda' i}=0$. Thus $p_{\lambda i}=0$.
	
	Now assume that $\lambda\neq\lambda'$. There exists $i'\in I$ such that $\parens{i',\lambda}\in R$. Since we also have $\parens{i,\lambda'}\in R$ (and $\parens{i',\lambda}\neq\parens{i,\lambda'}$), then part 1 of Lemma~\ref{0Rees: clique number, characterization of R} guarantees that $p_{\lambda i}=p_{\lambda' i'}=0$.
	
	Since $i$ and $\lambda$ are arbitrary elements of $I_{\lambda'}$ and $\Lambda'$, respectively, then this proves that $\Psubmatrix{P}{I_{\lambda'}}{\Lambda'}=O_{\abs{\Lambda'}\times\abs{I_{\lambda'}}}$. Hence $O_{\abs{\Lambda'}\times\abs{I_{\lambda'}}}$ is a \submatrix\ of $\timeszero{P}$, which implies that $\abs{\Lambda'}\cdot\abs{I_{\lambda'}}\in N$ and $\abs{\Lambda'}\cdot\abs{I_{\lambda'}}\leqslant \max N$. Therefore
	\begin{align*}
		\abs{C}&\leqslant\abs{G}\cdot\abs{R}\\
		&=\abs{G}\cdot\abs[\bigg]{\,\bigcup_{\lambda\in\Lambda'}\parens{I_\lambda\times\set{\lambda}}}\\
		&=\abs{G}\sum_{\lambda\in\Lambda'} \abs{I_\lambda\times\set{\lambda}}\\
		&=\abs{G}\sum_{\lambda\in\Lambda'}\abs{I_\lambda}\\
		&\leqslant\abs{G}\sum_{\lambda\in\Lambda'} \abs{I_{\lambda'}}\\
		&=\abs{G}\cdot\abs{\Lambda'}\cdot\abs{I_{\lambda'}}\\
		&\leqslant\abs{G}\cdot\max N.
	\end{align*}
	
	\smallskip
	
	In the previous four cases we saw that $\abs{C}\leqslant\abs{G}\cdot\max N$. Since $C$ is an arbitrary clique in $\commgraph{\Reeszero{G}{I}{\Lambda}{P}}$, then we can conclude that $\cliquenumber{\commgraph{\Reeszero{G}{I}{\Lambda}{P}}}\leqslant\abs{G}\cdot\max N$.
\end{proof}

\section{Girth of the commuting graph of a 0-Rees matrix semigroup over a group}\label{sec: girth 0Rees}

Let $G$ be a finite group, let $I$ and $\Lambda$ be finite index sets and let $P$ be a regular $\Lambda \times I$ matrix whose entries are elements of $G^0$. Suppose that $P$ contains at least one zero entry.

In this section we are going to study the cycles in $\commgraph{\Reeszero{G}{I}{\Lambda}{P}}$. More specifically, we are going to characterize the situations in which $\commgraph{\Reeszero{G}{I}{\Lambda}{P}}$ contains cycles, and in that case we are going to determine $\girth{\commgraph{\Reeszero{G}{I}{\Lambda}{P}}}$, the length of a shortest cycle in $\commgraph{\Reeszero{G}{I}{\Lambda}{P}}$.

\begin{lemma}\label{0Rees: lemma girth}
	Suppose that $G$ is trivial. Let
	\begin{displaymath}
		A=\begin{bmatrix}
			0&0&\times&\times\\
			\times&\times&0&0
		\end{bmatrix}\quad
		B=\begin{bmatrix}
			0&0&\times\\
			\times&0&0
		\end{bmatrix}.
	\end{displaymath}
	If $D_3$, $O_{1\times 3}$, $O_{3\times 1}$, $O_{2\times 2}$, $A$, $A^T$, $B$ and $B^T$ are not \submatrices\ of $\timeszero{P}$, then $\commgraph{\Reeszero{G}{I}{\Lambda}{P}}$ contains no cycles. 
\end{lemma}

\begin{proof}
	Assume that $D_3$, $O_{1\times 3}$, $O_{3\times 1}$, $O_{2\times 2}$, $A$, $A^T$, $B$ and $B^T$ are not \submatrices\ of $\timeszero{P}$. Since $O_{1\times 3}$ and $O_{3\times 1}$ are not \submatrices\ of $\timeszero{P}$, then all the rows and columns of $\timeszero{P}$ (and, consequently, $P$) have at most two zero entries. Furthermore, since $O_{2\times 2}$, $A$ and $B$ (respectively, $O_{2\times 2}$, $A^T$ and $B^T$) are not \submatrices\ of $\timeszero{P}$, then there is at most one row (respectively, column) of $\timeszero{P}$ (and, consequently, $P$) with exactly two zero entries.
	
	We divide the proof into two parts: in the first part we establish that if $Q$ is a \zeroclosuresub, then $\timeszero{Q}$ is \equivalent\ to $O_{1\times 1}$, $O_{1\times 2}$, $O_{2\times 1}$ or $\begin{bmatrix}0&0\\ 0&\times\end{bmatrix}$, and in the second part we use part 1 to prove that $\commgraph{\Reeszero{G}{I}{\Lambda}{P}}$ has no cycles.
	
	\medskip
	
	\textbf{Part 1.} Let $Q=\Psubmatrix{P}{I_Q}{\Lambda_Q}$ be a \zeroclosuresub.
	
	\smallskip
	
	\textit{Case 1:} Suppose that $\abs{I_Q}=1$. Since, by Lemma~\ref{0Rees: every row/column of 0-closure submatrix has a 0}, every row of $Q$ contains at least one zero entry, then $Q=\Psubmatrix{P}{I_Q}{\Lambda_Q}=O_{\abs{\Lambda_Q}\times 1}$. Furthermore, we have that every column of $P$ contains at most two zero entries, which implies that $\abs{\Lambda_Q}\leqslant 2$. Hence $Q=O_{1\times 1}$ or $Q=O_{2\times 1}$. Therefore $\timeszero{Q}$ is \equivalent\ to $O_{1\times 1}$ or $\timeszero{Q}$ is \equivalent\ to $O_{2\times 1}$.
	
	\smallskip
	
	\textit{Case 2:} Suppose that $\abs{\Lambda_Q}=1$. We can verify in an analogous way to case 1 that $\timeszero{Q}=O_{1\times 1}$ or $\timeszero{Q}=O_{1\times 2}$.
	
	\smallskip
	
	\textit{Case 3:} Suppose that $\abs{I_Q}>1$ and $\abs{\Lambda_Q}>1$. Let $i\in I_Q$ and $\lambda\in\Lambda_Q$ be such that $p_{\lambda i}=0$. For each $i\in\set{0,1,\ldots,\zeroindex{i}{\lambda}}$ let $Q_i$ be the submatrix of $P$ constructed at step $i$ of the $0$-closure method. We observe that, by Lemma~\ref{0Rees: comparing 0-closure submatrices}, $Q$ is the \zeroclosuresubmatrix{i}{\lambda}, that is, $Q$ corresponds to the matrix $Q_{\zeroindex{i}{\lambda}}$. We have that $Q_0=\Psubmatrix{P}{i}{\lambda}$. Since $\abs{I_Q}>1$, then $\zeroindex{i}{\lambda}>0$, which implies that row $\lambda$ of $P$ contains more than one zero entry or column $i$ of $P$ contains more than one zero entry. As a consequence of the fact that all rows and all columns of $P$ contain at most two zero entries, then we have that row $\lambda$ of $P$ contains (exactly) two zero entries or column $i$ of $P$ contains (exactly) two zero entries. We consider the following sub-cases:
	
	\smallskip
	
	\textsc{Sub-case 1:} Suppose that row $\lambda$ and column $i$ of $P$ both contain exactly two zero entries. Then there exist $j\in I\setminus\set{i}$ and $\mu\in\Lambda\setminus\set{\lambda}$ such that $p_{\lambda j}=p_{\mu i}=0$. Due to the fact that $P$ contains at most one row with exactly two zero entries, and row $\lambda$ has two zero entries, then row $\mu$ contains only one zero entry --- the $\parens{\mu,i}$-th entry of $P$. Consequently, $\timeszero{P}$ is \equivalent\ to
	\begin{displaymath}
		\begin{bNiceArray}{cc|cccccccccccc}[first-row,first-col,margin]
			& i & j & \Block{1-12}{I\setminus \set{i,j}} &&&&&&&&&&&\\
			\lambda & 0 & 0 & \Block{2-12}{\blocktimes{0.61em}{5em}{0em}} &&&&&&&\\
			\mu & 0 & \times &&&&&&&&&&&&\\
			\hline
			\Block{2-1}{\Lambda\setminus\set{\lambda,\mu}} & \Block{2-2}{\blocktimes{0.61em}{1.19em}{-0.25em}} & & \Block{2-12}{\Psubmatrix{\timeszero{P}}{I\setminus\set{i,j}}{\Lambda\setminus\set{\lambda,\mu}}} &&&&&&&&&&&\\
			\\
		\end{bNiceArray}
	\end{displaymath}
	and it is easy to conclude that $\timeszero{Q}=\Psubmatrix{\timeszero{P}}{\set{i,j}}{\set{\lambda,\mu}}$, which is \equivalent\ to $\begin{bmatrix}
		0&0\\
		0&\times
	\end{bmatrix}$.
	
	\smallskip
	
	\textsc{Sub-case 2:} Suppose that row $\lambda$ of $P$ contains exactly two zero entries and that column $i$ of $P$ contains only one zero entry. Then there exists $j\in I\setminus\set{i}$ such that $p_{\lambda j}=0$, and we have that the only zero entry of column $i$ is the $\parens{\lambda,i}$-th entry of $P$. Consequently, $Q_1=\Psubmatrix{P}{\set{i,j}}{\lambda}=O_{1\times 2}$. In addition, we have that $\abs{\Lambda_Q}>1$, which implies that $\zeroindex{i}{\lambda}>1$. Since column $i$ has no more zero entries other than the $\parens{\lambda,i}$-th entry of $P$, then this means that column $j$  must contain (exactly) two zero entries. Let $\mu\in\Lambda\setminus\set{\lambda}$ be such that $p_{\mu j}=0$. Since row $\lambda$ contains two zero entries, then row $\mu$ contains only one zero entry --- the $\parens{\mu,j}$-th entry of $P$. Thus $\timeszero{P}$ is \equivalent\ to
	\begin{displaymath}
		\begin{bNiceArray}{cc|cccccccccccc}[first-row,first-col,margin]
			& i & j & \Block{1-12}{I\setminus \set{i,j}} &&&&&&&&&&&\\
			\lambda & 0 & 0 & \Block{2-12}{\blocktimes{0.61em}{5em}{0em}} &&&&&&&\\
			\mu & \times & 0 &&&&&&&&&&&&\\
			\hline
			\Block{2-1}{\Lambda\setminus\set{\lambda,\mu}} & \Block{2-2}{\blocktimes{0.61em}{1.19em}{-0.25em}} & & \Block{2-12}{\Psubmatrix{\timeszero{P}}{I\setminus\set{i,j}}{\Lambda\setminus\set{\lambda,\mu}}} &&&&&&&&&&&\\
			\\
		\end{bNiceArray}
	\end{displaymath}
	and we can conclude that $\timeszero{Q}=\Psubmatrix{\timeszero{P}}{\set{i,j}}{\set{\lambda,\mu}}$, which is \equivalent\ to $\begin{bmatrix}
		0&0\\
		0&\times
	\end{bmatrix}$.
	
	\smallskip
	
	\textsc{Sub-case 3:} Suppose that column $i$ of $P$ contains exactly two zero entries and that row $\lambda$ of $P$ contains only one entry. We can verify, analogously to sub-case 2, that $\timeszero{Q}$ is \equivalent\ to $\begin{bmatrix}
		0&0\\
		0&\times
	\end{bmatrix}$.

	\medskip
	
	\textbf{Part 2.} Now we are going to check that $\commgraph{\Reeszero{G}{I}{\Lambda}{P}}$ contains no cycles. In order to see that, we are going to verify that none of the connected components of $\commgraph{\Reeszero{G}{I}{\Lambda}{P}}$ contains cycles. Let $\mathcal{C}$ be a connected component of $\commgraph{\Reeszero{G}{I}{\Lambda}{P}}$. We have three possibilities for $\mathcal{C}$:
	
	\smallskip
	
	\textit{Case 1:} Suppose that $\mathcal{C}$ is determined by a \zeroclosuresub. Assume that $Q=\Psubmatrix{P}{I_Q}{\Lambda_Q}$ is that submatrix. We know that $\timeszero{Q}$ is \equivalent\ to $O_{1\times 1}$, $O_{1\times 2}$, $O_{2\times 1}$ or $\begin{bmatrix}0&0\\ 0&\times\end{bmatrix}$.
	
	If $\timeszero{Q}$ is \equivalent\ to $O_{1\times 1}$, $O_{1\times 2}$ or $O_{2\times 1}$, then $\abs{I_Q\times G\times\Lambda_Q}=\abs{I_Q}\cdot\abs{G}\cdot\abs{\Lambda_Q}=\abs{I_Q}\cdot\abs{\Lambda_Q}\leqslant 2$, that is, $\mathcal{C}$ has at most two vertices, which implies that $\mathcal{C}$ has no cycles.
	
	If $\timeszero{Q}$ is \equivalent\ to $\begin{bmatrix}0&0\\ 0&\times\end{bmatrix}$, then there exist distinct $i,j\in I_Q$ and distinct $\lambda,\mu\in\Lambda_Q$ such that $\Psubmatrix{\timeszero{Q}}{i,j}{\lambda,\mu}=\Psubmatrix{\timeszero{P}}{i,j}{\lambda,\mu}=\begin{bmatrix}0&0\\ 0&\times\end{bmatrix}$. Then $\mathcal{C}$ is the graph
	
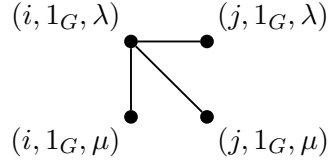
\begin{center}
		\begin{tikzpicture}
			
			\node[vertex] (11) at (0,0) {};
			\node[vertex] (12) at (1,0) {};
			\node[vertex] (21) at (0,-1) {};
			\node[vertex] (22) at (1,-1) {};

			\node[anchor=south east] at (11) {$\parens{i,1_G,\lambda}$};
			\node[anchor=south west] at (12) {$\parens{j,1_G,\lambda}$};
			\node[anchor=north east] at (21) {$\parens{i,1_G,\mu}$};
			\node[anchor=north west] at (22) {$\parens{j,1_G,\mu}$};

			\draw[edge] (11)--(12);
			\draw[edge] (11)--(21);
			\draw[edge] (11)--(22);
			
		\end{tikzpicture}
	\end{center}
	and, by inspection, $\mathcal{C}$ has no cycles.
	
	\smallskip
	
	\textit{Case 2:} Suppose that $\mathcal{C}$ is determined by two (distinct) $0$-closure submatrices of $P$. Let $Q=\Psubmatrix{P}{I_Q}{\lambda_Q}$ and $M=\Psubmatrix{P}{I_M}{\Lambda_M}$ be those two submatrices. We observe that $\timeszero{\Psubmatrix{P}{I_Q\cup I_M}{\Lambda_Q\cup\Lambda_M}}$ is \equivalent\ to
	
	\begin{displaymath}
		\begin{bNiceArray}{cc|cc}[first-row,first-col,margin]
			& \Block{1-2}{I_Q} & & \Block{1-2}{I_M} & \\
			\Block{2-1}{\Lambda_Q} & \Block{2-2}{\timeszero{Q}} & & \Block{2-2}{\blocktimes{0.61em}{0.61em}{-0.25em}} & \\
			\\
			\hline
			\Block{2-1}{\Lambda_M} & \Block{2-2}{\blocktimes{0.61em}{0.61em}{-0.25em}} & & \Block{2-2}{\timeszero{M}} & \\
			\\
		\end{bNiceArray}.
	\end{displaymath}
	
	\smallskip
	
	\textsc{Sub-case 1:} Suppose that $\timeszero{Q}$ is \equivalent\ to $O_{1\times 2}$ or $\timeszero{M}$ is \equivalent\ to $O_{1\times 2}$. Assume, without loss of generality, that $\timeszero{Q}$ is \equivalent\ to $O_{1\times 2}$. Then there exist distinct $i,i'\in I$ and $\lambda\in\Lambda$ such that $I_Q=\set{i,i'}$, $\Lambda_Q=\set{\lambda}$ and $p_{\lambda i}=p_{\lambda i'}=0$. Since row $\lambda$ of $\timeszero{P}$ contains two zero entries, then for all $\mu\in\Lambda\setminus\set{\lambda}$ we have that row $\mu$ of $\timeszero{P}$ contain at most one zero entry. In particular, this implies that $\timeszero{M}$ cannot have rows with two zero entries and, consequently, $\timeszero{M}$ is not \equivalent\ to $O_{1\times 2}$ or $\begin{bmatrix} 0&0\\ 0&\times \end{bmatrix}$. Thus $\timeszero{M}$ must be \equivalent\ to $O_{1\times 1}$ or $O_{2\times 1}$.
	
	If $\timeszero{M}$ is \equivalent\ to $O_{1\times 1}$, then there exist $j\in I$ and $\lambda'\in\Lambda$ such that $I_M=\set{j}$, $\Lambda_M=\set{\lambda'}$ and $p_{\lambda' j}=0$. Consequently, the connected component $\mathcal{C}$ is the graph
	
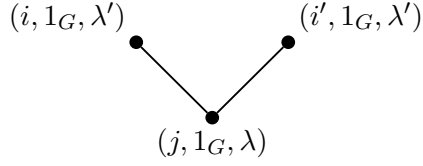
\begin{center}
		\begin{tikzpicture}
			
			\node[vertex] (11) at (0,0) {};
			\node[vertex] (21) at (1,-1) {};
			\node[vertex] (12) at (2,0) {};

			\node[anchor=south east] at (11) {$\parens{i,1_G,\lambda'}$};
			\node[anchor=south west] at (12) {$\parens{i',1_G,\lambda'}$};
			\node[anchor=north] at (21) {$\parens{j,1_G,\lambda}$};

			\draw[edge] (11)--(21);
			\draw[edge] (21)--(12);
			
		\end{tikzpicture}
	\end{center}
	which has no cycles.
	
	If $\timeszero{M}$ is \equivalent\ to $O_{2\times 1}$, then there exist $j\in I$ and distinct $\mu,\mu'\in\Lambda$ such that $I_M=\set{j}$, $\Lambda_M=\set{\mu,\mu'}$ and $p_{\mu j}=p_{\mu' j}=0$. Thus $\mathcal{C}$ is the graph
	
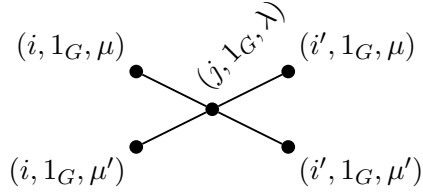
\begin{center}
		\begin{tikzpicture}
			
			\node[vertex] (11) at (0,0) {};
			\node[vertex] (12) at (2,0) {};
			\node[vertex] (21) at (0,-1) {};
			\node[vertex] (22) at (2,-1) {};
			\node[vertex] (3) at (1,-0.5) {};

			\node[anchor=south east] at (11) {$\parens{i,1_G,\mu}$};
			\node[anchor=south west] at (12) {$\parens{i',1_G,\mu}$};
			\node[anchor=north east] at (21) {$\parens{i,1_G,\mu'}$};
			\node[anchor=north west] at (22) {$\parens{i',1_G,\mu'}$};
			\node[anchor=south west,rotate=45] at (3) {$\parens{j,1_G,\lambda}$};

			\draw[edge] (3)--(11);
			\draw[edge] (3)--(12);
			\draw[edge] (3)--(21);
			\draw[edge] (3)--(22);
			
		\end{tikzpicture}
	\end{center}
	which also has no cycles.
	
	\smallskip
	
	\textsc{Sub-case 2:} Suppose that $\timeszero{Q}$ is \equivalent\ to $O_{2\times 1}$ or $\timeszero{M}$ is \equivalent\ to $O_{2\times 1}$. We can verify, in a similar way to sub-case 1, that $\mathcal{C}$ has no cycles.
	
	\smallskip
	
	\textsc{Sub-case 3:} Suppose that $\timeszero{Q}$ is \equivalent\ to $\begin{bmatrix} 0&0\\ 0&\times \end{bmatrix}$ or $\timeszero{M}$ is \equivalent\ to $\begin{bmatrix} 0&0\\ 0&\times \end{bmatrix}$. Assume, without loss of generality, that $\timeszero{Q}$ is \equivalent\ to $\begin{bmatrix} 0&0\\ 0&\times \end{bmatrix}$. Let $i,i'\in I_Q$ and $\lambda,\lambda'\in\Lambda_Q$ be such that $\Psubmatrix{\timeszero{Q}}{i,i'}{\lambda,\lambda'}=\Psubmatrix{\timeszero{P}}{i,i'}{\lambda,\lambda'}=\begin{bmatrix} 0&0\\ 0&\times \end{bmatrix}$. Since row $\lambda$ (respectively, column $i$) of $\timeszero{P}$ contains two zero entries, then all the other rows (respectively, columns) of $\timeszero{P}$ have at most one zero entry. Hence all rows and all columns of $\timeszero{M}$ contain at most one zero entry, which implies that $\timeszero{M}$ is \equivalent\ to $O_{1\times 1}$. Assume that $I_M=\set{j}$ and $\Lambda_M=\set{\mu}$. Then $p_{\mu j}=0$. We have that $\mathcal{C}$ is the graph
	
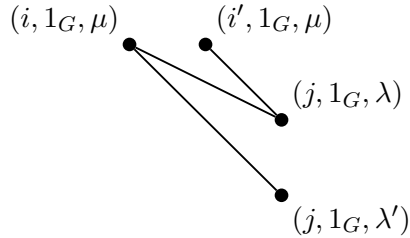
\begin{center}
		\begin{tikzpicture}
			
			\node[vertex] (11) at (0,0) {};
			\node[vertex] (21) at (1,0) {};
			\node[vertex] (12) at (2,-1) {};
			\node[vertex] (22) at (2,-2) {};

			\node[anchor=south east] at (11) {$\parens{i,1_G,\mu}$};
			\node[anchor=south west] at (12) {$\parens{j,1_G,\lambda}$};
			\node[anchor=south west] at (21) {$\parens{i',1_G,\mu}$};
			\node[anchor=north west] at (22) {$\parens{j,1_G,\lambda'}$};

			\draw[edge] (11)--(12);
			\draw[edge] (11)--(22);
			\draw[edge] (21)--(12);
			
		\end{tikzpicture}
	\end{center}
	and it is straightforward to check that $\mathcal{C}$ has no cycles.
	
	\smallskip
	
	\textsc{Sub-case 4:} Suppose that $\timeszero{Q}$ and $\timeszero{M}$ are not \equivalent\ to $O_{1\times 2}$, $O_{2\times 1}$ or $\begin{bmatrix} 0&0\\ 0&\times \end{bmatrix}$. Then $\timeszero{Q}$ and $\timeszero{M}$ are both \equivalent\ to $O_{1\times 1}$ and, consequently, $\abs{I_Q}=\abs{\Lambda_Q}=\abs{I_M}=\abs{\Lambda_M}=1$. Therefore $\mathcal{C}$ has $\abs{\parens{I_Q\times G\times\Lambda_M}\cup\parens{I_M\times G\times\Lambda_Q}}=\abs{I_Q}\cdot\abs{G}\cdot\abs{\Lambda_M}+\abs{I_M}\cdot\abs{G}\cdot\abs{\Lambda_Q}=1+1=2$ vertices, which implies that $\mathcal{C}$ has no cycles.
	
	\smallskip
	
	\textit{Case 3:} Assume that $\mathcal{C}$ is determined by $\parens{\lambda,i}$, for some $i\in I$ and $\lambda\in\Lambda$ such that column $i$ has no zero entries or row $\lambda$ has no zero entries. Then the vertex set of $\mathcal{C}$ is $\set{i}\times G\times\set{\lambda}$, which is a singleton. Thus $\mathcal{C}$ contains no cycles.
	
	\smallskip
	
	It follows from cases 1, 2 and 3 that none of the connected components of $\commgraph{\Reeszero{G}{I}{\Lambda}{P}}$ contains cycles. Thus $\commgraph{\Reeszero{G}{I}{\Lambda}{P}}$ contains no cycles.
\end{proof}

\begin{theorem}\label{0Rees: girth}
	Let
	\begin{displaymath}
		A=\begin{bmatrix}
			0&0&\times&\times\\
			\times&\times&0&0
		\end{bmatrix}\quad
		B=\begin{bmatrix}
			0&0&\times\\
			\times&0&0
		\end{bmatrix}.
	\end{displaymath}
	\begin{enumerate}
		\item Suppose that $|G|\geqslant 3$. Then $\commgraph{\Reeszero{G}{I}{\Lambda}{P}}$ contains cycles and $\girth{\commgraph{\Reeszero{G}{I}{\Lambda}{P}}}=3$.
		
		\item Suppose that $|G|=2$. Then $\commgraph{\Reeszero{G}{I}{\Lambda}{P}}$ contains cycles if and only if $P$ contains more than one zero entry, in which case $\girth{\commgraph{\Reeszero{G}{I}{\Lambda}{P}}}=3$.
		
		\item Suppose that $|G|=1$. Then $\commgraph{\Reeszero{G}{I}{\Lambda}{P}}$ contains cycles if and only if at least one of the matrices $D_3$, $O_{1\times 3}$, $O_{3\times 1}$, $O_{2\times 2}$, $A$, $A^T$, $B$ or $B^T$ is a \submatrix\ of $\timeszero{P}$. Furthermore,
		\begin{enumerate}
			\item If at least one of the matrices $D_3$, $O_{1\times 3}$, $O_{3\times 1}$ or $O_{2\times 2}$ is a \submatrix\ of $\timeszero{P}$, then $\girth{\commgraph{\Reeszero{G}{I}{\Lambda}{P}}}=3$.
			
			\item If at least one of the matrices $A$, $A^T$, $B$ or $B^T$ is a \submatrix\ of $\timeszero{P}$ and $D_3$, $O_{1\times 3}$, $O_{3\times 1}$ and $O_{2\times 2}$ are not \submatrices\ of $\timeszero{P}$, then $\girth{\commgraph{\Reeszero{G}{I}{\Lambda}{P}}}=4$.
		\end{enumerate}
	\end{enumerate}
\end{theorem}

\begin{proof}
	\textbf{Part 1.} Suppose that $\abs{G}\geqslant 3$. Then there exist distinct $x,y,z\in G$. Since $P$ contains at least one zero entry, there exist $i\in I$ and $\lambda\in \Lambda$ such that $p_{\lambda i}=0$. It follows from Lemma~\ref{0Rees: commutativity} that $\parens{i,x,\lambda}$, $\parens{i,y,\lambda}$ and $\parens{i,z,\lambda}$ commute with each other, which implies that
	\begin{displaymath}
		\parens{i,x,\lambda}-\parens{i,y,\lambda}-\parens{i,z,\lambda}-\parens{i,x,\lambda}
	\end{displaymath}
	is a cycle (of length $3$) in $\commgraph{\Reeszero{G}{I}{\Lambda}{P}}$. Thus $\girth{\commgraph{\Reeszero{G}{I}{\Lambda}{P}}}=3$.
	
	\medskip
	
	\textbf{Part 2.} Suppose that $\abs{G}=2$. First, we prove the forward implication. Assume that $\commgraph{\Reeszero{G}{I}{\Lambda}{P}}$ contains cycles. Let
	\begin{displaymath}
		\parens{i_1,x_1,\lambda_1}-\parens{i_2,x_2,\lambda_2}-\cdots-\parens{i_n,x_n,\lambda_n}-\parens{i_1,x_1,\lambda_1}
	\end{displaymath}
	be a cycle in $\commgraph{\Reeszero{G}{I}{\Lambda}{P}}$. We have $n\geqslant 3$. Assume that $\parens{i_1,\lambda_1}\neq\parens{i_2,\lambda_2}$. Then, it follows from Lemma~\ref{0Rees: commutativity} that $p_{\lambda_1 i_2}=p_{\lambda_2 i_1}=0$ and, consequently, $P$ contains (at least) two zero entries. Now assume that $\parens{i_1,\lambda_1}=\parens{i_2,\lambda_2}$. Then $x_1\neq x_2$ and $G=\set{x_1,x_2}$. We have that $x_3\in\set{x_1,x_2}$. Furthermore, $\parens{i_3,x_3,\lambda_3}$ is distinct from $\parens{i_1,x_1,\lambda_1}=\parens{i_2,x_1,\lambda_2}$ and $\parens{i_2,x_2,\lambda_2}$, which implies that $\parens{i_3,\lambda_3}\neq\parens{i_2,\lambda_2}$. Then, by Lemma~\ref{0Rees: commutativity}, we have that $p_{\lambda_2 i_3}=p_{\lambda_3 i_2}=0$. Hence $P$ contains (at least) two zero entries.
	
	Now we prove the reverse implication. Assume that $P$ contains at least two zero entries. Then there exist $i,j\in I$ and $\lambda,\mu\in\Lambda$ such that $\parens{i,\lambda}\neq\parens{j,\mu}$ and $p_{\lambda i}=p_{\mu j}=0$. Furthermore, since $\abs{G}=2$, then $G$ is abelian and there exist distinct $x,y\in G$. Due to the fact that $p_{\lambda i}=p_{\mu j}=0$, and as a consequence of Lemma~\ref{0Rees: commutativity}, we have that $\parens{i,x,\mu}$ commutes with $\parens{j,x,\lambda}$ and $\parens{j,y,\lambda}$. Moreover, if $p_{\lambda j}=0$, we have $xp_{\lambda j}y=0=yp_{\lambda j}x$ and, if $p_{\lambda j}\in G$, we have $xp_{\lambda j}y=yp_{\lambda j}x$ because $G$ is abelian. Hence Lemma~\ref{0Rees: commutativity} also ensures that $\parens{j,x,\lambda}$ and $\parens{j,y,\lambda}$ commute. Therefore
	\begin{displaymath}
		\parens{i,x,\mu}-\parens{j,x,\lambda}-\parens{j,y,\lambda}-\parens{i,x,\mu}
	\end{displaymath}
	is a cycle (of length $3$) in $\commgraph{\Reeszero{G}{I}{\Lambda}{P}}$ and we have $\girth{\commgraph{\Reeszero{G}{I}{\Lambda}{P}}}=3$.
	
	\medskip
	
	\textbf{Part 3.} Suppose that $\abs{G}=1$. We observe that the forward implication is an immediate consequence of Lemma~\ref{0Rees: lemma girth}. We are going to prove the reverse implication. In order to achieve this we consider the following cases:
	
	\smallskip
	
	\textit{Case 1:} Suppose that $D_3$ is a \submatrix\ of $\timeszero{P}$. Then there exist pairwise distinct $i_1,i_2,i_3\in I$ and pairwise distinct $\lambda_1,\lambda_2,\lambda_3\in\Lambda$ such that $\Psubmatrix{\timeszero{P}}{i_1,i_2,i_3}{\lambda_1,\lambda_2,\lambda_3}\allowbreak=D_3$. Hence $p_{\lambda_1 i_2}=p_{\lambda_1 i_3}=p_{\lambda_2 i_1}=p_{\lambda_2 i_3}=p_{\lambda_3 i_1}=p_{\lambda_3 i_2}=0$. Consequently, Lemma~\ref{0Rees: commutativity} ensures that $\parens{i_1,1_G,\lambda_1}$, $\parens{i_2,1_G,\lambda_2}$ and $\parens{i_3,1_G,\lambda_3}$ commute with each other. Thus
	\begin{displaymath}
		\parens{i_1,1_G,\lambda_1}-\parens{i_2,1_G,\lambda_2}-\parens{i_3,1_G,\lambda_3}-\parens{i_1,1_G,\lambda_1}
	\end{displaymath}
	is a cycle (of length $3$) in $\commgraph{\Reeszero{G}{I}{\Lambda}{P}}$, which implies that $\girth{\commgraph{\Reeszero{G}{I}{\Lambda}{P}}}=3$.
	
	\smallskip
	
	\textit{Case 2:} Suppose that $O_{1\times 3}$ is a \submatrix\ of $\timeszero{P}$. (The proof is symmetrical if $O_{3\times 1}$ is a \submatrix\ of $\timeszero{P}$.) Then there exist pairwise distinct $i_1,i_2,i_3\in I$ and $\lambda\in\Lambda$ such that $\Psubmatrix{\timeszero{P}}{i_1,i_2,i_3}{\lambda}=O_{1\times 3}$, which implies that $p_{\lambda i_1}=p_{\lambda i_2}=p_{\lambda i_3}=0$. Then, by Lemma~\ref{0Rees: commutativity}, $\parens{i_1,1_G,\lambda}$, $\parens{i_2,1_G,\lambda}$ and $\parens{i_3,1_G,\lambda}$ commute with one another and, consequently,
	\begin{displaymath}
		\parens{i_1,1_G,\lambda}-\parens{i_2,1_G,\lambda}-\parens{i_3,1_G,\lambda}-\parens{i_1,1_G,\lambda}
	\end{displaymath}
	is a cycle (of length $3$) in $\commgraph{\Reeszero{G}{I}{\Lambda}{P}}$. Thus $\girth{\commgraph{\Reeszero{G}{I}{\Lambda}{P}}}=3$.
	
	\smallskip
	
	\textit{Case 3:} Suppose that $O_{2\times 2}$ is a \submatrix\ of $\timeszero{P}$. Then $\Psubmatrix{\timeszero{P}}{i_1,i_2}{\lambda_1,\lambda_2}=O_{2\times 2}$ for some distinct $i_1,i_2\in I$ and distinct $\lambda_1,\lambda_2\in\Lambda$. We have $p_{\lambda_1 i_1}=p_{\lambda_1 i_2}=p_{\lambda_2 i_1}=p_{\lambda_2 i_2}=0$, which implies that $\parens{i_1,1_G,\lambda_1}$, $\parens{i_2,1_G,\lambda_2}$ and $\parens{i_2,1_G,\lambda_1}$ commute with each other. Therefore,
	\begin{displaymath}
		\parens{i_1,1_G,\lambda_1}-\parens{i_2,1_G,\lambda_2}-\parens{i_2,1_G,\lambda_1}-\parens{i_1,1_G,\lambda_1}
	\end{displaymath}
	is a cycle (of length $3$) in $\commgraph{\Reeszero{G}{I}{\Lambda}{P}}$ and we have $\girth{\commgraph{\Reeszero{G}{I}{\Lambda}{P}}}=3$.
	
	\smallskip
	
	\textit{Case 4:} Suppose that $A$ is a \submatrix\ of $\timeszero{P}$. (The proof is similar if we assume that $A^T$ is a \submatrix\ of $\timeszero{P}$.) It follows that there exist pairwise distinct $i_1,i_2,i_3,i_4\in I$ and distinct $\lambda_1,\lambda_2\in\Lambda$ such that $\Psubmatrix{\timeszero{P}}{i_1,i_2,i_3,i_4}{\lambda_1,\lambda_2}=A$. Then we have $p_{\lambda_1 i_1}=p_{\lambda_1 i_2}=p_{\lambda_2 i_3}=p_{\lambda_2 i_4}=0$ and, consequently, Lemma~\ref{0Rees: commutativity} ensures that $\parens{i_3,1_G,\lambda_1}$ commutes with $\parens{i_1,1_G,\lambda_2}$ and $\parens{i_2,1_G,\lambda_2}$, and $\parens{i_4,1_G,\lambda_1}$ commutes with $\parens{i_2,1_G,\lambda_2}$ and $\parens{i_1,1_G,\lambda_2}$. Thus
	\begin{displaymath}
		\parens{i_1,1_G,\lambda_2}-\parens{i_3,1_G,\lambda_1}-\parens{i_2,1_G,\lambda_2}-\parens{i_4,1_G,\lambda_1}-\parens{i_1,1_G,\lambda_2}
	\end{displaymath}
	is a cycle (of length $4$) in $\commgraph{\Reeszero{G}{I}{\Lambda}{P}}$, which implies that $\girth{\commgraph{\Reeszero{G}{I}{\Lambda}{P}}}\leqslant 4$.
	
	\smallskip
	
	\textit{Case 5:} Suppose that $B$ is a \submatrix\ of $\timeszero{P}$. (The proof is similar when $B^T$ is a \submatrix\ of $\timeszero{P}$.) There exist pairwise distinct $i_1,i_2,i_3\in I$ and distinct $\lambda_1,\lambda_2\in\Lambda$ such that $\Psubmatrix{\timeszero{P}}{i_1,i_2,i_3}{\lambda_1,\lambda_2}=B$. Since $p_{\lambda_1 i_1}=p_{\lambda_1 i_2}=p_{\lambda_2 i_2}=p_{\lambda_2 i_3}=0$, then Lemma~\ref{0Rees: commutativity} guarantees that $\parens{i_2,1_G,\lambda_1}$ commutes with $\parens{i_1,1_G,\lambda_2}$ and $\parens{i_2,1_G,\lambda_2}$, and that $\parens{i_3,1_G,\lambda_1}$ commutes with $\parens{i_2,1_G,\lambda_2}$ and $\parens{i_1,1_G,\lambda_2}$. Consequently,
	\begin{displaymath}
		\parens{i_1,1_G,\lambda_2}-\parens{i_2,1_G,\lambda_1}-\parens{i_2,1_G,\lambda_2}-\parens{i_3,1_G,\lambda_1}-\parens{i_1,1_G,\lambda_2}
	\end{displaymath}
	is a cycle (of length $4$) in $\commgraph{\Reeszero{G}{I}{\Lambda}{P}}$. Thus $\girth{\commgraph{\Reeszero{G}{I}{\Lambda}{P}}}\leqslant 4$.
	
	\smallskip
	
	This concludes the proof of the reverse implication. Furthermore, we notice that cases 1, 2 and 3 demonstrate part a). In order to prove part b), it is enough to show that when $\commgraph{\Reeszero{G}{I}{\Lambda}{P}}$ contains a cycle of length $3$, then at least one of the matrices $D_3$, $O_{1\times 3}$, $O_{3 \times 1}$ or $O_{2\times 2}$ is a \submatrix\ of $\timeszero{P}$. Assume that $\commgraph{\Reeszero{G}{I}{\Lambda}{P}}$ contains cycles of length $3$, and let
	\begin{displaymath}
		\parens{i_1,1_G,\lambda_1}-\parens{i_2,1_G,\lambda_2}-\parens{i_3,1_G,\lambda_3}-\parens{i_1,1_G,\lambda_1}
	\end{displaymath}
	be one of those cycles. Since $\parens{i_1,1_G,\lambda_1}$, $\parens{i_2,1_G,\lambda_2}$ and $\parens{i_3,1_G,\lambda_3}$ are pairwise distinct, then $\parens{i_1,\lambda_1}$, $\parens{i_2,\lambda_2}$ and $\parens{i_3,\lambda_3}$ are also pairwise distinct. Furthermore, $\parens{i_1,1_G,\lambda_1}$, $\parens{i_2,1_G,\lambda_2}$ and $\parens{i_3,1_G,\lambda_3}$ commute with each other. Thus Lemma~\ref{0Rees: commutativity} guarantees that $p_{\lambda_1 i_2}=p_{\lambda_2 i_1}=p_{\lambda_2 i_3}=p_{\lambda_3 i_2}=p_{\lambda_3 i_1}=p_{\lambda_1 i_3}=0$. We consider the following three cases.
	
	\smallskip
	
	\textit{Case 1:} Suppose that there exist distinct $j,k\in\set{1,2,3}$ such that $i_j=i_k$. Then we must have $\lambda_j\neq\lambda_k$. Let $l\in\set{1,2,3}\setminus\set{j,k}$. We have two possibilities: $\lambda_l,\lambda_j,\lambda_k$ are pairwise distinct, or $\lambda_l\in\set{\lambda_j,\lambda_k}$. If $\lambda_l,\lambda_j,\lambda_k$ are pairwise distinct, then we have $\Psubmatrix{\timeszero{P}}{i_j}{\lambda_l,\lambda_j,\lambda_k}=O_{3 \times 1}$. If $\lambda_l\in\set{\lambda_j,\lambda_k}$, then $i_l\neq i_j=i_k$ and $\Psubmatrix{\timeszero{P}}{\lambda_j,\lambda_k}{i_j,i_l}=O_{2\times 2}$. Therefore $O_{3\times 1}$ is a \submatrix\ of $\timeszero{P}$ or $O_{2\times 2}$ is a \submatrix\ of $\timeszero{P}$.
	
	\smallskip
	
	\textit{Case 2:} Suppose that there exist distinct $j,k\in\set{1,2,3}$ such that $\lambda_j=\lambda_k$. We can prove in a similar way to case 1 that $O_{1\times 3}$ is a \submatrix\ of $\timeszero{P}$ or $O_{2\times 2}$ is a \submatrix\ of $\timeszero{P}$.
	
	\smallskip
	
	\textit{Case 3:} Suppose that $i_1,i_2,i_3$ are pairwise distinct and $\lambda_1,\lambda_2,\lambda_3$ are pairwise distinct. If $p_{\lambda_1 i_1}, p_{\lambda_2 i_2}, p_{\lambda_3 i_3}\in G$, then $\Psubmatrix{\timeszero{P}}{i_1,i_2,i_3}{\lambda_1,\lambda_2,\lambda_3}=D_3$. If there exist $j\in\set{1,2,3}$ such that $p_{\lambda_j i_j}=0$, then $\Psubmatrix{\timeszero{P}}{i_1,i_2,i_3}{\lambda_j}=O_{1\times 3}$. Therefore $D_3$ is a \submatrix\ of $\timeszero{P}$ or $O_{1\times 3}$ is a submatrix of $\timeszero{P}$.
\end{proof}

\section{Chromatic number of the commuting graph of a 0-Rees matrix semigroup over a group}\label{sec: chromatic number 0Rees}

Let $G$ be a finite group, let $I$ and $\Lambda$ be finite index sets and let $P$ be a regular $\Lambda \times I$ matrix whose entries are elements of $G^0$ and such that $P$ contains at least one zero entry.

The aim of this section is to gain some insight regarding the chromatic number of $\commgraph{\Reeszero{G}{I}{\Lambda}{P}}$. In Theorem~\ref{0Rees: clique number} we exhibited the clique number of $\commgraph{\Reeszero{G}{I}{\Lambda}{P}}$. It is known that the clique number of a graph provides a lower bound for its chromatic number. Here we obtain an upper bound for the chromatic number of $\commgraph{\Reeszero{G}{I}{\Lambda}{P}}$. We will present two methods for determining such an upper bound, which are given by Theorems~\ref{0Rees: chromatic number edges} and \ref{0Rees: chromatic number max vertex degree}. Moreover, in Example~\ref{0Rees: example chromatic number} we will see that none of the methods is better than the other, in the sense that in some cases the smaller upper bound comes from the first method and in other cases it comes from the second method.

Although $\simplifiedgraph{I}{\Lambda}{P}$ was introduced to study $\commgraph{\Reeszero{G}{I}{\Lambda}{P}}$ in terms of connectedness and diameter (see Definition~\ref{0Rees: definition G(I,^,P)}), the next lemma shows that it can also be helpful in the determination of (an upper bound for) the chromatic number of $\commgraph{\Reeszero{G}{I}{\Lambda}{P}}$.

\begin{lemma}\label{0Rees: lemma chromatic number}
	Let $\mathcal{D}$ be a subgraph of $\simplifiedgraph{I}{\Lambda}{P}$ and let $D$ be its vertex set. Let $\mathcal{C}$ be the subgraph of $\commgraph{\Reeszero{G}{I}{\Lambda}{P}}$ induced by $\bigcup_{\parens{i,\lambda}\in D}\set{i}\times G\times\set{\lambda}$. Then $\chromaticnumber{\mathcal{C}}\leqslant\chromaticnumber{\mathcal{D}}\cdot\abs{G}$.
\end{lemma}

\begin{proof}
	Let $n=\chromaticnumber{\mathcal{D}}$. Then $\mathcal{D}$ is $n$-colourable, which implies that there exists a set $X$ of size $n$ and a map $\varphi:D\to X$ such that for all $x\in X$ the set $\set{x}\varphi^{-1}$ contains no adjacent vertices of $\mathcal{D}$.
	
	Let $\psi:\bigcup_{\parens{i,\lambda}\in D}\set{i}\times G\times\set{\lambda}\to X\times G$ be the map such that $\parens{i,g,\lambda}\psi=\parens{\parens{i,\lambda}\varphi,g}$ for all $\parens{i,\lambda}\in D$ and $g\in G$.
	
	We are going to see that for each $x\in X$ and $g\in G$ the set $\braces{\parens{x,g}}\psi^{-1}$ contains no adjacent vertices of $\mathcal{C}$. Let $x\in X$ and $g\in G$. Let $\parens{i,\lambda},\parens{j,\mu}\in D$ and $g_1,g_2\in G$ be such that $\parens{i,g_1,\lambda}\neq\parens{j,g_2,\mu}$ and $\parens{i,g_1,\lambda},\parens{j,g_2,\mu}\in\braces{\parens{x,g}}\psi^{-1}$. Then
	\begin{displaymath}
		\parens{\parens{i,\lambda}\varphi,g_1}=\parens{i,g_1,\lambda}\psi=\parens{x,g}=\parens{j,g_2,\mu}\psi=\parens{\parens{j,\mu}\varphi,g_2}
	\end{displaymath}
	and, consequently, $\parens{i,\lambda}\varphi=\parens{j,\mu}\varphi$ and $g_1=g_2$. The latter implies, together with the fact that $\parens{i,g_1,\lambda}\neq\parens{j,g_2,\mu}$, that $\parens{i,\lambda}\neq\parens{j,\mu}$. Moreover, since $\parens{i,\lambda}\varphi=\parens{j,\mu}\varphi$, we have that $\parens{i,\lambda},\parens{j,\mu}\in\set{\parens{i,\lambda}\varphi}\varphi^{-1}$ and, since $\set{\parens{i,\lambda}\varphi}\varphi^{-1}$ contains no adjacent vertices, then $\parens{i,\lambda}$ and $\parens{j,\mu}$ are (distinct) non-adjacent vertices of $\simplifiedgraph{I}{\Lambda}{P}$. Consequently, $p_{\lambda j}\neq 0$ or $p_{\mu i}\neq 0$ and it follows from Lemma~\ref{0Rees: commutativity}, together with the fact that $\parens{i,\lambda}\neq\parens{j,\mu}$, that $\parens{i,g_1,\lambda}$ and $\parens{j,g_2,\mu}$ are not adjacent (in $\commgraph{\Reeszero{G}{I}{\Lambda}{P}}$).
	
	This concludes the proof that $\set{\parens{x,g}}\psi^{-1}$ contains no adjacent vertices of $\mathcal{C}$ for all $x\in X$ and $g\in G$. Thus $\mathcal{C}$ is $\abs{X\times G}$-colourable and, consequently,
	\begin{displaymath}
		\chromaticnumber{\mathcal{C}}\leqslant \abs{X\times G}=\abs{X}\cdot\abs{G}=n\abs{G}=\chromaticnumber{\mathcal{D}}\cdot\abs{G}.\qedhere
	\end{displaymath}
\end{proof}

The key point of the first method (Theorem~\ref{0Rees: chromatic number edges}) for determining an upper bound for $\chromaticnumber{\commgraph{\Reeszero{G}{I}{\Lambda}{P}}}$ is the number of zero entries of each one of the (distinct) $0$-closure submatrices of $P$. Using this information we can easily determine the upper bound required.

\begin{theorem}\label{0Rees: chromatic number edges}
	\begin{enumerate}
		\item Suppose that $\commgraph{\Rees{G}{I}{\Lambda}{P}}$ is connected. Let $z$ be the number of zeros in $P$. Then
		\begin{displaymath}
			\chromaticnumber{\commgraph{\Rees{G}{I}{\Lambda}{P}}}\leqslant z|G|.
		\end{displaymath}
		
		\item Suppose that $\commgraph{\Rees{G}{I}{\Lambda}{P}}$ is not connected. Let $A_1,\ldots,A_n$ be all the $0$-closure submatrices of $P$ and let $z_1,\ldots,z_n$ be the number of zeros in $A_1,\ldots,A_n$, respectively. Then
		\begin{displaymath}
			\chromaticnumber{\commgraph{\Rees{G}{I}{\Lambda}{P}}} \leqslant \max\set{2,z_1,\ldots, z_n}|G|.
		\end{displaymath}
	\end{enumerate}
\end{theorem}

\begin{proof}
	Before demonstrating the statements of Theorem~\ref{0Rees: chromatic number edges} we are going to prove the following lemma.
	
	\begin{lemma}\label{0Rees: chromatic number edges lemma}
		Let $Q=\Psubmatrix{P}{I_Q}{\Lambda_Q}$ be a \zeroclosuresub\ and let $\mathcal{C}$ be the connected component of $\commgraph{\Reeszero{G}{I}{\Lambda}{P}}$ determined by $Q$. Let $z_Q$ be the number of zeros in $Q$. Then $\chromaticnumber{\mathcal{C}}\leqslant z_Q\abs{G}$.
	\end{lemma}
	
	\begin{proof}
		Let $\mathcal{D}$ be the subgraph of $\simplifiedgraph{I}{\Lambda}{P}$ induced by $I_Q\times\Lambda_Q$, which is a connected component of $\simplifiedgraph{I}{\Lambda}{P}$ (by part 1 of Lemma~\ref{0Rees: connected component G(I,^,P) G(M0[...])}). We notice that $Q$ contains at least one zero entry, which implies that $z_Q\geqslant 1$.
		
		We have that for all $i,j\in I_Q$ and $\lambda,\mu\in\Lambda_Q$
		\begin{align*}
			&\braces{\parens{i,\lambda},\parens{j,\mu}} \text{ is an edge of } \simplifiedgraph{I}{\Lambda}{P}\\
			\iff{}& \parens{i,\lambda}\neq\parens{j,\mu} \text{ and } p_{\lambda j}=p_{\mu i}=0\\
			\iff{}& \parens{j,\lambda}\neq\parens{i,\mu} \text{ and } p_{\lambda j}=p_{\mu i}=0.
		\end{align*}
		This implies that there is a bijection between edges of $\simplifiedgraph{I}{\Lambda}{P}$ and pairs of distinct zero entries of $P$. Hence the number of edges of $\simplifiedgraph{I}{\Lambda}{P}$ is equal to $\dbinom{z_Q}{2}$ and, by Lemma~\ref{preli: 0Rees chromatic number}, we have
		\begin{align*}
			&\chromaticnumber{\mathcal{D}}\parens{\chromaticnumber{\mathcal{D}}-1}\leqslant 2\dbinom{z_Q}{2}\\
			\iff{}& \parens{\chromaticnumber{\mathcal{D}}-1/2}^2-1/4\leqslant z_Q\parens{z_Q-1}\\
			\iff{}& \parens{\chromaticnumber{\mathcal{D}}-1/2}^2\leqslant z_Q\parens{z_Q-1}+1/4\\
			\iff{}& \parens{\chromaticnumber{\mathcal{D}}-1/2}^2\leqslant \parens{z_Q-1/2}^2\\
			\iff{}& \chromaticnumber{\mathcal{D}}-1/2\leqslant z_Q-1/2& \bracks{\text{because } \chromaticnumber{\mathcal{D}}\geqslant 1 \text{ and } z_Q\geqslant 1}\\
			\iff{}& \chromaticnumber{\mathcal{D}}\leqslant z_Q.
		\end{align*}
		
		Additionally, we have that the vertex set of $\mathcal{C}$ is
		\begin{displaymath}
			I_Q\times G\times\Lambda_Q=\bigcup_{\parens{i,\lambda}\in I_Q\times\Lambda_Q}\set{i}\times G\times\set{\lambda},
		\end{displaymath}
		where $I_Q\times\Lambda_Q$ is the vertex set of $\mathcal{D}$. Therefore, by Lemma~\ref{0Rees: lemma chromatic number}, we have that $\chromaticnumber{\mathcal{C}}\leqslant\chromaticnumber{\mathcal{D}}\cdot\abs{G}\leqslant z_Q\abs{G}$.
	\end{proof}
	
	Now we prove part 1 and part 2 of Theorem~\ref{0Rees: chromatic number edges}.
	
	\medskip
	
	\textbf{Part 1.} Suppose that $\commgraph{\Reeszero{G}{I}{\Lambda}{P}}$ is connected. It follows from Theorem~\ref{0Rees: connectedness} that $P=\Psubmatrix{P}{I}{\Lambda}$ is the unique \zeroclosuresub. Moreover, the connected component of $\commgraph{\Reeszero{G}{I}{\Lambda}{P}}$ determined by $P$ (whose vertex set is $I\times G\times\Lambda$) is $\commgraph{\Reeszero{G}{I}{\Lambda}{P}}$ itself. Hence it follows from Lemma~\ref{0Rees: chromatic number edges lemma} that $\chromaticnumber{\Reeszero{G}{I}{\Lambda}{P}}\leqslant z\abs{G}$.
	
	\medskip
	
	\textbf{Part 2:} Suppose that $\commgraph{\Reeszero{G}{I}{\Lambda}{P}}$ is not connected. Let $\mathcal{C}$ be a connected component of $\commgraph{\Reeszero{G}{I}{\Lambda}{P}}$. Then we have three possibilities for $\mathcal{C}$:
	
	\smallskip
	
	\textit{Case 1:} Suppose that $\mathcal{C}$ is the connected component determined by $A_l$, for some $l\in\Xn$. Then Lemma~\ref{0Rees: chromatic number edges lemma} ensures that
	\begin{displaymath}
		\chromaticnumber{\mathcal{C}}\leqslant z_l\abs{G}\leqslant\max\set{2,z_1,\ldots,z_n}\abs{G}.
	\end{displaymath}
	
	\smallskip
	
	\textit{Case 2:} Suppose that $\mathcal{C}$ is the connected component determined by $A_l$ and $A_m$, for some distinct $l,m\in\Xn$. Assume that $A_l=\Psubmatrix{P}{I_l}{\Lambda_l}$ and $A_m=\Psubmatrix{P}{I_m}{\Lambda_m}$. Let $\mathcal{D}$ be the subgraph of $\simplifiedgraph{I}{\Lambda}{P}$ induced by $\parens{I_l\times\Lambda_m} \cup \parens{I_m\times\Lambda_l}$. Due to Lemma~\ref{0Rees: lemma connected component determined Q M} we have that the vertices of $\mathcal{D}$ that belong to $I_l\times\Lambda_m$ (respectively, $I_m\times\Lambda_l$) are only adjacent to vertices that belong to $I_m\times\Lambda_l$ (respectively, $I_l\times\Lambda_m$). Consequently, if we associate one colour to the vertices that belong to $I_l\times\Lambda_m$ and another colour to the vertices that belong to $I_m\times\Lambda_l$, then $\mathcal{D}$ will not have adjacent vertices with the same colour. Thus $\chromaticnumber{\mathcal{D}}\leqslant 2$.
	
	Furthermore, the vertex set of $\mathcal{C}$ is
	\begin{displaymath}
		\parens{I_l\times G\times\Lambda_m}\cup\parens{I_m\times G\times\Lambda_l}=\bigcup_{\parens{i,\lambda}\in \parens{I_l\times\Lambda_m}\cup\parens{I_m\times\Lambda_l}}\set{i}\times G\times\set{\lambda},
	\end{displaymath}
	where $\parens{I_l\times\Lambda_m}\cup\parens{I_m\times\Lambda_l}$ is the vertex set of $\mathcal{D}$. Therefore, by Lemma~\ref{0Rees: lemma chromatic number}, we have
	\begin{displaymath}
		\chromaticnumber{\mathcal{C}}\leqslant\chromaticnumber{\mathcal{D}}\cdot\abs{G}\leqslant 2\abs{G}\leqslant\max\set{2,z_1,\ldots,z_n}\abs{G}.
	\end{displaymath}
	
	\smallskip
	
	\textit{Case 3:} Suppose that $\mathcal{C}$ is the connected component determined by $\parens{\lambda, i}$, for some $i\in I$ and $\lambda\in\Lambda$ such that column $i$ has no zero entries or row $\lambda$ has no zero entries. Then the set of vertices of $\mathcal{C}$ is $\set{i}\times G\times\set{\lambda}$ and we have
	\begin{displaymath}
		\chromaticnumber{\mathcal{C}}\leqslant\abs{\set{i}\times G\times\set{\lambda}}=\abs{G}\leqslant\max\set{2,z_1,\ldots,z_n}\abs{G}.
	\end{displaymath}
	
	Therefore cases 1, 2 and 3 allow us to conclude that
	\begin{align*}
		\MoveEqLeft\chromaticnumber{\commgraph{\Reeszero{G}{I}{\Lambda}{P}}}\\
		&=\max\gset{\chromaticnumber{\mathcal{C}}}{\mathcal{C} \text{ is a connected component of } \commgraph{\Reeszero{G}{I}{\Lambda}{P}}}\\
		&\leqslant \max\set{2,z_1,\ldots,z_n}\abs{G}.\qedhere
	\end{align*}
\end{proof}

We note that, when $\commgraph{\Reeszero{G}{I}{\Lambda}{P}}$ is not connected and $P$ contains only one \zeroclosuresub, then $\commgraph{\Reeszero{G}{I}{\Lambda}{P}}$ contains only one connected component determined by a \zeroclosuresub. Consequently, $\commgraph{\Reeszero{G}{I}{\Lambda}{P}}$ contains no connected components determined by (distinct) $0$-closure submatrices of $P$. This implies that case 2 of part 2 of the proof does not exist in this situation. Thus, if $z$ is the number of zeros in the unique \zeroclosuresub, then $z\abs{G}$ is an upper bound for $\chromaticnumber{
	\commgraph{\Reeszero{G}{I}{\Lambda}{P}}}$ (which is better than $\max\set{2,z}\abs{G}$, the upper bound provided by the statement of Theorem~\ref{0Rees: chromatic number edges}).

In order to provide a second method to determine an upper bound for the chromatic number of $\commgraph{\Reeszero{G}{I}{\Lambda}{P}}$, we need to prove two lemmata: Lemma~\ref{0Rees: chromatic number odd cycle lemma} characterizes the $0$-closure submatrices of $P$ that determine the connected components of $\commgraph{\Reeszero{G}{I}{\Lambda}{P}}$ that are odd cycles; and Lemma~\ref{0Rees: degree in G(I,^)} shows how to calculate the degree of a vertex of $\simplifiedgraph{I}{\Lambda}{P}$.

\begin{lemma}\label{0Rees: chromatic number odd cycle lemma}
	Let $Q=\Psubmatrix{P}{I_Q}{\Lambda_Q}$ be a \zeroclosuresub. If the subgraph of $\simplifiedgraph{I}{\Lambda}{P}$ induced by $I_Q\times\Lambda_Q$ is an odd cycle, then $Q=O_{\abs{\Lambda_Q}\times\abs{I_Q}}$.
\end{lemma}

\begin{proof}
	Let $\mathcal{C}$ be the subgraph of $\simplifiedgraph{I}{\Lambda}{P}$ induced by $I_Q\times\Lambda_Q$ (which, by Lemma~\ref{0Rees: connected component G(I,^,P) G(M0[...])}, is a connected component of $\simplifiedgraph{I}{\Lambda}{P}$) and assume that $\mathcal{C}$ is an odd cycle. Then $\mathcal{C}$ has an odd number of vertices, that is, $\abs{I_Q\times\Lambda_Q}=\abs{I_Q}\cdot\abs{\Lambda_Q}$ is odd, which implies that $\abs{ I_Q}$ and $\abs{\Lambda_Q }$ are odd.
	
	
	Assume, with the aim of obtaining a contradiction, that $Q$ is regular. It is easy to verify that $\mathcal{C}$ is isomorphic to $\commgraph{\Reeszero{\set{1_G}}{I_Q}{\Lambda_Q}{Q}}$. Consequently, $\commgraph{\Reeszero{\set{1_G}}{I_Q}{\Lambda_Q}{Q}}$ is also an odd cycle. Moreover, by Theorem~\ref{0Rees: girth}, we have that $\girth{\commgraph{\Reeszero{\set{1_G}}{I_Q}{\Lambda_Q}{Q}}}\in\set{3,4}$. Thus $\commgraph{\Reeszero{\set{1_G}}{I_Q}{\Lambda_Q}{Q}}$ (and, consequently, $\mathcal{C}$) must be a cycle of length $3$. Hence $\abs{I_Q}\cdot\abs{\Lambda_Q}=\abs{I_Q\times \Lambda_Q}=3$, which implies that $\abs{I_Q}=1$ or $\abs{\Lambda_Q}=1$. In addition, Lemma~\ref{0Rees: every row/column of 0-closure submatrix has a 0}, guarantees that each row and each column of $Q$ contains at least one zero entry. Therefore $Q=O_{\abs{\Lambda_Q}\times 1}$ or $Q=O_{1\times \abs{I_Q}}$. This implies that $Q$ is not regular, which is a contradiction.
	
	Thus $Q$ is not regular and, consequently, $Q$ contains a row of zeros or $Q$ contains a column of zeros. Assume, without loss of generality, that $Q$ contains a row of zeros. Let $\lambda\in\Lambda_Q$ be the index of that row.
	
	\smallskip
	
	\textit{Case 1:} Assume that $\abs{I_Q}\geqslant 3$. Then there exist distinct $i_1,i_2,i_3\in I_Q$. We have $p_{\lambda i_1}=p_{\lambda i_2}=p_{\lambda i_3}=0$, which implies that $\parens{i_1,\lambda}$, $\parens{i_2,\lambda}$ and $\parens{i_3,\lambda}$ are adjacent to each other (in $\simplifiedgraph{I}{\Lambda}{P}$). Hence
	\begin{displaymath}
		\parens{i_1,\lambda}-\parens{i_2,\lambda}-\parens{i_3,\lambda}-\parens{i_1,\lambda}
	\end{displaymath}
	is a cycle of length $3$ in $\mathcal{C}$. Thus $\mathcal{C}$ must be a cycle of length $3$ and, consequently, $\abs{I_Q}\cdot\abs{\Lambda_Q}=\abs{I_Q\times\Lambda_Q}=3$. Since $\abs{I_Q}\geqslant 3$, then $\abs{I_Q}=3$ and $\abs{\Lambda_Q}=1$. Thus $\Lambda_Q=\set{\lambda}$ and $Q=\Psubmatrix{P}{I_Q}{\Lambda_Q}=\Psubmatrix{P}{I_Q}{\lambda}=O_{1\times\abs{I_Q}}=O_{\abs{\Lambda_Q}\times\abs{I_Q}}$.
	
	\smallskip
	
	\textit{Case 2:} Assume that $\abs{I_Q}<3$. Due to the fact that $\abs{I_Q}$ is odd, we have $\abs{I_Q}=1$. Additionally, Lemma~\ref{0Rees: every row/column of 0-closure submatrix has a 0} ensures that all the rows of $Q$ contain at least one zero entry, which implies that $Q=O_{\abs{\Lambda_Q}\times 1}=O_{\abs{\Lambda_Q}\times\abs{I_Q}}$.
\end{proof}

\begin{lemma}\label{0Rees: degree in G(I,^)}
	Let $i\in I$ and $\lambda\in\Lambda$. Let $c_i$ (respectively, $r_\lambda$) be the number of zeros in column $i$ (respectively, row $\lambda$) of $P$. Then
	\begin{displaymath}
		\degree{\simplifiedgraph{I}{\Lambda}{P}}{\parens{i,\lambda}}=\begin{cases}
			c_ir_\lambda-1& \text{if } p_{\lambda i}=0,\\
			c_ir_\lambda& \text{if } p_{\lambda i}\in G.
		\end{cases}
	\end{displaymath}
\end{lemma}

\begin{proof}
	Let $C_i=\gset{\mu\in\Lambda}{p_{\mu i}=0}$ and $R_\lambda=\gset{j\in I}{p_{\lambda j}=0}$. Then $\abs{C_i}=c_i$ and $\abs{R_\lambda}=r_\lambda$. For all $j\in I$ and $\mu\in\Lambda$ we have
	\begin{align*}
		&\parens{j,\mu} \text{ is adjacent to } \parens{i,\lambda}\\
		\iff{}& \parens{j,\mu}\neq\parens{i,\lambda} \text{ and } p_{\mu i}=p_{\lambda j}=0\\
		\iff{}& \parens{j,\mu}\neq\parens{i,\lambda} \text{ and } \mu\in C_i \text{ and } j\in R_\lambda\\
		\iff{}& \parens{j,\mu}\in \parens{R_\lambda\times C_i}\setminus\set{\parens{i,\lambda}}.
	\end{align*}
	This implies that $\degree{\simplifiedgraph{I}{\Lambda}{P}}{\parens{i,\lambda}}=\abs{\parens{R_\lambda\times C_i}\setminus\set{\parens{i,\lambda}}}$.
	
	\smallskip
	
	\textit{Case 1:} Assume that $p_{\lambda i}=0$. Then $\parens{i,\lambda}\in R_\lambda\times C_i$ and, consequently,
	\begin{displaymath}
		\degree{\simplifiedgraph{I}{\Lambda}{P}}{\parens{i,\lambda}}=\abs{\parens{R_\lambda\times C_i}\setminus\set{\parens{i,\lambda}}}=\abs{R_\lambda\times C_i}-1=\abs{R_\lambda}\cdot\abs{C_i}-1=r_\lambda c_i-1.
	\end{displaymath}
	
	\smallskip
	
	\textit{Case 2:} Assume that $p_{\lambda i}\in G$ (that is, $p_{\lambda i}\neq 0$). Then $\parens{i,\lambda}\notin R_\lambda\times C_i$ and, consequently,
	\begin{displaymath}
		\degree{\simplifiedgraph{I}{\Lambda}{P}}{\parens{i,\lambda}}=\abs{\parens{R_\lambda\times C_i}\setminus\set{\parens{i,\lambda}}}=\abs{R_\lambda\times C_i}=\abs{R_\lambda}\cdot\abs{C_i}=r_\lambda c_i.\qedhere
	\end{displaymath}
\end{proof}

\begin{theorem}\label{0Rees: chromatic number max vertex degree}
	\begin{enumerate}
		\item Suppose that $\commgraph{\Rees{G}{I}{\Lambda}{P}}$ is connected. Let $c$ (respectively, $r$) be the maximum number of zeros in a column (respectively, row) of $P$. Then
		\begin{enumerate}
			\item If there exist $i\in I$ and $\lambda\in\Lambda$ such that $p_{\lambda i}\in G$, column $i$ of $P$ has $c$ zeros and row $\lambda$ of $P$ has $r$ zeros, then
			\begin{displaymath}
				\chromaticnumber{\commgraph{\Rees{G}{I}{\Lambda}{P}}}\leqslant cr\abs{G}.
			\end{displaymath}
			
			\item Otherwise,
			\begin{displaymath}
				\chromaticnumber{\commgraph{\Rees{G}{I}{\Lambda}{P}}}\leqslant \parens{cr-1}\abs{G}.
			\end{displaymath}
		\end{enumerate}
		
		\item Suppose that $\commgraph{\Rees{G}{I}{\Lambda}{P}}$ is not connected. Let $A_1,\ldots,A_n$ be all the $0$-closure submatrices of $P$ and assume that for each $l\in\Xn$ we have $A_l=\Psubmatrix{P}{I_l}{\Lambda_l}$. For each $l\in\Xn$ let $c_l$ (respectively, $r_l$) be the maximum number of zeros in a column (respectively, row) of $A_l$, and let $z_l$ be defined as follows:
		\begin{enumerate}
			\item If $A_l= O_{\abs{\Lambda_l}\times\abs{I_l}}$ or there exist $i\in I_l$ and $\lambda\in\Lambda_l$ such that $p_{\lambda i}\in G$, column $i$ of $A_l$ has $c_l$ zeros and row $\lambda$ of $A_l$ has $r_l$ zeros, then let $z_l=c_lr_l$.
			
			\item Otherwise, let $z_l=\parens{c_lr_l-1}$.
		\end{enumerate}
		Then
		\begin{displaymath}
			\chromaticnumber{\commgraph{\Rees{G}{I}{\Lambda}{P}}}\leqslant  \max\set{2,z_1,\ldots, z_n}\abs{G}.
		\end{displaymath}
	\end{enumerate}
\end{theorem}

\begin{proof}
	We begin by proving the following lemma.
	
	\begin{lemma}\label{0Rees: chromatic number max vertex degree lemma}
		Let $Q=\Psubmatrix{P}{I_Q}{\Lambda_Q}$ be a \zeroclosuresub\ and let $\mathcal{C}$ be the connected component of $\commgraph{\Reeszero{G}{I}{\Lambda}{P}}$ determined by $Q$. Let $c_Q$ (respectively, $r_Q$) be the maximum number of zeros in a column (respectively, row) of $Q$. Then
		\begin{enumerate}
			\item If $Q= O_{\abs{\Lambda_Q}\times\abs{I_Q}}$ or there exist $i\in I_Q$ and $\lambda\in\Lambda_Q$ such that $p_{\lambda i}\in G$, column $i$ of $Q$ has $c_Q$ zeros and row $\lambda$ of $Q$ has $r_Q$ zeros, then $\chromaticnumber{\mathcal{C}}\leqslant c_Qr_Q\abs{G}$.
			
			\item Otherwise, $\chromaticnumber{\mathcal{C}}\leqslant \parens{c_Qr_Q-1}\abs{G}$.
		\end{enumerate}
	\end{lemma}
	
	\begin{proof}
		Let $\mathcal{D}$ be the subgraph of $\simplifiedgraph{I}{\Lambda}{P}$ induced by $I_Q\times\Lambda_Q$. For each $i\in I_Q$ (respectively, $\lambda\in\Lambda_Q$) let $c_i$ (respectively, $r_\lambda$) be the number of zeros in column $i$ (respectively, row $\lambda$) of $Q$. Then $c_Q=\max\gset{c_i}{i\in I_Q}$ and $r_Q=\max\gset{r_\lambda}{\lambda\in\Lambda_Q}$.
		
		We note that, as a consequence of the fact that $Q$ is a \zeroclosuresub, then there are no zero entries in the rows and columns of $P$ intersecting $Q$ that are not in $Q$. This implies that for all $i\in I_Q$ (respectively, $\lambda\in\Lambda_Q$) the number of zeros of column $i$ (respectively, row $\lambda$) of $Q$ is equal to the number of zeros of column $i$ (respectively, row $\lambda$) of $P$. Thus for all $i\in I_Q$ (respectively, $\lambda\in\Lambda_Q$) we have that $c_i$ (respectively, $r_\lambda$) is equal to the number of zeros in column $i$ (respectively, row $\lambda$) of $P$.
		
		\smallskip
		
		\textit{Case 1:} Assume that $Q=O_{\abs{\Lambda_Q}\times\abs{I_Q}}$. Then each row (respectively, each column) of $Q$ has $\abs{I_Q}$ (respectively, $\abs{\Lambda_Q}$) zero entries and, consequently, $r_Q=\abs{I_Q}$ and $c_Q=\abs{\Lambda_Q}$. Furthermore, it follows from Lemma~\ref{0Rees: diameter 0 or 1 or 2} that $\diam{\mathcal{C}}\in\set{0,1}$, which implies that $\mathcal{C}$ is isomorphic to $K_{\abs{I_Q\times G\times\Lambda_Q}}$. Consequently,
		\begin{displaymath}
			\chromaticnumber{\mathcal{C}}=\abs{I_Q\times G\times \Lambda_Q}=\abs{I_Q}\cdot\abs{G}\cdot\abs{\Lambda_Q}=r_Q\abs{G}c_Q=c_Qr_Q\abs{G}.
		\end{displaymath}
		
		\smallskip
		
		\textit{Case 2:} Assume that $Q\neq O_{\abs{\Lambda_Q}\times\abs{I_Q}}$. Then there exist $j\in I_Q$ and $\mu\in\Lambda_Q$ such that $p_{\mu j}\neq 0$. Since $Q$ is a \zeroclosuresub, then Lemma~\ref{0Rees: every row/column of 0-closure submatrix has a 0} ensures that row $\mu$ and column $j$ of $Q$ both contain at least one zero entry. Consequently, there exist $j'\in I_Q$ and $\mu'\in\Lambda_Q$ such that $p_{\mu j'}=p_{\mu' j}=0$. As a consequence of the fact that $p_{\mu j}\neq 0$, then we have that $\parens{j,\mu'}$ and $\parens{j',\mu}$ are not adjacent (in $\simplifiedgraph{I}{\Lambda}{P}$), which implies that $\mathcal{D}$ is not a complete graph. Furthermore, it follows from part 1 of Lemma~\ref{0Rees: connected component G(I,^,P) G(M0[...])} that $\mathcal{D}$ is connected, and it follows from Lemma~\ref{0Rees: chromatic number odd cycle lemma} and the fact that $Q\neq O_{\abs{\Lambda_Q}\times\abs{I_Q}}$ that $\mathcal{D}$ is not an odd cycle. Then we can use Brooks' Theorem (Theorem~\ref{preli: ORees Brooks}) to conclude that $\chromaticnumber{\mathcal{D}}\leqslant\maxdegree{\mathcal{D}}$. Since the vertex set of $\mathcal{C}$ is
		\begin{displaymath}
			I_Q\times G\times\Lambda_Q=\bigcup_{\parens{i,\lambda}\in I_Q\times\Lambda_Q}\set{i}\times G\times\set{\lambda}
		\end{displaymath}
		and $I_Q\times\Lambda_Q$ is the vertex set of $\mathcal{D}$, then we can use Lemma~\ref{0Rees: lemma chromatic number} to see that
		\begin{displaymath}
			\chromaticnumber{\mathcal{C}}\leqslant \chromaticnumber{\mathcal{D}}\cdot\abs{G}\leqslant\maxdegree{\mathcal{D}}\cdot\abs{G}.
		\end{displaymath}
		In order to finish this case we only need to determine $\maxdegree{\mathcal{D}}$, which is done in the following two sub-cases.
		
		\smallskip
		
		\textsc{Sub-case 1:} Assume that there exist $i\in I_Q$ and $\lambda\in\Lambda_Q$ such that $p_{\lambda i}\in G$, $c_i=c_Q$ and $r_\lambda=r_Q$. We are going to verify that $\maxdegree{\mathcal{D}}=c_Qr_Q$. Let $i'\in I_Q$ and $\lambda'\in\Lambda_Q$. By Lemma~\ref{0Rees: degree in G(I,^)} we have
		\begin{displaymath}
			\degree{\simplifiedgraph{I}{\Lambda}{P}}{\parens{i',\lambda'}}\leqslant c_{i'}r_{\lambda'}\leqslant c_Qr_Q=c_ir_\lambda=\degree{\simplifiedgraph{I}{\Lambda}{P}}{\parens{i,\lambda}}.
		\end{displaymath}
		Furthermore, since $\mathcal{D}$ is a connected component of $\simplifiedgraph{I}{\Lambda}{P}$, we have that $\parens{i',\lambda'}$ and $\parens{i,\lambda}$ are only adjacent to vertices of $\mathcal{D}$. Hence $\degree{\simplifiedgraph{I}{\Lambda}{P}}{\parens{i',\lambda'}}=\degree{\mathcal{D}}{\parens{i',\lambda'}}$ and $\degree{\simplifiedgraph{I}{\Lambda}{P}}{\parens{i,\lambda}}=\degree{\mathcal{D}}{\parens{i,\lambda}}$ and, consequently, $\degree{\mathcal{D}}{\parens{i',\lambda'}}\leqslant\degree{\mathcal{D}}{\parens{i,\lambda}}$.
		
		Due to the fact that $i'$ and $\lambda'$ are arbitrary elements of $I_Q$ and $\Lambda_Q$, we can conclude that $\maxdegree{\mathcal{D}}=\degree{\mathcal{D}}{\parens{i,\lambda}}=c_Qr_Q$ and, consequently, $\chromaticnumber{\mathcal{C}}\leqslant c_Qr_Q\abs{G}$.
		
		\smallskip
		
		\textsc{Sub-case 2:} Assume that for all $i\in I_Q$ and $\lambda\in\Lambda_Q$ we have $p_{\lambda i}=0$ or $c_i<c_Q$ or $r_\lambda<r_Q$. We intend to show that $\maxdegree{\mathcal{D}}=c_Qr_Q-1$. Let $i\in I_Q$ and $\lambda\in\Lambda_Q$ be such that $c_i=c_Q$ and $r_\lambda=r_Q$. Then we must have $p_{\lambda i}=0$. Let $i'\in I_Q$ and $\lambda'\in\Lambda_Q$.
		
		If $p_{\lambda' i'}=0$, then Lemma~\ref{0Rees: degree in G(I,^)} implies that
		\begin{displaymath}
			\degree{\simplifiedgraph{I}{\Lambda}{P}}{\parens{i',\lambda'}}=c_{i'}r_{\lambda'}-1\leqslant c_Qr_Q-1=c_ir_\lambda-1=\degree{\simplifiedgraph{I}{\Lambda}{P}}{\parens{i,\lambda}}.
		\end{displaymath}
		
		If $p_{\lambda' i'}\in G$, then we must have $c_{i'}<c_Q$ or $r_{\lambda'}<r_Q$. Assume, without loss of generality, that $c_{i'}<c_Q$. Due to the fact that $r_Q\geqslant 1$ (because, by Lemma~\ref{0Rees: every row/column of 0-closure submatrix has a 0}, each row of $Q$ contains at least one zero entry), and due to Lemma~\ref{0Rees: degree in G(I,^)}, we have that
		\begin{align*}
			\degree{\simplifiedgraph{I}{\Lambda}{P}}{\parens{i',\lambda'}}&=c_{i'}r_{\lambda'}\leqslant\parens{c_Q-1}r_Q=c_Qr_Q-r_Q\leqslant c_Qr_Q-1=c_ir_\lambda-1\\
			&= \degree{\simplifiedgraph{I}{\Lambda}{P}}{\parens{i,\lambda}}.
		\end{align*}
		
		As a consequence of the fact that $\mathcal{D}$ is a connected component of $\simplifiedgraph{I}{\Lambda}{P}$, we have that $\degree{\simplifiedgraph{I}{\Lambda}{P}}{\parens{i',\lambda'}}=\degree{\mathcal{D}}{\parens{i',\lambda'}}$ and $\degree{\simplifiedgraph{I}{\Lambda}{P}}{\parens{i,\lambda}}=\degree{\mathcal{D}}{\parens{i,\lambda}}$. Therefore we have $\degree{\mathcal{D}}{\parens{i',\lambda'}}\leqslant\degree{\mathcal{D}}{\parens{i,\lambda}}$.
		Since $i'$ and $\lambda'$ are arbitrary elements of $I_Q$ and $\Lambda_Q$, we can conclude that $\maxdegree{\mathcal{D}}=\degree{\mathcal{D}}{\parens{i,\lambda}}=c_Qr_Q-1$ and, consequently, $\chromaticnumber{\mathcal{C}}\leqslant \parens{c_Qr_Q-1}\abs{G}$. 
	\end{proof}
	
	Now we prove statements 1 and 2 of Theorem~\ref{0Rees: chromatic number max vertex degree}.
	
	\medskip
	
	\textbf{Part 1.} Suppose that $\commgraph{\Reeszero{G}{I}{\Lambda}{P}}$ is connected. By Theorem~\ref{0Rees: connectedness} we have that $P=\Psubmatrix{P}{I}{\Lambda}$ is a \zeroclosuresub\ and we have that the connected component of $\commgraph{\Reeszero{G}{I}{\Lambda}{P}}$ determined by $P$ (whose vertex set is $I\times G\times\Lambda$) is $\commgraph{\Reeszero{G}{I}{\Lambda}{P}}$ itself. Furthermore, $P$ is regular, which implies that $P\neq O_{\abs{\Lambda}\times\abs{I}}$. Thus Lemma~\ref{0Rees: chromatic number max vertex degree lemma} implies that that:
	\begin{enumerate}
		\item[a)] If there exist $i\in I$ and $\lambda\in\Lambda$ such that $p_{\lambda i}\in G$, column $i$ of $P$ has $c$ zeros and row $\lambda$ of $P$ has $r$ zeros, then $\chromaticnumber{\Reeszero{G}{I}{\Lambda}{P}}\leqslant cr\abs{G}$.
		
		\item[b)] Otherwise, $\chromaticnumber{\Reeszero{G}{I}{\Lambda}{P}}\leqslant \parens{cr-1}\abs{G}$.
	\end{enumerate}
	
	\medskip
	
	\textbf{Part 2:} Suppose that $\commgraph{\Reeszero{G}{I}{\Lambda}{P}}$ is not connected. Let $\mathcal{C}$ be a connected component of $\commgraph{\Reeszero{G}{I}{\Lambda}{P}}$. Then we have three possibilities for $\mathcal{C}$:
	
	\smallskip
	
	\textit{Case 1:} Suppose that $\mathcal{C}$ is the connected component determined by $A_l$, for some $l\in\Xn$. Then, by Lemma~\ref{0Rees: chromatic number max vertex degree lemma}, we have that:
	\begin{enumerate}
		\item[a)] If $A_l= O_{\abs{\Lambda_l}\times\abs{I_l}}$ or there exist $i\in I_l$ and $\lambda\in\Lambda_l$ such that $p_{\lambda i}\in G$, column $i$ of $A_l$ has $c_l$ zeros and row $\lambda$ of $A_l$ has $r_l$ zeros, then
		\begin{displaymath}
			\chromaticnumber{\mathcal{C}}\leqslant c_lr_l\abs{G}=z_l\abs{G}\leqslant\max\set{2,z_1,\ldots,z_n}\abs{G}.
		\end{displaymath}
		
		\item[b)] Otherwise,
		\begin{displaymath}
			\chromaticnumber{\mathcal{C}}\leqslant \parens{c_lr_l-1}\abs{G}=z_l\abs{G}\leqslant\max\set{2,z_1,\ldots,z_n}\abs{G}.
		\end{displaymath}
	\end{enumerate}
	
	\smallskip
	
	\textit{Case 2:} Suppose that $\mathcal{C}$ is the connected component determined by $A_l$ and $A_m$, for some distinct $l,m\in\Xn$, and let $\mathcal{D}$ be the subgraph of $\simplifiedgraph{I}{\Lambda}{P}$ induced by $\parens{I_l\times\Lambda_m} \cup \parens{I_m\times\Lambda_l}$. By Lemma~\ref{0Rees: lemma connected component determined Q M}, we have that the vertices of $\mathcal{D}$ that belong to $I_l\times\Lambda_m$ (respectively, $I_m\times\Lambda_l$) are only adjacent to vertices that belong to $I_m\times\Lambda_l$ (respectively, $I_l\times\Lambda_m$). Therefore $\chromaticnumber{\mathcal{D}}\leqslant 2$.
	
	Due to the fact that the vertex set of $\mathcal{C}$ is
	\begin{displaymath}
		\parens{I_l\times G\times\Lambda_m}\cup\parens{I_m\times G\times\Lambda_l}=\bigcup_{\parens{i,\lambda}\in \parens{I_l\times\Lambda_m}\cup\parens{I_m\times\Lambda_l}}\set{i}\times G\times\set{\lambda}
	\end{displaymath}
	and the vertex set of $\mathcal{D}$ is $\parens{I_l\times\Lambda_m}\cup\parens{I_m\times\Lambda_l}$, then we can use Lemma~\ref{0Rees: lemma chromatic number} to conclude that
	\begin{displaymath}
		\chromaticnumber{\mathcal{C}}\leqslant\chromaticnumber{\mathcal{D}}\cdot\abs{G}\leqslant 2\abs{G} \leqslant\max\set{2,z_1,\ldots,z_n}\abs{G}.
	\end{displaymath}
	
	\smallskip
	
	\textit{Case 3:} Suppose that $\mathcal{C}$ is the connected component determined by $\parens{\lambda,i}$, for some $i\in I$ and $\lambda\in\Lambda$ such that column $i$ has no zero entries or row $\lambda$ has no zero entries. Then the vertex set of $\mathcal{C}$ is $\set{i}\times G\times\set{\lambda}$ and we have
	\begin{displaymath}
		\chromaticnumber{\mathcal{C}}\leqslant\abs{\set{i}\times G\times\set{\lambda}}=\abs{G}\leqslant\max\set{2,z_1,\ldots,z_n}\abs{G}.
	\end{displaymath}
	
	It follows from cases 1, 2 and 3 that
	\begin{align*}
		\MoveEqLeft\chromaticnumber{\commgraph{\Reeszero{G}{I}{\Lambda}{P}}}\\
		&=\max\gset{\chromaticnumber{\mathcal{C}}}{\mathcal{C} \text{ is a connected component of } \commgraph{\Reeszero{G}{I}{\Lambda}{P}}}\\
		&\leqslant \max\set{2,z_1,\ldots,z_n}\abs{G}.\qedhere
	\end{align*}
\end{proof}

When $\commgraph{\Reeszero{G}{I}{\Lambda}{P}}$ is not connected and $P$ contains only one $0$-closure submatrix $A_1$ of $P$, then the previous theorem states that $\max\set{2,z_1}\abs{G}$ is an upper bound for $\chromaticnumber{\commgraph{\Reeszero{G}{I}{\Lambda}{P}}}$. However, since $P$ contains only one \zeroclosuresub, then this implies that $\commgraph{\Reeszero{G}{I}{\Lambda}{P}}$ does not contain multiple connected components determined by (distinct) $0$-closure submatrices of $P$. Consequently, case 2 of part 2 of the proof of Theorem~\ref{0Rees: chromatic number max vertex degree} is not a possibility. Thus we can simplify the upper bound of $\chromaticnumber{\commgraph{\Reeszero{G}{I}{\Lambda}{P}}}$ to $z_1\abs{G}$.

\begin{example}\label{0Rees: example chromatic number}
	The aim of this example is to demonstrate that sometimes Theorem~\ref{0Rees: chromatic number edges} provides a smaller upper bound for the chromatic number than \ref{0Rees: chromatic number max vertex degree}), and sometimes it provides a greater upper bound.
	
	Let $G$ be a group, let $I=\set{1,2,3,4}$ and $\Lambda=\set{1,2,3}$, and let $P$ and $P'$ be $\Lambda\times I$ matrices such that
	\begin{displaymath}
		\timeszero{P}=\begin{bNiceMatrix}[first-row,first-col]
			&1&2&3&4\\
			1&0&0&\times&\times\\
			2&\times&0&0&\times\\
			3&\times&\times&0&\times
		\end{bNiceMatrix}\quad
		\timeszero{P'}=\begin{bNiceMatrix}[first-row,first-col]
			&1&2&3&4\\
			1&0&0&0&\times\\
			2&0&\times&\times&\times\\
			3&\times&\times&\times&\times
		\end{bNiceMatrix}.
	\end{displaymath}
	
	It is easy to see that both $P$ and $P'$ are regular. Additionally, we can easily see that the unique \zeroclosuresub\ is $\Psubmatrix{P}{\set{1,2,3}}{\set{1,2,3}}$ and the unique $0$-closure submatrix of $P'$ is $\Psubmatrix{P'}{\set{1,2,3}}{\set{1,2}}$.
	
	Since $P$ has $5$ zero entries, then Theorem~\ref{0Rees: chromatic number edges} implies that
	\begin{displaymath}
		\chromaticnumber{\commgraph{\Reeszero{G}{I}{\Lambda}{P}}}\leqslant \max\set{2,5}= 5\abs{G}.
	\end{displaymath}
	Furthermore, it is easy to verify that the maximum number of zeros of a row (respectively, column) of $\Psubmatrix{P}{\set{1,2,3}}{\set{1,2,3}}$ is $2$. We have that $p_{1 3}\in G$ and both row $1$ and column $3$ of $\Psubmatrix{P}{\set{1,2,3}}{\set{1,2,3}}$ have exactly two zero entries. Then, by Theorem~\ref{0Rees: chromatic number max vertex degree}, we have
	\begin{displaymath}
		\chromaticnumber{\commgraph{\Reeszero{G}{I}{\Lambda}{P}}}\leqslant \max\set{2,2\cdot 2}= 4\abs{G}.
	\end{displaymath}
	
	We have that $P'$ has $4$ zero entries. Hence Theorem~\ref{0Rees: chromatic number edges} implies that
	\begin{displaymath}
		\chromaticnumber{\commgraph{\Reeszero{G}{I}{\Lambda}{P'}}}\leqslant \max\set{2,4}= 4\abs{G}.
	\end{displaymath}
	In addition, the maximum number of zeros of a row (and of a column) of $\Psubmatrix{P'}{\set{1,2,3}}{\set{1,2}}$ is $3$ (respectively, $2$). We have that all the entries of $\Psubmatrix{P'}{\set{1,2,3}}{\set{1,2}}$ which are in a row with three zeros are zero entries. Consequently, Theorem~\ref{0Rees: chromatic number max vertex degree} guarantees that
	\begin{displaymath}
		\chromaticnumber{\commgraph{\Reeszero{G}{I}{\Lambda}{P'}}}\leqslant \max\set{2,3\cdot 2-1}= 5\abs{G}.
	\end{displaymath}
	
	We can see that, in the case of $\chromaticnumber{\commgraph{\Reeszero{G}{I}{\Lambda}{P}}}$, Theorem~\ref{0Rees: chromatic number max vertex degree} gives the smaller upper bound, and in the case of $\chromaticnumber{\commgraph{\Reeszero{G}{I}{\Lambda}{P'}}}$, Theorem~\ref{0Rees: chromatic number edges} gives the smaller upper bound.
\end{example}

\section{Knit degree of a 0-Rees matrix semigroup over a group}\label{sec: knit degree 0Rees}

Let $G$ be a finite group, let $I$ and $\Lambda$ be finite index sets and let $P$ be a regular $\Lambda \times I$ matrix whose entries are elements of $G^0$. Throughout this section we are going to assume that $P$ contains at least one zero entry.

The aim of this section is to ascertain under which conditions $\commgraph{\Reeszero{G}{I}{\Lambda}{P}}$ contains left paths and, in that case, determine the length of the shortest left paths.

\begin{theorem}\label{0Rees: knit degree}
	$\commgraph{\Reeszero{G}{I}{\Lambda}{P}}$ contains left paths if and only if at least one of the following conditions is satisfied:
	\begin{enumerate}
		\item $\abs{G}>1$.
		\item $O_{1\times 2}$ is a \submatrix\ of $P$.
		\item $O_{2\times 1}$ is a \submatrix\ of $P$.
	\end{enumerate}
	Furthermore, if $\commgraph{\Reeszero{G}{I}{\Lambda}{P}}$ contains left paths, then $\knitdegree{\commgraph{\Reeszero{G}{I}{\Lambda}{P}}}=1$.
\end{theorem}

\begin{proof}
	\textbf{Part 1.} We are going to prove the forward implication. We do this through a proof by contrapositive. Suppose that $\abs{G}=1$ (that is, $G=\set{1_G}$) and that $O_{1\times 2}$ and $O_{2 \times 1}$ are not \submatrices\ of $P$. Let $i\in I$ and $\lambda\in\Lambda$. We want to see that $\parens{i,1_G,\lambda}$ is adjacent to at most one vertex. Let $i_1,i_2\in I$ and $\lambda_1,\lambda_2\in\Lambda$ be such that $\parens{i_1,1_G,\lambda_1}$ and $\parens{i_2,1_G,\lambda_2}$ are adjacent to $\parens{i,1_G,\lambda}$ and distinct from $\parens{i,1_G,\lambda}$. Then we have $\parens{i,1_G,\lambda}\parens{i_1,1_G,\lambda_1}=\parens{i_1,1_G,\lambda_1}\parens{i,1_G,\lambda}$ and $\parens{i,1_G,\lambda}\parens{i_2,1_G,\lambda_2}=\parens{i_2,1_G,\lambda_2}\parens{i,1_G,\lambda}$, and we have $\parens{i,\lambda}\neq\parens{i_1,\lambda_1}$ and $\parens{i,\lambda}\neq\parens{i_2,\lambda_2}$. It follows from Lemma~\ref{0Rees: commutativity} that $p_{\lambda i_1}=p_{\lambda_1 i}=p_{\lambda i_2}=p_{\lambda_2 i}=0$. Since $O_{1\times 2}$ is not a \submatrix\ of $P$ and $p_{\lambda i_1}=p_{\lambda i_2}=0$, then we must have $i_1=i_2$ and, since $O_{2\times 1}$ is not a \submatrix\ of $P$ and $p_{\lambda_1 i}=p_{\lambda_2 i}=0$, then $\lambda_1=\lambda_2$.
	
	We just proved that all the vertices of $\commgraph{\Reeszero{G}{I}{\Lambda}{P}}$ are adjacent to at most one other vertex. This implies that all non-trivial paths in $\commgraph{\Reeszero{G}{I}{\Lambda}{P}}$ have length $1$. In order to show that $\commgraph{\Reeszero{G}{I}{\Lambda}{P}}$ contains no left paths, we just need to verify that all paths in $\commgraph{\Reeszero{G}{I}{\Lambda}{P}}$ of length $1$ are not left paths. Let $\parens{i,1_G,\lambda}-\parens{j,1_G,\mu}$ be a path in $\commgraph{\Reeszero{G}{I}{\Lambda}{P}}$. We have $\parens{i,\lambda}\neq\parens{j,\mu}$, which implies (by Lemma~\ref{0Rees: commutativity}), that $p_{\lambda j}=p_{\mu i}=0$. Since $O_{1\times 2}$ is not a \submatrix\ of $P$ and $p_{\lambda j}=0$, then we must have $p_{\lambda i}\neq 0$. Therefore, $\parens{i,1_G,\lambda}\parens{i,1_G,\lambda}=\parens{i,1_Gp_{\lambda_i}1_G,\lambda}\neq 0=\parens{j,1_G,\mu}\parens{i,1_G,\lambda}$. Thus $\parens{i,1_G,\lambda}-\parens{j,1_G,\mu}$ is not a left path.
	
	\medskip
	
	\textbf{Part 2.} Now we prove the reverse implication. We consider three cases.
	
	\smallskip
	
	\textit{Case 1:} Assume that $\abs{G}>1$. Then there exist distinct $x,y\in G$. Since $P$ contains at least one zero entry, then there exist $i\in I$ and $\lambda\in\Lambda$ such that $p_{\lambda i}=0$. Consequently, we have $\parens{i,x,\lambda}\parens{i,y,\lambda}=\parens{i,y,\lambda}\parens{i,x,\lambda}$ (by Lemma~\ref{0Rees: commutativity}), which implies that $\parens{i,x,\lambda}-\parens{i,y,\lambda}$ is a path in $\commgraph{\Reeszero{G}{I}{\Lambda}{P}}$. Additionally, we have $\parens{i,x,\lambda}\parens{i,x,\lambda}=0=\parens{i,y,\lambda}\parens{i,x,\lambda}$ and $\parens{i,x,\lambda}\parens{i,y,\lambda}=0=\parens{i,y,\lambda}\parens{i,y,\lambda}$. Hence $\parens{i,x,\lambda}-\parens{i,y,\lambda}$ is a left path (of length $1$) and $\knitdegree{\commgraph{\Reeszero{G}{I}{\Lambda}{P}}}=1$.
	
	\smallskip
	
	\textit{Case 2:} Assume that $O_{1\times 2}$ is a \submatrix\ of $P$. It follows that there exist distinct $i,j\in I$ and $\lambda\in\Lambda$ such that $p_{\lambda i}=p_{\lambda j}=0$. Let $x\in G$. As a consequence of Lemma~\ref{0Rees: commutativity}, we have $\parens{i,x,\lambda}\parens{j,x,\lambda}=\parens{j,x,\lambda}\parens{i,x,\lambda}$. Hence $\parens{i,x,\lambda}-\parens{j,x,\lambda}$ is a path in $\commgraph{\Reeszero{G}{I}{\Lambda}{P}}$. We also have $\parens{i,x,\lambda}\parens{i,x,\lambda}=0=\parens{j,x,\lambda}\parens{i,x,\lambda}$ and $\parens{i,x,\lambda}\parens{j,x,\lambda}=0=\parens{j,x,\lambda}\parens{j,x,\lambda}$. Thus $\parens{i,x,\lambda}-\parens{j,x,\lambda}$ is a left path (of length 1) in $\commgraph{\Reeszero{G}{I}{\Lambda}{P}}$ and $\knitdegree{\commgraph{\Reeszero{G}{I}{\Lambda}{P}}}=1$.
	
	\smallskip
	
	\textit{Case 3:} Assume that $O_{2\times 1}$ is a \submatrix\ of $P$. We can prove in a similar way to case 2 that $\commgraph{\Reeszero{G}{I}{\Lambda}{P}}$ contains a left path of length $1$. Thus $\knitdegree{\commgraph{\Reeszero{G}{I}{\Lambda}{P}}}=1$.
\end{proof}

\section{Properties of commuting graphs of completely 0-simple semigroups}\label{sec: completely 0-simple smg}

This section is devoted to researching properties of commuting graphs of completely $0$-simple semigroups. The properties we consider are the diameter, clique number, girth, chromatic number and knit degree. The main idea is to give an answer to the following question: for each $n\in\mathbb{N}$ is there a (finite non-commutative) completely $0$-simple semigroup $S$ such that
\begin{enumerate}
	\item $\diam{\commgraph{S}}=n$?
	\item $\cliquenumber{\commgraph{S}}=n$?
	\item $\chromaticnumber{\commgraph{S}}=n$?
	\item $\girth{\commgraph{S}}=n$?
	\item $\knitdegree{S}=n$?
\end{enumerate}

\begin{example}\label{example diameter}
	Let $n\in \mathbb{N}$ be such that $n\geqslant 2$. Let $G$ be a group, let $I=\Lambda=\X{n+1}$ be index sets and let $P$ be a $\Lambda\times I$ matrix such that	
	\begin{displaymath}
		\timeszero{P}=\begin{bNiceMatrix}[first-row,first-col]
			& 1 & 2 & 3 & 4 & \cdots & n & n+1 \\
			1      & 0 & 0 & \times & \times & \cdots & \times & \times \\
			2      & \times & 0 & 0 & \times & \cdots & \times & \times \\
			3      & \times & \times & 0 & 0 & \cdots & \times & \times \\
			\vdots & \vdots & \vdots & \vdots & \vdots & \ddots & \vdots & \vdots\\
			n      & \times & \times & \times & \times & \cdots & 0 & 0 \\
			n+1    & 0 & \times & \times & \times & \cdots & \times & 0
		\end{bNiceMatrix}.
	\end{displaymath}
	That is, for all $i\in I=\X{n+1}$ and $\lambda\in\Lambda=\X{n+1}$ we have that $p_{\lambda i}=0$ if and only if $i=\lambda$, or $\lambda<n+1$ and $i=\lambda+1$, or $\lambda=n+1$ and $i=1$.
	
	Since $n\geqslant 2$ and each row and each column contains $n+1$ entries, two of which are zeros, then this means that each row and each column contains at least one non-zero entry, which implies that $P$ is regular. Hence $\Reeszero{G}{I}{\Lambda}{P}$ is a completely $0$-simple semigroup. Furthermore, $\Reeszero{G}{I}{\Lambda}{P}$ is finite and non-commutative (by Proposition~\ref{0Rees: center}).
	
	We begin by verifying that $\commgraph{\Reeszero{G}{I}{\Lambda}{P}}$ is connected. In order to do this we are going to rearrange the rows and columns of $P$ in a more convenient way, and then we are going to choose the top-leftmost entry of the new matrix to start the $0$-closure method --- that entry will correspond to the $\parens{1,1}$-th entry of $P$. (We observe that rearranging rows and columns of a matrix does not change the sequence of matrices we obtain from the $0$-closure method --- or its length.)
	
	\smallskip
	
	\textit{Case 1:} Assume that $n$ is odd. We consider the following order for the rows of $P$
	\begin{center}
		\begin{tikzpicture}[baseline=-5mm]
			\node (1) at (0,0) {$1$};
			\node (2) at (1,0) {$2$};
			\node (3) at (2,0) {$3$};
			\node (4) at (3,0) {$4$};
			\node (frac-1) at (5,0) {$\smash{\frac{n+1}{2}}-1$};
			\node (frac) at (6.8,0) {$\mathstrut\smash{\frac{n+1}{2}}$};
			
			\node (n+1) at (0,-1) {$n+1$};
			\node (n) at (1,-1) {$n$};
			\node (n-1) at (2,-1) {$n-1$};
			\node (n-2) at (3,-1) {$n-2$};
			\node (frac+2) at (5,-1) {$\smash{\frac{n+1}{2}}+2$};
			\node (frac+1) at (6.8,-1) {$\smash{\frac{n+1}{2}}+1$};

			\begin{scope}[->,thick]
				\draw (1) -- (n+1);
				\draw (n+1) -- (2);
				\draw (2) -- (n);
				\draw (n) -- (3);
				\draw (3) -- (n-1);
				\draw (n-1) -- (4);
				\draw (4) -- (n-2);
				\draw (frac-1) -- (frac+2);
				\draw (frac+2) -- (frac);
				\draw (frac) -- (frac+1);
			\end{scope}
			
			\begin{scope}[dash pattern=on 0pt off 4pt,line cap=round, very thick]
				\draw (3.5,0) -- (4.2,0);
				\draw (3.8,-1) -- (4.2,-1);
			\end{scope}
		\end{tikzpicture}\;,
	\end{center}
	and the following order for the columns of $P$
	\begin{center}
		\begin{tikzpicture}[baseline=-5mm]
			\node (1) at (0,0) {$1$};
			\node (2) at (1,0) {$2$};
			\node (3) at (2,0) {$3$};
			\node (4) at (3,0) {$4$};
			\node (frac-1) at (5,0) {$\smash{\frac{n+1}{2}}-1$};
			\node (frac) at (6.8,0) {$\mathstrut\smash{\frac{n+1}{2}}$};
			
			\node (n+1) at (1,-1) {$n+1$};
			\node (n) at (2,-1) {$n$};
			\node (n-1) at (3,-1) {$n-1$};
			\node (frac+3) at (5,-1) {$\smash{\frac{n+1}{2}}+3$};
			\node (frac+2) at (6.8,-1) {$\smash{\frac{n+1}{2}}+2$};
			\node (frac+1) at (8.8,-1) {$\smash{\frac{n+1}{2}}+1$};

			\begin{scope}[->,thick]
				\draw (1) -- (2);
				\draw (2) -- (n+1);
				\draw (n+1) -- (3);
				\draw (3) -- (n);
				\draw (n) -- (4);
				\draw (4) -- (n-1);
				\draw (frac-1) -- (frac+3);
				\draw (frac+3) -- (frac);
				\draw (frac) -- (frac+2);
				\draw (frac+2) -- (frac+1);
			\end{scope}
			
			\begin{scope}[dash pattern=on 0pt off 4pt,line cap=round, very thick]
				\draw (3.5,0) -- (4.2,0);
				\draw (3.8,-1) -- (4.2,-1);
			\end{scope}
		\end{tikzpicture}\;.
	\end{center}
	Then the matrix $P'$ we obtain by reordering the rows and columns of $P$ (in the way described above) is such that
	\begin{displaymath}
		\timeszero{P'}=\begin{bNiceMatrix}[first-row,first-col]
			& 1 & 2 & n+1 & 3 & \cdots & \frac{n+1}{2} & \frac{n+1}{2}+2 & \frac{n+1}{2}+1 \\
			1        & 0 & 0 & \times & \times & \cdots & \times & \times & \times \\
			n+1      & 0 & \times & 0 & \times & \cdots & \times & \times & \times \\
			2        & \times & 0 & \times & 0 & \cdots & \times & \times & \times \\
			n        & \times & \times & 0 & \times & \cdots & \times & \times & \times \\
			\vdots   & \vdots & \vdots & \vdots & \vdots & \ddots & \vdots & \vdots & \vdots\\
			\frac{n+1}{2}+2       & \times & \times & \times & \times & \cdots & \times & 0 & \times \\
			\frac{n+1}{2}         & \times & \times & \times & \times & \cdots & 0 & \times & 0 \\
			\frac{n+1}{2}+1       & \times & \times & \times & \times & \cdots & \times & 0 & 0 \\
		\end{bNiceMatrix}.
	\end{displaymath}
	
	\smallskip
	
	\textit{Case 2:} Assume that $n$ is even. If we reorder the rows of $P$ in the following way
	\begin{center}
		\begin{tikzpicture}[baseline=-5mm]
			\node (1) at (0,0) {$1$};
			\node (2) at (1,0) {$2$};
			\node (3) at (2,0) {$3$};
			\node (4) at (3,0) {$4$};
			\node (frac-1) at (5,0) {$\smash{\frac{n}{2}}-1$};
			\node (frac) at (6.2,0) {$\mathstrut\smash{\frac{n}{2}}$};
			\node (frac+1) at (7.4,0) {$\mathstrut\smash{\frac{n}{2}}+1$};
			
			\node (n+1) at (0,-1) {$n+1$};
			\node (n) at (1,-1) {$n$};
			\node (n-1) at (2,-1) {$n-1$};
			\node (n-2) at (3,-1) {$n-2$};
			\node (frac+3) at (5,-1) {$\smash{\frac{n}{2}}+3$};
			\node (frac+2) at (6.2,-1) {$\smash{\frac{n}{2}}+2$};

			\begin{scope}[->,thick]
				\draw (1) -- (n+1);
				\draw (n+1) -- (2);
				\draw (2) -- (n);
				\draw (n) -- (3);
				\draw (3) -- (n-1);
				\draw (n-1) -- (4);
				\draw (4) -- (n-2);
				\draw (frac-1) -- (frac+3);
				\draw (frac+3) -- (frac);
				\draw (frac) -- (frac+2);
				\draw (frac+2) -- (frac+1);
			\end{scope}
			
			\begin{scope}[dash pattern=on 0pt off 4pt,line cap=round, very thick]
				\draw (3.5,0) -- (4.33,0);
				\draw (3.8,-1) -- (4.3,-1);
			\end{scope}
		\end{tikzpicture}\;,
	\end{center}
	and reorder the columns of $P$ in the following way
	\begin{center}
		\begin{tikzpicture}[baseline=-5mm]
			\node (1) at (0,0) {$1$};
			\node (2) at (1,0) {$2$};
			\node (3) at (2,0) {$3$};
			\node (4) at (3,0) {$4$};
			\node (frac-1) at (5,0) {$\smash{\frac{n}{2}}-1$};
			\node (frac) at (6.2,0) {$\mathstrut\smash{\frac{n}{2}}$};
			\node (frac+1) at (7.4,0) {$\mathstrut\smash{\frac{n}{2}}+1$};
			
			\node (n+1) at (1,-1) {$n+1$};
			\node (n) at (2,-1) {$n$};
			\node (n-1) at (3,-1) {$n-1$};
			\node (frac+4) at (5,-1) {$\smash{\frac{n}{2}}+4$};
			\node (frac+3) at (6.2,-1) {$\smash{\frac{n}{2}}+3$};
			\node (frac+2) at (7.4,-1) {$\smash{\frac{n}{2}}+2$};

			\begin{scope}[->,thick]
				\draw (1) -- (2);
				\draw (2) -- (n+1);
				\draw (n+1) -- (3);
				\draw (3) -- (n);
				\draw (n) -- (4);
				\draw (4) -- (n-1);
				\draw (frac-1) -- (frac+4);
				\draw (frac+4) -- (frac);
				\draw (frac) -- (frac+3);
				\draw (frac+3) -- (frac+1);
				\draw (frac+1) -- (frac+2);
			\end{scope}
			
			\begin{scope}[dash pattern=on 0pt off 4pt,line cap=round, very thick]
				\draw (3.5,0) -- (4.33,0);
				\draw (3.8,-1) -- (4.3,-1);
			\end{scope}
		\end{tikzpicture}\; ,
	\end{center}
	then the matrix $P'$ we obtain is such that
	\begin{displaymath}
		\timeszero{P'}=\begin{bNiceMatrix}[first-row,first-col]
			& 1 & 2 & n+1 & 3 & \cdots & \frac{n}{2}+3 & \frac{n}{2}+1 & \frac{n}{2}+2 \\
			1        & 0 & 0 & \times & \times & \cdots & \times & \times & \times \\
			n+1      & 0 & \times & 0 & \times & \cdots & \times & \times & \times \\
			2        & \times & 0 & \times & 0 & \cdots & \times & \times & \times \\
			n        & \times & \times & 0 & \times & \cdots & \times & \times & \times \\
			\vdots   & \vdots & \vdots & \vdots & \vdots & \ddots & \vdots & \vdots & \vdots\\
			\frac{n}{2}         & \times & \times & \times & \times & \cdots & \times & 0 & \times \\
			\frac{n}{2}+2       & \times & \times & \times & \times & \cdots & 0 & \times & 0 \\
			\frac{n}{2}+1       & \times & \times & \times & \times & \cdots & \times & 0 & 0 \\
		\end{bNiceMatrix}.
	\end{displaymath}
	
	\smallskip
	
	We are going to choose the top-leftmost entry of the new matrix (obtained in cases 1 and 2) and start the $0$-closure method. Below we illustrate how the method works when we start with the chosen entry: the entries in red, yellow, green and blue are selected at steps $0$, $1$, $2$ and $3$, respectively, of the $0$-closure method.
	\begin{displaymath}
		\begin{bNiceMatrix}
			\cellcolor{Red} 0 & \cellcolor{Goldenrod} 0 & \cellcolor{Green} \times & \cellcolor{cyan!40} \times & \cdots & \times & \times & \times \\
			\cellcolor{Goldenrod} 0 & \cellcolor{Goldenrod} \times & \cellcolor{Green} 0 & \cellcolor{cyan!40} \times & \cdots & \times & \times & \times \\
			\cellcolor{Green} \times & \cellcolor{Green} 0 & \cellcolor{Green} \times & \cellcolor{cyan!40} 0 & \cdots & \times & \times & \times \\
			\cellcolor{cyan!40} \times & \cellcolor{cyan!40} \times & \cellcolor{cyan!40} 0 & \cellcolor{cyan!40} \times & \cdots & \times & \times & \times \\
			\vdots & \vdots & \vdots & \vdots & \ddots & \vdots & \vdots & \vdots\\
			\times & \times & \times & \times & \cdots & \times & 0 & \times \\
			\times & \times & \times & \times & \cdots & 0 & \times & 0 \\
			\times & \times & \times & \times & \cdots & \times & 0 & 0 \\
		\end{bNiceMatrix}
	\end{displaymath}
	
	Due to the fact that this matrix has $n+1$ rows and $n+1$ rows, it is easy to understand that it will take $n$ steps to finish the $0$-closure method, that is, $\zeroindex{1}{1}=n$. The matrix we obtain at the end of step $n$ will be the entire matrix, which implies that $P$ is a \zeroclosuresub\ and, consequently, $\commgraph{\Reeszero{G}{I}{\Lambda}{P}}$ is connected (by Theorem~\ref{0Rees: connectedness}). Furthermore, $\commgraph{\Reeszero{G}{I}{\Lambda}{P}}$ is a connected component determined by $P$ and, by Theorem~\ref{0Rees: diameter}, we have that $\diam{\commgraph{\Reeszero{G}{I}{\Lambda}{P}}}\geqslant\zeroindex{1}{1}=n$.
	
	In order to see that $\diam{\commgraph{\Reeszero{G}{I}{\Lambda}{P}}}=n$ we are going to see that for all $i\in I=\X{n+1}$ and $\lambda\in\Lambda=\X{n+1}$ such that $p_{\lambda i}=0$ we have $\zeroindex{\lambda}{i}=\zeroindex{1}{1}=n$. We prove this by showing that for all $i\in I=\X{n+1}$ and $\lambda\in\Lambda=\X{n+1}$ such that $p_{\lambda i}=0$ we can reorder the rows and columns of $P$ and obtain a new matrix $P_{\lambda i}$ such that:
	\begin{itemize}
		\item The top-leftmost entry of $P_{\lambda i}$ corresponds to the $\parens{\lambda, i}$-th entry of $P$.
		\item $\timeszero{P_{\lambda i}}=\timeszero{P}$.
	\end{itemize}
	This shows that, for all $i\in I=\X{n+1}$ and $\lambda\in\Lambda=\X{n+1}$ such that $p_{\lambda i}=0$, starting the $0$-clocure method with the $\parens{\lambda, i}$-th entry of $P$ is equivalent to starting the $0$-closure method with the $\parens{1,1}$-th entry of $P$ and, consequently, $\zeroindex{\lambda}{i}=\zeroindex{1}{1}=n$ for all $i\in I=\X{n+1}$ and $\lambda\in\Lambda=\X{n+1}$ such that $p_{\lambda i}=0$. Thus Theorem~\ref{0Rees: diameter} implies that
	\begin{displaymath}
		\diam{\commgraph{\Reeszero{G}{I}{\Lambda}{P}}}=\max\gset{\zeroindex{i}{\lambda}}{i\in I, \text{ } \lambda\in\Lambda \text{ and } p_{\lambda i}=0 }=n.
	\end{displaymath}

	We are going to see that it is possible to achieve the reordering we described. Let $i\in I=\Xn$ and $\lambda\in\Lambda=\Xn$ be such that $p_{\lambda i}=0$. Then $i=\lambda$, or $\lambda<n+1$ and $i=\lambda+1$, or $\lambda=n+1$ and $i=1$.
	
	\smallskip
	
	\textit{Case 1:} Assume that $i=\lambda$. We have that
	\begin{align*}
		\timeszero{P_{\lambda i}} & = \Psubmatrix{\timeszero{P}}{i,\ldots,n+1,1,\ldots,i-1}{\lambda,\ldots,n+1,1,\ldots,\lambda-1}\\
		& = \Psubmatrix{\timeszero{P}}{\lambda,\ldots,n+1,1,\ldots,\lambda-1}{\lambda,\ldots,n+1,1,\ldots,\lambda-1}\\
		& = \begin{bNiceMatrix}[first-row,first-col]
			& \lambda & \lambda+1 & \lambda+2 & \cdots & n+1 & 1 & 2 & \cdots & \lambda-2 & \lambda-1 \\
			\lambda        & 0 & 0 & \times & \cdots & \times & \times & \times & \cdots & \times & \times \\
			\lambda+1      & \times & 0 & 0 & \cdots & \times & \times & \times & \cdots & \times & \times \\
			\vdots   & \vdots & \vdots & \vdots & \ddots & \vdots & \vdots & \vdots & \ddots & \vdots & \vdots \\
			n+1      & \times & \times & \times & \cdots & 0 & 0 & \times & \cdots & \times & \times \\
			1        & \times & \times & \times & \cdots & \times & 0 & 0 & \cdots & \times & \times \\
			\vdots   & \vdots & \vdots & \vdots & \ddots & \vdots & \vdots & \vdots & \ddots & \vdots & \vdots \\
			\lambda-2      & \times & \times & \times & \cdots & \times & \times & \times & \cdots & 0 & 0\\
			\lambda-1      & 0 & \times & \times & \cdots & \times & \times & \times & \cdots & \times & 0
		\end{bNiceMatrix}\\
		& = \timeszero{P}.
	\end{align*}
	Moreover, entry $\parens{\lambda, \lambda}=\parens{\lambda,i}$ of $P$ is the top-leftmost entry of $P_{\lambda i}$.
	
	\smallskip
	
	\textit{Case 2:} Assume that $\lambda<n+1$ and $i=\lambda+1$. We have that
	\begin{align*}
		\timeszero{P_{\lambda i}} &= \Psubmatrix{\timeszero{P}}{i,i-1\ldots,1,n+1,n,\ldots,i+1}{\lambda,\lambda-1,\ldots,1,n+1,n,\ldots,\lambda+1}\\
		& = \Psubmatrix{\timeszero{P}}{\lambda+1,\lambda,\ldots,1,n+1,n,\ldots,\lambda+2}{\lambda,\lambda-1,\ldots,1,n+1,n,\ldots,\lambda+1}\\
		& = \begin{bNiceMatrix}[first-row,first-col]
			& \lambda+1 & \lambda & \lambda-1 & \cdots & 2 & 1 & n+1 & n & \cdots & \lambda+3 & \lambda+2 \\
			\lambda   & 0 & 0 & \times & \cdots & \times & \times & \times & \times & \cdots & \times & \times \\
			\lambda-1 & \times & 0 & 0 & \cdots & \times & \times & \times & \times & \cdots & \times & \times \\
			\vdots    & \vdots & \vdots & \vdots & \ddots & \vdots & \vdots & \vdots & \vdots & \ddots & \vdots & \vdots \\
			1         & \times & \times & \times & \cdots & 0 & 0 & \times & \times & \cdots & \times & \times \\
			n+1       & \times & \times & \times & \cdots & \times & 0 & 0 & \times & \cdots & \times & \times \\
			n         & \times & \times & \times & \cdots & \times & \times & 0 & 0 & \cdots & \times & \times\\
			\vdots    & \vdots & \vdots & \vdots & \ddots & \vdots & \vdots & \vdots & \ddots & \vdots & \vdots & \vdots \\
			\lambda+2 & \times & \times & \times & \cdots & \times & \times & \times & \times & \cdots & 0 & 0\\
			\lambda+1 & 0 & \times & \times & \cdots & \times & \times & \times & \times & \cdots & \times & 0
		\end{bNiceMatrix}\\
		& = \timeszero{P}.
	\end{align*}
	We can also see that entry $\parens{\lambda,\lambda+1}=\parens{\lambda,i}$ of $P$ is the top-leftmost entry of $P_{\lambda i}$.
	
	\smallskip
	
	\textit{Case 3:} Assume that $\lambda=n+1$ and $i=1$. We have that
	\begin{displaymath}
		\timeszero{P_{\lambda i}} = \Psubmatrix{\timeszero{P}}{1,n+1,n\ldots,2}{n+1,n,\ldots,1} = \begin{bNiceMatrix}[first-row,first-col]
			& 1 & n+1 & n & \cdots & 3 & 2\\
			n+1   & 0 & 0 & \times & \cdots & \times & \times\\
			n & \times & 0 & 0 & \cdots & \times & \times\\
			\vdots    & \vdots & \vdots & \vdots & \ddots & \vdots & \vdots\\
			2         & \times & \times & \times & \cdots & 0 & 0\\
			1       & 0 & \times & \times & \cdots & \times & 0
		\end{bNiceMatrix} = \timeszero{P}
	\end{displaymath}
	and it is straightforward to verify that entry $\parens{n+1,1}=\parens{\lambda,i}$ of $P$ is the top-leftmost entry of $P_{\lambda i}$.
\end{example}

Example~\ref{example diameter} proves the following result.

\begin{corollary}\label{0Rees: completely 0-simple diameter}
	For each $n\in\mathbb{N}$ such that $n\geqslant 2$, there is a (finite non-commutative) completely $0$-simple semigroup whose commuting graph has diameter equal to $n$.
\end{corollary}

We observe that there is also a (finite non-commutative) completely $0$-simple semigroup whose commuting graph has diameter equal to infinity --- see Example~\ref{0Rees: example connectedness}. Furthermore, there are no (finite non-commutative) completely $0$-simple semigroups whose commuting graph has diameter equal to $1$ (because the diameter of any connected commuting graph of a semigroup is at least $2$).

\begin{example}\label{example clique chromatic numbers}
	Let $n\in\mathbb{N}$. Let $G$ be a trivial group, let $I=\X{n+1}$ and $\Lambda=\set{1,2}$, and let $P$ be any $\Lambda\times I$ matrix such that $p_{11}=p_{12}=\cdots=p_{1n}=0$ and $p_{1\parens{n+1}},p_{21},p_{22},\ldots,p_{2\parens{n+1}}\in G$, that is, such that
	\begin{displaymath}
		\timeszero{P}=\begin{bNiceMatrix}[first-row,first-col]
			& 1 & 2 & \cdots & n & n+1 \\
			1      & 0 & 0 & \cdots & 0 & \times \\
			2      & \times & \times & \cdots & \times & \times
		\end{bNiceMatrix}.
	\end{displaymath}
	It is straightforward to see that $P$ is a regular matrix. Hence $\Reeszero{G}{I}{\Lambda}{P}$ is a completely $0$-simple semigroup. In addition, $\Reeszero{G}{I}{\Lambda}{P}$ is finite (because $G$, $I$ and $\Lambda$ are finite) and Lemma~\ref{0Rees: center} ensures that $\Reeszero{G}{I}{\Lambda}{P}$ is non-commutative.
	
	It is easy to see that $D_3$ and $D_2$ are not \submatrices\ of $P$, and that $O_{\abs{\set{1}}\times\abs{\Xn}}=O_{1\times n}$ is the largest \submatrix\ of zeros of $\timeszero{P}$. Consequently, Theorem~\ref{0Rees: clique number} implies that
	\begin{align*}
		\MoveEqLeft \cliquenumber{\commgraph{\Rees{G}{I}{\Lambda}{P}}}\\
		&=\abs{G}\cdot\max\{km: O_{k\times m} \text{ is a \submatrix\ of } \timeszero{P}\}\kern -4mm\\
		&=\max\{km: O_{k\times m} \text{ is a \submatrix\ of } \timeszero{P}\}& \bracks{\text{because } G \text{ is trivial}}\\
		&=n.&&\qedhere
	\end{align*}
	
	Furthermore, it is straightforward to see that the unique \zeroclosuresub\ is $\Psubmatrix{P}{1,\ldots,n}{1}=O_{1\times n}$, which contains $n$ zero entries. We are going to see that $\chromaticnumber{\commgraph{\Reeszero{G}{I}{\Lambda}{P}}}\leqslant n$.
	
	\smallskip
	
	\textit{Case 1:} Assume that $n\geqslant 2$. Hence Theorem~\ref{0Rees: chromatic number edges} implies that
	\begin{align*}
		\chromaticnumber{\commgraph{\Reeszero{G}{I}{\Lambda}{P}}}&\leqslant\abs{G}\cdot\max\set{2,n}\\
		&=\max\set{2,n}& \bracks{\text{because } G \text{ is trivial}}\\
		&=n.& \bracks{\text{because } n\geqslant 2}
	\end{align*}
	
	\smallskip
	
	\textit{Case 2:} Assume that $n=1$. Due to the fact that $P$ contains a row with no zero entries, Theorem~\ref{0Rees: connectedness} guarantees that $\commgraph{\Reeszero{G}{I}{\Lambda}{P}}$ is not connected. Moreover, $P$ contains only one \zeroclosuresub\ and, consequently, the remark we made after Theorem~\ref{0Rees: chromatic number edges} ensures that $\chromaticnumber{\commgraph{\Reeszero{G}{I}{\Lambda}{P}}}\leqslant n\abs{G}=n$.
	
	\smallskip
	
	The previous two cases imply that $\chromaticnumber{\commgraph{\Reeszero{G}{I}{\Lambda}{P}}}\leqslant n$. Since $\cliquenumber{\commgraph{\Reeszero{G}{I}{\Lambda}{P}}}=n$, then we also have $\chromaticnumber{\commgraph{\Reeszero{G}{I}{\Lambda}{P}}}\geqslant n$ (because the clique number provides a lower bound for the chromatic number). Therefore $\chromaticnumber{\commgraph{\Reeszero{G}{I}{\Lambda}{P}}}=n$.
\end{example}

Example~\ref{example clique chromatic numbers} proves the following two corollaries.

\begin{corollary}\label{0Rees: completely 0-simple clique number}
	For each $n\in\mathbb{N}$, there is a (finite non-commutative) completely $0$-simple semigroup whose commuting graph has clique number equal to $n$.
\end{corollary}

\begin{corollary}\label{0Rees: completely 0-simple chromatic number}
	For each $n\in\mathbb{N}$, there is a (finite non-commutative) completely $0$-simple semigroup whose commuting graph has chromatic number equal to $n$.
\end{corollary}

\begin{corollary}\label{0Rees: completely 0-simple girth}
	Let $S$ be a (finite non-commutative) completely $0$-simple semigroup. If $\commgraph{S}$ contains cycles, then $\girth{\commgraph{S}}\in\set{3,4}$.
\end{corollary}

In order to prove this corollary, we require the following theorem:

\begin{theorem}{\cite[Theorem 4.3]{Completely_simple_semigroups_paper}}\label{girth completely simple semigroup}
	Let $S$ be a finite non-commutative completely simple semigroup. If $\commgraph{S}$ contains at least one cycle, then $\girth{\commgraph{S}}=3$.
\end{theorem}

\begin{proof}
	Since $S$ is a completely $0$-simple semigroup, there exist a group $G$, index sets $I$ and $\Lambda$, and a regular $\Lambda\times I$ matrix $P$ with entries belonging to $G^0$ such that $S\simeq\Reeszero{G}{I}{\Lambda}{P}$. We observe that, since $S$ is finite, then so are $G$, $I$ and $\Lambda$. Moreover, we have that $\commgraph{S}$ is isomorphic to $\commgraph{\Reeszero{G}{I}{\Lambda}{P}}$. We consider two cases.
	
	\smallskip
	
	\textit{Case 1:} Assume that all the entries of $P$ belong to $G$ (that is, $P$ has no zero entries). This implies that $\Reeszero{G}{I}{\Lambda}{P}\simeq\parens{\Rees{G}{I}{\Lambda}{P}}^0$. Then we have that $\commgraph{\Reeszero{G}{I}{\Lambda}{P}}$ is isomorphic to $\commgraph{\parens{\Rees{G}{I}{\Lambda}{P}}^0}$, which is isomorphic to $\commgraph{\Rees{G}{I}{\Lambda}{P}}$ (because $0$ is a central element and, consequently, not a vertex of the commuting graph). Due to the fact that $\commgraph{S}$ contains cycles, then $\commgraph{\Rees{G}{I}{\Lambda}{P}}$ also contains cycles. Furthermore, $\Rees{G}{I}{\Lambda}{P}$ is a completely simple semigroup and, by Theorem~\ref{girth completely simple semigroup}, this implies that $\girth{S}=\girth{\commgraph{\Reeszero{G}{I}{\Lambda}{P}}}=\girth{\commgraph{\parens{\Rees{G}{I}{\Lambda}{P}}^0}}=\girth{\commgraph{\Rees{G}{I}{\Lambda}{P}}}=3$.
	
	\smallskip
	
	\textit{Case 2:} Assume that $P$ contains a zero entry. Then Theorem~\ref{0Rees: girth} immediately implies that $\girth{S}=\girth{\commgraph{\Reeszero{G}{I}{\Lambda}{P}}}\in\set{3,4}$.  
\end{proof}

We observe that it is possible to find a (finite non-commutative) completely $0$-simple semigroup whose commuting graph has no cycles. For instance, if we consider the cyclic group $C_2$ of order $2$, and if $I=\Lambda=\set{1,2}$ and $P$ is a $\Lambda\times I$ matrix such that
\begin{displaymath}
	\timeszero{P}=\begin{bNiceMatrix}[first-row, first-col]
		& 1 & 2 \\
		1 & 0 & \times \\
		2 & \times & \times
	\end{bNiceMatrix},
\end{displaymath}
then $\Reeszero{C_2}{I}{\Lambda}{P}$ is a completely $0$-simple semigroup and Theorem~\ref{0Rees: girth} implies that $\commgraph{\Reeszero{C_2}{I}{\Lambda}{P}}$ has no cycles (because $P$ contains only one zero entry).

\begin{corollary}\label{0Rees: completely 0-simple knit degree}
	Let $S$ be a (finite non-commutative) completely $0$-simple semigroup. If $\commgraph{S}$ contains left paths, then $\knitdegree{S}=1$.
\end{corollary}

The following result is required to prove Corollary~\ref{0Rees: completely 0-simple knit degree}:

\begin{corollary}{\cite[Corollary 4.5]{Completely_simple_semigroups_paper}}\label{knit degree completely simple semigroup}
	Let $S$ be a finite non-commutative completely simple semigroup. Then $\commgraph{S}$ has no left paths.
\end{corollary}

\begin{proof}
	It follows from the fact that $S$ is a completely $0$-simple semigroup that there exist a group $G$, index sets $I$ and $\Lambda$, and a regular $\Lambda\times I$ matrix $P$ whose entries belong to $G^0$ such that $S\simeq\Reeszero{G}{I}{\Lambda}{P}$. Hence $\commgraph{S}$ is isomorphic to $\commgraph{\Reeszero{G}{I}{\Lambda}{P}}$. Moreover, we have that $G$, $I$ and $\Lambda$ are finite because $S$ is finite.
	
	Assume, with the aim of obtaining a contradiction, that all the entries of $P$ are elements of $G$. Then we have $\Reeszero{G}{I}{\Lambda}{P}\simeq\parens{\Rees{G}{I}{\Lambda}{P}}^0$ and, consequently, $\commgraph{\Reeszero{G}{I}{\Lambda}{P}}$ is isomorphic to $\commgraph{\parens{\Rees{G}{I}{\Lambda}{P}}^0}$. Since $0$ is not a vertex of $\commgraph{\parens{\Rees{G}{I}{\Lambda}{P}}^0}$ (because it is a central element), then we also have that $\commgraph{\parens{\Rees{G}{I}{\Lambda}{P}}^0}$ is isomorphic to $\commgraph{\Rees{G}{I}{\Lambda}{P}}$. As a consequence of Corollary~\ref{knit degree completely simple semigroup}, we have that $\commgraph{\Rees{G}{I}{\Lambda}{P}}$ contains no left paths. Hence $\commgraph{S}$ also has no left paths, which is a contradiction.
	
	Thus $P$ must contain at least one zero entry. Therefore, Theorem~\ref{0Rees: knit degree} ensures that $\knitdegree{S}=\knitdegree{\Reeszero{G}{I}{\Lambda}{P}}=1$.
\end{proof}

We observe that it also follows from Theorem~\ref{0Rees: knit degree} that there exist completely $0$-simple semigroups whose commuting graph has no left paths. For instance, if $G$ is a trivial group, and if $I=\Lambda=\set{1,2}$ and $P$ is a $\Lambda\times I$ matrix such that
\begin{displaymath}
	\timeszero{P}=\begin{bNiceMatrix}[first-row, first-col]
		& 1 & 2 \\
		1 & 0 & \times \\
		2 & \times & 0
	\end{bNiceMatrix},
\end{displaymath}
then $\Reeszero{G}{I}{\Lambda}{P}$ is a completely $0$-simple semigroup and $\commgraph{\Reeszero{G}{I}{\Lambda}{P}}$ has no left paths (because $\abs{G}=1$ and $O_{1 \times 2}$ and $O_{2\times 1}$ are not \submatrices\ of $P$).

\section{Open problems}\label{sec: open problems}

In this section we discuss some unanswered questions concerning commuting graphs of completely $0$-simple semigroups.

\begin{problem}
	Determine the chromatic number of the commuting graph of a $0$-Rees matrix semigroup over a group.
	
	In Theorems~\ref{0Rees: chromatic number edges} and \ref{0Rees: chromatic number max vertex degree} we obtained two upper bounds for the chromatic number of the commuting graph of a $0$-Rees matrix semigroup over a group. This leaves the question of what is the exact value of the chromatic number. In the paper \cite{Brandt_semigroups_1} this has already been answered for particular Brandt semigroups (these semigroups are $0$-Rees matrix semigroups over groups with $I=\Lambda$ and $P$ being a $I\times I$ matrix whose entries are all equal to $0$, except the ones in the diagonal --- these are all equal to $1_G$).
\end{problem}

\begin{problem}
	Describe the simple graphs that are isomorphic to the commuting graph of some completely $0$-simple semigroup.
\end{problem}

\begin{problem}
	Characterize the completely $0$-simple semigroups whose commuting graphs are isomorphic.
\end{problem}

    \bibliography{Bibliography} 
\bibliographystyle{alphaurl}

\end{document}